\documentclass[12pt,a4paper]{amsart}
\usepackage{amssymb,amsfonts,amsthm,amsmath,graphicx,xcolor,csquotes}
\usepackage[scr]{rsfso}
\usepackage{chngcntr}
\usepackage{enumitem}

\theoremstyle{plain}
\newtheorem{theorem}{Theorem}
\newtheorem{lemma}{Lemma}
\newtheorem{definition}{Definition}
\newtheorem{conjecture}{Conjecture}
\newtheorem{corollary}{Corollary}
\newtheorem{statement}{Statement}

\numberwithin{theorem}{section}
\numberwithin{equation}{section}
\numberwithin{statement}{section}
\numberwithin{lemma}{section}
\numberwithin{definition}{section}
\numberwithin{conjecture}{section}
\numberwithin{corollary}{section}
\begin{document}
\dedicatory{Dedicated to Professor Rodney J Baxter on his $83$rd birthday.}
\title[Vector partition identities]{Vector Partition Identities for $2$D, $3$D and $n$D Lattices}
\author{Geoffrey B Campbell}
\address{Mathematical Sciences Institute,
         The Australian National University,
         Canberra, ACT, 0200, Australia}

\email{Geoffrey.Campbell@anu.edu.au}

\thanks{Thanks to Professor Dr Henk Koppelaar, whose suggestions helped summarize, for this paper, some chapters of the author's draft book.}


\begin{abstract}
We prove identities generating higher dimensional vector partitions. We derive theorems for integer lattice points in the 2D first quadrant, then generalize the approach to find 3D and $n$-space lattice point vector region extensions. We also state combinatorial identities for Visible Point Vectors in 2D up to 5D and $n$D first hyperquadrant and hyperpyramid lattices. 2D and 3D theorems for vector partitions with binary components are also derived.
\end{abstract}

\maketitle

\section{Preview of results}

We show the following selected sample results from this paper. These examples are intended to inform the reader of a few notable points of content along the way.

\textbf{Preview 1.} The number of first quadrant 2D partitions of vectors into exactly three parts $\clubsuit_2(a,b)$, is for $|y|<1$, $|z|<1$, generated by
\begin{equation}  \nonumber
\sum_{a,b\geq0}^{\infty} \clubsuit_2(a,b) y^a z^b = \frac{1}{3!}
 \begin{vmatrix}
    \frac{1}{\left(1-y\right) \left(1-z\right)} & -1 & 0 \\
    \frac{1}{\left(1-y^2\right) \left(1-z^2\right)} & \frac{1}{\left(1-y\right) \left(1-z\right)} & -2 \\
    \frac{1}{\left(1-y^3\right) \left(1-z^3\right)} & \frac{1}{\left(1-y^2\right) \left(1-z^2\right)} & \frac{}{\left(1-y\right) \left(1-z\right)} \\
  \end{vmatrix}
  \end{equation}
 \begin{equation*}
 =  \frac{y^3 z^3 + y^2 z^2 + y^2 z + y z^2 + y z + 1}{(1-y)^3 (1+y) (y^2 + y + 1) (1-z)^3 (1+z) (z^2 + z + 1)}.
 \end{equation*}

\textbf{Preview 2.} If $|v|<1, |w|<1, |x|<1, |y|<1, |z|<1 $ and $q+r+s+t+u=1$, then
  \begin{equation}   \nonumber
    \prod_{\substack{ \gcd(a,b,c,d,e)=1 \\ a,b,c,d,e \geq 1}} \left( \frac{1}{1-v^a w^b x^c y^d z^e} \right)^{1/(a^q b^r c^s d^t e^u)}
    \end{equation}
  \begin{equation} \nonumber
  = \exp\left\{ \left( \sum_{g=1}^{\infty} \frac{v^g}{g^q}\right) \left( \sum_{h=1}^{\infty} \frac{w^h}{h^r}\right) \left( \sum_{i=1}^{\infty} \frac{x^i}{i^s}\right)
    \left( \sum_{j=1}^{\infty} \frac{y^j}{j^t}\right) \left( \sum_{k=1}^{\infty} \frac{z^k}{k^u}\right)\right\}.
  \end{equation}

\textbf{Preview 3.} For each of $|y|, |z|<1,$
    \begin{equation} \nonumber
    \prod_{\substack{(a,b,c,d,e)=1 \\ a,b,c,d<e \\ a,b,c,d\geq0,e>0}} \left( \frac{1}{1-y^{a+b+c+d} z^e} \right)^{1/e}
       =   \left(\frac{(1-yz)^4 (1-y^3z)^4}{(1-z)(1-y^2z)^6 (1-y^4z)}\right)^{1/(1-y)^4}
   \end{equation}
    \begin{equation}  \nonumber
= 1 + \frac{z}{1!} + \begin{vmatrix}
    1 & -1 \\
    \left(\frac{1-y^2}{1-y}\right)^4 & 1 \\
  \end{vmatrix} \frac{z^2}{2!}
  + \begin{vmatrix}
    1 & -1 & 0 \\
    \left(\frac{1-y^2}{1-y}\right)^4 & 1 & -2 \\
    \left(\frac{1-y^3}{1-y}\right)^4 & \left(\frac{1-y^2}{1-y}\right)^4 & 1 \\
  \end{vmatrix} \frac{z^3}{3!}
      \end{equation}
  \begin{equation}  \nonumber
+ \begin{vmatrix}
    1 & -1 & 0 & 0 \\
    \left(\frac{1-y^2}{1-y}\right)^4 & 1 & -2 & 0 \\
    \left(\frac{1-y^3}{1-y}\right)^4 & \left(\frac{1-y^2}{1-y}\right)^4 & 1 & -3 \\
    \left(\frac{1-y^4}{1-y}\right)^4 & \left(\frac{1-y^3}{1-y}\right)^4 & \left(\frac{1-y^2}{1-y}\right)^4 & 1 \\
  \end{vmatrix} \frac{z^4}{4!}
+ etc.
 \end{equation}

\textbf{Preview 4.} For $|y|<1$, $|z|<1$,
 \begin{equation}   \nonumber
    \prod_{\substack{m,n \geq 1 \\ m \leq n ; \, (m,n)=1}} \left( \frac{1}{1- y^m z^n} \right)^{m/n^2}
    = \left(1-yz\right)^{\frac{y}{1-y}} \exp\left\{ \frac{y}{(1-y)^2} \left(Li_2(z)- Li_2(yz)\right) \right\}.
\end{equation}

\section{Defining Vector Partitions}

With respect to vector partitions, herein we study the additive decomposition of vectors $v = \langle a_1, a_2, \ldots , a_n \rangle$ where each of the components $a_1, a_2, \ldots , a_n$ are an integer, usually a non-negative integer, but this may vary by situation. A partition, whether it be for \textit{integer partitions} or it be for \textit{vector partitions}, is a partially ordered set or \textit{poset}. For an integer partition of $n$ say $\lambda_1+\lambda_2+\ldots+\lambda_r=n$ this means we impose that $\lambda_1 \geq \lambda_2 \geq \ldots \geq \lambda_r \geq 0$ for $\lambda$ as integers. For a vector partition of $v$ say $\lambda_1+\lambda_2+\ldots+\lambda_r=v$ this means we impose that $\| \lambda_1 \| \geq \| \lambda_2 \| \geq \ldots \geq \| \lambda_r \|$ for each $\lambda$ as an nD ($n$-dimensional, or $n$-component) vector.

\begin{definition}
  A partition of an $n$-dimensional vector $v = \langle a_1, a_2, \ldots , a_n \rangle$ is a finite partially ordered sequence of vectors $\lambda_1, \lambda_2, \ldots , \lambda_r$ such that $\sum_{i=1}^{r} \lambda_i = v$ under vector addition. The $\lambda_i$ vectors are called \textit{parts} of the vector partition. A suitable partial ordering for $v$ is for example, a non-increasing norm defined by $\| v \| =\sqrt{\sum_{i=1}^{n} a_i^2}$ applied to the $\lambda_i$.
\end{definition}

We will strive to obtain theorems and conjectures with regard to nD vectors generally with positive integer $n$, but mostly we shall begin with or attempt to illustrate or justify these theorems and conjectures by first considering the 2D and 3D situations, since they are most easily visualized in the Cartesian co-ordinate settings. We say the $p_n(\langle a_1, a_2, \ldots , a_n \rangle)$ abbreviated to $p_n(a_1, a_2, \ldots , a_n)$ or simply $p_n(v)$ is the number of nD partitions of vector $v$. Here are a few 2D examples taking $a,b >0$ so excluding the often included $a=0, \, b=0$.

$p_2(1,1) = 2; \quad \langle 1,1 \rangle = (\langle 1,1 \rangle^1);$

$p_2(2,1) = 1; \quad \langle 2,1 \rangle = (\langle 2,1 \rangle^1);$

$p_2(1,2) = 1; \quad \langle 1,2 \rangle = (\langle 1,2 \rangle^1);$

$p_2(2,2) = 2; \quad \langle 2,2 \rangle = (\langle 2,2 \rangle^1); \quad \langle 1,1 \rangle + \langle 1,1 \rangle = (\langle 1,1 \rangle^2);$

$p_2(3,2) = 2; \quad \langle 3,2 \rangle = (\langle 3,2 \rangle^1); \quad \langle 2,1 \rangle + \langle 1,1 \rangle
          = (\langle 2,1 \rangle^1 \langle 1,1 \rangle^1).$

\section{Vector or Multipartite Partitions from George E Andrews}

Andrews \cite[Chapter 12, page 203]{gA1976} defines $P(\textbf{n})= P(n_1,n_2,...,n_r)$ as the number of unrestricted partitions of the "$r$-partite" or "multipartite" number $(n_1,n_2,...,n_r)$ as an ordered $r$-tuple of non-negative integers not all zero. In our vector partitions theory we mostly would say that this multipartitie number is the vector $\mathbf{n}=\langle n_1,n_2,...,n_r \rangle$, defining a lattice point in Euclidean $r$-space. Similarly the $Q(\textbf{n})$ function is the number of partitions of vectors \textbf{n} into distinct parts. Andrews states the two generating functions
\begin{eqnarray}
\label{8.00a}  \sum_{n_1,n_2,...,n_r\geq0} P(\textbf{n})x_{1}^{n_1}...x_{r}^{n_r} &=& \prod_{\substack{n_1,n_2,...,n_r\geq0 \\ not\, all\, zero}} \frac{1}{1-x_{1}^{n_1}...x_{r}^{n_r}} \\
\label{8.00b}  \sum_{n_1,n_2,...,n_r\geq0} Q(\textbf{n})x_{1}^{n_1}...x_{r}^{n_r} &=& \prod_{n_1,n_2,...,n_r\geq0 } (1+x_{1}^{n_1}...x_{r}^{n_r})
\end{eqnarray}

and also gives the result due to Cheema and Motzkin \cite{mC1971} extending the well-known Euler theorem that the number of integer partitions into "odd parts" is same as the number of integer partitions into "distinct parts". In our work however we often replace $n_i\geq0$ by $n_i\geq1$ or even apply mixed versions of these.

In the theory resulting from our present paper we make use of the fact that equations (\ref{8.00a}) and (\ref{8.00b}) correspond to $n$D grids of their coefficients for $x_{1}^{n_1}...x_{r}^{n_r}$; which here are $P(n_1,n_2,...,n_r)$, respectively $Q(n_1,n_2,...,n_r)$. Although the lattice points of theses grids are $n$-space lattice configurations, we find that a straight line from the origin to the point $\langle n_1,n_2,...,n_r \rangle$ has an easily derived generating function for each point along the line.

The above equations (\ref{8.00a}) and (\ref{8.00b}) when represented by their 2D, 3D, 4D and 5D grids in their respective first quadrants and first hyperquadrants, obviously become unwieldy to visualize after the 2D extended rectangular lattice, and 3D extended cubic lattice. What human consciousness can say they see a 4D extended tesseract (4D cube) lattice, or a 5D extended hypercube lattice? However we understand this dimensional concept, there is a logical higher space extension also of the 2D first quadrant diagonal of lattice points upon the line $y=x$. In the 3D cube first hyperquadrant in an $x$-$y$-$z$ plane of lattice points there is the diagonal line represented by the equation $x=y=z$ and the coefficients of those lattice point vectors arise from setting the generating function $f(x,y,z)$ to say $f(z,z,z)$ which then is the generating function for the 3D diagonal lattice point vectors in that hyperquadrant. Similarly for the 4D case generating function $f(w,x,y,z)$ calculate $f(z,z,z,z)$ to give the 4D hyperdiagonal generating function for lattice points along the line with equation $w=x=y=z$ through the 4D extended tesseract (hypecube). And so on for the 5D function generated by $f(v,w,x,y,z)$ giving us the function $f(z,z,z,z,z)$ generating the coefficients (vector partition sums) for the lattice points along the 5D line $v=w=x=y=z$. This concept can be applied to any of our 2D, 3D, 4D, etc lattice point vector grids throughout our present volume. It applies to VPV identities as \textit{hyperquadrant} lattice functions, \textit{square hyperpyramid} lattice functions and \textit{skewed hyperpyramid} lattice point vector identities. This includes the many possibilities of applying \textit{polylogarithm} formulas and \textit{Parametric Euler sum} identities; examples given in our oncoming papers being just the starting place of a large number of possibilities.

\section{Hyperspace line generating functions for different nD slopes}  \label{nDlattice02}

In the previous section we saw how to calculate the hyperdiagonal generating function for an nD generated vector partition grid. Basically, we have the
\begin{statement}  \label{nDlattice03}
  If the nD generating function for the entire nD grid is the $n$ variable $f_n(z_1,z_2,z_3,...,z_n)$, the equation of the nD hyperdiagonal line from the origin is $z_1=z_2=z_3=...=z_n$ and so the generating function for the nD vector partitions along that line is given by the single variable function $f_n(z,z,,...,z)$.
\end{statement}
That is, even though an nD space grid is impossible for humans to envision, we can define a straight line through the hyperdiagonal from the nD origin point, and formulate an exact value of the vector partition function at any point along that nD hyperspace line.

Furthermore, taking an arbitrary lattice point vector $\langle a_1,a_2,a_3,...,a_n \rangle$ in the nD grid we can show that the vector partition function (or nD coefficient) for that lattice point is exactly evaluated as follows.

\begin{statement}  \label{nDlattice04}
  If the nD generating function for the entire nD grid is the $n$ variable $f_n(z_1,z_2,z_3,...,z_n)$, the equation of the nD \textit{hyper-radial-from-origin} line from the origin to the arbitrary point $\langle a_1,a_2,a_3,...,a_n \rangle$ is based on knowing values of
  \begin{equation} \label{nDlattice05}
    \langle a_1,a_2,a_3,...,a_n \rangle = r \langle c_1,c_2,c_3,...,c_n \rangle
  \end{equation}
  where $\langle c_1,c_2,c_3,...,c_n \rangle$ is a VPV with $\gcd( c_1,c_2,c_3,...,c_n ) =1$ and $r$ is the unique positive integer that makes this true.   So the straight line from the origin to the point $\langle a_1,a_2,a_3,...,a_n \rangle$ has the defining equation  $c_1z_1=c_2z_2=c_3z_3=...=c_nz_n$ and so the generating function for the nD vector partitions along that line is given by the single variable function $f_n(c_1z,c_2z,c_3z,...,c_nz)$, which has the coefficient of $z^r$ equal to the vector partition function (or nD coefficient) for that lattice point $\langle a_1,a_2,a_3,...,a_n \rangle$ exactly evaluated.
\end{statement}

We will give numerous examples in later papers where the Visible Point Vector (VPV) identities yield exact results for nD generating functions.

\bigskip

\section{Defining Vector Grids and resulting Partition Grids}

The set of first quadrant cartesian 2D lattice point vectors is a \text{countable set}. The lattice points radial from the origin in the radial region from $y=x$ to the positive $y$ axis can be replicated in the region between $y=x$ and the positive $x$ axis. Therefore, the set of all 2D lattice point vectors in the first quadrant are also countable. They are shown as follows.

\begin{definition} The first quadrant 2D vector grid, denoted as $V_2(y\in[1,\infty);z\in[1,\infty))$ is the set of vectors $\langle a,b\rangle$ such that a and b are positive integers. We say that a vector grid generally is a collection of vectors in a defined region. For example, in the following picture, \textbf{first quadrant 2D vector grid} is split into an $\overline{overlined Upper}$ vector grid, a \textbf{Bold Diagonal} vector grid and an \underline{underlined Lower} vector grid.

  \begin{table}[ht]
  \centering
  \caption{First quadrant 2D vector grid} \label{Table 1}
\begin{tabular}{c|ccccccccc}
  $\vdots$ & $\vdots$ & $\vdots $ & $\vdots $ & $\vdots $ & $ \vdots$ & $\vdots$ & $\vdots$ & $\vdots$ & $\ddots$  \\
 $ 8 $ & $\overline{\langle1,8\rangle}$ & $\overline{\langle2,8\rangle}$ & $\overline{\langle3,8\rangle}$ & $\overline{\langle4,8\rangle}$ & $\overline{\langle5,8\rangle}$ & $\overline{\langle6,8\rangle}$ & $\overline{\langle7,8\rangle}$ & $\mathbf{\langle8,8\rangle}$ & $\cdots$  \\
 $ 7 $ & $\overline{\langle1,7\rangle}$ & $\overline{\langle2,7\rangle}$ & $\overline{\langle3,7\rangle}$ & $\overline{\langle4,7\rangle}$ & $\overline{\langle5,7\rangle}$ & $\overline{\langle6,7\rangle}$ & $\mathbf{\langle7,7\rangle}$ & $\underline{\langle8,7\rangle}$ & $\cdots$  \\
 $ 6 $ & $\overline{\langle1,6\rangle}$ & $\overline{\langle2,6\rangle}$ & $\overline{\langle3,6\rangle}$ & $\overline{\langle4,6\rangle}$ & $\overline{\langle5,6\rangle}$ & $\mathbf{\langle6,6\rangle} $ & $\underline{\langle7,6\rangle} $ & $\underline{\langle8,6\rangle} $ & $\cdots$  \\
 $ 5 $ & $\overline{\langle1,5\rangle}$ & $\overline{\langle2,5\rangle}$ & $\overline{\langle3,5\rangle}$ & $\overline{\langle4,5\rangle}$ & $\mathbf{\langle5,5\rangle}$ & $\underline{\langle6,5\rangle}$ & $\underline{\langle7,5\rangle}$ & $\underline{\langle8,5\rangle}$ & $\cdots$  \\
 $ 4 $ & $\overline{\langle1,4\rangle}$ & $\overline{\langle2,4\rangle}$ & $\overline{\langle3,4\rangle}$ & $\mathbf{\langle4,4\rangle}$ & $\underline{\langle5,4\rangle}$ & $\underline{\langle6,4\rangle}$ & $ \underline{\langle7,4\rangle}$ & $\underline{\langle8,4\rangle}$ & $\cdots$  \\
 $ 3 $ & $\overline{\langle1,3\rangle}$ & $\overline{\langle2,3\rangle}$ & $\mathbf{\langle3,3\rangle}$ & $\underline{\langle4,3\rangle}$ & $\underline{\langle5,3\rangle}$ & $\underline{\langle6,3\rangle}$ & $\underline{\langle7,3\rangle}$ & $\underline{\langle8,3\rangle}$ & $\cdots$  \\
 $ 2 $ & $\overline{\langle1,2\rangle}$ & $\mathbf{\langle2,2\rangle}$ & $\underline{\langle3,2\rangle}$ & $\underline{\langle4,2\rangle}$ & $\underline{\langle5,2\rangle}$ & $\underline{\langle6,2\rangle}$ & $\underline{\langle7,2\rangle}$ & $\underline{\langle8,2\rangle}$ & $\cdots$  \\
 $ 1 $ & $\mathbf{\langle1,1\rangle}$ & $\underline{\langle2,1\rangle}$ & $\underline{\langle3,1\rangle}$ & $\underline{\langle4,1\rangle}$ & $\underline{\langle5,1\rangle}$ & $\underline{\langle6,1\rangle}$ & $\underline{\langle7,1\rangle}$ & $\underline{\langle8,1\rangle}$ & $\cdots$  \\ \hline
  $z/y$ & $1$  &  $2$  &   $3$  &   $4$  &   $5$  &   $6$  &   $7$  &   $8$  &  $\cdots$
  \end{tabular}
  \end{table}
\end{definition}

In certain situations the set of all overlined \textit{upper} vectors can be mapped to the set of all underlined \textit{lower} vectors. The bold diagonal is defined by $y=z$. We have in fact

\begin{statement}
For any positive integers such that $a<b$,
\begin{equation}  \nonumber
\overline{\langle a,b\rangle} \mapsto \underline{\langle b,a\rangle},
\end{equation}
and
\begin{equation}  \nonumber
\underline{\langle b,a\rangle} \mapsto \overline{\langle a,b\rangle}.
\end{equation}
Hence, the Upper grid can map onto the Lower grid, and the Lower grid can map onto the Upper grid.
\end{statement}

While the consideration of the \textbf{upper vectors} and the \textbf{lower vectors} may seem unnecessary in the present discussion, in later papers we shall make good use of the symmetries and possible mappings, especially when we come to discuss the Visible Point Vector (VPV) partitions wherein these symmetries can be applied to practical examples.

\begin{definition} \label{defn1.2} A finite 2D vector grid is a collection of vectors in a defined finite rectangular region. $V_2(y \in [a,b];z \in [c,d])$ is notation for

\begin{equation*}
     \begin{tabular}{c|ccccc}
     $d$ & $\langle a,d\rangle$ & $\langle a+1,d\rangle$ & $\hdots$ & $\langle b-1,d\rangle$ & $\langle b,d\rangle$  \\
     $d-1$ & $\langle a,d-1\rangle$ & $\langle a+1,d-1\rangle$ & $\hdots$ & $\langle b-1,d-1\rangle$ & $\langle b,d-1\rangle$   \\
     $d-2$ & $\langle a,d-2\rangle$ & $\langle a+1,d-2\rangle$ & $\hdots$ & $\langle b-1,d-2\rangle$ & $\langle b,d-2\rangle$   \\
     $\vdots$ & $\vdots$ & $\vdots$ & $\ddots$ & $\vdots$ & $\vdots$    \\
     $c+2$ & $\langle a,c+2\rangle$ & $\langle a+1,c+2\rangle$ & $\hdots$ & $\langle b-1,c+2\rangle$ & $\langle b,c+2\rangle$   \\
     $c+1$ & $\langle a,c+1\rangle$ & $\langle a+1,c+1\rangle$ & $\hdots$ & $\langle b-1,c+1\rangle$ & $\langle b,c+1\rangle$   \\
     $c$ & $\langle a,c\rangle$ & $\langle a+1,c\rangle$ & $\hdots$ & $\langle b-1,c\rangle$ & $\langle b,c\rangle$  \\ \hline
     $z/y$ & $a$ & $a+1$ & $\hdots$ & $b-1$ & $b$   \\
   \end{tabular}
\end{equation*}

and we see from this, that the first 2D quadrant finite grid starting at $\langle a,c\rangle$ can have $b$ and $d$ extended "to infinity", which could then be written $V_2(y \in [a,\infty];z \in [c,\infty])$.
\end{definition}

We give some examples now.

 \begin{equation}  \nonumber
 V_2(y\in[-1,1];z\in[-1,1])=
 \begin{tabular}{c|cccc}
  $1$  & $\langle -1,1\rangle$ & $\langle 0,1\rangle$ & $\langle 1,1\rangle$    \\
  $0$  & $\langle -1,0\rangle$ & $\langle 0,0\rangle$ & $\langle 1,0\rangle$   \\
  $-1$  & $\langle -1,-1\rangle$ & $\langle 0,-1\rangle$ & $\langle 1,-1\rangle$   \\ \hline
  $z/y$ & $-1$  & $0$  & $1$
  \end{tabular}
\end{equation}

 \begin{equation}  \nonumber
 V_2(y\in[-2,2];z\in[-2,2])=
 \begin{tabular}{c|ccccc}
 $2$ & $\langle -2,2\rangle$ & $\langle -1,2\rangle$ & $\langle 0,2\rangle$ & $\langle 1,2\rangle$ & $\langle 2,2\rangle$    \\
 $1$ & $\langle -2,1\rangle$ & $\langle -1,1\rangle$ & $\langle 0,1\rangle$ & $\langle 1,1\rangle$ & $\langle 2,1\rangle$    \\
 $0$ & $\langle -2,0\rangle$ & $\langle -1,0\rangle$ & $\langle 0,0\rangle$ & $\langle 1,0\rangle $ & $\langle 2,0\rangle$   \\
 $-1$ & $\langle -2,-1\rangle$ & $\langle -1,-1\rangle$ & $\langle 0,-1\rangle$ & $\langle 1,-1\rangle$ & $\langle 2,-1\rangle$   \\
$-2$ & $\langle -2,-2\rangle$ & $\langle -1,-2\rangle$ & $\langle 0,-2\rangle$ & $\langle 1,-2\rangle$ & $\langle 2,-2\rangle$   \\ \hline
 $z/y$  &    $-2$   &  $-1$   &    $0$   &   $1$   &    $2$
  \end{tabular}
\end{equation}

In our work on vector partitions, it will be a standard approach that we specify the vector grid, and then based on the type of partition function applied to the grid, specify a corresponding partition grid. Hence our partition grid related to the vector grid of Definition \ref{defn1.2}, is as follows.

\begin{definition} The first quadrant 2D \textbf{partition grid} is the set of 2D partitions denoted
 \begin{equation}  \nonumber
p_2(V_2(\langle m,n\rangle \subseteq \{y\in[a,b] \cup z\in[c,d]\}))
 \end{equation}

of $\langle a,b\rangle$ such that a and b are positive integers. We say that a \textbf{partition grid} is a range of a partition function applied to a \textbf{vector grid} domain in a defined region. For example, in the following picture, first quadrant 2D partition grid is split into an "$\overline{overlined \; Upper}$" partition grid, a "\textit{Diagonal upslope}" grid, and an "\underline{underlined Lower}" partition grid.
 \begin{equation*} \tiny{
\begin{tabular}{c|cccccccccc}
 ${\vdots}$ & ${\vdots}$ & ${\vdots}$ & ${\vdots}$  & ${\vdots}$  & ${\vdots}$  & ${\vdots}$ & ${\vdots}$  & ${\vdots}$  & ${\vdots}$ & ${\ddots}$   \\
  $8$ & ${\overline{p_2(0,8)}}$ & ${\overline{p_2(1,8)}}$ & ${\overline{p_2(2,8)}}$ & ${\overline{p_2(3,8)}}$ & ${\overline{p_2(4,8)}}$ & ${\overline{p_2(5,8)}}$ & ${\overline{p_2(6,8)}}$ & ${\overline{p_2(7,8)}}$ & ${p_2(8,8)}$ & ${\cdots}$  \\
  $7$ & ${\overline{p_2(0,7)}}$ & ${\overline{p_2(1,7)}}$ & ${\overline{p_2(2,7)}}$ & ${\overline{p_2(3,7)}}$ & ${\overline{p_2(4,7)}}$ & ${\overline{p_2(5,7)}}$ & ${\overline{p_2(6,7)}}$ & ${{p_2(7,7)}}$ & ${\underline{p_2(8,7)}}$ & ${\cdots}$  \\
  $6$ & ${\overline{p_2(0,6)}}$ & ${\overline{p_2(1,6)}}$ & ${\overline{p_2(2,6)}}$ & ${\overline{p_2(3,6)}}$ & ${\overline{p_2(4,6)}}$ & ${\overline{p_2(5,6)}}$ & ${p_2(6,6)}$  & ${\underline{p_2(7,6)}}$  & ${\underline{p_2(8,6)}}$  & ${\cdots}$  \\
  $5$ & ${\overline{p_2(0,5)}}$ & ${\overline{p_2(1,5)}}$ & ${\overline{p_2(2,5)}}$ & ${\overline{p_2(3,5)}}$ & ${\overline{p_2(4,5)}}$ & ${p_2(5,5)}$ & ${\underline{p_2(6,5)}}$ & ${\underline{p_2(7,5)}}$ & ${\underline{p_2(8,5)}}$ & ${\cdots}$  \\
  $4$ & ${\overline{p_2(0,4)}}$ & ${\overline{p_2(1,4)}}$ & ${\overline{p_2(2,4)}}$ & ${\overline{p_2(3,4)}}$ & ${p_2(4,4)}$ & ${\underline{p_2(5,4)}}$ & ${\underline{p_2(6,4)}}$ & ${\underline{p_2(7,4)}}$ & ${\underline{p_2(8,4)}}$ & ${\cdots}$  \\
  $3$ & ${\overline{p_2(0,3)}}$ & ${\overline{p_2(1,3)}}$ & ${\overline{p_2(2,3)}}$ & ${p_2(3,3)}$ & ${\underline{p_2(4,3)}}$ & ${\underline{p_2(5,3)}}$ & ${\underline{p_2(6,3)}}$ & ${\underline{p_2(7,3)}}$ & ${\underline{p_2(8,3)}}$ & ${\cdots}$  \\
  $2$ & ${\overline{p_2(0,2)}}$ & ${\overline{p_2(1,2)}}$ & ${p_2(2,2)}$ & ${\underline{p_2(3,2)}}$ & ${\underline{p_2(4,2)}}$ & ${\underline{p_2(5,2)}}$ & ${\underline{p_2(6,2)}}$ & ${\underline{p_2(7,2)}}$ & ${\underline{p_2(8,2)}}$ & ${\cdots}$  \\
  $1$ & ${\overline{p_2(0,1)}}$ & ${p_2(1,1)}$ & ${\underline{p_2(2,1)}}$ & ${\underline{p_2(3,1)}}$ & ${\underline{p_2(4,1)}}$ & ${\underline{p_2(5,1)}}$ &
  ${\underline{p_2(6,1)}}$ & ${\underline{p_2(7,1)}}$ & ${\underline{p_2(8,1)}}$ & ${\cdots}$\\
  $0$ & ${p_2(0,0)}$ & ${\underline{p_2(1,0)}}$ & ${\underline{p_2(2,0)}}$ & ${\underline{p_2(3,0)}}$ & ${\underline{p_2(4,0)}}$ & ${\underline{p_2(5,0)}}$ & ${\underline{p_2(6,0)}}$ & ${\underline{p_2(7,0)}}$ & ${\underline{p_2(8,0)}}$ & ${\cdots}$\\ \hline
 $z/y$ & $0$ & $1$ & $2$ & $3$ & $4$ & $5$ & $6$ & $7$ & $8$ & ${\cdots}$
  \end{tabular} }
  \end{equation*}
\end{definition}

Depending on how the partition function is defined, it may be that the set of all \textit{Upper} partitions can be mapped to the set of all  \textit{Lower} partitions. The \textit{Diagonal} upslope is defined here by $y=z$. We may often seek out or observe these features and apply them to generating functions for vector partitions in 2D, or generalized to 3D, 4D, and so on. As analogy with partitions into integers having Ferrers graphs, vector partitions can have generating functions over different grids, yielding theorems on vector partitions that arise simply from different equivalences of algebraic rational functions of several variables. What we often call weighted vector partitions are in fact functions mapped from the vector partition sums.

For example a 2D vector sum may map to a weighted value as in,
\begin{equation*}
  \langle 1, 2 \rangle + \langle 1, 3 \rangle + \langle 2, 3 \rangle = \langle 4, 8 \rangle \quad  \mapsto \quad  f_2(4,8) = \binom{2}{1}\binom{3}{1}\binom{3}{2},
\end{equation*}
so a vector partition sum maps to a \textit{function of the sum of that partition} $\langle 4, 8 \rangle$, the right side showing the function as a product of binomial coefficients in this instance.

\section{Partitions into exactly one part, two parts, or three parts}

We now try to lay the foundation for the language of vector partitions needed to progress our work in this paper.

\begin{statement}
Obviously, the number of partitions (of any kind) into one part, is 1. The \textbf{partition grid} for 2D vector partitions into one part as defined above is therefore trivially,

\bigskip
\begin{equation*}
\begin{tabular}{c|cccccccccc}
$\vdots$ & $\vdots$ & $\vdots$ & $\vdots$ & $\vdots$ & $\vdots$ & $\vdots$ & $\vdots$ & $\vdots$ & $\vdots$ & $\ddots$ \\
$8$ & $1$ & $1$ & $1$ & $1$ & $1$ & $1$ & $1$ & $1$ & $1$ & $\cdots$ \\
$7$ & $1$ & $1$ & $1$ & $1$ & $1$ & $1$ & $1$ & $1$ & $1$ & $\cdots$ \\
$6$ & $1$ & $1$ & $1$ & $1$ & $1$ & $1$ & $1$ & $1$ & $1$ & $\cdots$ \\
$5$ & $1$ & $1$ & $1$ & $1$ & $1$ & $1$ & $1$ & $1$ & $1$ & $\cdots$ \\
$4$ & $1$ & $1$ & $1$ & $1$ & $1$ & $1$ & $1$ & $1$ & $1$ & $\cdots$ \\
$3$ & $1$ & $1$ & $1$ & $1$ & $1$ & $1$ & $1$ & $1$ & $1$ & $\cdots$ \\
$2$ & $1$ & $1$ & $1$ & $1$ & $1$ & $1$ & $1$ & $1$ & $1$ & $\cdots$ \\
$1$ & $1$ & $1$ & $1$ & $1$ & $1$ & $1$ & $1$ & $1$ & $1$ & $\cdots$ \\
$0$ & $1$ & $1$ & $1$ & $1$ & $1$ & $1$ & $1$ & $1$ & $1$ & $\cdots$ \\ \hline
$z/y$ & $0$ & $1$ & $2$ & $3$ & $4$ & $5$ & $6$ & $7$ & $8$ & $\cdots$
  \end{tabular}
\end{equation*}
  \bigskip

and the generating function for this, say $\hat{O}_2(y,z)$, for $|y|$ and $|z|$ both less than $1$, is
\begin{equation} \label{8.01}
  \hat{O}_2(y,z) = \sum_{a,b \geq0} \hat{o}_2(a,b) y^a z^b = \sum_{a,b \geqq 0} y^a z^b = \frac{1}{(1-y)(1-z)},
\end{equation}
where $\hat{o}_2(a,b)$ is the number of partitions of 2D vectors $\langle a,b \rangle$, with each of $a$, and $b$ non-negative integers, into exactly one part.
\end{statement}

So we have described the simplest vector partition scenario, where there is only one part in each partition in the 2D first quadrant. In the 3D first hyperquadrant, the extended cubic lattice of points or vectors $\langle a,b,c \rangle$ with each of $a$, $b$ and $c$ non-negative integers, we have the generating function, say $\hat{O}_3(x,y,z)$, for $|x|$, $|y|$ and $|z|$ all less than $1$, is
\begin{equation} \label{8.01a}
 \hat{O}_3(x,y,z)= \sum_{a,b,c \geq0} \hat{o}_2(a,b,c) x^a y^b z^c = \sum_{a,b,c \geqq 0} x^a y^b z^c = \frac{1}{(1-x)(1-y)(1-z)},
\end{equation}
where $\hat{o}_2(a,b,c)$ is the number of partitions of 3D vectors $\langle a,b,c \rangle$ with each of $a$, $b$ and $c$ non-negative integers into exactly one part.
Generally speaking of $n$-space, we have: in the $n$D first hyperquadrant, the extended hypercubic lattice of points or vectors $\langle a_1,a_2,a_3,\ldots,a_n \rangle$ with each of $a_1$, $a_2$, $a_3$ up to $a_n$ all non-negative integers, we have the generating function, say $\hat{O}_n(a_1,a_2,a_3,\ldots,a_n)$, for $|z_1|$, $|z_2|$ up to $|z_n|$ all less than $1$, is
\begin{equation} \label{8.01b}
  \sum_{a_1,a_2,a_3,\ldots,a_n \geq0} \hat{o}_n(a_1,a_2,a_3,\ldots,a_n) {z_1}^{a_1} {z_2}^{a_2} {z_3}^{a_3} \cdots {z_n}^{a_n}
\end{equation}
\begin{equation*}
  = \sum_{a_1,a_2,\ldots,a_n \geqq 0} {z_1}^{a_1} {z_2}^{a_2} \cdots {z_n}^{a_n}
\end{equation*}
\begin{equation*}
  = \frac{1}{(1-z_1)(1-z_2)\cdots(1-z_n)},
\end{equation*}
where $\hat{o}_2(a_1,a_2,\ldots,a_n)$ is the number of partitions of $n$D vectors $\langle a_1,a_2,\ldots,a_n \rangle$, with each of $a_1$, $a_2$ up to $a_n$ non-negative integers, into exactly one part.

The next simplest case is then the scenario where each partition $\langle a,b\rangle$ is into exactly two parts, based on the vector grid for $a$ and $b$ both greater than or equal to 1. An arbitrary example where the vector partition is into two parts is for $\langle 5,2\rangle$,
\begin{equation} \nonumber
\langle 5,2\rangle = \langle 4,1\rangle + \langle 1,1\rangle = \langle 3,1\rangle + \langle 2,1\rangle.
\end{equation}

A reminder here, that \textit{compositions} count different orders of summands whereas \textit{partitions} do not. We recall that partitions are partially ordered sets, so these two vector equations are the only two partitions into two parts using vectors from the 2D first quadrant grid.

\bigskip

\section{Partitions into exactly two parts, and exactly three parts}

It is easy to find the number of partitions of 2D first quadrant vectors into exactly two parts, by a process of observation, then by formalizing this. However, a much neater approach is given from generating functions in Campbell \cite{gC2019} leading to a determinant form solution given here.

\begin{theorem} \label{thm2D2parts}
The number of first quadrant 2D partitions of vectors into two parts $\spadesuit_2(a,b)$, and the number of first quadrant 2D partitions of vectors into three parts $\clubsuit_2(a,b)$, are generated by the following equations valid for $|y|<1$, $|z|<1$,
\begin{equation}  \nonumber
\sum_{a,b\geq0}^{\infty} \spadesuit_2(a,b) y^a z^b = \frac{1}{2!}
  \begin{vmatrix}
    \frac{1}{\left(1-y\right) \left(1-z\right)} & -1 \\
    \frac{1}{\left(1-y^2\right) \left(1-z^2\right)} & \frac{1}{\left(1-y\right) \left(1-z\right)} \\
  \end{vmatrix}
\end{equation}
\begin{equation}   \nonumber
  = \frac{(1+yz)}{2(1+y)(1+z)(1-y^2)^2(1-z^2)^2};
\end{equation}
and
\begin{equation}  \nonumber
\sum_{a,b\geq0}^{\infty} \clubsuit_2(a,b) y^a z^b = \frac{1}{3!}
 \begin{vmatrix}
    \frac{1}{\left(1-y\right) \left(1-z\right)} & -1 & 0 \\
    \frac{1}{\left(1-y^2\right) \left(1-z^2\right)} & \frac{1}{\left(1-y\right) \left(1-z\right)} & -2 \\
    \frac{1}{\left(1-y^3\right) \left(1-z^3\right)} & \frac{1}{\left(1-y^2\right) \left(1-z^2\right)} & \frac{}{\left(1-y\right) \left(1-z\right)} \\
  \end{vmatrix}
  \end{equation}
 \begin{equation*}
 =  \frac{y^3 z^3 + y^2 z^2 + y^2 z + y z^2 + y z + 1}{(1-y)^3 (1+y) (y^2 + y + 1) (1-z)^3 (1+z) (z^2 + z + 1)}.
 \end{equation*}

  \bigskip

The two related partition grids for these are:

  \bigskip
for $\spadesuit_2(a,b)$,

\begin{tabular}{c|cccccccccc}
 $\vdots$ & $\vdots$ & $\vdots $ & $\vdots $ & $\vdots $ & $ \vdots$ & $\vdots $ & $\vdots $ & $\vdots$ & $\vdots$ & $\ddots$   \\
 $ 8$ & $5$ & $9$ & $14$ & $18$ & $23$ & $27$ & $32$ & $36$ & $\mathbf{41}$ & $\cdots$  \\
 $ 7$ & $4$ & $8$ & $12$ & $16$ & $20$ & $24$ & $28$ & $\mathbf{32}$ & $36$ & $\cdots$  \\
 $ 6$ & $4$ & $7$ & $11$ & $14$ & $18$ & $21$ & $\mathbf{25}$ & $28$ & $32$ & $\cdots$  \\
 $ 5$ & $3$ & $6$ & $9$ & $12$ & $15$ & $\mathbf{18}$ & $21$ & $24$ & $27$ & $\cdots$  \\
 $ 4$ & $3$ & $5$ & $8$ & $10$ & $\mathbf{13}$ & $15$ & $18$ & $20$ & $23$ & $\cdots$  \\
 $ 3$ & $2$ & $4$ & $6$ & $\mathbf{8}$ & $10$ & $12$ & $14$ & $16$ & $18$ & $\cdots$  \\
 $ 2$ & $2$ & $3$ & $\mathbf{5}$ & $6$ & $8$ & $9$ & $11$ & $12$ & $14$ & $\cdots$  \\
 $ 1$ & $1$ & $\mathbf{2}$ & $3$ & $4$ & $5$ & $6$ & $7$ & $8$ & $9$ & $\cdots$  \\
 $ 0$ & $\mathbf{1}$ & $1$ & $2$ & $2$ & $3$ & $3$ & $4$ & $4$ & $5$ & $\cdots$ \\ \hline
 $z\uparrow y\rightarrow$ & $0$ & $1$ & $2$ & $3$ & $4$ & $5$ & $6$ & $7$ & $8$ & $\cdots$
  \end{tabular}

  \bigskip

 and for $\clubsuit_2(a,b)$,

\begin{tabular}{c|cccccccccc}
 $\vdots$ & $\vdots$ & $\vdots $ & $\vdots $ & $\vdots $ & $ \vdots$ & $\vdots $ & $\vdots $ & $\vdots$ & $\vdots$ & $\ddots$   \\
 $ 8$ & $10$ & $25$ & $50$ & $80$ & $120$ & $165$ & $220$ & $280$ & $\mathbf{350}$ & $\cdots$  \\
 $ 7$ & $8$ & $20$ & $40$ & $64$ & $96$ & $132$ & $176$ & $\mathbf{224}$ & $280$ & $\cdots$  \\
 $ 6$ & $7$ & $16$ & $32$ & $51$ & $76$ & $104$ & $\mathbf{139}$ & $176$ & $220$ & $\cdots$  \\
 $ 5$ & $5$ & $12$ & $24$ & $38$ & $57$ & $\mathbf{78}$ & $104$ & $132$ & $165$ & $\cdots$  \\
 $ 4$ & $4$ & $9$ & $18$ & $28$ & $\mathbf{42}$ & $57$ & $76$ & $96$ & $120$ & $\cdots$  \\
 $ 3$ & $3$ & $6$ & $12$ & $\mathbf{19}$ & $28$ & $38$ & $51$ & $64$ & $80$ & $\cdots$  \\
 $ 2$ & $2$ & $4$ & $\mathbf{8}$ & $12$ & $18$ & $24$ & $32$ & $12$ & $40$ & $\cdots$  \\
 $ 1$ & $1$ & $\mathbf{2}$ & $4$ & $6$ & $9$ & $12$ & $16$ & $20$ & $25$ & $\cdots$  \\
 $ 0$ & $\mathbf{1}$ & $1$ & $2$ & $3$ & $4$ & $5$ & $7$ & $8$ & $10$ & $\cdots$ \\ \hline
 $z\uparrow y\rightarrow$ & $0$ & $1$ & $2$ & $3$ & $4$ & $5$ & $6$ & $7$ & $8$ & $\cdots$
  \end{tabular}
\end{theorem}

\bigskip

\textbf{Some notes on the partition grid sequences in theorem \ref{thm2D2parts}}

Observations on the diagonal sequences, and symmetries:

\bigskip

\underline{Observation 1}: The \textbf{bold} diagonal sequence $1,2,5,8,13,18,25,...$ in the $\spadesuit_2(a,b)$ grid is easily found in the OEIS as sequence number A000982, and is given by the ceiling function, so we have the exact formula $\spadesuit_2(n,n) = \lceil \frac{(n+1)^2}{2} \rceil$. In plain words this means "The number of 2D vector partitions of $\langle n,n \rangle$ into exactly two parts where $n$ is a positive integer is equal to $\lceil \frac{(n+1)^2}{2} \rceil$."

\bigskip

\underline{Observation 2}: The $\spadesuit_2(a,b)$ grid diagonal has generating function valid for $|q|<1$,
\begin{equation}\nonumber
  \frac{1+q^2}{(1+q)(1-q)^3} = 1 + 2 q + 5 q^2 + 8 q^3 + 13 q^4 + 18 q^5 + 25 q^6 + \hdots.
\end{equation}

\bigskip

\underline{Observation 3}: The sequence $\spadesuit_2(n,n)$ is also given as a "floor" function by $\lfloor \frac{n^2+1}{2}\rfloor $

\bigskip

\underline{Observation 4}: If we partition $n$ in two parts, say $r$ and $s$ so that $r^2 + s^2$ is minimal, then $\spadesuit_2(n,n) = r^2 +s^2$. Geometrical significance: folding a rod with length $n$ units at right angles in such a way that the end points are at the least distance, which is given by $\sqrt{\spadesuit_2(n,n)}$ as the hypotenuse of a right triangle with the sum of the base and height $= n$ units.

\bigskip

\underline{Observation 5}: \textit{A Dirichlet summation}. With respect to the diagonal sequence $\spadesuit_2(n,n)$ it is known that

\begin{equation}\nonumber
  \sum_{n > 0} \frac{1}{\lfloor \frac{n^2+1}{2} \rfloor} = \sum_{n > 0} \frac{1}{2n^2} + \sum_{n \geq 0} \frac{1}{2n^2 + 2n + 1}
\end{equation}
\begin{equation}\nonumber
1 + \frac{1}{2} + \frac{1}{5} + \frac{1}{8} + \frac{1}{13} + \frac{1}{18} + \frac{1}{25} + \hdots
         = \frac{1}{2} \left(\frac{\pi^2}{6} + \pi \tanh\left(\frac{\pi}{2}\right) \right)= 2.26312655....
\end{equation}

  \bigskip

\underline{Observation 6}: The \textbf{bold} diagonal sequence $1,2,8,19,42,78,139,224,350,...$ in the $\clubsuit_2(a,b)$ grid is found in the OEIS sequence number A101427, and is given by the following cases,
 \begin{equation*}
\clubsuit_2(n,n) =
    \left\{
      \begin{array}{ll}
        ((n+2)^2 (n+1)^2 + 12 (\lfloor n/2\rfloor +1)^2+8)/24, & \hbox{for $n$ divisible by 3$$;} \\
        ((n+2)^2 (n+1)^2 + 12 (\lfloor n/2\rfloor +1)^2)/24, & \hbox{otherwise.}
      \end{array}
    \right.
 \end{equation*}

Therefore we have an exact formula for $\clubsuit_2(n,n)$; the number of 2D vector partitions of $\langle n,n \rangle$ into exactly three parts; for $n$ any positive integer.  The description of this sequence A101427 from OEIS is "Number of different cuboids with volume $(pq)^n$, where $p$, $q$ are distinct prime numbers".

As an illustrative example of the formula, the first ten terms of $((n+2)^2 (n+1)^2 + 12 (\lfloor n/2\rfloor +1)^2+8)/24$ are $1, \frac{7}{3}, \frac{25}{3}, 19, \frac{127}{3}, \frac{235}{3}, 139, \frac{673}{3}, \frac{1051}{3}, 517$, and the first ten terms of $((n+2)^2 (n+1)^2 + 12 (\lfloor n/2\rfloor +1)^2)/24$ are $\frac{2}{3}, 2, 8, \frac{56}{3}, 42, 78, \frac{416}{3}, 224, 350, \frac{1550}{3}$.

\bigskip

\underline{Observation 7}: The $\clubsuit_2(a,b)$ grid diagonal has generating function valid for $|q|<1$,
\begin{equation}\nonumber
  \frac{q^6+3q^4+4q^3+3q^2+1}{(1-q^3)(1-q^2)^2 (1-q)^2} = 1 + 2 q + 8q^2 + 19q^3 + 42q^4 + 78q^5 + 139q^6 + \hdots.
\end{equation}

\bigskip

\underline{Observation 8}: \textit{Horizontal rows and Vertical columns sequences generating functions}.

Row and column sequences for $\spadesuit_2(a,b)$ and $\clubsuit_2(a,b)$ have certain observable properties worth mentioning:
\begin{enumerate}
 \item The $j$th row sequence is the same as the $j$th column sequence;
 \item The $k$th term in the $j$th row sequence is the same as $k$ term in the $j$th column sequence;
 \item That is, $\spadesuit_2(a,b) = \spadesuit_2(b,a)$ and $\clubsuit_2(a,b) = \clubsuit_2(b,a)$.
 \item The generating function for the $j$th row sequence is the same as the generating function for the $j$th column sequence.
 \end{enumerate}

  \bigskip

\textbf{Proof of theorem \ref{thm2D2parts}.}

In Campbell \cite{gC2019}, we have the following 2D corollary of the main $n$D extension of Cauchy's $q$-binomial theorem, which is equation (3.3) in that paper. The proof of this $n$D theorem used Cramer's Rule and gave us the determinants as coefficients. So, setting $a=0$ in identity (3.3) of \cite{gC2019} we obtain, for $y$ and $z$ both less than unity,

\begin{equation}  \label{8.02}
  F_2(y,z;0,t)= \prod_{j,k\geq0}\frac{1}{1-y^j z^k t} \\
\end{equation}
\begin{equation}  \nonumber
  = 1 + \frac{1}{\left(1-y\right) \left(1-z\right)} \frac{t}{1!}
+ \begin{vmatrix}
    \frac{1}{\left(1-y\right) \left(1-z\right)} & -1 \\
    \frac{1}{\left(1-y^2\right) \left(1-z^2\right)} & \frac{1}{\left(1-y\right) \left(1-z\right)} \\
  \end{vmatrix} \frac{t^2}{2!}
  \end{equation}
\begin{equation}  \nonumber
+ \begin{vmatrix}
    \frac{1}{\left(1-y\right) \left(1-z\right)} & -1 & 0 \\
    \frac{1}{\left(1-y^2\right) \left(1-z^2\right)} & \frac{1}{\left(1-y\right) \left(1-z\right)} & -2 \\
    \frac{1}{\left(1-y^3\right) \left(1-z^3\right)} & \frac{1}{\left(1-y^2\right) \left(1-z^2\right)} & \frac{1}{\left(1-y\right) \left(1-z\right)} \\
  \end{vmatrix} \frac{t^3}{3!}
  \end{equation}
\begin{equation}  \nonumber
+ \begin{vmatrix}
    \frac{1}{\left(1-y\right) \left(1-z\right)} & -1 & 0 & 0 \\
    \frac{1}{\left(1-y^2\right) \left(1-z^2\right)} & \frac{1}{\left(1-y\right) \left(1-z\right)} & -2 & 0 \\
    \frac{1}{\left(1-y^3\right) \left(1-z^3\right)} & \frac{1}{\left(1-y^2\right) \left(1-z^2\right)} & \frac{1}{\left(1-y\right) \left(1-z\right)} & -3 \\
    \frac{1}{\left(1-y^4\right) \left(1-z^4\right)} & \frac{1}{\left(1-y^3\right) \left(1-z^3\right)} & \frac{1}{\left(1-y^2\right) \left(1-z^2\right)} & \frac{1}{\left(1-y\right) \left(1-z\right)} \\
  \end{vmatrix} \frac{t^4}{4!}
  \end{equation}
\begin{equation}  \nonumber
+ etc.
\end{equation}

The proof of the theorem comes from recognizing that $\spadesuit_2(b,a)$ and $\clubsuit_2(a,b)$ are the coefficients of $t^2$ and $t^3$ respectively in (\ref{8.02}). The determinants are simplified by applying the below identities, setting
\begin{equation*}
  a_1= \frac{1}{(1-y)(1-z)}, \quad  a_2= \frac{1}{(1-y^2)(1-z^2)}, \quad a_3= \frac{1}{(1-y^3)(1-z^3)},
\end{equation*}
substituted into
   \begin{equation}  \nonumber
   \begin{vmatrix}
    {a_1} & -1 \\
    {a_2} & {a_1} \\
  \end{vmatrix} = {a_1}^2 + {a_2};
  \end{equation}
and
  \begin{equation}  \nonumber
  \begin{vmatrix}
    {a_1} & -1 & 0 \\
    {a_2} & {a_1} & -2 \\
    {a_3} & {a_2} & {a_1} \\
  \end{vmatrix} = {a_2}^3 + 3 {a_1} {a_2} + 2 {a_3}
  \end{equation}
respectively. This gives us the appropriate functions generating the partitions $\spadesuit_2(b,a)$ and $\clubsuit_2(a,b)$. Expanding these generating functions allows us to fill in the relevant partition grids.  $\quad \blacksquare$

\subsection{The $n$-dimensional $q$-binomial theorem} \label{S:Setting}
\bigskip

In his classical account of the theory of partitions, Andrews \cite[chapter 2]{gA1976} shows that many of the time honoured partition identities first given by Euler, Gauss, Heine and Jacobi derive from the $q$-binomial theorem originally given by Cauchy \cite{aC1893},

\begin{theorem}
The $q$-binomial theorem. If $\left|q\right|<1, \left|t\right|<1$, for all complex $a$,
\begin{equation}\label{9.01}
  1+\sum_{k=1}^\infty \frac{(1-a)(1-aq)(1-aq^2)...(1-aq^{k-1})t^k}{(1-q)(1-q^2)(1-q^3)...(1-q^k)} = \prod_{k=0}^\infty \frac{1-atq^k}{1-tq^k}.
\end{equation}
  \end{theorem}

The following extension to $n$D of the $q$-binomial theorem was proved in Campbell \cite{gC2019} and 22 years earlier in his thesis \cite{gC1997}. We begin with the
\begin{definition}
  Define the function $F_n (t)$ for all complex numbers $a,t$
with $|a|,|t|<1$, and for $n\geq 1$ with $|x_1|,|x_2|,...,|x_n|<1$ by the sequence of functions $F_n(t)$ with
\begin{equation}\label{9.02}
  F_0(-;a,t)= \frac{1-at}{1-t}
\end{equation}
and for all of $\left|x_1\right|,\left|x_2\right|,\left|x_3\right|,...,\left|x_n\right|<1$, by
\begin{equation}\label{9.03}
  F_n(x_1,x_2,x_3,...,x_n;a,t)= \prod_{\alpha_1,\alpha_2,\alpha_3,...,\alpha_n \geq0}\frac{1-x_1^{\alpha_1} x_2^{\alpha_2} ... x_n^{\alpha_n}at}{1-x_1^{\alpha_1} x_2^{\alpha_2} ... x_n^{\alpha_n}t} \equiv \sum_{k=0}^\infty {_{n}}{A_{k}} t^k
\end{equation}
where $\alpha_1,\alpha_2,\alpha_3,...,\alpha_n$ may be all zero.
\end{definition}

By the way, note that from this definition the right side of (\ref{9.01}) is $F_1(q;a,t)$, so it is evident we are working on a generalised $q$-binomial expression. Next, based on this definition we can assert the $n$-space $q$-binomial theorem, extended making $q$ redundant replaced by $n$ separate $x_i$ variables,

\begin{theorem} \label{thm 2.1} For each $|x_i|<1$, with $1 \leq i \leq n$,
\begin{equation}\label{9.04}
  {_{n}}{A_{k}}(x_1,x_2,x_3,...,x_n;a)= \det(a_{ij})/k!
\end{equation}
where the determinant is of order $k$, and

  \[
    a_{ij}=
    \begin{cases}
      \frac{1-a^{i-j+1}}{(1-{x_1}^{i-j+1})(1-{x_2}^{i-j+1})(1-{x_3}^{i-j+1})...(1-{x_n}^{i-j+1})},             &\text{$i \geq j$;}\\
      -i,      &\text{$i=j-1$;}\\
      0,      &\text{$otherwise$.}
    \end{cases}
  \]
\end{theorem}

We write theorem \ref{thm 2.1}, restating it as a single equation identity as follows, and rate it as a distinct theorem for our purposes.

\begin{theorem} \label{th3.1}  For each $|x_i|<1$, with $1 \leq i \leq n$,
\begin{equation}  \label{9.05}
  F_n(x_1,x_2,x_3,...,x_n;a,t)= \prod_{\alpha_1,\alpha_2,\alpha_3,...,\alpha_n \geq0}\frac{1-x_1^{\alpha_1} x_2^{\alpha_2} ... x_n^{\alpha_n}at}{1-x_1^{\alpha_1} x_2^{\alpha_2} ... x_n^{\alpha_n}t} \\
\end{equation}
\begin{equation}  \nonumber
  = 1 + \frac{1-a}{\left(1-{x_1}\right) \left(1-{x_2}\right) \cdots \left(1-{x_n}\right)} \frac{t}{1!}
\end{equation}
\begin{equation}  \nonumber
+ \begin{vmatrix}
    \frac{1-a}{\left(1-{x_1}\right) \left(1-{x_2}\right) \cdots \left(1-{x_n}\right)} & -1 \\
    \frac{1-a^2}{\left(1-{x_1}^2\right) \left(1-{x_2}^2\right) \cdots \left(1-{x_n}^2\right)} & \frac{1-a}{\left(1-{x_1}\right) \left(1-{x_2}\right) \cdots \left(1-{x_n}\right)} \\
  \end{vmatrix} \frac{t^2}{2!}
  \end{equation}
\begin{equation}  \nonumber
+ \begin{vmatrix}
    \frac{1-a}{\left(1-{x_1}\right) \left(1-{x_2}\right) \cdots \left(1-{x_n}\right)} & -1 & 0 \\
    \frac{1-a^2}{\left(1-{x_1}^2\right) \left(1-{x_2}^2\right) \cdots \left(1-{x_n}^2\right)} & \frac{1-a}{\left(1-{x_1}\right) \left(1-{x_2}\right) \cdots \left(1-{x_n}\right)} & -2 \\
    \frac{1-a^3}{\left(1-{x_1}^3\right) \left(1-{x_2}^3\right) \cdots \left(1-{x_n}^3\right)} & \frac{1-a^2}{\left(1-{x_1}^2\right) \left(1-{x_2}^2\right) \cdots \left(1-{x_n}^2\right)} & \frac{1-a}{\left(1-{x_1}\right) \left(1-{x_2}\right) \cdots \left(1-{x_n}\right)} \\
  \end{vmatrix} \frac{t^3}{3!}
  \end{equation}
\begin{equation}  \nonumber
+ \begin{vmatrix}
    \frac{1-a}{\left(1-{x_1}\right) \left(1-{x_2}\right) \cdots \left(1-{x_n}\right)} & -1 & 0 & 0 \\
    \frac{1-a^2}{\left(1-{x_1}^2\right) \left(1-{x_2}^2\right) \cdots \left(1-{x_n}^2\right)} & \frac{1-a}{\left(1-{x_1}\right) \left(1-{x_2}\right) \cdots \left(1-{x_n}\right)} & -2 & 0 \\
    \frac{1-a^3}{\left(1-{x_1}^3\right) \left(1-{x_2}^3\right) \cdots \left(1-{x_n}^3\right)} & \frac{1-a^2}{\left(1-{x_1}^2\right) \left(1-{x_2}^2\right) \cdots \left(1-{x_r}^2\right)} & \frac{1-a}{\left(1-{x_1}\right) \left(1-{x_2}\right) \cdots \left(1-{x_n}\right)} & -3 \\
    \frac{1-a^4}{\left(1-{x_1}^4\right) \left(1-{x_2}^4\right) \cdots \left(1-{x_n}^4\right)} & \frac{1-a^3}{\left(1-{x_1}^3\right) \left(1-{x_2}^3\right) \cdots \left(1-{x_n}^3\right)} & \frac{1-a^2}{\left(1-{x_1}^2\right) \left(1-{x_2}^2\right) \cdots \left(1-{x_n}^2\right)} & \frac{1-a}{\left(1-{x_1}\right) \left(1-{x_2}\right) \cdots \left(1-{x_n}\right)} \\
  \end{vmatrix} \frac{t^4}{4!}
  \end{equation}
\begin{equation}  \nonumber
 + etc.
\end{equation}
\end{theorem}

We are now in a position to state the $n$-dimensional vector partition generalization of theorem \ref{thm2D2parts}.

\begin{theorem} \label{thm-nD2parts}
The number of first hyperquadrant $n$D partitions of vectors into two parts $\spadesuit_n(a_1,a_2,\ldots,a_n)$, and the number of first hyperquadrant $n$D partitions of vectors into three parts $\clubsuit_n(a_1,a_2,\ldots,a_n)$, are generated by the following equations valid for $|z_i|<1$, with $1 \leq i \leq n$,
\begin{equation}  \nonumber
\sum_{a_1,a_2,\ldots,a_n\geq0}^{\infty} \spadesuit_n(a_1,a_2,\ldots,a_n) {z_1}^{a_1} {z_2}^{a_2} \cdots {z_n}^{a_n} = \frac{1}{2!}
  \begin{vmatrix}
    \frac{1}{\left(1-{z_1}\right) \left(1-{z_2}\right) \cdots \left(1-{z_n}\right)} & -1 \\
    \frac{1}{\left(1-{z_1}^2\right) \left(1-{z_2}^2\right) \cdots \left(1-{z_n}^2\right)} & \frac{1}{\left(1-{z_1}\right) \left(1-{z_2}\right) \cdots \left(1-{z_n}\right)} \\
  \end{vmatrix}
\end{equation}
\begin{equation}   \nonumber
  = {c_1}^2 + {c_2};
\end{equation}
\begin{equation}   \nonumber
  where \; c_1= \frac{1}{\left(1-{z_1}\right) \left(1-{z_2}\right) \cdots \left(1-{z_n}\right)} \; and \;
                           c_2= \frac{1}{\left(1-{z_1^2}\right) \left(1-{z_2^2}\right) \cdots \left(1-{z_n^2}\right)}.
\end{equation}
Also
\begin{equation}  \nonumber
\sum_{a_1,a_2,\ldots,a_n}^{\infty} \clubsuit_n(a_1,a_2,\ldots,a_n) {z_1}^{a_1} {z_2}^{a_2} \cdots {z_n}^{a_n}
  \end{equation}
 \begin{equation*} = \frac{1}{3!}
 \begin{vmatrix}
    \frac{1}{\left(1-{z_1}\right) \left(1-{z_2}\right) \cdots \left(1-{z_n}\right)} & -1 & 0 \\
    \frac{1}{\left(1-{z_1}^2\right) \left(1-{z_2}^2\right) \cdots \left(1-{z_n}^2\right)} & \frac{1}{\left(1-{z_1}\right) \left(1-{z_2}\right) \cdots \left(1-{z_n}\right)} & -2 \\
    \frac{1}{\left(1-{z_1}^3\right) \left(1-{z_2}^3\right) \cdots \left(1-{z_n}^3\right)} & \frac{1}{\left(1-{z_1}^2\right) \left(1-{z_2}^2\right) \cdots \left(1-{z_n}^2\right)} & \frac{1}{\left(1-{z_1}\right) \left(1-{z_2}\right) \cdots \left(1-{z_n}\right)} \\
  \end{vmatrix}
  \end{equation*}
 \begin{equation*}
  = {c_2}^3 + 3 {c_1} {c_2} + 2 {c_3}; \; where \; c_1= \frac{1}{\left(1-{z_1}\right) \left(1-{z_2}\right) \cdots \left(1-{z_n}\right)},
\end{equation*}
\begin{equation*}
  c_2= \frac{1}{\left(1-{z_1^2}\right) \left(1-{z_2^2}\right) \cdots \left(1-{z_n^2}\right)}  \; and \;
                           c_3= \frac{1}{\left(1-{z_1^3}\right) \left(1-{z_2^3}\right) \cdots \left(1-{z_n^3}\right)}.
\end{equation*}
\end{theorem}

\textbf{PROOF}: From the coefficients of $t^2$ respectively $t^3$ and mapping $x_1 \mapsto z_i$ in theorem \ref{th3.1} and setting $a=0$ we obtain the number of relevant $n$D partitions into exactly two parts and exactly three parts. $ \quad \quad \blacksquare$

Hence, in this section we have demonstrated that in $n$D first hyperquadrants, the number of vector partitions into exactly integer $m$ number of parts is derivable to a closed form of a determinant.

\bigskip

\section{First 3D hyperquadrant tableau reduced to a 2D tableau}

The set of all lattice points in 3D space first hyperquadrant (extended cube) are similarly a countable set. The lattice point vectors in nD space first hyperquadrant comprised of points $\left\langle a_1,a_2,a_3,...,a_n \right\rangle$ with each $a_i$ a positive integer is a countable set.

In 3D the set of first hyperquadrant points can be portrayed as shown below, as 2D horizontal layers or lamina, each one as follows.

 \begin{equation} \label{8.04}
\begin{tabular}  {c c c c c c c c}
  \multicolumn{1}{c|}{$\vdots$}  & \multicolumn{1}{c}{$\vdots$}  & \multicolumn{1}{c}{$\vdots$}  & \multicolumn{1}{c}{$\vdots$}  & \multicolumn{1}{c}{$\vdots$}  & \multicolumn{1}{c}{$\vdots$}  & \multicolumn{1}{c}{$\vdots$} & \multicolumn{1}{c}{$.\cdot$} \\
  \multicolumn{1}{c|}{$6$}  & \multicolumn{1}{c}{$\langle1,6,a\rangle$}  & \multicolumn{1}{c}{$\langle2,6,a\rangle$}  & \multicolumn{1}{c}{$\langle3,6,a\rangle$}  & \multicolumn{1}{c}{$\langle4,6,a\rangle$}  & \multicolumn{1}{c}{$\langle5,6,a\rangle$}  & \multicolumn{1}{c}{$\langle6,6,a\rangle$} & \multicolumn{1}{c}{$\cdots$} \\
  \multicolumn{1}{c|}{$5$}  & \multicolumn{1}{c}{$\langle1,5,a\rangle$}  & \multicolumn{1}{c}{$\langle2,5,a\rangle$}  & \multicolumn{1}{c}{$\langle3,5,a\rangle$}  & \multicolumn{1}{c}{$\langle4,5,a\rangle$}  & \multicolumn{1}{c}{$\langle5,5,a\rangle$}  & \multicolumn{1}{c}{$\langle6,5,a\rangle$} & \multicolumn{1}{c}{$\cdots$} \\
  \multicolumn{1}{c|}{$4$}  & \multicolumn{1}{c}{$\langle1,4,a\rangle$}  & \multicolumn{1}{c}{$\langle2,4,a\rangle$}  & \multicolumn{1}{c}{$\langle3,4,a\rangle$}  & \multicolumn{1}{c}{$\langle4,4,a\rangle$}  & \multicolumn{1}{c}{$\langle5,4,a\rangle$}  & \multicolumn{1}{c}{$\langle6,4,a\rangle$} & \multicolumn{1}{c}{$\cdots$} \\
  \multicolumn{1}{c|}{$3$}  & \multicolumn{1}{c}{$\langle1,3,a\rangle$}  & \multicolumn{1}{c}{$\langle2,3,a\rangle$}  & \multicolumn{1}{c}{$\langle3,3,a\rangle$}  & \multicolumn{1}{c}{$\langle4,3,a\rangle$}  & \multicolumn{1}{c}{$\langle5,3,a\rangle$}  & \multicolumn{1}{c}{$\langle6,3,a\rangle$} & \multicolumn{1}{c}{$\cdots$} \\
  \multicolumn{1}{c|}{$2$}  & \multicolumn{1}{c}{$\langle1,2,a\rangle$}  & \multicolumn{1}{c}{$\langle2,2,a\rangle$}  & \multicolumn{1}{c}{$\langle3,2,a\rangle$}  & \multicolumn{1}{c}{$\langle4,2,a\rangle$}  & \multicolumn{1}{c}{$\langle5,2,a\rangle$}  & \multicolumn{1}{c}{$\langle6,2,a\rangle$} & \multicolumn{1}{c}{$\cdots$} \\
  \multicolumn{1}{c|}{$1$}  & \multicolumn{1}{c}{$\langle1,1,a\rangle$}  & \multicolumn{1}{c}{$\langle2,1,a\rangle$}  & \multicolumn{1}{c}{$\langle3,1,a\rangle$}  & \multicolumn{1}{c}{$\langle4,1,a\rangle$}  & \multicolumn{1}{c}{$\langle5,1,a\rangle$}  & \multicolumn{1}{c}{$\langle6,1,a\rangle$} & \multicolumn{1}{c}{$\cdots$} \\ \cline{1-8}
  \multicolumn{1}{c|}{$y \uparrow, x \rightarrow, z=a$}  & \multicolumn{1}{c}{$1$}  & \multicolumn{1}{c}{$2$}  & \multicolumn{1}{c}{$3$}  & \multicolumn{1}{c}{$4$}  & \multicolumn{1}{c}{$5$}  & \multicolumn{1}{c}{$6$} & \multicolumn{1}{c}{$\cdots$} \\
\end{tabular}
\end{equation}
where $a \in \{1,2,3,\ldots \}$.

In order to count each lattice point vector in the entire first 3D hyperquadrant, it suffices to count the following vectors $\left\langle a,b,c \right\rangle$ such that $0<a<b<c$ and apply symmetries. We can start with an arrangement like
\begin{equation}  \label{8.05}
\begin{array}{ccccccccc}
  \textmd{etc.} &   &   &   &   &   &   &   & \\
  \langle1,3,6\rangle & \langle2,3,7\rangle & \langle3,6,7\rangle & \langle4,6,8\rangle & \langle5,6,9\rangle & \langle6,8,9\rangle & \langle7,8,10\rangle & \langle8,9,10\rangle &   \\
  \langle1,2,6\rangle & \langle2,5,6\rangle & \langle3,5,7\rangle & \langle4,5,8\rangle & \langle5,7,8\rangle & \langle6,7,9\rangle & \langle7,8,9\rangle & &   \\
  \langle1,4,5\rangle & \langle2,4,6\rangle & \langle3,4,7\rangle & \langle4,6,7\rangle & \langle5,6,8\rangle & \langle6,7,8\rangle & & &   \\
  \langle1,3,5\rangle & \langle2,3,6\rangle & \langle3,5,6\rangle & \langle4,5,7\rangle & \langle5,6,7\rangle &   &   &   &   \\
  \langle1,2,5\rangle & \langle2,4,5\rangle & \langle3,4,6\rangle & \langle4,5,6\rangle &   &   &   &   &   \\
  \langle1,3,4\rangle & \langle2,3,5\rangle & \langle3,4,5\rangle &   &   &   &   &   &   \\
  \langle1,2,4\rangle & \langle2,3,4\rangle &   &   &   &   &   &   &   \\
  \langle1,2,3\rangle &   &   &   &   &   &   &   &
\end{array}
\end{equation}
where the bottom entry for "column $a$" is $\langle a,a+1,a+2 \rangle$ and each column has terms $\bigcup_{a<j<k} \langle a,j,k \rangle$ in an order that exhausts successive $k$ integer terms with respect to the $j$ terms between $a$ and $k$. The above arrangement of 3D vectors enables us to write the generating function

\begin{theorem}   \label{3D-2D01}
\begin{equation}  \label{8.06}
  \prod_{0<a<b<c} (1+ x^ay^bz^c) = 1+ \sum_{0<i<j<k} p_3(\mathfrak{D};i,j,k) x^i y^j z^k
\end{equation}
\begin{eqnarray*}
     &=& (1+ x^1y^2z^3) \\
     & & (1+ x^1y^2z^4)  (1+ x^2y^3z^4) \\
     & & (1+ x^1y^3z^4)  (1+ x^2y^3z^5)  (1+ x^3y^4z^5) \\
     & & (1+ x^1y^2z^5)  (1+ x^2y^4z^5)  (1+ x^3y^4z^6)  (1+ x^4y^5z^6) \\
     & & (1+ x^1y^3z^5)  (1+ x^2y^3z^6)  (1+ x^3y^5z^6)  (1+ x^4y^5z^7)  (1+ x^5y^6z^7) \\
     & & (1+ x^1y^4z^5)  (1+ x^2y^4z^6)  (1+ x^3y^4z^7)  (1+ x^4y^6z^7)  (1+ x^5y^6z^8)  (1+ x^6y^7z^8) \\
     & & (1+ x^1y^2z^6)  (1+ x^2y^5z^6)  (1+ x^3y^5z^7)  (1+ x^4y^5z^8)  (1+ x^5y^7z^8)  (1+ x^6y^7z^9)  (1+ x^7y^8z^9) \\
     & & \textmd{etc.} \\
\end{eqnarray*}
where $p_3(\mathfrak{D};i,j,k)$ is the number of 3D partitions of $\left\langle i,j,k \right\rangle$ into distinct vector parts of type $\left\langle a,b,c \right\rangle$ such that $0<a<b<c$ with positive integer values of $a$, $b$, and $c$. In other words, $p_3(\mathfrak{D};i,j,k)$ is the number of 3D partitions into distinct parts from (\ref{8.05}).
\end{theorem}
\bigskip

As a logical companion to theorem \ref{3D-2D01}, also easily inferred from our reduction of 3D tableau to 2D tableau (\ref{8.05}) is the following

\begin{theorem}   \label{3D-2D02}
\begin{equation}  \label{8.07}
  \prod_{0<a<b<c} \frac{1}{(1- x^ay^bz^c)} = 1+ \sum_{0<i<j<k} p_3(\mathfrak{U};i,j,k) x^i y^j z^k
\end{equation}
\begin{eqnarray*}
     &=& \frac{1}{(1- x^1y^2z^3)} \\
     & & \frac{1}{(1- x^1y^2z^4)  (1- x^2y^3z^4)} \\
     & & \frac{1}{(1- x^1y^3z^4)  (1- x^2y^3z^5)  (1- x^3y^4z^5)} \\
     & & \frac{1}{(1- x^1y^2z^5)  (1- x^2y^4z^5)  (1- x^3y^4z^6)  (1- x^4y^5z^6)} \\
     & & \frac{1}{(1- x^1y^3z^5)  (1- x^2y^3z^6)  (1- x^3y^5z^6)  (1- x^4y^5z^7) (1- x^5y^6z^7)} \\
     & & \frac{1}{(1- x^1y^4z^5)  (1- x^2y^4z^6)  (1- x^3y^4z^7)  (1- x^4y^6z^7)  (1- x^5y^6z^8)  (1- x^6y^7z^8)} \\
     & & \frac{1}{(1- x^1y^2z^6)  (1- x^2y^5z^6)  (1- x^3y^5z^7)  (1- x^4y^5z^8)  (1- x^5y^7z^8)  (1- x^6y^7z^9)  (1- x^7y^8z^9)} \\
     & & \textmd{etc.} \\
\end{eqnarray*}
where $p_3(\mathfrak{U};i,j,k)$ is the number of 3D partitions of $\left\langle i,j,k \right\rangle$ into unrestricted vector parts of type $\left\langle a,b,c \right\rangle$ such that $0<a<b<c$ with positive integer values of $a$, $b$, and $c$. In other words, $p_3(\mathfrak{U};i,j,k)$ is the number of 3D partitions into unrestricted parts from (\ref{8.05}).
\end{theorem}


\section{2D and 3D \textit{Upper Radial Regions} of vectors for partition identities}

In the next sections we study 2D and 3D sets of vectors in regions radial from the origin. A 2D or 3D vector partition is essentially:
\begin{enumerate}
  \item a bunch of vector arrows
  \item in order of non-increasing length,
  \item each with positive rational gradients in 2D or in 3D wrt the relevant axes,
  \item joined head to tail
  \item starting at the origin
  \item with ending destination an integer co-ordinate point,
  \item say $\left\langle a,b \right\rangle$ with $a$ and $b$ both non-negative integers in 2D, or
  \item say $\left\langle a,b,c \right\rangle$ with $a$, $b$ and $c$ all non-negative integers in 3D.
\end{enumerate}

The number of arrows in this vector partition is the number of parts of the partition.

Since a vector partition is a partially ordered set summed in a defined way, this means there are potentially many ways to define the vector partition. Rules used in defining a vector partition are possibly one or more of:
\begin{itemize}
  \item non-increasing arrow length (or non-decreasing),
  \item non-increasing co-ordinate numbering (or non-decreasing) or co-ordinate orderings, say X then Y; or Y then X axes,
  \item non-increasing (respectively non-decreasing) angular measure from an axis as polar co-ordinate systemic,
  \item non-increasing (respectively non-decreasing) color gradations.
\end{itemize}

We consider 2D vectors $\left\langle a,b \right\rangle$ where $0<a<b$. We mainly look into two types of such vectors for our vector partition aggregations; both types being in "radial from the origin" regions:
\begin{enumerate}
  \item All of the 2D vectors $\left\langle a,b \right\rangle$ where $0\leq a<b$ are in the infinitely extended region radial from the origin. This may involve examination of defined "subset finite patches" of that radial region.
  \item All VPV 2D vectors are $\left\langle a,b \right\rangle$ where $0<a<b$ and $\gcd(a,b)=1$ are in the infinitely extended region radial from the origin. This also may involve examination of defined "subset finite patches" of that radial region.
 \end{enumerate}
      Using "All 2D first quadrant vectors" we have the examples
\begin{equation} \label{8.07.01}
\prod_{j,k\geq1}\frac{1}{1-x^j y^k}
=  \prod_{j,k\geq1} \left\{ 1 + x^j y^k + x^{2j} y^{2k} +  x^{3j} y^{3k} + \ldots \right\} = \sum_{a,b\geq0} \mathbf{p}_2(a,b)x^a y^b.
\end{equation}
Here we generate $\mathbf{p}_2(a,b)$ as the number of partitions of $\langle j,k \rangle$ into unrestricted parts $\langle a,b \rangle$ with $a,b\geq0$. In this case each partition has a value of 1, so the number of partitions is unsurprisingly just a count of each partition.
\begin{equation} \label{8.07.02}
\prod_{j,k\geq1}\left(1+ \frac{x^j y^k}{(1-x^j y^k)^2}\right)
=  \prod_{j,k\geq1} \left\{ 1 + x^j y^k + 2x^{2j} y^{2k} + 3x^{3j} y^{3k} + \ldots \right\} = \sum_{a,b\geq0} \mathbf{q}_2(a,b)x^a y^b.
\end{equation}
This generates $\mathbf{q}_2(a,b)$ as a weighted sum of partitions of $\langle j,k \rangle$ into unrestricted parts $\langle a,b \rangle$ with $a,b\geq0$. This is a \textit{function of a partition} applied across all possible partitions of $\langle j,k \rangle$. In this case each partition is assigned a value (called also a weight) of
\begin{equation*}
  n_1^{a_1} n_2^{a_2} ...n_r^{a_r}
\end{equation*}
for a partition of $\langle j,k \rangle$
\begin{equation*}
  (\langle n_1j,n_1k \rangle^{a_1} \langle n_2j,n_2k \rangle^{a_2} ... \langle n_rj,n_rk \rangle^{a_r}),
\end{equation*}
so the function of partitions (as distinct from number of partitions) on $\langle j,k \rangle$ is derived as a combined function of each partition. So $\mathbf{q}_2(a,b)$ is a sum of products $\sum n_1^{a_1} n_2^{a_2} ...n_r^{a_r}$ taken over every partition of $\langle j,k \rangle$.

\begin{equation} \label{8.07.03}
\prod_{j,k\geq1}\left(1 + \frac{x^j y^k (1+x^j y^k)}{(1-x^j y^k)^3} \right)
=  \prod_{j,k\geq1} \left\{ 1 + 1^2x^j y^k + 2^2x^{2j} y^{2k} + 3^2x^{3j} y^{3k} + \ldots \right\}
\end{equation}
\begin{equation} \nonumber
 = \sum_{a,b\geq0} \mathbf{r}_2(a,b)x^a y^b.
\end{equation}
In this instance $\mathbf{r}_2(a,b)$ is a sum of products $\sum n_1^{2a_1} n_2^{2a_2} ...n_r^{2a_r}$ taken over every partition $(\langle n_1j,n_1k \rangle^{a_1} \langle n_2j,n_2k \rangle^{a_2} ... \langle n_rj,n_rk \rangle^{a_r})$ of $\langle j,k \rangle$.
\begin{equation} \label{8.07.04}
\prod_{j,k\geq1}\left(1+ \frac{x^j y^k}{1-x^{2j} y^{2k}}\right)
=  \prod_{j,k\geq1} \left\{ 1 + x^j y^k + 3x^{3j} y^{3k} + 5x^{5j} y^{5k} + \ldots \right\}
\end{equation}
\begin{equation} \nonumber
 = \sum_{a,b\geq0} \mathbf{s}_2(a,b)x^a y^b.
\end{equation}
In this case $\mathbf{x}_2(a,b)$ is a sum of products $\sum (2n_1+1)^{a_1} (2n_2+1)^{a_2} ...(2n_r+1)^{a_r}$ taken over every partition $(\langle n_1j,n_1k \rangle^{a_1} \langle n_3j,n_3k \rangle^{a_2} ... \langle n_{2r+1}j,n_{2r+1}k \rangle^{a_r})$ of $\langle j,k \rangle$.

     In the examples that follow here we use $\gcd(j,k)=1$ abbreviated to $(j,k)=1$. It will become evident that when applying these generating function forms to radial-from-origin regions over vector partitions, the theory of $\langle j,k \rangle$ where $(j,k)=1$ simplifies things.
\begin{equation} \label{8.07.01a}
\prod_{\substack{j,k\geq1; \\ (j,k)=1}}\frac{1}{1-x^j y^k}
=  \prod_{\substack{j,k\geq1; \\ (j,k)=1}} \left\{1 + x^j y^k + x^{2j} y^{2k} + x^{3j} y^{3k} + \ldots \right\}
= \sum_{a,b\geq0} \mathbf{v}_2(a,b)x^a y^b.
\end{equation}
Note that the curly-bracketed term in (\ref{8.07.01a}) contributes vector parts of rational gradient $j/k$ with $(j,k)=1$ in the first 2D quadrant exactly one part per partition. So $\mathbf{v}_2(a,b)$ is the number of partitions of $\left\langle a,b \right\rangle$ into vectors each part of distinct positive rational gradient. In traditional integer partition parlance this would say $\mathbf{v}_2(a,b)$ is the number of partitions into unrestricted Visible Point Vector (VPV) parts.

\begin{equation} \label{8.07.02a}
\prod_{\substack{j,k\geq1; \\ (j,k)=1}}\left(1+ \frac{x^j y^k}{(1-x^j y^k)^2}\right)
=  \prod_{\substack{j,k\geq1; \\ (j,k)=1}} \left\{ 1 + x^j y^k + 2x^{2j} y^{2k} + 3x^{3j} y^{3k} + \ldots \right\}
\end{equation}
\begin{equation} \nonumber
= \sum_{a,b\geq0} \mathbf{w}_2(a,b)x^a y^b.
\end{equation}

\begin{equation} \label{8.07.03a}
\prod_{\substack{j,k\geq1; \\ (j,k)=1}}\left(1 + \frac{x^j y^k (1+x^j y^k)}{(1-x^j y^k)^3} \right)
=  \prod_{\substack{j,k\geq1; \\ (j,k)=1}} \left\{ 1 + 1^2x^j y^k + 2^2x^{2j} y^{2k} + 3^2x^{3j} y^{3k} + \ldots \right\}
\end{equation}
\begin{equation} \nonumber
= \sum_{a,b\geq0} \mathbf{x}_2(a,b)x^a y^b.
\end{equation}

\begin{equation} \label{8.07.04a}
\prod_{\substack{j,k\geq1; \\ \gcd(j,k)=1}}\left(1+ \frac{x^j y^k}{1-x^{2j} y^{2k}}\right)
=  \prod_{\substack{j,k\geq1; \\ \gcd(j,k)=1}} \left\{ 1 + x^j y^k + 3x^{3j} y^{3k} + 5x^{5j} y^{5k} + \ldots \right\}
\end{equation}
\begin{equation} \nonumber
= \sum_{a,b\geq0} \mathbf{y}_2(a,b)x^a y^b.
\end{equation}

We also consider 3D vectors $\left\langle a,b,c \right\rangle$ where $0<a<b<c$. Again we mainly look into two types of such vectors for our vector partition aggregations:
\begin{enumerate}
  \item All of the 3D vectors $\left\langle a,b,c \right\rangle$ where $0<a<b<c$ are in the infinitely extended region radial from the origin. This on occasion means examination of defined "finite patches" of that radial region.
  \item All of the 3D vectors $\left\langle a,b,c \right\rangle$ where $0<a<b<c$ and $gcd(a,b,c)=1$ are in the infinitely extended region radial from the origin. This on occasion means examination of defined "finite patches" of that radial region.
\end{enumerate}

\section{Defining 2D \textit{Upper Visible Point Vectors} in origin-radial regions}

These lattice point vectors \textit{count} the visible from the origin points in the infinite radial region of the first quadrant bounded by the positive $y$ axis and the line $y=x$. The visible lattice points in 3D space are similarly a countable set, and in fact in nD space with $n$ a positive integer greater than unity. In 3D the set of visible points can be counted as shown below. In order to count the entire first 3D hyperquadrant it suffices to count the following vectors $\left\langle a,b,c \right\rangle$ such that $0<a<b<c$ has

A set of identities involving $n$ dimensional visible lattice points was discovered by Campbell (1994). The visible lattice points in the 2 dimensional first quadrant (see Weisstein, Eric W. "Visible Point." From MathWorld--A Wolfram Web Resource.

http://mathworld.wolfram.com/VisiblePoint.html) correspond to the countable set of rational numbers between 0 and 1, listed here as

\begin{equation}  \nonumber
\begin{array}{ccccccccc}
  \frac{1}{2} &   &   &   &   &   &   &   &   \\
  \frac{1}{3} & \frac{2}{3} &   &   &   &   &   &   &   \\
  \frac{1}{4} &   & \frac{3}{4} &   &   &   &   &   &   \\
  \frac{1}{5} & \frac{2}{5} & \frac{3}{5} & \frac{4}{5} &   &   &   &   &   \\
  \frac{1}{6} &   &   &   & \frac{5}{6} &   &   &   &   \\
  \frac{1}{7} & \frac{2}{7} & \frac{3}{7} & \frac{4}{7} & \frac{5}{7} & \frac{6}{7} &   &   &   \\
  \frac{1}{8} &   & \frac{3}{8} &   & \frac{5}{8} &   & \frac{7}{8} &   &   \\
  \frac{1}{9} & \frac{2}{9} &   & \frac{4}{9} & \frac{5}{9} &   & \frac{7}{9} & \frac{8}{9} &   \\
  \frac{1}{10} &   & \frac{3}{10}  &   &   &   & \frac{7}{10}  &   & \frac{9}{10} \\
     &   &  &   & etc.  &   & &   &
  \end{array}
\end{equation}

giving rise to the \textit{visible from the origin vector} (VPV) lattice points $\langle x,y \rangle$ with $x < y$ and $\gcd(x,y)=1$ in the first 2D quadrant.

\begin{equation}  \nonumber
\begin{array}{ccccccccc}
   &   &   &   & etc.  &   &   &   &  \\
  \langle1,10\rangle &   & \langle3,10\rangle  &   &   &   & \langle7,10\rangle  &   & \langle9,10\rangle \\
  \langle1,9\rangle & \langle2,9\rangle &   & \langle4,9\rangle & \langle5,9\rangle &   & \langle7,9\rangle & \langle8,9\rangle &   \\
  \langle1,8\rangle &   & \langle3,8\rangle &   & \langle5,8\rangle &   & \langle7,8\rangle &   &   \\
  \langle1,7\rangle & \langle2,7\rangle & \langle3,7\rangle & \langle4,7\rangle & \langle5,7\rangle & \langle6,7\rangle &   &   &   \\
  \langle1,6\rangle &   &   &   & \langle5,6\rangle &   &   &   &   \\
  \langle1,5\rangle & \langle2,5\rangle & \langle3,5\rangle & \langle4,5\rangle &   &   &   &   &   \\
  \langle1,4\rangle &   & \langle3,4\rangle &   &   &   &   &   &   \\
  \langle1,3\rangle & \langle2,3\rangle &   &   &   &   &   &   &   \\
  \langle1,2\rangle &   &   &   &   &   &   &   &
\end{array}
\end{equation}
These are the countable list of \textit{upper first quadrant} 2D Visible Point Vectors. That is, the VPVs in the infinitely extended radial region of the 2D first quadrant above the line $y=x$.


\section{Examples of 2D VPV finite generating functions}


\subsection{2D Distinct \textit{Upper VPV} Coefficients - Order 2}
\begin{equation}  \label{8.08}
  (1+xy^2) \quad  \quad \quad   and \quad  \quad  \quad (1-xy^2)
\end{equation}
 \begin{equation*}
\begin{tabular}  {c c c c c c}
  \multicolumn{1}{c|}{$3$}          &    &   &   &   & \multicolumn{1}{c}{$\sum r$}\\  \cline{6-6}
  \multicolumn{1}{c|}{$2$}  &  & \multicolumn{1}{c}{$1$}  &   &   & \multicolumn{1}{c}{$1$}\\ \cline{6-6}
  \multicolumn{1}{c|}{$1$} &  &  &   &   & \multicolumn{1}{c}{$0$} \\  \cline{6-6}
  \multicolumn{1}{c|}{$0$}  & \multicolumn{1}{c}{$1$}  &  &   &   & \multicolumn{1}{c}{$1$} \\ \cline{1-4} \cline{6-6}
  \multicolumn{1}{c|}{$y/x$}  & \multicolumn{1}{c}{$0$}  & \multicolumn{1}{c}{$1$}  & \multicolumn{1}{c}{$2$}  &   &  \\
                                  &    &   &   &   & \\
  \multicolumn{1}{c|}{$\sum c$}  & \multicolumn{1}{|c|}{$1$}  & \multicolumn{1}{|c|}{$1$}  & \multicolumn{1}{|c|}{$ $}  &   & \\
\end{tabular}
\quad  \quad  \quad
\begin{tabular}  {c c c c c c}
  \multicolumn{1}{c|}{$3$}          &    &   &   &   & \multicolumn{1}{c}{$\sum r$}\\  \cline{6-6}
  \multicolumn{1}{c|}{$2$}  &  & \multicolumn{1}{c}{$-1$}  &   &   & \multicolumn{1}{c}{$-1$}\\ \cline{6-6}
  \multicolumn{1}{c|}{$1$} &  &  &   &   & \multicolumn{1}{c}{$0$} \\  \cline{6-6}
  \multicolumn{1}{c|}{$0$}  & \multicolumn{1}{c}{$1$}  &  &   &   & \multicolumn{1}{c}{$1$} \\ \cline{1-4} \cline{6-6}
  \multicolumn{1}{c|}{$y/x$}  & \multicolumn{1}{c}{$0$}  & \multicolumn{1}{c}{$1$}  & \multicolumn{1}{c}{$2$}  &   &  \\
                                  &    &   &   &   & \\
  \multicolumn{1}{c|}{$\sum c$}  & \multicolumn{1}{|c|}{$1$}  & \multicolumn{1}{|c|}{$-1$}  & \multicolumn{1}{|c|}{$ $}  &   & \\
\end{tabular}
\end{equation*}


\subsection{2D Distinct \textit{Upper VPV} Coefficients - Order 3}
\begin{equation}  \label{8.09.01}
  (1+xy^2)(1+xy^3)(1+x^2y^3)  \quad \quad   and \quad  \quad  (1-xy^2)(1-xy^3)(1-x^2y^3)
\end{equation}
 \begin{equation*}
\begin{tabular}  {cccccccc}
    &  &   &   &   &   &   & \multicolumn{1}{c}{$\sum r$} \\  \cline{8-8}
  \multicolumn{1}{c|}{$8$}  &  &   &   &   & \multicolumn{1}{c}{$1$} &   & \multicolumn{1}{c}{$1$} \\  \cline{8-8}
  \multicolumn{1}{c|}{$7$}  &  &   &   &   &   &   & \multicolumn{1}{c}{$0$} \\  \cline{8-8}
  \multicolumn{1}{c|}{$6$}  &  &   &   & \multicolumn{1}{c}{$1$}  &   &   & \multicolumn{1}{c}{$1$} \\  \cline{8-8}
  \multicolumn{1}{c|}{$5$}  &  &   & \multicolumn{1}{c}{$1$}  & \multicolumn{1}{c}{$1$}  &   &   & \multicolumn{1}{c}{$2$} \\  \cline{8-8}
  \multicolumn{1}{c|}{$4$}  &  &   &   &   &   &   & \multicolumn{1}{c}{$0$} \\  \cline{8-8}
  \multicolumn{1}{c|}{$3$}  &  & \multicolumn{1}{c}{$1$}  & \multicolumn{1}{c}{$1$}  &   &   &   & \multicolumn{1}{c}{$2$} \\  \cline{8-8}
  \multicolumn{1}{c|}{$2$}  &  & \multicolumn{1}{c}{$1$}  &   &   &   &   & \multicolumn{1}{c}{$1$} \\ \cline{8-8}
  \multicolumn{1}{c|}{$1$}  &  &  &   &   &   &   & \multicolumn{1}{c}{$0$} \\  \cline{8-8}
  \multicolumn{1}{c|}{$0$}  & \multicolumn{1}{c}{$1$}  &  &   &   &   &   & \multicolumn{1}{c}{$1$}  \\ \cline{1-6} \cline{8-8}
  \multicolumn{1}{c|}{$y/x$}  & \multicolumn{1}{c}{$0$}  & \multicolumn{1}{c}{$1$}  & \multicolumn{1}{c}{$2$}  & \multicolumn{1}{c}{$3$}  & \multicolumn{1}{c}{$4$}  &  &   \\
                                  &    &   &   &   &   &   &  \\
  \multicolumn{1}{c|}{$\sum c$}  & \multicolumn{1}{|c|}{$1$}  & \multicolumn{1}{|c|}{$2$}  & \multicolumn{1}{|c|}{$2$}  & \multicolumn{1}{|c|}{$2$}  & \multicolumn{1}{|c|}{$1$} &   &  \\
\end{tabular}
\quad \quad \quad
\begin{tabular}  {cccccccc}
    &  &   &   &   &   &   & \multicolumn{1}{c}{$\sum r$} \\  \cline{8-8}
  \multicolumn{1}{c|}{$8$}  &  &   &   &   & \multicolumn{-1}{c}{$-1$} &   & \multicolumn{1}{c}{$-1$} \\  \cline{8-8}
  \multicolumn{1}{c|}{$7$}  &  &   &   &   &   &   & \multicolumn{1}{c}{$0$} \\  \cline{8-8}
  \multicolumn{1}{c|}{$6$}  &  &   &   & \multicolumn{1}{c}{$1$}  &   &   & \multicolumn{1}{c}{$1$} \\  \cline{8-8}
  \multicolumn{1}{c|}{$5$}  &  &   & \multicolumn{1}{c}{$1$}  & \multicolumn{1}{c}{$1$}  &   &   & \multicolumn{1}{c}{$2$} \\  \cline{8-8}
  \multicolumn{1}{c|}{$4$}  &  &   &   &   &   &   & \multicolumn{1}{c}{$0$} \\  \cline{8-8}
  \multicolumn{1}{c|}{$3$}  &  & \multicolumn{1}{c}{$-1$}  & \multicolumn{1}{c}{$-1$}  &   &   &   & \multicolumn{1}{c}{$-2$} \\  \cline{8-8}
  \multicolumn{1}{c|}{$2$}  &  & \multicolumn{1}{c}{$-1$}  &   &   &   &   & \multicolumn{1}{c}{$-1$} \\ \cline{8-8}
  \multicolumn{1}{c|}{$1$}  &  &  &   &   &   &   & \multicolumn{1}{c}{$0$} \\ \cline{8-8}
  \multicolumn{1}{c|}{$0$}  & \multicolumn{1}{c}{$1$}  &  &   &   &   &   & \multicolumn{1}{c}{$1$}  \\ \cline{1-6} \cline{8-8}
  \multicolumn{1}{c|}{$y/x$}  & \multicolumn{1}{c}{$0$}  & \multicolumn{1}{c}{$1$}  & \multicolumn{1}{c}{$2$}  & \multicolumn{1}{c}{$3$}  & \multicolumn{1}{c}{$4$}  &  &   \\
                                  &    &   &   &   &   &   &  \\
  \multicolumn{1}{c|}{$\sum c$}  & \multicolumn{1}{|c|}{$1$}  & \multicolumn{1}{|c|}{$-2$}  & \multicolumn{1}{|c|}{$0$}  & \multicolumn{1}{|c|}{$2$}  & \multicolumn{1}{|c|}{$-1$} &   &  \\
\end{tabular}
\end{equation*}

\bigskip
Note that row totals in the left side grid generate from
\begin{equation*}
  (1+xy^2)(1+xy^3)(1+x^2y^3) \quad \textmd{where} \quad  x=1,
\end{equation*}
expanding to
\begin{equation*}
  y^8 + y^6 + 2 y^5 + 2 y^3 + y^2 + 1.
\end{equation*}
Also row totals in the right side grid generate from
\begin{equation*}
  (1-xy^2)(1-xy^3)(1-x^2y^3) \quad \textmd{where} \quad  x=1,
\end{equation*}
expanding to
\begin{equation*}
  -y^8 + y^6 + 2 y^5 - 2 y^3 - y^2 + 1.
\end{equation*}

Likewise, the left side grid column totals generate from
\begin{equation*}
  (1+xy^2)(1+xy^3)(1+x^2y^3) \quad \textmd{where} \quad  y=1,
\end{equation*}
expanding to
\begin{equation*}
    x^4 + 2 x^3 + 2 x^2 + 2 x + 1;
\end{equation*}
whilst the right side grid column totals generate from
\begin{equation*}
  (1-xy^2)(1-xy^3)(1-x^2y^3) \quad \textmd{where} \quad  y=1,
\end{equation*}
expanding to
\begin{equation*}
    -x^4 + 2 x^3 - 2 x + 1.
\end{equation*}

\textbf{Combinatorial interpretation}: Consider the three vectors
\begin{equation}  \label{8.09.02}
   \begin{array}{cc}
     \langle 1,3 \rangle & \langle 2,3 \rangle   \\
     \langle 1,2 \rangle &
   \end{array}
\end{equation}
Let $p_{\,2, \, \textmd{up}}(\mathfrak{D}_{vpv}, \langle a,b \rangle)$, denote the number of partitions of $\langle a,b \rangle$ into distinct parts from (\ref{8.09.02}). Then
\begin{equation}  \label{8.09.03}
   \prod_{k=2}^{3} \prod_{\substack{j=1; \\ \gcd(j,k)=1}}^{k-1} (1+x^j y^k)
   = \sum_{b=2}^{8} \sum_{a=1}^{b-1} p_{\,2, \, \textmd{up}}(\mathfrak{D}_{vpv}, \langle a,b \rangle) x^a y^b,
\end{equation}
and each entry in the left side grid after (\ref{8.09.01}) gives each numerical value.

Let $p_{\,2, \, \textmd{up}}(\mathfrak{D}_{vpv,\textmd{odd}}, \langle a,b \rangle)$, denote the number of partitions of $\langle a,b \rangle$ into an odd number of distinct parts from (\ref{8.09.02}). Let $p_{\,2, \, \textmd{up}}(\mathfrak{D}_{vpv,\textmd{even}}, \langle a,b \rangle)$, denote the number of partitions of $\langle a,b \rangle$ into an even number of distinct parts from (\ref{8.09.02}).Then
\begin{equation}  \label{8.09.04}
   \prod_{k=2}^{3} \prod_{\substack{j=1; \\ \gcd(j,k)=1}}^{k-1} (1+x^j y^k)
   = \sum_{b=2}^{8} \sum_{a=1}^{b-1} [ p_{\,2, \, \textmd{up}}(\mathfrak{D}_{vpv,\textmd{odd}}, \langle a,b \rangle)
     - p_{\,2, \, \textmd{up}}(\mathfrak{D}_{vpv,\textmd{even}}, \langle a,b \rangle) ] x^a y^b,
\end{equation}
and each entry in the right side grid after (\ref{8.09.01}) gives each numerical value.



\subsection{2D Distinct \textit{Upper VPV} Coefficients - Order 4}
\begin{equation}  \label{8.10.01}
  (1+xy^2)(1+xy^3)(1+x^2y^3)(1+xy^4)(1+x^3y^4)
\end{equation}
 \begin{equation*}
\begin{tabular}  {cccccccccccc}
    &  &   &   &   &   &   &   &  &   &   & \multicolumn{1}{c}{$\sum r$} \\  \cline{12-12}
  \multicolumn{1}{c|}{$16$}  &   &   &   &  &  &   &   &   & \multicolumn{1}{c}{$1$} &   & \multicolumn{1}{c}{$1$} \\  \cline{12-12}
  \multicolumn{1}{c|}{$15$}  &   &   &   &  &  &   &   &   &  &   & \multicolumn{1}{c}{$0$} \\  \cline{12-12}
  \multicolumn{1}{c|}{$14$}  &   &   &   &  &  &   &   & \multicolumn{1}{c}{$1$}  &  &   & \multicolumn{1}{c}{$1$} \\  \cline{12-12}
  \multicolumn{1}{c|}{$13$}  &   &   &   &  &  &   & \multicolumn{1}{c}{$1$}  & \multicolumn{1}{c}{$1$}  &  &   & \multicolumn{1}{c}{$2$} \\  \cline{12-12}
  \multicolumn{1}{c|}{$12$}  &   &   &   &  &  & \multicolumn{1}{c}{$1$}  &  & \multicolumn{1}{c}{$1$}  &  &   & \multicolumn{1}{c}{$2$} \\  \cline{12-12}
  \multicolumn{1}{c|}{$11$}  &   &   &   &  &  & \multicolumn{1}{c}{$1$}  & \multicolumn{1}{c}{$1$}  &   &  &   & \multicolumn{1}{c}{$2$} \\  \cline{12-12}
  \multicolumn{1}{c|}{$10$}  &   &   &   &  & \multicolumn{1}{c}{$1$} & \multicolumn{1}{c}{$1$}  & \multicolumn{1}{c}{$1$}  &   &  &   & \multicolumn{1}{c}{$3$} \\  \cline{12-12}
  \multicolumn{1}{c|}{$9$}  &   &   &   & \multicolumn{1}{c}{$1$} & \multicolumn{1}{c}{$1$} & \multicolumn{1}{c}{$1$}  & \multicolumn{1}{c}{$1$}  &   &  &   & \multicolumn{1}{c}{$4$} \\  \cline{12-12}
  \multicolumn{1}{c|}{$8$}  &   &   &   &  & \multicolumn{1}{c}{$2$} &   &   &   &  &   & \multicolumn{1}{c}{$2$} \\  \cline{12-12}
  \multicolumn{1}{c|}{$7$}  &   &   & \multicolumn{1}{c}{$1$}  & \multicolumn{1}{c}{$1$} & \multicolumn{1}{c}{$1$} & \multicolumn{1}{c}{$1$}  &   &   &   &   & \multicolumn{1}{c}{$4$} \\  \cline{12-12}
  \multicolumn{1}{c|}{$6$}  &  &   & \multicolumn{1}{c}{$1$}  & \multicolumn{1}{c}{$1$}  & \multicolumn{1}{c}{$1$}  &   &   &  &   &   & \multicolumn{1}{c}{$3$} \\  \cline{12-12}
  \multicolumn{1}{c|}{$5$}  &  &   & \multicolumn{1}{c}{$1$}  & \multicolumn{1}{c}{$1$}  &   &   &   &  &   &   & \multicolumn{1}{c}{$2$} \\  \cline{12-12}
  \multicolumn{1}{c|}{$4$}  &  & \multicolumn{1}{c}{$1$}  &   & \multicolumn{1}{c}{$1$} &  &   &   &   &   &   & \multicolumn{1}{c}{$2$} \\  \cline{12-12}
  \multicolumn{1}{c|}{$3$}  &  & \multicolumn{1}{c}{$1$}  & \multicolumn{1}{c}{$1$}  &   &   &   &  &   &   &   & \multicolumn{1}{c}{$2$} \\  \cline{12-12}
  \multicolumn{1}{c|}{$2$}  &  & \multicolumn{1}{c}{$1$} &   &   &   &  &   &   &   &   & \multicolumn{1}{c}{$1$} \\ \cline{12-12}
  \multicolumn{1}{c|}{$1$}  &   &   &   &  &  &  &   &   &   &   & \multicolumn{1}{c}{$0$} \\  \cline{12-12}
  \multicolumn{1}{c|}{$0$}  & \multicolumn{1}{c}{$1$}&   &   &   &   &  &   &   &   &   & \multicolumn{1}{c}{$1$}  \\ \cline{1-10} \cline{12-12}
  \multicolumn{1}{c|}{$y/x$}  & \multicolumn{1}{c}{$0$}  & \multicolumn{1}{c}{$1$}  & \multicolumn{1}{c}{$2$}  & \multicolumn{1}{c}{$3$}  & \multicolumn{1}{c}{$4$}& \multicolumn{1}{c}{$5$}  & \multicolumn{1}{c}{$6$}  & \multicolumn{1}{c}{$7$}  & \multicolumn{1}{c}{$8$} &  &   \\
                                  &    &   &   &   &   &   &  &   &   &   &  \\
  \multicolumn{1}{c|}{$\sum c$}  & \multicolumn{1}{|c|}{$1$}  & \multicolumn{1}{|c|}{$3$}  & \multicolumn{1}{|c|}{$4$}  & \multicolumn{1}{|c|}{$5$}  & \multicolumn{1}{|c|}{$6$} & \multicolumn{1}{|c|}{$5$}  & \multicolumn{1}{|c|}{$4$}  & \multicolumn{1}{|c|}{$3$}  & \multicolumn{1}{|c|}{$1$}  &  &  \\
\end{tabular}
\end{equation*}

\bigskip
Note that column totals generate from
\begin{equation*}
  (1+xy^2)(1+xy^3)(1+x^2y^3)(1+xy^4)(1+x^3y^4) \quad \textmd{where} \quad  y=1,
\end{equation*}
expanding to
\begin{equation*}
  x^8 + 3 x^7 + 4 x^6 + 5 x^5 + 6 x^4 + 5 x^3 + 4 x^2 + 3 x + 1.
\end{equation*}

Likewise, row totals generate from
\begin{equation*}
  (1+xy^2)(1+xy^3)(1+x^2y^3)(1+xy^4)(1+x^3y^4) \quad \textmd{where} \quad  x=1,
\end{equation*}
expanding to
\begin{equation*}
    y^{16} + y^{14} + 2 y^{13} + 2 y^{12} + 2 y^{11} + 3 y^{10} + 4 y^9
\end{equation*}
\begin{equation*}
  + 2 y^8 + 4 y^7 + 3 y^6 + 2 y^5 + 2 y^4 + 2 y^3 + y^2 + 1.
\end{equation*}

\textbf{Combinatorial interpretation}: Consider the aggregate of vectors
\begin{equation}  \label{8.10.02}
   \begin{array}{ccc}
     \langle 1,4 \rangle &   & \langle 3,4 \rangle   \\
     \langle 1,3 \rangle & \langle 2,3 \rangle &     \\
     \langle 1,2 \rangle &   &
   \end{array}
\end{equation}
Let $p_{\,2, \, \textmd{up}}(\mathfrak{D}_{vpv}, \langle a,b \rangle)$, denote the number of partitions of $\langle a,b \rangle$ into distinct parts from (\ref{8.10.02}). Then
\begin{equation}  \label{8.10.03}
   \prod_{k=2}^{4} \prod_{\substack{j=1; \\ \gcd(j,k)=1}}^{k-1} (1+x^j y^k)
   = \sum_{b=2}^{20} \sum_{a=1}^{b-1} p_{\,2, \, \textmd{up}}(\mathfrak{D}_{vpv}, \langle a,b \rangle) x^a y^b,
\end{equation}
and each entry in the grid after (\ref{8.10.01}) above gives each numerical value.


We continue in this way by considering the generating function
\begin{equation}  \label{8.11.01}
  (1-xy^2)(1-xy^3)(1-x^2y^3)(1-xy^4)(1-x^3y^4)
\end{equation}
which expanded gives us the grid,
 \begin{equation*}
\begin{tabular}  {cccccccccccc}
    &  &   &   &   &   &   &   &  &   &   & \multicolumn{1}{c}{$\sum r$} \\  \cline{12-12}
  \multicolumn{1}{c|}{$16$}  &   &   &   &  &  &   &   &   & \multicolumn{1}{c}{$-1$} &   & \multicolumn{1}{c}{$-1$} \\  \cline{12-12}
  \multicolumn{1}{c|}{$15$}  &   &   &   &  &  &   &   &   &  &   & \multicolumn{1}{c}{$0$} \\  \cline{12-12}
  \multicolumn{1}{c|}{$14$}  &   &   &   &  &  &   &   & \multicolumn{1}{c}{$1$}  &  &   & \multicolumn{1}{c}{$1$} \\  \cline{12-12}
  \multicolumn{1}{c|}{$13$}  &   &   &   &  &  &   & \multicolumn{1}{c}{$1$}  & \multicolumn{1}{c}{$1$}  &  &   & \multicolumn{1}{c}{$2$} \\  \cline{12-12}
  \multicolumn{1}{c|}{$12$}  &   &   &   &  &  & \multicolumn{1}{c}{$1$}  &  & \multicolumn{1}{c}{$1$}  &  &   & \multicolumn{1}{c}{$2$} \\  \cline{12-12}
  \multicolumn{1}{c|}{$11$}  &   &   &   &  &  & \multicolumn{1}{c}{$-1$}  & \multicolumn{1}{c}{$-1$}  &   &  &   & \multicolumn{1}{c}{$-2$} \\  \cline{12-12}
  \multicolumn{1}{c|}{$10$}  &   &   &   &  & \multicolumn{1}{c}{$-1$} & \multicolumn{1}{c}{$-1$}  & \multicolumn{1}{c}{$-1$}  &   &  &   & \multicolumn{1}{c}{$-3$} \\  \cline{12-12}
  \multicolumn{1}{c|}{$9$}  &   &   &   & \multicolumn{1}{c}{$-1$} & \multicolumn{1}{c}{$-1$} & \multicolumn{1}{c}{$-1$}  & \multicolumn{1}{c}{$-1$}  &   &  &   & \multicolumn{1}{c}{$-4$} \\  \cline{12-12}
  \multicolumn{1}{c|}{$8$}  &   &   &   &  & \multicolumn{1}{c}{$2$} &   &   &   &  &   & \multicolumn{1}{c}{$2$} \\  \cline{12-12}
  \multicolumn{1}{c|}{$7$}  &   &   & \multicolumn{1}{c}{$1$}  & \multicolumn{1}{c}{$1$} & \multicolumn{1}{c}{$1$} & \multicolumn{1}{c}{$1$}  &   &   &   &   & \multicolumn{1}{c}{$4$} \\  \cline{12-12}
  \multicolumn{1}{c|}{$6$}  &  &   & \multicolumn{1}{c}{$1$}  & \multicolumn{1}{c}{$1$}  & \multicolumn{1}{c}{$1$}  &   &   &  &   &   & \multicolumn{1}{c}{$3$} \\  \cline{12-12}
  \multicolumn{1}{c|}{$5$}  &  &   & \multicolumn{1}{c}{$1$}  & \multicolumn{1}{c}{$1$}  &   &   &   &  &   &   & \multicolumn{1}{c}{$2$} \\  \cline{12-12}
  \multicolumn{1}{c|}{$4$}  &  & \multicolumn{1}{c}{$-1$}  &   & \multicolumn{1}{c}{$-1$} &  &   &   &   &   &   & \multicolumn{1}{c}{$-2$} \\  \cline{12-12}
  \multicolumn{1}{c|}{$3$}  &  & \multicolumn{1}{c}{$-1$}  & \multicolumn{1}{c}{$-1$}  &   &   &   &  &   &   &   & \multicolumn{1}{c}{$-2$} \\  \cline{12-12}
  \multicolumn{1}{c|}{$2$}  &  & \multicolumn{1}{c}{$-1$} &   &   &   &  &   &   &   &   & \multicolumn{1}{c}{$-1$} \\ \cline{12-12}
  \multicolumn{1}{c|}{$1$}  &   &   &   &  &  &  &   &   &   &   & \multicolumn{1}{c}{$0$} \\  \cline{12-12}
  \multicolumn{1}{c|}{$0$}  & \multicolumn{1}{c}{$1$}&   &   &   &   &  &   &   &   &   & \multicolumn{1}{c}{$1$}  \\ \cline{1-10} \cline{12-12}
  \multicolumn{1}{c|}{$y/x$}  & \multicolumn{1}{c}{$0$}  & \multicolumn{1}{c}{$1$}  & \multicolumn{1}{c}{$2$}  & \multicolumn{1}{c}{$3$}  & \multicolumn{1}{c}{$4$}& \multicolumn{1}{c}{$5$}  & \multicolumn{1}{c}{$6$}  & \multicolumn{1}{c}{$7$}  & \multicolumn{1}{c}{$8$} &  &   \\
                                  &    &   &   &   &   &   &  &   &   &   &  \\
  \multicolumn{1}{c|}{$\sum c$}  & \multicolumn{1}{|c|}{$1$}  & \multicolumn{1}{|c|}{$-3$}  & \multicolumn{1}{|c|}{$2$}  & \multicolumn{1}{|c|}{$1$}  & \multicolumn{1}{|c|}{$2$} & \multicolumn{1}{|c|}{$-1$}  & \multicolumn{1}{|c|}{$-2$}  & \multicolumn{1}{|c|}{$3$}  & \multicolumn{1}{|c|}{$-1$}  &  &  \\
\end{tabular}
\end{equation*}

\bigskip
Note that column totals generate from
\begin{equation*}
  (1-xy^2)(1-xy^3)(1-x^2y^3)(1-xy^4)(1-x^3y^4) \quad \textmd{where} \quad  y=1,
\end{equation*}
expanding to
\begin{equation*}
  -x^8 + 3 x^7 - 2 x^6 - x^5 + x^3 + 2 x^2 - 3 x + 1.
\end{equation*}

\newpage
Likewise, row totals generate from
\begin{equation*}
  (1-xy^2)(1-xy^3)(1-x^2y^3)(1-xy^4)(1-x^3y^4) \quad \textmd{where} \quad  x=1,
\end{equation*}
expanding to
\begin{equation*}
    -y^{16} + y^{14} + 2 y^{13} + 2 y^{12} - 2 y^{11} - 3 y^{10} - 4 y^9
  + 4 y^7 + 3 y^6 + 2 y^5 - 2 y^4 - 2 y^3 - y^2 + 1.
\end{equation*}

\textbf{Combinatorial interpretation}: Consider the aggregate of vectors
\begin{equation}  \label{8.11.02}
   \begin{array}{ccc}
     \langle 1,4 \rangle &   & \langle 3,4 \rangle   \\
     \langle 1,3 \rangle & \langle 2,3 \rangle &     \\
     \langle 1,2 \rangle &   &
   \end{array}
\end{equation}
Let $p_{\,2, \, \textmd{up}}(\mathfrak{D}_{(vpv,\textmd{odd})}; \langle a,b \rangle)$, denote the number of partitions of $\langle a,b \rangle$ into distinct odd number of parts from (\ref{8.10.02}). Also let $p_{\,2, \, \textmd{up}}(\mathfrak{D}_{(vpv,\textmd{even})}; \langle a,b \rangle)$, denote the number of partitions of $\langle a,b \rangle$ into distinct even number of parts from (\ref{8.10.02}).Then
\begin{equation}  \label{8.11.03}
   \prod_{k=2}^{4} \prod_{\substack{j=1; \\ \gcd(j,k)=1}}^{k-1} (1-x^j y^k)
   = \sum_{b=2}^{20} \sum_{a=1}^{b-1} [p_{\,2, \, \textmd{up}}(\mathfrak{D}_{(vpv,\textmd{odd})}; \langle a,b \rangle)
   -  p_{\,2, \, \textmd{up}}(\mathfrak{D}_{(vpv,\textmd{even})}; \langle a,b \rangle)] x^a y^b,
\end{equation}
and each entry in the grid after (\ref{8.11.01}) above gives each numerical value.



\subsection{2D Distinct \textit{Upper VPV} Coefficients - Order 5}
\begin{equation}  \nonumber
  (1+xy^2)(1+xy^3)(1+x^2y^3)(1+xy^4)(1+x^3y^4)(1+xy^5)(1+x^2y^5)(1+x^3y^5)(1+x^4y^5)
\end{equation}
generates the grid
 \begin{equation*}  \tiny{
\begin{tabular}  {cccccccccccccccccccccc}
    &  &   &   &   &   &   &   &   &   &  &  &   &   &   &  &  &  &  &   &   & \multicolumn{1}{c}{$\sum r$} \\  \cline{22-22}
\multicolumn{1}{c|}{$36$}  &   &   &   &  &  &   &   &   &  &  &   &   &   &  &  &  &   &   & \multicolumn{1}{c}{$1$} &   & \multicolumn{1}{c}{$1$} \\  \cline{22-22}
\multicolumn{1}{c|}{$35$}  &   &   &   &  &  &   &   &   &  &  &   &   &   &  &  &  &   &   &  &   & \multicolumn{1}{c}{$0$} \\  \cline{22-22}
\multicolumn{1}{c|}{$34$}  &   &   &   &  &  &   &   &   &  &  &   &   &   &  &  &  &   & \multicolumn{1}{c}{$1$}   & &   & \multicolumn{1}{c}{$1$} \\  \cline{22-22}
\multicolumn{1}{c|}{$33$}  &   &   &   &  &  &   &   &   &  &  &   &   &   &  &  &  & \multicolumn{1}{c}{$1$}   & \multicolumn{1}{c}{$1$}   & &   & \multicolumn{1}{c}{$2$} \\  \cline{22-22}
\multicolumn{1}{c|}{$32$}  &   &   &   &  &  &   &   &   &  &  &   &   &   &  &  & \multicolumn{1}{c}{$1$}  &   & \multicolumn{1}{c}{$1$}   &   &   & \multicolumn{1}{c}{$2$} \\  \cline{22-22}
\multicolumn{1}{c|}{$31$}  &   &   &   &  &  &   &   &   &  &  &   &   &   &   & \multicolumn{1}{c}{$1$}  & \multicolumn{1}{c}{$2$}  & \multicolumn{1}{c}{$2$}   & \multicolumn{1}{c}{$1$}   & &   & \multicolumn{1}{c}{$6$} \\  \cline{22-22}
\multicolumn{1}{c|}{$30$}  &   &   &   &  &  &   &   &   &  &  &   &   &   &  & \multicolumn{1}{c}{$1$}  & \multicolumn{1}{c}{$1$}  & \multicolumn{1}{c}{$1$}   &   &   &   & \multicolumn{1}{c}{$3$} \\  \cline{22-22}
\multicolumn{1}{c|}{$29$}  &   &   &   &  &  &   &   &   &  &  &   &   &   & \multicolumn{1}{c}{$2$}  & \multicolumn{1}{c}{$2$}  & \multicolumn{1}{c}{$2$}  & \multicolumn{1}{c}{$2$}   &   & &   & \multicolumn{1}{c}{$8$} \\  \cline{22-22}
\multicolumn{1}{c|}{$28$}  &   &   &   &  &  &   &   &   &  &  &   &   & \multicolumn{1}{c}{$1$}   & \multicolumn{1}{c}{$2$}  & \multicolumn{1}{c}{$4$}  & \multicolumn{1}{c}{$2$}  & \multicolumn{1}{c}{$1$}   &   & &   & \multicolumn{1}{c}{$10$} \\  \cline{22-22}
\multicolumn{1}{c|}{$27$}  &   &   &   &  &  &   &   &   &  &  &   & \multicolumn{1}{c}{$1$} & \multicolumn{1}{c}{$2$}   & \multicolumn{1}{c}{$3$} & \multicolumn{1}{c}{$3$}  & \multicolumn{1}{c}{$2$}  & \multicolumn{1}{c}{$1$}   &   & &   & \multicolumn{1}{c}{$12$} \\  \cline{22-22}
\multicolumn{1}{c|}{$26$}  &   &   &   &  &  &   &   &   &  &  &   & \multicolumn{1}{c}{$2$} & \multicolumn{1}{c}{$4$}   & \multicolumn{1}{c}{$5$} & \multicolumn{1}{c}{$4$}  & \multicolumn{1}{c}{$2$}  &   &   & &   & \multicolumn{1}{c}{$17$} \\  \cline{22-22}
\multicolumn{1}{c|}{$25$}  &   &   &   &  &  &   &   &   &  &  & \multicolumn{1}{c}{$1$}   & \multicolumn{1}{c}{$2$}   & \multicolumn{1}{c}{$4$}   & \multicolumn{1}{c}{$4$}  & \multicolumn{1}{c}{$2$}  & \multicolumn{1}{c}{$1$} &   &   & &   & \multicolumn{1}{c}{$14$} \\  \cline{22-22}
\multicolumn{1}{c|}{$24$}  &   &   &   &  &  &   &   &   &  &\multicolumn{1}{c}{$1$}   &\multicolumn{1}{c}{$3$}    &\multicolumn{1}{c}{$5$}    &\multicolumn{1}{c}{$6$}    &\multicolumn{1}{c}{$5$}   &\multicolumn{1}{c}{$3$}   &\multicolumn{1}{c}{$1$}  &   &   & &   & \multicolumn{1}{c}{$24$} \\  \cline{22-22}
\multicolumn{1}{c|}{$23$}  &   &   &   &  &  &   &   &   &  &\multicolumn{1}{c}{$1$}   &\multicolumn{1}{c}{$4$}    &\multicolumn{1}{c}{$6$}    &\multicolumn{1}{c}{$6$}    &\multicolumn{1}{c}{$4$}   &\multicolumn{1}{c}{$1$}   &  &   &   & &   & \multicolumn{1}{c}{$22$} \\  \cline{22-22}
\multicolumn{1}{c|}{$22$}  &   &   &   &  &  &   &   &   &\multicolumn{1}{c}{$2$}   &\multicolumn{1}{c}{$3$}   &\multicolumn{1}{c}{$6$}    &\multicolumn{1}{c}{$7$}    &\multicolumn{1}{c}{$6$}    &\multicolumn{1}{c}{$3$}   &\multicolumn{1}{c}{$2$}   &  &   &   & &   & \multicolumn{1}{c}{$29$} \\  \cline{22-22}
\multicolumn{1}{c|}{$21$}  &   &   &   &  &  &   &   &   &\multicolumn{1}{c}{$2$}   &\multicolumn{1}{c}{$5$}   &\multicolumn{1}{c}{$7$}    &\multicolumn{1}{c}{$7$}    &\multicolumn{1}{c}{$5$}    &\multicolumn{1}{c}{$2$}   &  &  &   &   & &   & \multicolumn{1}{c}{$28$} \\  \cline{22-22}
\multicolumn{1}{c|}{$20$}  &   &   &   &  &  &   &   &\multicolumn{1}{c}{$1$}    &\multicolumn{1}{c}{$3$}   &\multicolumn{1}{c}{$6$}   &\multicolumn{1}{c}{$7$}    &\multicolumn{1}{c}{$6$}    &\multicolumn{1}{c}{$3$}    &\multicolumn{1}{c}{$1$}   &  &  &   &   & &   & \multicolumn{1}{c}{$27$} \\  \cline{22-22}
\multicolumn{1}{c|}{$19$}  &   &   &   &  &  &   &\multicolumn{1}{c}{$1$}    &\multicolumn{1}{c}{$3$}    &\multicolumn{1}{c}{$6$}   &\multicolumn{1}{c}{$8$}   &\multicolumn{1}{c}{$8$}    &\multicolumn{1}{c}{$6$}    &\multicolumn{1}{c}{$3$}    &\multicolumn{1}{c}{$1$}   &  &  &   &   & &   & \multicolumn{1}{c}{$36$} \\  \cline{22-22}
\multicolumn{1}{c|}{$18$}  &   &   &   &  &  &   &   &\multicolumn{1}{c}{$4$}    &\multicolumn{1}{c}{$6$}   &\multicolumn{1}{c}{$8$}   &\multicolumn{1}{c}{$6$}    &\multicolumn{1}{c}{$4$}    &   &  &  &  &   &   & &   & \multicolumn{1}{c}{$28$} \\  \cline{22-22}
\multicolumn{1}{c|}{$17$}  &   &   &   &  &  &\multicolumn{1}{c}{$1$}    &\multicolumn{1}{c}{$3$}    &\multicolumn{1}{c}{$6$}   &\multicolumn{1}{c}{$8$}   &\multicolumn{1}{c}{$8$}    &\multicolumn{1}{c}{$6$}    &\multicolumn{1}{c}{$3$}    &\multicolumn{1}{c}{$1$}  & &  &  &   &   & &   & \multicolumn{1}{c}{$36$} \\  \cline{22-22}
\multicolumn{1}{c|}{$16$}  &   &   &   &  &  & \multicolumn{1}{c}{$1$}    &\multicolumn{1}{c}{$3$}   &\multicolumn{1}{c}{$6$}   &\multicolumn{1}{c}{$7$}    &\multicolumn{1}{c}{$6$}    &\multicolumn{1}{c}{$3$}    &\multicolumn{1}{c}{$1$}  & & &  &  &   &   & &   & \multicolumn{1}{c}{$27$} \\  \cline{22-22}
  \multicolumn{1}{c|}{$15$}  &   &   &   &  &  &\multicolumn{1}{c}{$2$}   &\multicolumn{1}{c}{$5$}   &\multicolumn{1}{c}{$7$}    &\multicolumn{1}{c}{$7$}    &\multicolumn{1}{c}{$5$}    &\multicolumn{1}{c}{$2$}   &  &   &   &   &  &  &  &  &   & \multicolumn{1}{c}{$28$} \\  \cline{22-22}
  \multicolumn{1}{c|}{$14$}  &   &   &   &  &\multicolumn{1}{c}{$2$}   &\multicolumn{1}{c}{$3$}   &\multicolumn{1}{c}{$6$}    &\multicolumn{1}{c}{$7$}    &\multicolumn{1}{c}{$6$}    &\multicolumn{1}{c}{$3$}   &\multicolumn{1}{c}{$2$}   &  &  &   &   &   &  &  &  &   & \multicolumn{1}{c}{$29$} \\  \cline{22-22}
  \multicolumn{1}{c|}{$13$}  &   &   &   &  &\multicolumn{1}{c}{$1$}   &\multicolumn{1}{c}{$4$}    &\multicolumn{1}{c}{$6$}    &\multicolumn{1}{c}{$6$}    &\multicolumn{1}{c}{$4$}   &\multicolumn{1}{c}{$1$}   &   &   &  &  &   &   &   &  &  &   & \multicolumn{1}{c}{$22$} \\  \cline{22-22}
  \multicolumn{1}{c|}{$12$}  &   &   &   &\multicolumn{1}{c}{$1$}   &\multicolumn{1}{c}{$3$}    &\multicolumn{1}{c}{$5$}    &\multicolumn{1}{c}{$6$}    &\multicolumn{1}{c}{$5$}   &\multicolumn{1}{c}{$3$}   &\multicolumn{1}{c}{$1$}  &   &   &  &  &   &   &   &  &  &   & \multicolumn{1}{c}{$24$} \\  \cline{22-22}
  \multicolumn{1}{c|}{$11$}  &   &   &   & \multicolumn{1}{c}{$1$}   & \multicolumn{1}{c}{$2$}   & \multicolumn{1}{c}{$4$}   & \multicolumn{1}{c}{$4$}  & \multicolumn{1}{c}{$2$}  & \multicolumn{1}{c}{$1$} &   &   &   &  &  &   &   &   &  &  &   & \multicolumn{1}{c}{$14$} \\  \cline{22-22}
  \multicolumn{1}{c|}{$10$}  &   &   &   & \multicolumn{1}{c}{$2$} & \multicolumn{1}{c}{$4$}   & \multicolumn{1}{c}{$5$} & \multicolumn{1}{c}{$4$}  & \multicolumn{1}{c}{$2$}  &  &   &   &   &  &  &   &   &   &  &  &   & \multicolumn{1}{c}{$17$} \\  \cline{22-22}
  \multicolumn{1}{c|}{$9$}  &   &   &\multicolumn{1}{c}{$1$} & \multicolumn{1}{c}{$2$}   & \multicolumn{1}{c}{$3$} & \multicolumn{1}{c}{$3$}  & \multicolumn{1}{c}{$2$}  & \multicolumn{1}{c}{$1$}   &  &   &   &   &  &  &   &   &   &  &  &   & \multicolumn{1}{c}{$12$} \\  \cline{22-22}
  \multicolumn{1}{c|}{$8$}  &   &   & \multicolumn{1}{c}{$1$}   & \multicolumn{1}{c}{$2$}  & \multicolumn{1}{c}{$4$}  & \multicolumn{1}{c}{$2$}  & \multicolumn{1}{c}{$1$}   &   &   &   &   &   &  &  &   &   &   &  &  &   & \multicolumn{1}{c}{$10$} \\  \cline{22-22}
  \multicolumn{1}{c|}{$7$}  &   &   & \multicolumn{1}{c}{$2$}  & \multicolumn{1}{c}{$2$}  & \multicolumn{1}{c}{$2$}  & \multicolumn{1}{c}{$2$}   &   &   &   &   &   &   &  &  &   &   &   &  &  &   & \multicolumn{1}{c}{$8$} \\  \cline{22-22}
  \multicolumn{1}{c|}{$6$}  &  &   & \multicolumn{1}{c}{$1$}  & \multicolumn{1}{c}{$1$}  & \multicolumn{1}{c}{$1$}  &   &   &  &   &   &   &   &  &  &   &   &   &  &  &   & \multicolumn{1}{c}{$3$} \\  \cline{22-22}
  \multicolumn{1}{c|}{$5$}  &  & \multicolumn{1}{c}{$1$}  & \multicolumn{1}{c}{$2$}  & \multicolumn{1}{c}{$2$}   & \multicolumn{1}{c}{$1$}   &   &   &  &   &   &   &   &  &  &   &   &   &  &  &   & \multicolumn{1}{c}{$6$} \\  \cline{22-22}
  \multicolumn{1}{c|}{$4$}  &  & \multicolumn{1}{c}{$1$}  &   & \multicolumn{1}{c}{$1$} &  &   &   &   &   &   &   &   &  &  &   &   &   &  &  &   & \multicolumn{1}{c}{$2$} \\  \cline{22-22}
  \multicolumn{1}{c|}{$3$}  &  & \multicolumn{1}{c}{$1$}  & \multicolumn{1}{c}{$1$}  &   &   &   &  &   &   &   &   &   &  &  &   &   &   &  &  &   & \multicolumn{1}{c}{$2$} \\  \cline{22-22}
  \multicolumn{1}{c|}{$2$}  &  & \multicolumn{1}{c}{$1$} &   &   &   &  &   &   &   &   &   &   &  &  &   &   &   &  &  &   & \multicolumn{1}{c}{$1$} \\ \cline{22-22}
  \multicolumn{1}{c|}{$1$}  &   &   &   &  &  &  &   &   &   &   &   &   &  &  &   &   &   &  &  &   & \multicolumn{1}{c}{$0$} \\  \cline{22-22}
  \multicolumn{1}{c|}{$0$}  & \multicolumn{1}{c}{$1$}&   &   &   &   &  &   &   &   &   &   &  &  &   &   &   &  &  &   &   & \multicolumn{1}{c}{$1$}  \\ \cline{1-20} \cline{22-22}
  \multicolumn{1}{c|}{$y/x$}  & \multicolumn{1}{c}{$0$}  & \multicolumn{1}{c}{$1$}  & \multicolumn{1}{c}{$2$}  & \multicolumn{1}{c}{$3$}  & \multicolumn{1}{c}{$4$}& \multicolumn{1}{c}{$5$}  & \multicolumn{1}{c}{$6$}  & \multicolumn{1}{c}{$7$}  & \multicolumn{1}{c}{$8$} & \multicolumn{1}{c}{$9$} & \multicolumn{1}{c}{$10$} & \multicolumn{1}{c}{$11$} & \multicolumn{1}{c}{$12$} & \multicolumn{1}{c}{$13$} & \multicolumn{1}{c}{$14$} & \multicolumn{1}{c}{$15$} & \multicolumn{1}{c}{$16$} & \multicolumn{1}{c}{$17$} & \multicolumn{1}{c}{$18$} &  &   \\
                                  &    &   &   &   &   &   &  &   &   &   &   &   &  &  &   &   &   &  &  &   &  \\
  \multicolumn{1}{c|}{$\sum c$}  & \multicolumn{1}{|c|}{$1$}  & \multicolumn{1}{|c|}{$4$}  & \multicolumn{1}{|c|}{$8$}  & \multicolumn{1}{|c|}{$14$}  & \multicolumn{1}{|c|}{$23$} & \multicolumn{1}{|c|}{$32$}  & \multicolumn{1}{|c|}{$41$}  & \multicolumn{1}{|c|}{$50$}  & \multicolumn{1}{|c|}{$55$}  & \multicolumn{1}{|c|}{$56$}  & \multicolumn{1}{|c|}{$55$}  & \multicolumn{1}{|c|}{$50$}  & \multicolumn{1}{|c|}{$41$}  & \multicolumn{1}{|c|}{$32$}  & \multicolumn{1}{|c|}{$23$} & \multicolumn{1}{|c|}{$14$} & \multicolumn{1}{|c|}{$8$} & \multicolumn{1}{|c|}{$4$} & \multicolumn{1}{|c|}{$1$} & &  \\
\end{tabular} }
\end{equation*}

\bigskip
Note that column totals generate from
\begin{equation*}
  (1+xy^2)(1+xy^3)(1+x^2y^3)(1+xy^4)(1+x^3y^4)(1+xy^5)(1+x^2y^5)(1+x^3y^5)(1+x^4y^5)
\end{equation*}
where $y=1$, expanding to
\begin{equation*}
1 + 4 x + 8 x^2 + 14 x^3 + 23 x^4 + 32 x^5 + 41 x^6 + 50 x^7 + 55 x^8 + 56 x^9 + 55 x^{10} + 50 x^{11}
\end{equation*}
\begin{equation*}
    + 41 x^{12} + 32 x^{13} + 23 x^{14} + 14 x^{15} + 8 x^{16} + 4 x^{17} + x^{18}.
\end{equation*}

Likewise, row totals generate from
\begin{equation*}
  (1+xy^2)(1+xy^3)(1+x^2y^3)(1+xy^4)(1+x^3y^4)(1+xy^5)(1+x^2y^5)(1+x^3y^5)(1+x^4y^5)
\end{equation*}
where $x=1$, expanding to
\begin{equation*}
  1 + y^2 + 2 y^3 + 2 y^4 + 6 y^5 + 3 y^6 + 8 y^7 + 10 y^8 + 12 y^9 + 17 y^{10} + 14 y^{11}
\end{equation*}
\begin{equation*}
 + 24 y^{12} + 22 y^{13} + 29 y^{14} + 28 y^{15} + 27 y^{16} + 36 y^{17} + 28 y^{18} + 36 y^{19}
\end{equation*}
\begin{equation*}
 + 27 y^{20} + 28 y^{21} + 29 y^{22} + 22 y^{23} + 24 y^{24} + 14 y^{25} + 17 y^{26} + 12 y^{27}
\end{equation*}
\begin{equation*}
  + 10 y^{28} + 8 y^{29} + 3 y^{30} + 6 y^{31} + 2 y^{32} + 2 y^{33} + y^{34} + y^{36}.
\end{equation*}

\bigskip

\textbf{Combinatorial interpretation}: Consider the aggregate of vectors
\begin{equation}  \label{8.12.01}
   \begin{array}{cccc}
     \langle 1,5 \rangle & \langle 2,5 \rangle & \langle 3,5 \rangle & \langle 4,5 \rangle  \\
     \langle 1,4 \rangle &                     & \langle 3,4 \rangle &   \\
     \langle 1,3 \rangle & \langle 2,3 \rangle &   &     \\
     \langle 1,2 \rangle &   &   &
   \end{array}
\end{equation}
Let $p_{\, 2, \textmd{vpv}}(\mathfrak{D}, \langle a,b \rangle)$, denote the number of partitions of $\langle a,b \rangle$ into distinct parts from (\ref{8.12.01}). Then we have the generating function
\begin{equation}  \label{8.12.02}
   \prod_{k=2}^{5} \prod_{\substack{j=1; \\ \gcd(j,k)=1}}^{k-1} (1+x^j y^k)
   = \sum_{b=2}^{36} \sum_{a=1}^{b-1} p_{\, 2, \textmd{vpv}}(\mathfrak{D}, \langle a,b \rangle) x^a y^b,
\end{equation}
encoding all possible partitions of this kind, so each entry in the previous page $18 \times 36$ grid gives all numerical values of $p_{2}(\mathfrak{D}, \langle a,b \rangle)$.

Next we see that
\begin{equation} \label{8.13.01}
  (1-xy^2)(1-xy^3)(1-x^2y^3)(1-xy^4)(1-x^3y^4)(1-xy^5)(1-x^2y^5)(1-x^3y^5)(1-x^4y^5)
\end{equation}
is a generating function for the grid
 \begin{equation*} \tiny{
\begin{tabular}  {cccccccccccccccccccccc}
    &  &   &   &   &   &   &   &   &   &  &  &   &   &   &  &  &  &  &   &   & \multicolumn{1}{c}{$\sum r$} \\  \cline{22-22}
\multicolumn{1}{c|}{$36$}  &   &   &   &  &  &   &   &   &  &  &   &   &   &  &  &  &   &   & \multicolumn{1}{c}{$-1$} &   & \multicolumn{1}{c}{$-1$} \\  \cline{22-22}
\multicolumn{1}{c|}{$35$}  &   &   &   &  &  &   &   &   &  &  &   &   &   &  &  &  &   &   &  &   & \multicolumn{1}{c}{$0$} \\  \cline{22-22}
\multicolumn{1}{c|}{$34$}  &   &   &   &  &  &   &   &   &  &  &   &   &   &  &  &  &   & \multicolumn{1}{c}{$1$}   & &   & \multicolumn{1}{c}{$1$} \\  \cline{22-22}
\multicolumn{1}{c|}{$33$}  &   &   &   &  &  &   &   &   &  &  &   &   &   &  &  &  & \multicolumn{1}{c}{$1$}   & \multicolumn{1}{c}{$1$}   & &   & \multicolumn{1}{c}{$2$} \\  \cline{22-22}
\multicolumn{1}{c|}{$32$}  &   &   &   &  &  &   &   &   &  &  &   &   &   &  &  & \multicolumn{1}{c}{$1$}  &   & \multicolumn{1}{c}{$1$}   &   &   & \multicolumn{1}{c}{$2$} \\  \cline{22-22}
\multicolumn{1}{c|}{$31$}  &   &   &   &  &  &   &   &   &  &  &   &   &   &   & \multicolumn{1}{c}{$1$}  &   &    & \multicolumn{1}{c}{$1$}   & &   & \multicolumn{1}{c}{$2$} \\  \cline{22-22}
\multicolumn{1}{c|}{$30$}  &   &   &   &  &  &   &   &   &  &  &   &   &   &  & \multicolumn{1}{c}{$-1$}  & \multicolumn{1}{c}{$-1$}  & \multicolumn{1}{c}{$-1$}   &   &   &   & \multicolumn{1}{c}{$-3$} \\  \cline{22-22}
\multicolumn{1}{c|}{$29$}  &   &   &   &  &  &   &   &   &  &  &   &   &   & \multicolumn{1}{c}{$-2$}  & \multicolumn{1}{c}{$-2$}  & \multicolumn{1}{c}{$-2$}  & \multicolumn{1}{c}{$-2$}   &   & &   & \multicolumn{1}{c}{$-8$} \\  \cline{22-22}
\multicolumn{1}{c|}{$28$}  &   &   &   &  &  &   &   &   &  &  &   &   & \multicolumn{1}{c}{$-1$}   & \multicolumn{1}{c}{$-2$}  & \multicolumn{1}{c}{$-2$}  & \multicolumn{1}{c}{$-2$}  & \multicolumn{1}{c}{$-1$}   &   & &   & \multicolumn{1}{c}{$-8$} \\  \cline{22-22}
\multicolumn{1}{c|}{$27$}  &   &   &   &  &  &   &   &   &  &  &   & \multicolumn{1}{c}{$-1$} &    & \multicolumn{1}{c}{$-1$} & \multicolumn{1}{c}{$-1$}  & & \multicolumn{1}{c}{$-1$}   &   & &   & \multicolumn{1}{c}{$-4$} \\  \cline{22-22}
\multicolumn{1}{c|}{$26$}  &   &   &   &  &  &   &   &   &  &  &   &   & \multicolumn{1}{c}{$2$} & \multicolumn{1}{c}{$1$}   & \multicolumn{1}{c}{$2$} & &   &   &   &   & \multicolumn{1}{c}{$5$} \\  \cline{22-22}
\multicolumn{1}{c|}{$25$}  &   &   &   &  &  &   &   &   &  &  & \multicolumn{1}{c}{$1$}   & \multicolumn{1}{c}{$2$}   & \multicolumn{1}{c}{$4$}   & \multicolumn{1}{c}{$4$}  & \multicolumn{1}{c}{$2$}  & \multicolumn{1}{c}{$1$} &   &   & &   & \multicolumn{1}{c}{$14$} \\  \cline{22-22}
\multicolumn{1}{c|}{$24$}  &   &   &   &  &  &   &   &   &  &\multicolumn{1}{c}{$1$}   &\multicolumn{1}{c}{$3$}    &\multicolumn{1}{c}{$3$}    &\multicolumn{1}{c}{$6$}    &\multicolumn{1}{c}{$3$}   &\multicolumn{1}{c}{$3$}   &\multicolumn{1}{c}{$1$}  &   &   & &   & \multicolumn{1}{c}{$20$} \\  \cline{22-22}
\multicolumn{1}{c|}{$23$}  &   &   &   &  &  &   &   &   &  &\multicolumn{1}{c}{$1$}   &\multicolumn{1}{c}{$2$}    &\multicolumn{1}{c}{$2$}    &\multicolumn{1}{c}{$2$}    &\multicolumn{1}{c}{$2$}   &\multicolumn{1}{c}{$1$}   &  &   &   & &   & \multicolumn{1}{c}{$10$} \\  \cline{22-22}
\multicolumn{1}{c|}{$22$}  &   &   &   &  &  &   &   &   &\multicolumn{-1}{c}{$-1$}   &    &\multicolumn{1}{c}{$-3$}    &     &\multicolumn{1}{c}{$-1$}    &    &    &  &   &   & &   & \multicolumn{1}{c}{$-5$} \\  \cline{22-22}
\multicolumn{1}{c|}{$21$}  &   &   &   &  &  &   &   &   &\multicolumn{1}{c}{$-2$}   &\multicolumn{1}{c}{$-3$}   &\multicolumn{1}{c}{$-5$}    &\multicolumn{1}{c}{$-5$}    &\multicolumn{1}{c}{$-3$}    &\multicolumn{1}{c}{$-2$}   &  &  &   &   & &   & \multicolumn{1}{c}{$-20$} \\  \cline{22-22}
\multicolumn{1}{c|}{$20$}  &   &   &   &  &  &   &   &\multicolumn{1}{c}{$-1$}    &\multicolumn{1}{c}{$-3$}   &\multicolumn{1}{c}{$-6$}   &\multicolumn{1}{c}{$-5$}    &\multicolumn{1}{c}{$-6$}    &\multicolumn{1}{c}{$-3$}    &\multicolumn{1}{c}{$-1$}   &  &  &   &   & &   & \multicolumn{1}{c}{$-25$} \\  \cline{22-22}
\multicolumn{1}{c|}{$19$}  &   &   &   &  &  &   &\multicolumn{1}{c}{$-1$}    &\multicolumn{1}{c}{$-1$}    &\multicolumn{-1}{c}{$4$}   &\multicolumn{1}{c}{$-4$}   &\multicolumn{1}{c}{$-4$}    &\multicolumn{1}{c}{$-4$}    &\multicolumn{1}{c}{$-1$}    &\multicolumn{1}{c}{$-1$}   &  &  &   &   & &   & \multicolumn{1}{c}{$-20$} \\  \cline{22-22}
\multicolumn{1}{c|}{$18$}  &   &   &   &  &  &   &   &     &   &    &     &     &   &  &  &  &   &   & &   & \multicolumn{1}{c}{$0$} \\  \cline{22-22}
\multicolumn{1}{c|}{$17$}  &   &   &   &  &  &\multicolumn{1}{c}{$1$}    &\multicolumn{1}{c}{$1$}    &\multicolumn{1}{c}{$4$}   &\multicolumn{1}{c}{$4$}   &\multicolumn{1}{c}{$4$}    &\multicolumn{1}{c}{$4$}    &\multicolumn{1}{c}{$1$}    &\multicolumn{1}{c}{$1$}  & &  &  &   &   & &   & \multicolumn{1}{c}{$20$} \\  \cline{22-22}
\multicolumn{1}{c|}{$16$}  &   &   &   &  &  & \multicolumn{1}{c}{$1$}    &\multicolumn{1}{c}{$3$}   &\multicolumn{1}{c}{$6$}   &\multicolumn{1}{c}{$5$}    &\multicolumn{1}{c}{$6$}    &\multicolumn{1}{c}{$3$}    &\multicolumn{1}{c}{$1$}  & & &  &  &   &   & &   & \multicolumn{1}{c}{$25$} \\  \cline{22-22}
  \multicolumn{1}{c|}{$15$}  &   &   &   &  &  &\multicolumn{1}{c}{$2$}   &\multicolumn{1}{c}{$3$}   &\multicolumn{1}{c}{$5$}    &\multicolumn{1}{c}{$5$}    &\multicolumn{1}{c}{$3$}    &\multicolumn{1}{c}{$2$}   &  &   &   &   &  &  &  &  &   & \multicolumn{1}{c}{$20$} \\  \cline{22-22}
  \multicolumn{1}{c|}{$14$}  &   &   &   &  &   &\multicolumn{1}{c}{$1$}   &     &\multicolumn{1}{c}{$3$}    &    &\multicolumn{1}{c}{$1$}   &   &  &  &   &   &   &  &  &  &   & \multicolumn{1}{c}{$5$} \\  \cline{22-22}
  \multicolumn{1}{c|}{$13$}  &   &   &   &  &\multicolumn{1}{c}{$-1$}   &\multicolumn{1}{c}{$-2$}    &\multicolumn{1}{c}{$-2$}    &\multicolumn{1}{c}{$-2$}    &\multicolumn{1}{c}{$-2$}   &\multicolumn{1}{c}{$-1$}   &   &   &  &  &   &   &   &  &  &   & \multicolumn{1}{c}{$-10$} \\  \cline{22-22}
  \multicolumn{1}{c|}{$12$}  &   &   &   &\multicolumn{1}{c}{$-1$}   &\multicolumn{1}{c}{$-3$}    &\multicolumn{1}{c}{$-3$}    &\multicolumn{1}{c}{$-6$}    &\multicolumn{1}{c}{$-3$}   &\multicolumn{1}{c}{$-3$}   &\multicolumn{1}{c}{$-1$}  &   &   &  &  &   &   &   &  &  &   & \multicolumn{1}{c}{$-20$} \\  \cline{22-22}
  \multicolumn{1}{c|}{$11$}  &   &   &   & \multicolumn{1}{c}{$-1$}   & \multicolumn{1}{c}{$-2$}   & \multicolumn{1}{c}{$-4$}   & \multicolumn{1}{c}{$-4$}  & \multicolumn{1}{c}{$-2$}  & \multicolumn{1}{c}{$-1$} &   &   &   &  &  &   &   &   &  &  &   & \multicolumn{1}{c}{$-14$} \\  \cline{22-22}
  \multicolumn{1}{c|}{$10$}  &   &   &   &  & \multicolumn{1}{c}{$-2$}   & \multicolumn{1}{c}{$-1$} & \multicolumn{1}{c}{$-2$}  & &  &   &   &   &  &  &   &   &   &  &  &   & \multicolumn{1}{c}{$-5$} \\  \cline{22-22}
  \multicolumn{1}{c|}{$9$}  &   &   &\multicolumn{1}{c}{$1$} &    & \multicolumn{1}{c}{$1$} & \multicolumn{1}{c}{$1$}  &   & \multicolumn{1}{c}{$1$}   &  &   &   &   &  &  &   &   &   &  &  &   & \multicolumn{1}{c}{$4$} \\  \cline{22-22}
  \multicolumn{1}{c|}{$8$}  &   &   & \multicolumn{1}{c}{$1$}   & \multicolumn{1}{c}{$2$}  & \multicolumn{1}{c}{$2$}  & \multicolumn{1}{c}{$2$}  & \multicolumn{1}{c}{$1$}   &   &   &   &   &   &  &  &   &   &   &  &  &   & \multicolumn{1}{c}{$8$} \\  \cline{22-22}
  \multicolumn{1}{c|}{$7$}  &   &   & \multicolumn{1}{c}{$2$}  & \multicolumn{1}{c}{$2$}  & \multicolumn{1}{c}{$2$}  & \multicolumn{1}{c}{$2$}   &   &   &   &   &   &   &  &  &   &   &   &  &  &   & \multicolumn{1}{c}{$8$} \\  \cline{22-22}
  \multicolumn{1}{c|}{$6$}  &  &   & \multicolumn{1}{c}{$1$}  & \multicolumn{1}{c}{$1$}  & \multicolumn{1}{c}{$1$}  &   &   &  &   &   &   &   &  &  &   &   &   &  &  &   & \multicolumn{1}{c}{$3$} \\  \cline{22-22}
  \multicolumn{1}{c|}{$5$}  &  & \multicolumn{1}{c}{$-1$}  &   &   & \multicolumn{1}{c}{$-1$}   &   &   &  &   &   &   &   &  &  &   &   &   &  &  &   & \multicolumn{1}{c}{$-2$} \\  \cline{22-22}
  \multicolumn{1}{c|}{$4$}  &  & \multicolumn{1}{c}{$-1$}  &   & \multicolumn{1}{c}{$-1$} &  &   &   &   &   &   &   &   &  &  &   &   &   &  &  &   & \multicolumn{1}{c}{$-2$} \\  \cline{22-22}
  \multicolumn{1}{c|}{$3$}  &  & \multicolumn{1}{c}{$-1$}  & \multicolumn{1}{c}{$-1$}  &   &   &   &  &   &   &   &   &   &  &  &   &   &   &  &  &   & \multicolumn{1}{c}{$-2$} \\  \cline{22-22}
  \multicolumn{1}{c|}{$2$}  &  & \multicolumn{1}{c}{$-1$} &   &   &   &  &   &   &   &   &   &   &  &  &   &   &   &  &  &   & \multicolumn{1}{c}{$-1$} \\ \cline{22-22}
  \multicolumn{1}{c|}{$1$}  &   &   &   &  &  &  &   &   &   &   &   &   &  &  &   &   &   &  &  &   & \multicolumn{1}{c}{$0$} \\  \cline{22-22}
  \multicolumn{1}{c|}{$0$}  & \multicolumn{1}{c}{$1$}&   &   &   &   &  &   &   &   &   &   &  &  &   &   &   &  &  &   &   & \multicolumn{1}{c}{$1$}  \\ \cline{1-20} \cline{22-22}
  \multicolumn{1}{c|}{$y/x$}  & \multicolumn{1}{c}{$0$}  & \multicolumn{1}{c}{$1$}  & \multicolumn{1}{c}{$2$}  & \multicolumn{1}{c}{$3$}  & \multicolumn{1}{c}{$4$}& \multicolumn{1}{c}{$5$}  & \multicolumn{1}{c}{$6$}  & \multicolumn{1}{c}{$7$}  & \multicolumn{1}{c}{$8$} & \multicolumn{1}{c}{$9$} & \multicolumn{1}{c}{$10$} & \multicolumn{1}{c}{$11$} & \multicolumn{1}{c}{$12$} & \multicolumn{1}{c}{$13$} & \multicolumn{1}{c}{$14$} & \multicolumn{1}{c}{$15$} & \multicolumn{1}{c}{$16$} & \multicolumn{1}{c}{$17$} & \multicolumn{1}{c}{$18$} &  &   \\
                                  &    &   &   &   &   &   &  &   &   &   &   &   &  &  &   &   &   &  &  &   &  \\
  \multicolumn{1}{c|}{$\sum c$}  & \multicolumn{1}{|c|}{$1$}  & \multicolumn{1}{|c|}{$-4$}  & \multicolumn{1}{|c|}{$4$}  & \multicolumn{1}{|c|}{$2$}  & \multicolumn{1}{|c|}{$-3$} & \multicolumn{1}{|c|}{$0$}  & \multicolumn{1}{|c|}{$-7$}  & \multicolumn{1}{|c|}{$10$}   & \multicolumn{1}{|c|}{$-1$}  & \multicolumn{1}{|c|}{$0$}  & \multicolumn{1}{|c|}{$1$}  & \multicolumn{1}{|c|}{$-10$}  & \multicolumn{1}{|c|}{$7$}  & \multicolumn{1}{|c|}{$0$} & \multicolumn{1}{|c|}{$3$} & \multicolumn{1}{|c|}{$-2$} & \multicolumn{1}{|c|}{$-4$} & \multicolumn{1}{|c|}{$4$} & \multicolumn{1}{|c|}{$-1$} & & \\
\end{tabular} }
\end{equation*}

\bigskip
Note that column totals generate from
\begin{equation*}
  (1-xy^2)(1-xy^3)(1-x^2y^3)(1-xy^4)(1-x^3y^4)(1-xy^5)(1-x^2y^5)(1-x^3y^5)(1-x^4y^5),
\end{equation*}
where $y=1$ expanding to
\begin{equation*}
  1 - 4 x + 4 x^2 + 2 x^3 - 3 x^4 - 7 x^6 + 10 x^7 - x^8 + x^{10} - 10 x^{11}
  \end{equation*}
\begin{equation*}
   + 7 x^{12} + 3 x^{14} - 2 x^{15} - 4 x^{16} + 4 x^{17} - x^{18}.
\end{equation*}

Likewise, row totals generate from
\begin{equation*}
  (1-xy^2)(1-xy^3)(1-x^2y^3)(1-xy^4)(1-x^3y^4)(1-xy^5)(1-x^2y^5)(1-x^3y^5)(1-x^4y^5),
\end{equation*}
where $x=1$ expanding to
\begin{equation*}
    1 - y^2 - 2 y^3 - 2 y^4 - 2 y^5 + 3 y^6 + 8 y^7 + 8 y^8 + 4 y^9 - 5 y^{10} - 14 y^{11}
\end{equation*}
\begin{equation*}
  - 20 y^{12} - 10 y^{13} + 5 y^{14} + 20 y^{15} + 25 y^{16} + 20 y^{17} - 20 y^{19} - 25 y^{20}
\end{equation*}
\begin{equation*}
  - 20 y^{21} - 5 y^{22} + 10 y^{23} + 20 y^{24} + 14 y^{25} + 5 y^{26} - 4 y^{27} - 8 y^{28}
\end{equation*}
\begin{equation*}
  - 8 y^{29} - 3 y^{30} + 2 y^{31} + 2 y^{32} + 2 y^{33} + y^{34} - y^{36}.
\end{equation*}

\textbf{Combinatorial interpretation}: Consider the aggregate of vectors
\begin{equation}  \label{8.13.02}
   \begin{array}{cccc}
     \langle 1,5 \rangle & \langle 2,5 \rangle  & \langle 3,5 \rangle & \langle 4,5 \rangle     \\
     \langle 1,4 \rangle &   & \langle 3,4 \rangle &      \\
     \langle 1,3 \rangle & \langle 2,3 \rangle &   &        \\
     \langle 1,2 \rangle &   &   &
   \end{array}
\end{equation}
Let $p_{\,2, \, \textmd{up}}(\mathfrak{D}_{(vpv,\textmd{odd})}; \langle a,b \rangle)$, denote the number of partitions of $\langle a,b \rangle$ into distinct odd number of parts from (\ref{8.13.02}). Also let $p_{\,2, \, \textmd{up}}(\mathfrak{D}_{(vpv,\textmd{even})}; \langle a,b \rangle)$, denote the number of partitions of $\langle a,b \rangle$ into distinct even number of parts from (\ref{8.13.02}).Then
\begin{equation}  \label{8.13.03}
   \prod_{k=2}^{5} \prod_{\substack{j=1; \\ \gcd(j,k)=1}}^{k-1} (1-x^j y^k)
   = \sum_{b=2}^{36} \sum_{a=1}^{b-1} [p_{\,2, \, \textmd{up}}(\mathfrak{D}_{(vpv,\textmd{odd})}; \langle a,b \rangle)
   -  p_{\,2, \, \textmd{up}}(\mathfrak{D}_{(vpv,\textmd{even})}; \langle a,b \rangle)] x^a y^b,
\end{equation}
and each entry in the grid after (\ref{8.13.01}) above gives each numerical value.

\bigskip

\subsection{2D weighted \textit{Upper VPV} Coefficients - Order 5}
The following product is five of the factors in the known VPV infinite product for $\left(\frac{1-y}{1-xy}\right)^{\frac{1}{1-x}}$.
\begin{equation} \label{8.14.01}
  (1-xy^2)^{1/2}
  \end{equation}
\begin{equation*}
  (1-xy^3)^{1/3}(1-x^2y^3)^{1/3}
  \end{equation*}
\begin{equation*}
  (1-xy^4)^{1/4}(1-x^3y^4)^{1/4}
  \end{equation*}
\begin{equation*}
  (1-xy^5)^{1/5}(1-x^2y^5)^{1/5}(1-x^3y^5)^{1/5}(1-x^4y^5)^{1/5}.
\end{equation*}
This product is encoded by the grid

 \begin{equation*}
\begin{tabular}  {ccccccccccccc}
    &  &   &   &   &   &   &   &  &   &   & &   \multicolumn{1}{c}{$\sum r$} \\  \cline{13-13}
  \multicolumn{1}{c|}{$13$}  &   &   &   & \multicolumn{1}{c}{$\vdots$} & \multicolumn{1}{c}{$\vdots$} & \multicolumn{1}{c}{$\vdots$} & \multicolumn{1}{c}{$\vdots$} & \multicolumn{1}{c}{$\vdots$} & \multicolumn{1}{c}{$\vdots$} & \multicolumn{1}{c}{$\vdots$} &  & \multicolumn{1}{c}{$\vdots$}  \\  \cline{13-13}
  \multicolumn{1}{c|}{$12$}  &   &   &   & \multicolumn{1}{c}{$\frac{-301}{9600}$} & \multicolumn{1}{c}{$\frac{-163279}{1555200}$} & \multicolumn{1}{c}{$\frac{86923}{3110400}$} & \multicolumn{1}{c}{$\frac{360869}{2073600}$} & \multicolumn{1}{c}{$\frac{86923}{3110400}$} & \multicolumn{1}{c}{$\frac{-163279}{1555200}$} & \multicolumn{1}{c}{$\frac{-301}{9600}$} &   &   \multicolumn{1}{c}{$\frac{-2431243}{6220800}$} \\  \cline{13-13}
  \multicolumn{1}{c|}{$11$}  &   &   &   & \multicolumn{1}{c}{$\frac{41}{1440}$} & \multicolumn{1}{c}{$\frac{127}{3240}$} & \multicolumn{1}{c}{$\frac{-899}{17280}$} & \multicolumn{1}{c}{$\frac{-899}{17280}$}  & \multicolumn{1}{c}{$\frac{127}{3240}$} & \multicolumn{1}{c}{$\frac{41}{1440}$} &   &   &   \multicolumn{1}{c}{$\frac{811}{25920}$} \\  \cline{13-13}
  \multicolumn{1}{c|}{$10$}  &   &   & \multicolumn{1}{c}{$\frac{-2}{25}$} & \multicolumn{1}{c}{$\frac{1171}{14400}$} & \multicolumn{1}{c}{$\frac{-1511}{14400}$} & \multicolumn{1}{c}{$\frac{23}{6400}$}  & \multicolumn{1}{c}{$\frac{-1511}{14400}$}  & \multicolumn{1}{c}{$\frac{1171}{14400}$} & \multicolumn{1}{c}{$\frac{-2}{25}$} &   &   &   \multicolumn{1}{c}{$\frac{-12143}{57600}$} \\  \cline{13-13}
  \multicolumn{1}{c|}{$9$}  &   &   & \multicolumn{1}{c}{$\frac{1}{20}$} & \multicolumn{1}{c}{$\frac{-23}{810}$} & \multicolumn{1}{c}{$\frac{61}{432}$} & \multicolumn{1}{c}{$\frac{61}{432}$} & \multicolumn{1}{c}{$\frac{-23}{810}$} & \multicolumn{1}{c}{$\frac{1}{20}$} &  &   &   &   \multicolumn{1}{c}{$211/648$} \\  \cline{13-13}
  \multicolumn{1}{c|}{$8$}  &   &   & \multicolumn{1}{c}{$\frac{-13}{480}$}  & \multicolumn{1}{c}{$\frac{317}{1440}$} & \multicolumn{1}{c}{$\frac{583}{5760}$} & \multicolumn{1}{c}{$\frac{317}{1440}$}  & \multicolumn{1}{c}{$\frac{-13}{480}$}  &   &  &   &   &   \multicolumn{1}{c}{$2807/5760$} \\  \cline{13-13}
  \multicolumn{1}{c|}{$7$}  &   &   & \multicolumn{1}{c}{$\frac{11}{60}$}  & \multicolumn{1}{c}{$\frac{9}{40}$} & \multicolumn{1}{c}{$\frac{9}{40}$} & \multicolumn{1}{c}{$\frac{11}{60}$}  &   &   &   &   &   &   \multicolumn{1}{c}{$49/60$} \\  \cline{13-13}
  \multicolumn{1}{c|}{$6$}  &  &   & \multicolumn{1}{c}{$\frac{1}{72}$}  & \multicolumn{1}{c}{$\frac{1}{144}$}  & \multicolumn{1}{c}{$\frac{1}{72}$}  &   &   &  &   &   &   &   \multicolumn{1}{c}{$11/144$} \\  \cline{13-13}
  \multicolumn{1}{c|}{$5$}  &  &  \multicolumn{1}{c}{$\frac{-1}{5}$}  & \multicolumn{1}{c}{$\frac{-1}{30}$}  & \multicolumn{1}{c}{$\frac{-1}{30}$}  & \multicolumn{1}{c}{$\frac{-1}{5}$}  & &  &  &   &   &   &   \multicolumn{1}{c}{$-7/15$} \\  \cline{13-13}
  \multicolumn{1}{c|}{$4$}  &  & \multicolumn{1}{c}{$\frac{1}{4}$}  & \multicolumn{1}{c}{$\frac{1}{8}$}  & \multicolumn{1}{c}{$\frac{1}{4}$} & &   &   &   &   &   &   &   \multicolumn{1}{c}{$-5/8$} \\  \cline{13-13}
  \multicolumn{1}{c|}{$3$}  &  & \multicolumn{1}{c}{$\frac{-1}{3}$}  & \multicolumn{1}{c}{$\frac{-1}{3}$}  &   &   &   &  &   &   &   &   & \multicolumn{1}{c}{$-2/3$}  \\  \cline{13-13}
  \multicolumn{1}{c|}{$2$}  &  & \multicolumn{1}{c}{$\frac{-1}{2}$} &   &   &   &  &   &   &   &   &   & \multicolumn{1}{c}{$-1/2$} \\ \cline{13-13}
  \multicolumn{1}{c|}{$1$}  &   &   &   &  &  &  &   &   &   &   &   &   \multicolumn{1}{c}{$0$} \\  \cline{13-13}
  \multicolumn{1}{c|}{$0$}  & \multicolumn{1}{c}{$1$}&   &   &   &   &  &   &   &   &   &   &   \multicolumn{1}{c}{$1$}  \\ \cline{1-11} \cline{13-13}
  \multicolumn{1}{c|}{$y/x$}  & \multicolumn{1}{c}{$0$}  & \multicolumn{1}{c}{$1$}  & \multicolumn{1}{c}{$2$}  & \multicolumn{1}{c}{$3$}  & \multicolumn{1}{c}{$4$}& \multicolumn{1}{c}{$5$}  & \multicolumn{1}{c}{$6$}  & \multicolumn{1}{c}{$7$}  & \multicolumn{1}{c}{$8$} & \multicolumn{1}{c}{$9$} &   &     \\
                                  &    &   &   &   &   &   &  &   &   &   &   &    \\
  \multicolumn{1}{c|}{$\sum c$}  & \multicolumn{1}{|c|}{$1$}  & \multicolumn{1}{|c|}{$\frac{-77}{60}$}  & \multicolumn{1}{|c|}{$\frac{-2531}{7200}$}  & \multicolumn{1}{|c|}{$\ldots$} & \multicolumn{1}{|c|}{$see$}& \multicolumn{1}{|c|}{$below$}  & \multicolumn{1}{|c|}{$ $}  & \multicolumn{1}{|c|}{$ $}  & \multicolumn{1}{|c|}{$ $}  & \multicolumn{1}{|c|}{$ $} &   &    \\
\end{tabular}
\end{equation*}

\bigskip
Note that column totals generate from
\begin{equation*}
  (1-xy^2)^{1/2}(1-xy^3)^{1/3}(1-x^2y^3)^{1/3}(1-xy^4)^{1/4}(1-x^2y^4)^{1/4}(1-x^3y^4)^{1/4}
\end{equation*}
\begin{equation*}
  (1-xy^5)^{1/5}(1-x^2y^5)^{1/5}(1-x^3y^5)^{1/5}(1-x^4y^5)^{1/5} \quad \textmd{where} \quad  y=1,
\end{equation*}
expanding to
\begin{equation*}
  1 - \frac{77}{60} x - \frac{2531}{7200} x^2 + \frac{360127}{1296000} x^3 + \frac{54348601}{311040000} x^4 + \frac{52396725643}{93312000000} x^5 - \frac{15725249008811}{33592320000000} x^6
\end{equation*}
\begin{equation*}
   + \frac{536095858573681}{2015539200000000} x^7 - \frac{221227534655582777}{967458816000000000} x^8- \frac{15078454659730017851}{522427760640000000000} x^9
\end{equation*}
\begin{equation*}
    + \frac{55218367281675862901707}{313456656384000000000000} x^{10} + O(x^{11})
\end{equation*}

However, only the first three grid columns are complete. Likewise, row totals generate from
\begin{equation*}
  (1-xy^2)^{1/2}(1-xy^3)^{1/3}(1-x^2y^3)^{1/3}(1-xy^4)^{1/4}(1-x^2y^4)^{1/4}(1-x^3y^4)^{1/4}
\end{equation*}
\begin{equation*}
  (1-xy^5)^{1/5}(1-x^2y^5)^{1/5}(1-x^3y^5)^{1/5}(1-x^4y^5)^{1/5} \quad \textmd{where} \quad  x=1,
\end{equation*}
expanding to
\begin{equation*}
    1 - \frac{1}{2}y^2 - \frac{2}{3} y^3 - \frac{5}{8} y^4 - \frac{7}{15} y^5 + \frac{11}{144} y^6 + \frac{49}{60} y^7 + \frac{2807}{5760} y^8 + \frac{211}{648} y^9 - \frac{12143}{57600} y^{10}
\end{equation*}
\begin{equation*}
   + \frac{811}{25920} y^{11}  - \frac{2431243}{6220800} y^{12} + \frac{19889}{259200} y^{13} + \frac{626597}{2488320} y^{14} - \frac{7647377}{46656000} y^{15} + O(y^{15}).
\end{equation*}

\bigskip

\textbf{Combinatorial interpretation}: Consider the sub-aggregate of the countable list of \textit{upper first quadrant} 2D Visible Point Vectors.
\begin{equation}  \label{8.14.02}
\begin{array}{cccc}
  \langle1,5\rangle & \langle2,5\rangle & \langle3,5\rangle & \langle4,5\rangle  \\
  \langle1,4\rangle &   & \langle3,4\rangle &   \\
  \langle1,3\rangle & \langle2,3\rangle &   &   \\
  \langle1,2\rangle &   &   &
\end{array}
\end{equation}

Let $p_{\, 2,\textmd{up}}(\mathfrak{W}, \langle a,b \rangle)$, denote the \textit{weighted sum of partitions} of $\langle a,b \rangle$ into parts of the form $\binom{1/k}{n} \langle j,k \rangle$ for integers $n \geq 0$, with $1 \leq j<k \leq 5$, $\gcd(j,k)=1$, so that each $\langle j,k \rangle$ from (\ref{8.14.02}). Then
\begin{equation}  \label{8.14.03}
   \prod_{k=2}^{5} \prod_{\substack{j=1, \\ \gcd(j,k)=1}}^{k-1} (1-x^j y^k)^{1/k}
   = \sum_{b=2}^{\infty} \sum_{a=1}^{b-1} p_{\, 2,\textmd{up}}(\mathfrak{W}, \langle a,b \rangle) x^a y^b,
\end{equation}
and each entry in the grid after (\ref{8.14.01}) gives each numerical value of $p_{\, 2,\textmd{up}}(\mathfrak{W}, \langle a,b \rangle)$.

While the above argument of partitions of vectors looks a bit disheveled and onerous, \textit{a priori} there is another interpretation which appears more natural.

\bigskip

\textbf{The Light Diffusion Model}: Consider the following mapping of the (\ref{8.14.02}) vectors plus the origin point $\langle0,0\rangle$, where each $\circ$ is considered as an infinitesimal small lens that receives light intensity $L_n$, and radiates it out at intensity $L_{n+1}$ with each $L_{n+1} = f(L_n)$. ie. $L_{n+1}$ is a function of $L_{n}$.
\begin{equation}  \label{8.14.04}
\begin{array}{cccccc}
 \multicolumn{1}{c|}{$5$}  &  &  \langle1,5\rangle & \langle2,5\rangle & \langle3,5\rangle & \langle4,5\rangle  \\  \cline{1-1}
 \multicolumn{1}{c|}{$4$}  &  &  \langle1,4\rangle &   & \langle3,4\rangle &   \\  \cline{1-1}
 \multicolumn{1}{c|}{$3$}  &  &  \langle1,3\rangle & \langle2,3\rangle &   &   \\  \cline{1-1}
 \multicolumn{1}{c|}{$2$}  &  &  \langle1,2\rangle &   &   & \\  \cline{1-1}
 \multicolumn{1}{c|}{$1$}  &  &                    &   &   & \\  \cline{1-1}
 \multicolumn{1}{c|}{$0$}  &  \langle0,0\rangle    &   &   &   & \\  \cline{1-6}
 \multicolumn{1}{c|}{$y/x$} & \multicolumn{1}{|c|}{$0$} &  \multicolumn{1}{|c|}{$1$} & \multicolumn{1}{|c|}{$2$}  & \multicolumn{1}{|c|}{$3$}  & \multicolumn{1}{|c|}{$4$}
\end{array}
\mapsto
\begin{array}{cccccc}
 \multicolumn{1}{c|}{$5$}  & \uparrow &  \circ & \circ & \circ & \circ  \\  \cline{1-1}
 \multicolumn{1}{c|}{$4$}  & \uparrow &  \circ &   & \circ &  \nearrow \\  \cline{1-1}
 \multicolumn{1}{c|}{$3$}  & \uparrow &  \circ & \circ &  \nearrow &   \\  \cline{1-1}
 \multicolumn{1}{c|}{$2$}  & \uparrow &  \circ & \nearrow  &   & \\  \cline{1-1}
 \multicolumn{1}{c|}{$1$}  & \uparrow &  \nearrow           &   &   & \\  \cline{1-1}
 \multicolumn{1}{c|}{$0$}  & \circ &                    &   &   & \\  \cline{1-6}
 \multicolumn{1}{c|}{$y/x$} & \multicolumn{1}{|c|}{$0$} &  \multicolumn{1}{|c|}{$1$} & \multicolumn{1}{|c|}{$2$}  & \multicolumn{1}{|c|}{$3$}  & \multicolumn{1}{|c|}{$4$}
\end{array}
\end{equation}
Suppose the origin lens $\circ$ radiates with intensity $\Upsilon$ upward along $x=0$ and across to the diagonal $y=x$ in the radial region between lines. The arrows in the right side grid here depict this phenomenon. Each visible point designated by vectors $\langle a,b \rangle$ (with $\gcd(a,b)=1$ and $1 \leq a < b \leq 5$) will receive the light ray at intensity $\Upsilon$, then in turn radiate that light onward in the same shaped radial upward sector around to the diagonal with intensity $f(\Upsilon)$, where $f$ is some defined function. Each lens receiving light transmits that light further to a point visible to it with a new intensity of $f(f(\Upsilon))$.

The product generating function (\ref{8.14.01}) for our present exercise gives us a means to examine light intensity radiating from an initial point, the origin, and extending via visible points in a radial region from the origin. The Visible Point Vector (VPV) identities can be construed as encoding this type of light diffusion.








\section{Defining radial from origin region 2D \textit{Upper All Vectors} aggregates}

These are \textit{all lattice point vectors} \textit{count} in the infinite radial region of the first quadrant bounded by the positive $y$ axis and the line $y=x$. The visible lattice points in 3D space are similarly a countable set, and so also are the n-component Vectors in nD space with $n$ a positive integer greater than unity. In order to count the entire first 3D hyperquadrant comprising all $a$, $b$, and $c$ positive integer components of $\langle a, b, c \rangle$ it suffices to count the following vectors $\left\langle a,b,c \right\rangle$ such that $0<a<b<c$ and apply symmetries in ways to be shown later in this paper.

A set of identities involving $n$ dimensional \textit{visible lattice points} were discovered by Campbell (1994). However, there has not been any emphasis on the all vector analogues of the VPV identities, despite it seeming to be an important area of research requiring a clear framework base from which to launch researches.

The \textit{Upper All Vector} (UAV) lattice points $\langle x,y \rangle$ with $x < y$ in the first 2D quadrant are:

\begin{equation}  \nonumber
\begin{array}{ccccccccc}
   &   &   &   & etc.  &   &   &   &  \\
  \langle1,10\rangle & \langle2,10\rangle  & \langle3,10\rangle  & \langle4,10\rangle  & \langle5,10\rangle  & \langle6,10\rangle  & \langle7,10\rangle  & \langle8,10\rangle  & \langle9,10\rangle \\
  \langle1,9\rangle & \langle2,9\rangle & \langle3,9\rangle  & \langle4,9\rangle & \langle5,9\rangle & \langle6,9\rangle  & \langle7,9\rangle & \langle8,9\rangle &   \\
  \langle1,8\rangle & \langle2,8\rangle  & \langle3,8\rangle & \langle4,8\rangle  & \langle5,8\rangle & \langle6,8\rangle  & \langle7,8\rangle &   &   \\
  \langle1,7\rangle & \langle2,7\rangle & \langle3,7\rangle & \langle4,7\rangle & \langle5,7\rangle & \langle6,7\rangle &   &   &   \\
  \langle1,6\rangle & \langle2,6\rangle  & \langle3,6\rangle  & \langle4,6\rangle  & \langle5,6\rangle &   &   &   &   \\
  \langle1,5\rangle & \langle2,5\rangle & \langle3,5\rangle & \langle4,5\rangle &   &   &   &   &   \\
  \langle1,4\rangle & \langle2,4\rangle  & \langle3,4\rangle &   &   &   &   &   &   \\
  \langle1,3\rangle & \langle2,3\rangle &   &   &   &   &   &   &   \\
  \langle1,2\rangle &   &   &   &   &   &   &   &
\end{array}
\end{equation}
These are the countable list of \textit{upper first quadrant} 2D Vectors. That is, the vectors in the infinitely extended radial region of the 2D first quadrant between the $y$-axis and the line $y=x$.

\bigskip

\section{Examples of 2D \textit{Upper All Vectors} finite generating functions}


\subsection{2D Distinct \textit{Upper All Vectors} Coefficients - Order 2}

Note that for Orders 2 and 3 the \textit{Upper All Vectors} set is identical to the \textit{Upper VPSs}.
\begin{equation}  \label{8.12}
  (1+xy^2) \quad  \quad \quad   and \quad  \quad  \quad (1-xy^2)
\end{equation}
 \begin{equation*}
\begin{tabular}  {c c c c c c}
  \multicolumn{1}{c|}{$3$}          &    &   &   &   & \multicolumn{1}{c}{$\sum r$}\\  \cline{6-6}
  \multicolumn{1}{c|}{$2$}  &  & \multicolumn{1}{c}{$1$}  &   &   & \multicolumn{1}{c}{$1$}\\ \cline{6-6}
  \multicolumn{1}{c|}{$1$} &  &  &   &   & \multicolumn{1}{c}{$0$} \\  \cline{6-6}
  \multicolumn{1}{c|}{$0$}  & \multicolumn{1}{c}{$1$}  &  &   &   & \multicolumn{1}{c}{$1$} \\ \cline{1-4} \cline{6-6}
  \multicolumn{1}{c|}{$y/x$}  & \multicolumn{1}{c}{$0$}  & \multicolumn{1}{c}{$1$}  & \multicolumn{1}{c}{$2$}  &   &  \\
                                  &    &   &   &   & \\
  \multicolumn{1}{c|}{$\sum c$}  & \multicolumn{1}{|c|}{$1$}  & \multicolumn{1}{|c|}{$1$}  & \multicolumn{1}{|c|}{$ $}  &   & \\
\end{tabular}
\quad  \quad  \quad
\begin{tabular}  {c c c c c c}
  \multicolumn{1}{c|}{$3$}          &    &   &   &   & \multicolumn{1}{c}{$\sum r$}\\  \cline{6-6}
  \multicolumn{1}{c|}{$2$}  &  & \multicolumn{1}{c}{$-1$}  &   &   & \multicolumn{1}{c}{$-1$}\\ \cline{6-6}
  \multicolumn{1}{c|}{$1$} &  &  &   &   & \multicolumn{1}{c}{$0$} \\  \cline{6-6}
  \multicolumn{1}{c|}{$0$}  & \multicolumn{1}{c}{$1$}  &  &   &   & \multicolumn{1}{c}{$1$} \\ \cline{1-4} \cline{6-6}
  \multicolumn{1}{c|}{$y/x$}  & \multicolumn{1}{c}{$0$}  & \multicolumn{1}{c}{$1$}  & \multicolumn{1}{c}{$2$}  &   &  \\
                                  &    &   &   &   & \\
  \multicolumn{1}{c|}{$\sum c$}  & \multicolumn{1}{|c|}{$1$}  & \multicolumn{1}{|c|}{$-1$}  & \multicolumn{1}{|c|}{$ $}  &   & \\
\end{tabular}
\end{equation*}


\subsection{2D Distinct \textit{Upper All Vectors} Coefficients - Order 3}

Note that for Order 3 the \textit{Upper All Vectors} set is identical to the \textit{Upper VPVs} set.
\begin{equation}  \label{8.13}
  (1+xy^2)(1+xy^3)(1+x^2y^3)  \quad \quad   and \quad  \quad  (1-xy^2)(1-xy^3)(1-x^2y^3)
\end{equation}
 \begin{equation*}
\begin{tabular}  {cccccccc}
    &  &   &   &   &   &   & \multicolumn{1}{c}{$\sum r$} \\  \cline{8-8}
  \multicolumn{1}{c|}{$8$}  &  &   &   &   & \multicolumn{1}{c}{$1$} &   & \multicolumn{1}{c}{$1$} \\  \cline{8-8}
  \multicolumn{1}{c|}{$7$}  &  &   &   &   &   &   & \multicolumn{1}{c}{$0$} \\  \cline{8-8}
  \multicolumn{1}{c|}{$6$}  &  &   &   & \multicolumn{1}{c}{$1$}  &   &   & \multicolumn{1}{c}{$1$} \\  \cline{8-8}
  \multicolumn{1}{c|}{$5$}  &  &   & \multicolumn{1}{c}{$1$}  & \multicolumn{1}{c}{$1$}  &   &   & \multicolumn{1}{c}{$2$} \\  \cline{8-8}
  \multicolumn{1}{c|}{$4$}  &  &   &   &   &   &   & \multicolumn{1}{c}{$0$} \\  \cline{8-8}
  \multicolumn{1}{c|}{$3$}  &  & \multicolumn{1}{c}{$1$}  & \multicolumn{1}{c}{$1$}  &   &   &   & \multicolumn{1}{c}{$2$} \\  \cline{8-8}
  \multicolumn{1}{c|}{$2$}  &  & \multicolumn{1}{c}{$1$}  &   &   &   &   & \multicolumn{1}{c}{$1$} \\ \cline{8-8}
  \multicolumn{1}{c|}{$1$}  &  &  &   &   &   &   & \multicolumn{1}{c}{$0$} \\  \cline{8-8}
  \multicolumn{1}{c|}{$0$}  & \multicolumn{1}{c}{$1$}  &  &   &   &   &   & \multicolumn{1}{c}{$1$}  \\ \cline{1-6} \cline{8-8}
  \multicolumn{1}{c|}{$y/x$}  & \multicolumn{1}{c}{$0$}  & \multicolumn{1}{c}{$1$}  & \multicolumn{1}{c}{$2$}  & \multicolumn{1}{c}{$3$}  & \multicolumn{1}{c}{$4$}  &  &   \\
                                  &    &   &   &   &   &   &  \\
  \multicolumn{1}{c|}{$\sum c$}  & \multicolumn{1}{|c|}{$1$}  & \multicolumn{1}{|c|}{$2$}  & \multicolumn{1}{|c|}{$2$}  & \multicolumn{1}{|c|}{$2$}  & \multicolumn{1}{|c|}{$1$} &   &  \\
\end{tabular}
\quad \quad \quad
\begin{tabular}  {cccccccc}
    &  &   &   &   &   &   & \multicolumn{1}{c}{$\sum r$} \\  \cline{8-8}
  \multicolumn{1}{c|}{$8$}  &  &   &   &   & \multicolumn{-1}{c}{$-1$} &   & \multicolumn{1}{c}{$-1$} \\  \cline{8-8}
  \multicolumn{1}{c|}{$7$}  &  &   &   &   &   &   & \multicolumn{1}{c}{$0$} \\  \cline{8-8}
  \multicolumn{1}{c|}{$6$}  &  &   &   & \multicolumn{1}{c}{$1$}  &   &   & \multicolumn{1}{c}{$1$} \\  \cline{8-8}
  \multicolumn{1}{c|}{$5$}  &  &   & \multicolumn{1}{c}{$1$}  & \multicolumn{1}{c}{$1$}  &   &   & \multicolumn{1}{c}{$2$} \\  \cline{8-8}
  \multicolumn{1}{c|}{$4$}  &  &   &   &   &   &   & \multicolumn{1}{c}{$0$} \\  \cline{8-8}
  \multicolumn{1}{c|}{$3$}  &  & \multicolumn{1}{c}{$-1$}  & \multicolumn{1}{c}{$-1$}  &   &   &   & \multicolumn{1}{c}{$-2$} \\  \cline{8-8}
  \multicolumn{1}{c|}{$2$}  &  & \multicolumn{1}{c}{$-1$}  &   &   &   &   & \multicolumn{1}{c}{$1$} \\ \cline{8-8}
  \multicolumn{1}{c|}{$1$}  &  &  &   &   &   &   & \multicolumn{1}{c}{$0$} \\ \cline{8-8}
  \multicolumn{1}{c|}{$0$}  & \multicolumn{1}{c}{$1$}  &  &   &   &   &   & \multicolumn{1}{c}{$1$}  \\ \cline{1-6} \cline{8-8}
  \multicolumn{1}{c|}{$y/x$}  & \multicolumn{1}{c}{$0$}  & \multicolumn{1}{c}{$1$}  & \multicolumn{1}{c}{$2$}  & \multicolumn{1}{c}{$3$}  & \multicolumn{1}{c}{$4$}  &  &   \\
                                  &    &   &   &   &   &   &  \\
  \multicolumn{1}{c|}{$\sum c$}  & \multicolumn{1}{|c|}{$1$}  & \multicolumn{1}{|c|}{$-2$}  & \multicolumn{1}{|c|}{$0$}  & \multicolumn{1}{|c|}{$2$}  & \multicolumn{1}{|c|}{$-1$} &   &  \\
\end{tabular}
\end{equation*}

\bigskip

\bigskip


\subsection{2D Distinct \textit{Upper All Vectors} Coefficients - Order 4}
\begin{equation}  \label{8.14}
  (1+xy^2)(1+xy^3)(1+x^2y^3)(1+xy^4)(1+x^2y^4)(1+x^3y^4)
\end{equation}
 \begin{equation*}
\begin{tabular}  {cccccccccccccc}
    &  &   &   &   &   &   &   &  &   &   &   &   &  \multicolumn{1}{c}{$\sum r$} \\  \cline{14-14}
  \multicolumn{1}{c|}{$20$}  &   &   &   &  &  &   &   &   &  &   & \multicolumn{1}{c}{$1$}  &   & \multicolumn{1}{c}{$1$} \\  \cline{14-14}
  \multicolumn{1}{c|}{$19$}  &   &   &   &  &  &   &   &   &  &   &   &   & \multicolumn{1}{c}{$0$} \\  \cline{14-14}
  \multicolumn{1}{c|}{$18$}  &   &   &   &  &  &   &   &   &  & \multicolumn{1}{c}{$1$}  & &   & \multicolumn{1}{c}{$1$} \\  \cline{14-14}
  \multicolumn{1}{c|}{$17$}  &   &   &   &  &  &   &   &   & \multicolumn{1}{c}{$1$} & \multicolumn{1}{c}{$1$}  &   &   & \multicolumn{1}{c}{$2$} \\  \cline{14-14}
  \multicolumn{1}{c|}{$16$}  &   &   &   &  &  &   &   &  \multicolumn{1}{c}{$1$} & \multicolumn{1}{c}{$1$} & \multicolumn{1}{c}{$1$}  &   &   & \multicolumn{1}{c}{$3$} \\  \cline{14-14}
  \multicolumn{1}{c|}{$15$}  &   &   &   &  &  &   &   & \multicolumn{1}{c}{$1$}  & \multicolumn{1}{c}{$1$} &   &   &   &   \multicolumn{1}{c}{$2$} \\  \cline{14-14}
  \multicolumn{1}{c|}{$14$}  &   &   &   &  &  &   & \multicolumn{1}{c}{$1$}  & \multicolumn{1}{c}{$2$}  & \multicolumn{1}{c}{$1$} &   &   &   &   \multicolumn{1}{c}{$4$} \\  \cline{14-14}
  \multicolumn{1}{c|}{$13$}  &   &   &   &  &  & \multicolumn{1}{c}{$1$}  & \multicolumn{1}{c}{$2$}  & \multicolumn{1}{c}{$2$}  & \multicolumn{1}{c}{$1$} &   &   &   &   \multicolumn{1}{c}{$6$} \\  \cline{14-14}
  \multicolumn{1}{c|}{$12$}  &   &   &   &  &  & \multicolumn{1}{c}{$1$}  & \multicolumn{1}{c}{$2$} & \multicolumn{1}{c}{$1$}  &  &   &   &   &   \multicolumn{1}{c}{$4$} \\  \cline{14-14}
  \multicolumn{1}{c|}{$11$}  &   &   &   &  & \multicolumn{1}{c}{$1$} & \multicolumn{1}{c}{$2$}  & \multicolumn{1}{c}{$2$}  & \multicolumn{1}{c}{$1$}  &  &   &   &   &   \multicolumn{1}{c}{$6$} \\  \cline{14-14}
  \multicolumn{1}{c|}{$10$}  &   &   &   &  & \multicolumn{1}{c}{$2$} & \multicolumn{1}{c}{$2$}  & \multicolumn{1}{c}{$2$}  &   &  &   &   &   &   \multicolumn{1}{c}{$6$} \\  \cline{14-14}
  \multicolumn{1}{c|}{$9$}  &   &   &   & \multicolumn{1}{c}{$1$} & \multicolumn{1}{c}{$2$} & \multicolumn{1}{c}{$2$}  & \multicolumn{1}{c}{$1$}  &   &  &   &   &   &   \multicolumn{1}{c}{$6$} \\  \cline{14-14}
  \multicolumn{1}{c|}{$8$}  &   &   &   & \multicolumn{1}{c}{$1$} & \multicolumn{1}{c}{$2$} & \multicolumn{1}{c}{$1$}  &   &   &  &   &   &   &   \multicolumn{1}{c}{$4$} \\  \cline{14-14}
  \multicolumn{1}{c|}{$7$}  &   &   & \multicolumn{1}{c}{$1$}  & \multicolumn{1}{c}{$2$} & \multicolumn{1}{c}{$2$} & \multicolumn{1}{c}{$1$}  &   &   &   &   &   &   &   \multicolumn{1}{c}{$6$} \\  \cline{14-14}
  \multicolumn{1}{c|}{$6$}  &  &   & \multicolumn{1}{c}{$1$}  & \multicolumn{1}{c}{$2$}  & \multicolumn{1}{c}{$1$}  &   &   &  &   &   &   &   &   \multicolumn{1}{c}{$4$} \\  \cline{14-14}
  \multicolumn{1}{c|}{$5$}  &  &   & \multicolumn{1}{c}{$1$}  & \multicolumn{1}{c}{$1$}  &   &   &   &  &   &   &   &   &   \multicolumn{1}{c}{$2$} \\  \cline{14-14}
  \multicolumn{1}{c|}{$4$}  &  & \multicolumn{1}{c}{$1$}  & \multicolumn{1}{c}{$1$}  & \multicolumn{1}{c}{$1$} & &   &   &   &   &   &   &   &   \multicolumn{1}{c}{$3$} \\  \cline{14-14}
  \multicolumn{1}{c|}{$3$}  &  & \multicolumn{1}{c}{$1$}  & \multicolumn{1}{c}{$1$}  &   &   &   &  &   &   &   &   &   &   \multicolumn{1}{c}{$2$} \\  \cline{14-14}
  \multicolumn{1}{c|}{$2$}  &  & \multicolumn{1}{c}{$1$} &   &   &   &  &   &   &   &   &   &   &   \multicolumn{1}{c}{$1$} \\ \cline{14-14}
  \multicolumn{1}{c|}{$1$}  &   &   &   &  &  &  &   &   &   &   &   &   &   \multicolumn{1}{c}{$0$} \\  \cline{14-14}
  \multicolumn{1}{c|}{$0$}  & \multicolumn{1}{c}{$1$}&   &   &   &   &  &   &   &   &   &   &   &   \multicolumn{1}{c}{$1$}  \\ \cline{1-12} \cline{14-14}
  \multicolumn{1}{c|}{$y/x$}  & \multicolumn{1}{c}{$0$}  & \multicolumn{1}{c}{$1$}  & \multicolumn{1}{c}{$2$}  & \multicolumn{1}{c}{$3$}  & \multicolumn{1}{c}{$4$}& \multicolumn{1}{c}{$5$}  & \multicolumn{1}{c}{$6$}  & \multicolumn{1}{c}{$7$}  & \multicolumn{1}{c}{$8$} & \multicolumn{1}{c}{$9$} & \multicolumn{1}{c}{$10$}  &   &     \\
                                  &    &   &   &   &   &   &  &   &   &   &   &   &    \\
  \multicolumn{1}{c|}{$\sum c$}  & \multicolumn{1}{|c|}{$1$}  & \multicolumn{1}{|c|}{$3$}  & \multicolumn{1}{|c|}{$5$}  & \multicolumn{1}{|c|}{$8$}  & \multicolumn{1}{|c|}{$10$} & \multicolumn{1}{|c|}{$10$}  & \multicolumn{1}{|c|}{$10$}  & \multicolumn{1}{|c|}{$8$}  & \multicolumn{1}{|c|}{$5$}  & \multicolumn{1}{|c|}{$3$} & \multicolumn{1}{|c|}{$1$}  &   &    \\
\end{tabular}
\end{equation*}

\bigskip
Note that column totals generate from
\begin{equation*}
  (1+xy^2)(1+xy^3)(1+x^2y^3)(1+xy^4)(1+x^2y^4)(1+x^3y^4) \quad \textmd{where} \quad  y=1,
\end{equation*}
expanding to
\begin{equation*}
  x^{10} + 3 x^9 + 5 x^8 + 8 x^7 + 10 x^6 + 10 x^5 + 10 x^4 + 8 x^3 + 5 x^2 + 3 x + 1.
\end{equation*}

Likewise, row totals generate from
\begin{equation*}
  (1+xy^2)(1+xy^3)(1+x^2y^3)(1+xy^4)(1+x^2y^4)(1+x^3y^4) \quad \textmd{where} \quad  x=1,
\end{equation*}
expanding to
\begin{equation*}
  y^{20} + y^{18} + 2 y^{17} + 3 y^{16} + 2 y^{15} + 4 y^{14} + 6 y^{13} + 4 y^{12} + 6 y^{11}
\end{equation*}
\begin{equation*}
  + 6 y^{10} + 6 y^9 + 4 y^8 + 6 y^7 + 4 y^6 + 2 y^5 + 3 y^4 + 2 y^3 + y^2 + 1.
\end{equation*}

\textbf{Combinatorial interpretation}: Consider the aggregate of vectors
\begin{equation}  \label{8.14a}
   \begin{array}{ccc}
     \langle 1,4 \rangle & \langle 2,4 \rangle & \langle 3,4 \rangle   \\
     \langle 1,3 \rangle & \langle 2,3 \rangle &     \\
     \langle 1,2 \rangle &   &
   \end{array}
\end{equation}
Let $p_{2}(\mathfrak{D}, \langle a,b \rangle)$, denote the number of partitions of $\langle a,b \rangle$ into distinct parts from (\ref{8.14a}). Then
\begin{equation}  \label{8.14b}
   \prod_{k=2}^{4} \prod_{j=1}^{k-1} (1+x^j y^k)
   = \sum_{b=2}^{20} \sum_{a=1}^{b-1} p_{2}(\mathfrak{D}, \langle a,b \rangle) x^a y^b,
\end{equation}
and each entry in the grid after (\ref{8.14}) above gives each numerical value.

\bigskip

\bigskip

Next we consider the following product

\begin{equation}  \label{8.15}
  (1-xy^2)(1-xy^3)(1-x^2y^3)(1-xy^4)(1-x^2y^4)(1-x^3y^4)
\end{equation}

which is the generating function for the grid

 \begin{equation*}
\begin{tabular}  {cccccccccccccc}
    &  &   &   &   &   &   &   &  &   &   &   &   &  \multicolumn{1}{c}{$\sum r$} \\  \cline{14-14}
  \multicolumn{1}{c|}{$20$}  &   &   &   &  &  &   &   &   &  &   & \multicolumn{1}{c}{$1$}  &   & \multicolumn{1}{c}{$1$} \\  \cline{14-14}
  \multicolumn{1}{c|}{$19$}  &   &   &   &  &  &   &   &   &  &   &   &   & \multicolumn{1}{c}{$0$} \\  \cline{14-14}
  \multicolumn{1}{c|}{$18$}  &   &   &   &  &  &   &   &   &  & \multicolumn{-1}{c}{$-1$}  & &   & \multicolumn{-1}{c}{$-1$} \\  \cline{14-14}
  \multicolumn{1}{c|}{$17$}  &   &   &   &  &  &   &   &   & \multicolumn{1}{c}{$-1$} & \multicolumn{1}{c}{$-1$}  &   &   & \multicolumn{1}{c}{$-2$} \\  \cline{14-14}
  \multicolumn{1}{c|}{$16$}  &   &   &   &  &  &   &   &  \multicolumn{1}{c}{$-1$} & \multicolumn{1}{c}{$-1$} & \multicolumn{-1}{c}{$1$}  &   &   & \multicolumn{1}{c}{$-3$} \\  \cline{14-14}
  \multicolumn{1}{c|}{$15$}  &   &   &   &  &  &   &   & \multicolumn{1}{c}{$1$}  & \multicolumn{1}{c}{$1$} &   &   &   &   \multicolumn{1}{c}{$2$} \\  \cline{14-14}
  \multicolumn{1}{c|}{$14$}  &   &   &   &  &  &   & \multicolumn{1}{c}{$1$}  & \multicolumn{1}{c}{$2$}  & \multicolumn{1}{c}{$1$} &   &   &   &   \multicolumn{1}{c}{$4$} \\  \cline{14-14}
  \multicolumn{1}{c|}{$13$}  &   &   &   &  &  & \multicolumn{1}{c}{$1$}  & \multicolumn{1}{c}{$2$}  & \multicolumn{1}{c}{$2$}  & \multicolumn{1}{c}{$1$} &   &   &   &   \multicolumn{1}{c}{$6$} \\  \cline{14-14}
  \multicolumn{1}{c|}{$12$}  &   &   &   &  &  & \multicolumn{1}{c}{$1$}  &  & \multicolumn{1}{c}{$1$}  &  &   &   &   &   \multicolumn{1}{c}{$2$} \\  \cline{14-14}
  \multicolumn{1}{c|}{$11$}  &   &   &   &  & \multicolumn{1}{c}{$-1$} & \multicolumn{1}{c}{$-2$}  & \multicolumn{1}{c}{$-2$}  & \multicolumn{1}{c}{$-1$}  &  &   &   &   &   \multicolumn{1}{c}{$-6$} \\  \cline{14-14}
  \multicolumn{1}{c|}{$10$}  &   &   &   &  & \multicolumn{1}{c}{$-2$} & \multicolumn{1}{c}{$-2$}  & \multicolumn{1}{c}{$-2$}  &   &  &   &   &   &   \multicolumn{1}{c}{$-6$} \\  \cline{14-14}
  \multicolumn{1}{c|}{$9$}  &   &   &   & \multicolumn{1}{c}{$-1$} & \multicolumn{1}{c}{$-2$} & \multicolumn{1}{c}{$-2$}  & \multicolumn{1}{c}{$-1$}  &   &  &   &   &   &   \multicolumn{1}{c}{$-6$} \\  \cline{14-14}
  \multicolumn{1}{c|}{$8$}  &   &   &   & \multicolumn{1}{c}{$1$} &  & \multicolumn{1}{c}{$1$}  &   &   &  &   &   &   &   \multicolumn{1}{c}{$2$} \\  \cline{14-14}
  \multicolumn{1}{c|}{$7$}  &   &   & \multicolumn{1}{c}{$1$}  & \multicolumn{1}{c}{$2$} & \multicolumn{1}{c}{$2$} & \multicolumn{1}{c}{$1$}  &   &   &   &   &   &   &   \multicolumn{1}{c}{$6$} \\  \cline{14-14}
  \multicolumn{1}{c|}{$6$}  &  &   & \multicolumn{1}{c}{$1$}  & \multicolumn{1}{c}{$2$}  & \multicolumn{1}{c}{$1$}  &   &   &  &   &   &   &   &   \multicolumn{1}{c}{$4$} \\  \cline{14-14}
  \multicolumn{1}{c|}{$5$}  &  &   & \multicolumn{1}{c}{$1$}  & \multicolumn{1}{c}{$1$}  &   &   &   &  &   &   &   &   &   \multicolumn{1}{c}{$2$} \\  \cline{14-14}
  \multicolumn{1}{c|}{$4$}  &  & \multicolumn{1}{c}{$-1$}  & \multicolumn{1}{c}{$-1$}  & \multicolumn{1}{c}{$-1$} & &   &   &   &   &   &   &   &   \multicolumn{1}{c}{$-3$} \\  \cline{14-14}
  \multicolumn{1}{c|}{$3$}  &  & \multicolumn{1}{c}{$-1$}  & \multicolumn{1}{c}{$-1$}  &   &   &   &  &   &   &   &   &   &   \multicolumn{1}{c}{$-2$} \\  \cline{14-14}
  \multicolumn{1}{c|}{$2$}  &  & \multicolumn{1}{c}{$-1$} &   &   &   &  &   &   &   &   &   &   &   \multicolumn{1}{c}{$-1$} \\ \cline{14-14}
  \multicolumn{1}{c|}{$1$}  &   &   &   &  &  &  &   &   &   &   &   &   &   \multicolumn{1}{c}{$0$} \\  \cline{14-14}
  \multicolumn{1}{c|}{$0$}  & \multicolumn{1}{c}{$1$}&   &   &   &   &  &   &   &   &   &   &   &   \multicolumn{1}{c}{$1$}  \\ \cline{1-12} \cline{14-14}
  \multicolumn{1}{c|}{$y/x$}  & \multicolumn{1}{c}{$0$}  & \multicolumn{1}{c}{$1$}  & \multicolumn{1}{c}{$2$}  & \multicolumn{1}{c}{$3$}  & \multicolumn{1}{c}{$4$}& \multicolumn{1}{c}{$5$}  & \multicolumn{1}{c}{$6$}  & \multicolumn{1}{c}{$7$}  & \multicolumn{1}{c}{$8$} & \multicolumn{1}{c}{$9$} & \multicolumn{1}{c}{$10$}  &   &     \\
                                  &    &   &   &   &   &   &  &   &   &   &   &   &    \\
  \multicolumn{1}{c|}{$\sum c$}  & \multicolumn{1}{|c|}{$1$}  & \multicolumn{1}{|c|}{$-3$}  & \multicolumn{1}{|c|}{$1$}  & \multicolumn{1}{|c|}{$4$}  & \multicolumn{1}{|c|}{$-2$} & \multicolumn{1}{|c|}{$-2$}  & \multicolumn{1}{|c|}{$-2$}  & \multicolumn{1}{|c|}{$4$}  & \multicolumn{1}{|c|}{$1$}  & \multicolumn{1}{|c|}{$-3$} & \multicolumn{1}{|c|}{$1$}  &   &    \\
\end{tabular}
\end{equation*}

\bigskip

We see that column totals generate from
\begin{equation*}
  (1-xy^2)(1-xy^3)(1-x^2y^3)(1-xy^4)(1-x^2y^4)(1-x^3y^4) \quad \textmd{where} \quad  y=1,
\end{equation*}
expanding to
\begin{equation*}
  x^{10} - 3 x^9 + x^8 + 4 x^7 - 2 x^6 - 2 x^5 - 2 x^4 + 4 x^3 + x^2 - 3 x + 1.
\end{equation*}

Likewise, row totals generate from
\begin{equation*}
  (1-xy^2)(1-xy^3)(1-x^2y^3)(1-xy^4)(1-x^2y^4)(1-x^3y^4) \quad \textmd{where} \quad  x=1,
\end{equation*}
expanding to
\begin{equation*}
  y^{20} - y^{18} - 2 y^{17} - 3 y^{16} + 2 y^{15} + 4 y^{14} + 6 y^{13} + 2 y^{12} - 6 y^{11}
\end{equation*}
\begin{equation*}
  - 6 y^{10} - 6 y^9 + 2 y^8 + 6 y^7 + 4 y^6 + 2 y^5 - 3 y^4 - 2 y^3 - y^2 + 1.
\end{equation*}

\bigskip

\textbf{Combinatorial interpretation}: Consider the aggregate of vectors
\begin{equation}  \label{8.16}
   \begin{array}{ccc}
     \langle 1,4 \rangle & \langle 2,4 \rangle & \langle 3,4 \rangle   \\
     \langle 1,3 \rangle & \langle 2,3 \rangle &     \\
     \langle 1,2 \rangle &   &
   \end{array}
\end{equation}
Let $p_{2,\,e}(\mathfrak{D}, \langle a,b \rangle)$, respectively $p_{2,\,o}(\mathfrak{D}, \langle a,b \rangle)$, denote the number of partitions of $\langle a,b \rangle$ from (\ref{8.16}) into an \textit{even} respectively \textit{odd} number of distinct parts. Then
\begin{equation}  \label{8.17}
   \prod_{k=2}^{4} \prod_{j=1}^{k-1} (1-x^j y^k)
   = \sum_{b=2}^{20} \sum_{a=1}^{b-1} \left[ p_{2,\,e}(\mathfrak{D}, \langle a,b \rangle) - p_{2,\,o}(\mathfrak{D}, \langle a,b \rangle) \right] x^a y^b,
\end{equation}
and each entry in the grid after (\ref{8.15}) above gives each numerical value.

\bigskip

\bigskip

\subsection{2D Unrestricted \textit{Upper All Vectors} Partitions - Order 4}
\begin{equation}  \label{8.18a}
  \frac{1}{(1-xy^2)(1-xy^3)(1-x^2y^3)(1-xy^4)(1-x^2y^4)(1-x^3y^4)}
\end{equation}
encodes the following part sector of the infinitely extended radial grid
 \begin{equation*}
\begin{tabular}  {cccccccccccccccccc}
    &  &   &   &   &   &   & \multicolumn{1}{c}{$\vdots$} & \multicolumn{1}{c}{$\vdots$} & \multicolumn{1}{c}{$\vdots$} & \multicolumn{1}{c}{$\vdots$} & \multicolumn{1}{c}{$\vdots$} & \multicolumn{1}{c}{$\vdots$} & \multicolumn{1}{c}{$\vdots$} & \multicolumn{1}{c}{$\vdots$} & \multicolumn{1}{c}{$\vdots$} &  &  \multicolumn{1}{c}{$\sum r$} \\  \cline{18-18}
  \multicolumn{1}{c|}{$20$}  &   &   &   &  &  & \multicolumn{1}{c}{$1$} & \multicolumn{1}{c}{$4$} & \multicolumn{1}{c}{$11$} & \multicolumn{1}{c}{$21$} & \multicolumn{1}{c}{$31$} & \multicolumn{1}{c}{$37$} & \multicolumn{1}{c}{$31$} & \multicolumn{1}{c}{$21$} & \multicolumn{1}{c}{$11$} & \multicolumn{1}{c}{$4$} &  & \multicolumn{1}{c}{$173$} \\  \cline{18-18}
  \multicolumn{1}{c|}{$19$}  &   &   &   &  &  & \multicolumn{1}{c}{$1$} & \multicolumn{1}{c}{$1$} & \multicolumn{1}{c}{$5$} & \multicolumn{1}{c}{$12$} & \multicolumn{1}{c}{$21$} & \multicolumn{1}{c}{$28$} & \multicolumn{1}{c}{$21$} & \multicolumn{1}{c}{$12$} & \multicolumn{1}{c}{$5$} & \multicolumn{1}{c}{$1$} &   & \multicolumn{1}{c}{$134$} \\  \cline{18-18}
  \multicolumn{1}{c|}{$18$}  &   &   &   &  &  & \multicolumn{1}{c}{$2$} & \multicolumn{1}{c}{$7$} & \multicolumn{1}{c}{$15$} & \multicolumn{1}{c}{$23$} & \multicolumn{1}{c}{$28$}  & \multicolumn{1}{c}{$23$} & \multicolumn{1}{c}{$15$} & \multicolumn{1}{c}{$7$} & \multicolumn{1}{c}{$2$} &   &   &\multicolumn{1}{c}{$122$} \\  \cline{18-18}
  \multicolumn{1}{c|}{$17$}  &   &   &   &  &  & \multicolumn{1}{c}{$2$} & \multicolumn{1}{c}{$7$} & \multicolumn{1}{c}{$14$} & \multicolumn{1}{c}{$20$} & \multicolumn{1}{c}{$20$}  & \multicolumn{1}{c}{$14$} & \multicolumn{1}{c}{$7$} & \multicolumn{1}{c}{$2$} &   &   &  & \multicolumn{1}{c}{$86$} \\  \cline{18-18}
  \multicolumn{1}{c|}{$16$}  &   &   &   &  & \multicolumn{1}{c}{$1$} & \multicolumn{1}{c}{$4$}  & \multicolumn{1}{c}{$10$}  &  \multicolumn{1}{c}{$17$} & \multicolumn{1}{c}{$21$} & \multicolumn{1}{c}{$17$}    & \multicolumn{1}{c}{$10$}& \multicolumn{1}{c}{$4$}  & \multicolumn{1}{c}{$1$}  &   &   &  & \multicolumn{1}{c}{$85$} \\  \cline{18-18}
  \multicolumn{1}{c|}{$15$}  &   &   &   &  & \multicolumn{1}{c}{$1$} & \multicolumn{1}{c}{$5$}  & \multicolumn{1}{c}{$10$}  & \multicolumn{1}{c}{$15$}  & \multicolumn{1}{c}{$15$} & \multicolumn{1}{c}{$10$}  & \multicolumn{1}{c}{$5$}  & \multicolumn{1}{c}{$1$}  &   &   &   &   &  \multicolumn{1}{c}{$62$} \\  \cline{18-18}
  \multicolumn{1}{c|}{$14$}  &   &   &   &  & \multicolumn{1}{c}{$2$} & \multicolumn{1}{c}{$6$}  & \multicolumn{1}{c}{$12$}  & \multicolumn{1}{c}{$15$}  & \multicolumn{1}{c}{$12$} & \multicolumn{1}{c}{$6$}  & \multicolumn{1}{c}{$2$}  &   &   &   &   &   &   \multicolumn{1}{c}{$55$} \\  \cline{18-18}
  \multicolumn{1}{c|}{$13$}  &   &   &   &  & \multicolumn{1}{c}{$2$} & \multicolumn{1}{c}{$6$}  & \multicolumn{1}{c}{$10$}  & \multicolumn{1}{c}{$10$}  & \multicolumn{1}{c}{$6$} & \multicolumn{1}{c}{$2$} &   &   &   &   &   &   &   \multicolumn{1}{c}{$36$} \\  \cline{18-18}
  \multicolumn{1}{c|}{$12$}  &   &   &   & \multicolumn{1}{c}{$1$} & \multicolumn{1}{c}{$4$} & \multicolumn{1}{c}{$8$}  & \multicolumn{1}{c}{$11$} & \multicolumn{1}{c}{$8$}  & \multicolumn{1}{c}{$4$} & \multicolumn{1}{c}{$1$}  &   &   &   &   &   &   &   \multicolumn{1}{c}{$37$} \\  \cline{18-18}
  \multicolumn{1}{c|}{$11$}  &   &   &   & \multicolumn{1}{c}{$1$} & \multicolumn{1}{c}{$4$} & \multicolumn{1}{c}{$7$}  & \multicolumn{1}{c}{$7$}  & \multicolumn{1}{c}{$4$}  & \multicolumn{1}{c}{$1$} &   &   &   &   &   &   &   &   \multicolumn{1}{c}{$24$} \\  \cline{18-18}
  \multicolumn{1}{c|}{$10$}  &   &   &   & \multicolumn{1}{c}{$2$}  & \multicolumn{1}{c}{$5$} & \multicolumn{1}{c}{$8$} & \multicolumn{1}{c}{$5$}  & \multicolumn{1}{c}{$2$}  &   &   &   &   &   &   &   &   &   \multicolumn{1}{c}{$22$} \\  \cline{18-18}
  \multicolumn{1}{c|}{$9$}  &   &   &   & \multicolumn{1}{c}{$2$} & \multicolumn{1}{c}{$4$} & \multicolumn{1}{c}{$4$}  & \multicolumn{1}{c}{$2$}  &   &  &   &   &   &   &   &   &  &   \multicolumn{1}{c}{$12$} \\  \cline{18-18}
  \multicolumn{1}{c|}{$8$}  &   &   & \multicolumn{1}{c}{$1$}  & \multicolumn{1}{c}{$3$} & \multicolumn{1}{c}{$5$} & \multicolumn{1}{c}{$3$}  & \multicolumn{1}{c}{$1$}  &   &  &   &   &   &   &   &   &   &   \multicolumn{1}{c}{$13$} \\  \cline{18-18}
  \multicolumn{1}{c|}{$7$}  &   &   & \multicolumn{1}{c}{$1$}  & \multicolumn{1}{c}{$3$} & \multicolumn{1}{c}{$3$} & \multicolumn{1}{c}{$1$}  &   &   &   &   &   &   &   &   &   &  &   \multicolumn{1}{c}{$8$} \\  \cline{18-18}
  \multicolumn{1}{c|}{$6$}  &  &   & \multicolumn{1}{c}{$2$}  & \multicolumn{1}{c}{$3$}  & \multicolumn{1}{c}{$2$}  &   &   &  &   &   &   &   &   &   &   &   &   \multicolumn{1}{c}{$7$} \\  \cline{18-18}
  \multicolumn{1}{c|}{$5$}  &  &   & \multicolumn{1}{c}{$1$}  & \multicolumn{1}{c}{$1$}  &   &   &   &  &   &   &   &   &   &   &   &   &   \multicolumn{1}{c}{$2$} \\  \cline{18-18}
  \multicolumn{1}{c|}{$4$}  &  & \multicolumn{1}{c}{$1$}  & \multicolumn{1}{c}{$2$}  & \multicolumn{1}{c}{$1$} & &   &   &   &   &   &   &   &   &   &   &   &   \multicolumn{1}{c}{$4$} \\  \cline{18-18}
  \multicolumn{1}{c|}{$3$}  &  & \multicolumn{1}{c}{$1$}  & \multicolumn{1}{c}{$1$}  &   &   &   &  &   &   &   &   &   &   &   &   &   &   \multicolumn{1}{c}{$2$} \\  \cline{18-18}
  \multicolumn{1}{c|}{$2$}  &  & \multicolumn{1}{c}{$1$} &   &   &   &  &   &   &   &   &   &   &   &   &   &   &   \multicolumn{1}{c}{$1$} \\ \cline{18-18}
  \multicolumn{1}{c|}{$1$}  &   &   &   &  &  &  &   &   &   &   &   &   &   &   &   &   &   \multicolumn{1}{c}{$0$} \\  \cline{18-18}
  \multicolumn{1}{c|}{$0$}  & \multicolumn{1}{c}{$1$}&   &   &   &   &  &   &   &   &   &   &   &   &   &   &   &   \multicolumn{1}{c}{$1$}  \\ \cline{1-16} \cline{18-18}
  \multicolumn{1}{c|}{$y/x$}  & \multicolumn{1}{c}{$0$}  & \multicolumn{1}{c}{$1$}  & \multicolumn{1}{c}{$2$}  & \multicolumn{1}{c}{$3$}  & \multicolumn{1}{c}{$4$}& \multicolumn{1}{c}{$5$}  & \multicolumn{1}{c}{$6$}  & \multicolumn{1}{c}{$7$}  & \multicolumn{1}{c}{$8$} & \multicolumn{1}{c}{$9$} & \multicolumn{1}{c}{$10$}  & \multicolumn{1}{c}{$11$}  & \multicolumn{1}{c}{$12$}  & \multicolumn{1}{c}{$13$}  & \multicolumn{1}{c}{$14$}  &   &     \\
                                  &    &   &   &   &   &   &  &   &   &   &   &   &   &   &   &   &    \\
  \multicolumn{1}{c|}{$\sum c$}  & \multicolumn{1}{|c|}{$1$}  & \multicolumn{1}{|c|}{$3$}  & \multicolumn{1}{|c|}{$8$}  & \multicolumn{1}{|c|}{$17$}  & \multicolumn{1}{|c|}{$33$} & \multicolumn{1}{|c|}{$58$}  & \multicolumn{1}{|c|}{$97$}  & \multicolumn{1}{|c|}{$153$}  & \multicolumn{1}{|c|}{$233$}  & \multicolumn{1}{|c|}{$342$} & \multicolumn{1}{|c|}{$489$}  & \multicolumn{1}{|c|}{$681$}  & \multicolumn{1}{|c|}{$930$}  & \multicolumn{1}{|c|}{$1245$}  & \multicolumn{1}{|c|}{$1641$}  &   &    \\
\end{tabular}
\end{equation*}

\bigskip
Note that column totals generate from
\begin{equation*}
  \frac{1}{(1-xy^2)(1-xy^3)(1-x^2y^3)(1-xy^4)(1-x^2y^4)(1-x^3y^4)} \quad \textmd{where} \quad  y=1,
\end{equation*}
expanding to
\begin{equation*}
  1 + 3 x + 8 x^2 + 17 x^3 + 33 x^4 + 58 x^5 + 97 x^6 + 153 x^7 + 233 x^8 + 342 x^9
\end{equation*}
\begin{equation*}
  + 489 x^{10} + 681 x^{11} + 930 x^{12} + 1245 x^{13} + 1641 x^{14} + 2130 x^{15} + 2730 x^{16}
\end{equation*}
\begin{equation*}
  + 3456 x^{17} + 4330 x^{18} + O(x^{19})
\end{equation*}

Likewise, row totals generate from
\begin{equation*}
  \frac{1}{(1-xy^2)(1-xy^3)(1-x^2y^3)(1-xy^4)(1-x^2y^4)(1-x^3y^4)} \quad \textmd{where} \quad  x=1,
\end{equation*}
expanding to
\begin{equation*}
  1 + y^2 + 2 y^3 + 4 y^4 + 2 y^5 + 7 y^6 + 8 y^7 + 13 y^8 + 12 y^9 + 22 y^{10} + 24 y^{11} + 37 y^{12}
\end{equation*}
\begin{equation*}
  + 36 y^{13} + 55 y^{14} + 62 y^{15} + 85 y^{16} + 86 y^{17} + 122 y^{18} + 134 y^{19} + 173 y^{20} + 182 y^{21}
\end{equation*}
\begin{equation*}
  + 239 y^{22} + 260 y^{23} + 326 y^{24} + 344 y^{25} + 431 y^{26} + 470 y^{27} + 569 y^{28} + 602 y^{29}
\end{equation*}
\begin{equation*}
  + 734 y^{30} + 794 y^{31} + 938 y^{32} + O(y^{33})
\end{equation*}

\textbf{Combinatorial interpretation}: Consider the aggregate of vectors
\begin{equation}  \label{8.19a}
   \begin{array}{ccc}
     \langle 1,4 \rangle & \langle 2,4 \rangle & \langle 3,4 \rangle   \\
     \langle 1,3 \rangle & \langle 2,3 \rangle &     \\
     \langle 1,2 \rangle &   &
   \end{array}
\end{equation}

Let $p_{2, \textmd{up}}(\mathfrak{U}, \langle a,b \rangle)$, denote the number of partitions of $\langle a,b \rangle$ into unrestricted parts from (\ref{8.19a}). Then
\begin{equation}  \label{8.20}
   \prod_{k=2}^{4} \prod_{j=1}^{k-1} \frac{1}{(1-x^j y^k)}
   = \sum_{b=2}^{\infty} \sum_{a=1}^{b-1} p_{2, \textmd{up}}(\mathfrak{U}, \langle a,b \rangle) x^a y^b,
\end{equation}
and each entry in the grid after (\ref{8.18a}) above gives each numerical value.

\textbf{Observable Features.}
Just quickly inspecting the grid suggests the following points for the 2D Unrestricted \textit{Upper All Vectors} Partitions of Order 4:
\begin{enumerate}
  \item The grid table extends indefinitely and therefore is not simply a finite patch as are most of the distinct partitions grids so far given.
  \item The grid bounds are best defined by the radial region between the lines $y=4x$ and $y=x$ with some of the vectors actually on $y=4x$.
  \item Each row of number entries is a symmetric finite sequence.
\end{enumerate}


\subsection{2D Distinct \textit{Upper All Vectors} Coefficients - Order 5}
\begin{equation}  \nonumber
  (1+xy^2)(1+xy^3)(1+x^2y^3)(1+xy^4)(1+x^2y^4)(1+x^3y^4)(1+xy^5)(1+x^2y^5)(1+x^3y^5)(1+x^4y^5)
\end{equation}
is encoded by the grid
 \begin{equation*} \tiny{
\begin{tabular}  {cccccccccccccccccccccccc}
    &  &   &   &   &   &   &   &   &   &  &  &   &   &   &  &  &  &  &   &   &   &   &  \multicolumn{1}{c}{$\sum r$} \\  \cline{24-24}
\multicolumn{1}{c|}{$40$}  &   &   &  &  &   &   &   &  &  &   &   &   &  &  &  &   &   &   &    &   & \multicolumn{1}{c}{$1$} && \multicolumn{1}{c}{$1$} \\  \cline{24-24}
\multicolumn{1}{c|}{$39$}  &   &   &  &  &   &   &   &  &  &   &   &   &  &  &  &   &   &   &    &  & &   & \multicolumn{1}{c}{$0$} \\  \cline{24-24}
\multicolumn{1}{c|}{$38$}  &   &   &  &  &   &   &   &  &  &   &   &   &  &  &  &   &   &   &    & \multicolumn{1}{c}{$1$} &   & & \multicolumn{1}{c}{$1$} \\  \cline{24-24}
\multicolumn{1}{c|}{$37$}  &   &   &  &  &   &   &   &  &  &   &   &   &  &  &  &   &   &   & \multicolumn{1}{c}{$1$} & \multicolumn{1}{c}{$1$} &   & &  \multicolumn{1}{c}{$2$} \\  \cline{24-24}
\multicolumn{1}{c|}{$36$}  &   &   &  &  &   &   &   &  &  &   &   &   &  &  &  &   &   & \multicolumn{1}{c}{$1$}  & \multicolumn{1}{c}{$1$} & \multicolumn{1}{c}{$1$} &   &  &  \multicolumn{1}{c}{$3$} \\  \cline{24-24}
\multicolumn{1}{c|}{$35$}  &   &   &   &  &  &   &   &   &  &  &   &   &   &  &  &  & \multicolumn{1}{c}{$1$} & \multicolumn{1}{c}{$2$} & \multicolumn{1}{c}{$2$}  & \multicolumn{1}{c}{$1$}  &   &   & \multicolumn{1}{c}{$6$} \\  \cline{24-24}
\multicolumn{1}{c|}{$34$}  &   &   &   &  &  &   &   &   &  &  &   &   &   &  & &  & \multicolumn{1}{c}{$1$} & \multicolumn{1}{c}{$2$}  & \multicolumn{1}{c}{$1$}  &   & &   & \multicolumn{1}{c}{$4$} \\  \cline{24-24}
\multicolumn{1}{c|}{$33$}  &   &   &   &  &  &   &   &  & &  &  &   &   &   &   & \multicolumn{1}{c}{$2$} & \multicolumn{1}{c}{$3$} & \multicolumn{1}{c}{$3$} & \multicolumn{1}{c}{$2$} &  & & & \multicolumn{1}{c}{$10$} \\  \cline{24-24}
\multicolumn{1}{c|}{$32$}  &   &   &   &  &  &   &   &   &  &  &   &   &   &   & \multicolumn{1}{c}{$1$} & \multicolumn{1}{c}{$3$} & \multicolumn{1}{c}{$4$} & \multicolumn{1}{c}{$3$}  & \multicolumn{1}{c}{$1$} &  &  &   & \multicolumn{1}{c}{$12$} \\  \cline{24-24}
\multicolumn{1}{c|}{$31$}  &   &   &   &  &  &   &   &   &  &  &   &   &   & \multicolumn{1}{c}{$1$}  & \multicolumn{1}{c}{$3$}  & \multicolumn{1}{c}{$5$}   & \multicolumn{1}{c}{$5$} & \multicolumn{1}{c}{$3$} & \multicolumn{1}{c}{$1$} &   &  &   & \multicolumn{1}{c}{$18$} \\  \cline{24-24}
\multicolumn{1}{c|}{$30$}  &   &   &   &  &  &   &   &   &  &  &   &   &   & \multicolumn{1}{c}{$2$} & \multicolumn{1}{c}{$5$}  & \multicolumn{1}{c}{$6$}  & \multicolumn{1}{c}{$5$} & \multicolumn{1}{c}{$2$} &  &   &   &   &  \multicolumn{1}{c}{$20$} \\  \cline{24-24}
\multicolumn{1}{c|}{$29$}  &   &   &   &  &  &   &   &   &  &  &   &  & \multicolumn{1}{c}{$1$}  & \multicolumn{1}{c}{$4$}  & \multicolumn{1}{c}{$6$}  & \multicolumn{1}{c}{$6$} & \multicolumn{1}{c}{$4$} & \multicolumn{1}{c}{$1$} & &  &   &   &  \multicolumn{1}{c}{$8$} \\  \cline{24-24}
\multicolumn{1}{c|}{$28$}  &   &   &   &  &  &   &   &  &  &   &   & \multicolumn{1}{c}{$1$}   & \multicolumn{1}{c}{$4$}  & \multicolumn{1}{c}{$7$}  & \multicolumn{1}{c}{$10$}  & \multicolumn{1}{c}{$7$}   & \multicolumn{1}{c}{$4$}  & \multicolumn{1}{c}{$1$} & &  &   &   &  \multicolumn{1}{c}{$34$} \\  \cline{24-24}
\multicolumn{1}{c|}{$27$}  &   &   &   &  &  &   &   &   &  &  &   & \multicolumn{1}{c}{$2$} & \multicolumn{1}{c}{$6$}   & \multicolumn{1}{c}{$9$} & \multicolumn{1}{c}{$9$}  & \multicolumn{1}{c}{$6$}  & \multicolumn{1}{c}{$2$} &   & &   &   &    & \multicolumn{1}{c}{$34$} \\  \cline{24-24}
\multicolumn{1}{c|}{$26$}  &   &   &   &  &  &   &   &   &  &  & \multicolumn{1}{c}{$2$}  & \multicolumn{1}{c}{$5$} & \multicolumn{1}{c}{$10$}   & \multicolumn{1}{c}{$12$} & \multicolumn{1}{c}{$10$}  & \multicolumn{1}{c}{$5$}  &  \multicolumn{1}{c}{$2$}  &    &  &   &   &   & \multicolumn{1}{c}{$46$} \\  \cline{24-24}
\multicolumn{1}{c|}{$25$}  &   &   &   &  &  &   &   &   &  &  & \multicolumn{1}{c}{$3$} & \multicolumn{1}{c}{$7$} & \multicolumn{1}{c}{$11$} & \multicolumn{1}{c}{$11$}  & \multicolumn{1}{c}{$7$}  & \multicolumn{1}{c}{$3$} &  &   &   &   &  &   & \multicolumn{1}{c}{$42$} \\  \cline{24-24}
\multicolumn{1}{c|}{$24$}  &   &   &   &  &  &   &   &   &  &\multicolumn{1}{c}{$2$}   &\multicolumn{1}{c}{$6$}    &\multicolumn{1}{c}{$11$}    &\multicolumn{1}{c}{$13$}    &\multicolumn{1}{c}{$11$}   &\multicolumn{1}{c}{$6$}   &\multicolumn{1}{c}{$2$}  &   &   &   &   &  &   & \multicolumn{1}{c}{$51$} \\  \cline{24-24}
\multicolumn{1}{c|}{$23$}  &   &   &   &  &  &   &   &   & \multicolumn{1}{c}{$1$} &\multicolumn{1}{c}{$4$}   &\multicolumn{1}{c}{$10$}    &\multicolumn{1}{c}{$14$}    &\multicolumn{1}{c}{$14$}    &\multicolumn{1}{c}{$10$}   &\multicolumn{1}{c}{$4$}   &\multicolumn{1}{c}{$1$}  &   &   &   &   &  &   & \multicolumn{1}{c}{$58$} \\  \cline{24-24}
\multicolumn{1}{c|}{$22$}  &   &   &   &  &  &   &   &   &\multicolumn{1}{c}{$2$}   &\multicolumn{1}{c}{$7$}   &\multicolumn{1}{c}{$12$}    &\multicolumn{1}{c}{$15$}    &\multicolumn{1}{c}{$12$}    &\multicolumn{1}{c}{$7$}   &\multicolumn{1}{c}{$2$}   &  &   &   &   &   &  &   & \multicolumn{1}{c}{$53$} \\  \cline{24-24}
\multicolumn{1}{c|}{$21$}  &   &   &   &  &  &   &   & \multicolumn{1}{c}{$1$}  &\multicolumn{1}{c}{$5$}   &\multicolumn{1}{c}{$11$}   &\multicolumn{1}{c}{$15$}    &\multicolumn{1}{c}{$15$}    &\multicolumn{1}{c}{$11$}    &\multicolumn{1}{c}{$5$}   &\multicolumn{1}{c}{$1$}  &  &   &   &   &   &  &   & \multicolumn{1}{c}{$64$} \\  \cline{24-24}
\multicolumn{1}{c|}{$20$}  &   &   &   &  &  &   &   &\multicolumn{1}{c}{$2$}    &\multicolumn{1}{c}{$6$}   &\multicolumn{1}{c}{$12$}   &\multicolumn{1}{c}{$14$}    &\multicolumn{1}{c}{$12$}    &\multicolumn{1}{c}{$6$}    &\multicolumn{1}{c}{$2$}   &  &  &   &   &   &   &  &   & \multicolumn{1}{c}{$54$} \\  \cline{24-24}
\multicolumn{1}{c|}{$19$}  &   &   &   &  &  &   &\multicolumn{1}{c}{$1$}    &\multicolumn{1}{c}{$5$}    &\multicolumn{1}{c}{$11$}   &\multicolumn{1}{c}{$15$}   &\multicolumn{1}{c}{$15$}    &\multicolumn{1}{c}{$11$}    &\multicolumn{1}{c}{$5$}    &\multicolumn{1}{c}{$1$}   &  &  &   &   & &   &   &    & \multicolumn{1}{c}{$64$} \\  \cline{24-24}
\multicolumn{1}{c|}{$18$}  &   &   &   &  &  &   &\multicolumn{1}{c}{$2$} &\multicolumn{1}{c}{$7$}    &\multicolumn{1}{c}{$12$}   &\multicolumn{1}{c}{$15$}   &\multicolumn{1}{c}{$12$}    &\multicolumn{1}{c}{$7$} &\multicolumn{1}{c}{$2$}   &  &  &  &   &   &   &   &  &   & \multicolumn{1}{c}{$57$} \\  \cline{24-24}
\multicolumn{1}{c|}{$17$}  &   &   &   &  &  &\multicolumn{1}{c}{$1$}    &\multicolumn{1}{c}{$4$}    &\multicolumn{1}{c}{$10$}   &\multicolumn{1}{c}{$14$}   &\multicolumn{1}{c}{$14$}    &\multicolumn{1}{c}{$10$}    &\multicolumn{1}{c}{$4$}    &\multicolumn{1}{c}{$1$}  & &  &  &   &   & &   &   &    & \multicolumn{1}{c}{$58$} \\  \cline{24-24}
\multicolumn{1}{c|}{$16$}  &   &   &   &  &  & \multicolumn{1}{c}{$2$}    &\multicolumn{1}{c}{$6$}   &\multicolumn{1}{c}{$11$}   &\multicolumn{1}{c}{$13$}    &\multicolumn{1}{c}{$11$}    &\multicolumn{1}{c}{$6$}    &\multicolumn{1}{c}{$2$}  & & &  &  &   &   &   &   &  &   & \multicolumn{1}{c}{$51$} \\  \cline{24-24}
  \multicolumn{1}{c|}{$15$}  &   &   &   &  &  &\multicolumn{1}{c}{$3$}   &\multicolumn{1}{c}{$7$}   &\multicolumn{1}{c}{$11$}    &\multicolumn{1}{c}{$11$}    &\multicolumn{1}{c}{$7$}    &\multicolumn{1}{c}{$3$}   &  &   &   &   &  &  &  &  &   &   &    & \multicolumn{1}{c}{$42$} \\  \cline{24-24}
  \multicolumn{1}{c|}{$14$}  &   &   &   &  &\multicolumn{1}{c}{$2$}   &\multicolumn{1}{c}{$5$}   &\multicolumn{1}{c}{$10$}    &\multicolumn{1}{c}{$12$}    &\multicolumn{1}{c}{$10$}    &\multicolumn{1}{c}{$5$}   &\multicolumn{1}{c}{$2$}   &  &  &   &   &   &  &  &   &   &   &   & \multicolumn{1}{c}{$46$} \\  \cline{24-24}
  \multicolumn{1}{c|}{$13$}  &   &   &   &  &\multicolumn{1}{c}{$2$}   &\multicolumn{1}{c}{$6$}    &\multicolumn{1}{c}{$9$}    &\multicolumn{1}{c}{$9$}    &\multicolumn{1}{c}{$6$}   &\multicolumn{1}{c}{$2$}   &   &   &  &  &   &   &   &  &  &   &   &    & \multicolumn{1}{c}{$34$} \\  \cline{24-24}
  \multicolumn{1}{c|}{$12$}  &   &   &   &\multicolumn{1}{c}{$1$}   &\multicolumn{1}{c}{$4$}    &\multicolumn{1}{c}{$7$}    &\multicolumn{1}{c}{$10$}    &\multicolumn{1}{c}{$7$}   &\multicolumn{1}{c}{$4$}   &\multicolumn{1}{c}{$1$}  &   &   &  &  &   &   &   &  &   &   &   &   & \multicolumn{1}{c}{$34$} \\  \cline{24-24}
  \multicolumn{1}{c|}{$11$}  &   &   &   & \multicolumn{1}{c}{$1$}   & \multicolumn{1}{c}{$4$}   & \multicolumn{1}{c}{$6$}   & \multicolumn{1}{c}{$6$}  & \multicolumn{1}{c}{$4$}  & \multicolumn{1}{c}{$1$} &   &   &   &  &  &   &   &   &  &   &   &   &   & \multicolumn{1}{c}{$22$} \\  \cline{24-24}
  \multicolumn{1}{c|}{$10$}  &   &   &   & \multicolumn{1}{c}{$2$} & \multicolumn{1}{c}{$5$}   & \multicolumn{1}{c}{$6$} & \multicolumn{1}{c}{$5$}  & \multicolumn{1}{c}{$2$}  &  &   &   &   &  &  &   &   &   &  &  &   &   &    & \multicolumn{1}{c}{$20$} \\  \cline{24-24}
  \multicolumn{1}{c|}{$9$}  &   &   &\multicolumn{1}{c}{$1$} & \multicolumn{1}{c}{$3$}   & \multicolumn{1}{c}{$5$} & \multicolumn{1}{c}{$5$}  & \multicolumn{1}{c}{$3$}  & \multicolumn{1}{c}{$1$}   &  &   &   &   &  &  &   &   &   &  &  &   &   &    & \multicolumn{1}{c}{$18$} \\  \cline{24-24}
  \multicolumn{1}{c|}{$8$}  &   &   & \multicolumn{1}{c}{$1$}   & \multicolumn{1}{c}{$3$}  & \multicolumn{1}{c}{$4$}  & \multicolumn{1}{c}{$3$}  & \multicolumn{1}{c}{$1$}   &   &   &   &   &   &  &  &   &   &   &  &  &   &   &    & \multicolumn{1}{c}{$12$} \\  \cline{24-24}
  \multicolumn{1}{c|}{$7$}  &   &   & \multicolumn{1}{c}{$2$}  & \multicolumn{1}{c}{$3$}  & \multicolumn{1}{c}{$3$}  & \multicolumn{1}{c}{$2$}   &   &   &   &   &   &   &  &  &   &   &   &  &  &   &   &    & \multicolumn{1}{c}{$10$} \\  \cline{24-24}
  \multicolumn{1}{c|}{$6$}  &  &   & \multicolumn{1}{c}{$1$}  & \multicolumn{1}{c}{$2$}  & \multicolumn{1}{c}{$1$}  &   &   &  &   &   &   &   &  &  &   &   &   &  &  &   &   &    & \multicolumn{1}{c}{$4$} \\  \cline{24-24}
  \multicolumn{1}{c|}{$5$}  &  & \multicolumn{1}{c}{$1$}  & \multicolumn{1}{c}{$2$}  & \multicolumn{1}{c}{$2$}   & \multicolumn{1}{c}{$1$}   &   &   &  &   &   &   &   &  &  &   &   &   &  &   &   &   &   & \multicolumn{1}{c}{$6$} \\  \cline{24-24}
  \multicolumn{1}{c|}{$4$}  &  & \multicolumn{1}{c}{$1$}  & \multicolumn{1}{c}{$2$}  & \multicolumn{1}{c}{$1$} &  &   &   &   &   &   &   &   &  &  &   &   &   &   &    &  &  &   & \multicolumn{1}{c}{$4$} \\  \cline{24-24}
  \multicolumn{1}{c|}{$3$}  &  & \multicolumn{1}{c}{$1$}  & \multicolumn{1}{c}{$1$}  &   &   &   &  &   &   &   &   &   &  &  &   &   &   &    &   &  &  &   & \multicolumn{1}{c}{$2$} \\  \cline{24-24}
  \multicolumn{1}{c|}{$2$}  &  & \multicolumn{1}{c}{$1$} &   &   &   &  &   &   &   &   &   &   &  &  &   &   &   &   &   &   &  &   & \multicolumn{1}{c}{$1$} \\ \cline{24-24}
  \multicolumn{1}{c|}{$1$}  &   &   &   &  &  &  &   &   &   &   &   &   &  &  &   &   &   &  &   &   &   &   & \multicolumn{1}{c}{$0$} \\  \cline{24-24}
  \multicolumn{1}{c|}{$0$}  & \multicolumn{1}{c}{$1$}&   &   &   &   &  &   &   &   &   &   &  &  &   &   &   &  &   &   &   &   &   & \multicolumn{1}{c}{$1$}  \\ \cline{1-22} \cline{24-24}
  \multicolumn{1}{c|}{$y/x$}  & \multicolumn{1}{c}{$0$}  & \multicolumn{1}{c}{$1$}  & \multicolumn{1}{c}{$2$}  & \multicolumn{1}{c}{$3$}  & \multicolumn{1}{c}{$4$}& \multicolumn{1}{c}{$5$}  & \multicolumn{1}{c}{$6$}  & \multicolumn{1}{c}{$7$}  & \multicolumn{1}{c}{$8$} & \multicolumn{1}{c}{$9$} & \multicolumn{1}{c}{$10$} & \multicolumn{1}{c}{$11$} & \multicolumn{1}{c}{$12$} & \multicolumn{1}{c}{$13$} & \multicolumn{1}{c}{$14$} & \multicolumn{1}{c}{$15$} & \multicolumn{1}{c}{$16$} & \multicolumn{1}{c}{$17$} & \multicolumn{1}{c}{$18$} & \multicolumn{1}{c}{$19$}  & \multicolumn{1}{c}{$20$}  &   &   \\
                                  &    &   &   &   &   &   &  &   &   &   &   &   &  &  &   &   &   &  &   &   &   &   &  \\
  \multicolumn{1}{c|}{$\sum c$}  & \multicolumn{1}{|c|}{$1$}  & \multicolumn{1}{|c|}{$4$}  & \multicolumn{1}{|c|}{$9$}  & \multicolumn{1}{|c|}{$18$}  & \multicolumn{1}{|c|}{$31$} & \multicolumn{1}{|c|}{$46$}  & \multicolumn{1}{|c|}{$64$}  & \multicolumn{1}{|c|}{$82$}  & \multicolumn{1}{|c|}{$96$}  & \multicolumn{1}{|c|}{$106$}  & \multicolumn{1}{|c|}{$110$}  & \multicolumn{1}{|c|}{$106$}  & \multicolumn{1}{|c|}{$96$}  & \multicolumn{1}{|c|}{$82$}  & \multicolumn{1}{|c|}{$64$} & \multicolumn{1}{|c|}{$46$} & \multicolumn{1}{|c|}{$31$} & \multicolumn{1}{|c|}{$18$} & \multicolumn{1}{|c|}{$9$} & \multicolumn{1}{|c|}{$4$}  & \multicolumn{1}{|c|}{$1$}  &  &  \\
\end{tabular} }
\end{equation*}

\bigskip
Note that column totals generate from
\begin{equation*}
  (1+xy^2)(1+xy^3)(1+x^2y^3)(1+xy^4)(1+x^2y^4)(1+x^3y^4)(1+xy^5)(1+x^2y^5)(1+x^3y^5)(1+x^4y^5)
\end{equation*}
where $y=1$, expanding to
\begin{equation*}
   x^{20} + 4 x^{19} + 9 x^{18} + 18 x^{17} + 31 x^{16} + 46 x^{15} + 64 x^{14} + 82 x^{13} + 96 x^{12} + 106 x^{11}
\end{equation*}
\begin{equation*}
    + 110 x^{10} + 106 x^9 + 96 x^8 + 82 x^7 + 64 x^6 + 46 x^5 + 31 x^4 + 18 x^3 + 9 x^2 + 4 x + 1.
\end{equation*}

Likewise, row totals generate from again
\begin{equation*}
  (1+xy^2)(1+xy^3)(1+x^2y^3)(1+xy^4)(1+x^2y^4)(1+x^3y^4)(1+xy^5)(1+x^2y^5)(1+x^3y^5)(1+x^4y^5)
\end{equation*}
where $x=1$, expanding to
\begin{equation*}
  1 + y^2 + 2 y^3 + 3 y^4 + 6 y^5 + 4 y^6 + 10 y^7 + 12 y^8 + 18 y^9 + 20 y^{10} + 22 y^{11} + 34 y^{12}
\end{equation*}
\begin{equation*}
 + 34 y^{13} + 46 y^{14} + 42 y^{15} + 51 y^{16} + 58 y^{17} + 57 y^{18} + 64 y^{19} + 54 y^{20} + 64 y^{21}
\end{equation*}
\begin{equation*}
 + 57 y^{22} + 58 y^{23} + 51 y^{24} + 42 y^{25} + 46 y^{26} + 34 y^{27} + 34 y^{28} + 22 y^{29} + 20 y^{30}
\end{equation*}
\begin{equation*}
  + 18 y^{31} + 12 y^{32} + 10 y^{33} + 4 y^{34} + 6 y^{35} + 3 y^{36} + 2 y^{37} + y^{38} + y^{40}.
\end{equation*}


\textbf{Combinatorial interpretation}: Consider the aggregate of vectors
\begin{equation}  \label{8.21a}
   \begin{array}{cccc}
     \langle 1,5 \rangle & \langle 2,5 \rangle & \langle 3,5 \rangle & \langle 4,5 \rangle  \\
     \langle 1,4 \rangle & \langle 2,4 \rangle & \langle 3,4 \rangle &   \\
     \langle 1,3 \rangle & \langle 2,3 \rangle &   &     \\
     \langle 1,2 \rangle &   &   &
   \end{array}
\end{equation}
Let $p_{2}(\mathfrak{D}, \langle a,b \rangle)$, denote the number of partitions of $\langle a,b \rangle$ into distinct parts from (\ref{8.21a}). Then we have the generating function
\begin{equation}  \label{8.22a}
   \prod_{k=2}^{4} \prod_{j=1}^{k-1} (1+x^j y^k)
   = \sum_{b=2}^{20} \sum_{a=1}^{b-1} p_{2}(\mathfrak{D}, \langle a,b \rangle) x^a y^b,
\end{equation}
encoding all possible partitions of this kind, so each entry in the previous page $20 \times 40$ grid gives all numerical values of $p_{2}(\mathfrak{D}, \langle a,b \rangle)$.

\bigskip

\bigskip

Next we see that
\begin{equation*}
  (1-xy^2)(1-xy^3)(1-x^2y^3)(1-xy^4)(1-x^2y^4)(1-x^3y^4)(1-xy^5)(1-x^2y^5)(1-x^3y^5)(1-x^4y^5)
\end{equation*}
is encoded by the grid
 \begin{equation*} \tiny{
\begin{tabular}  {cccccccccccccccccccccc}
    &  &   &   &   &   &   &   &   &   &  &  &   &   &   &  &  &  &  &   &   & \multicolumn{1}{c}{$\sum r$} \\  \cline{22-22}
\multicolumn{1}{c|}{$36$}  &   &   &   &  &  &   &   &   &  &  &   &   &   &  &  &  &   &   & \multicolumn{1}{c}{$-1$} &   & \multicolumn{1}{c}{$-1$} \\  \cline{22-22}
\multicolumn{1}{c|}{$35$}  &   &   &   &  &  &   &   &   &  &  &   &   &   &  &  &  &   &   &  &   & \multicolumn{1}{c}{$0$} \\  \cline{22-22}
\multicolumn{1}{c|}{$34$}  &   &   &   &  &  &   &   &   &  &  &   &   &   &  &  &  &   & \multicolumn{1}{c}{$1$}   & &   & \multicolumn{1}{c}{$1$} \\  \cline{22-22}
\multicolumn{1}{c|}{$33$}  &   &   &   &  &  &   &   &   &  &  &   &   &   &  &  &  & \multicolumn{1}{c}{$1$}   & \multicolumn{1}{c}{$1$}   & &   & \multicolumn{1}{c}{$2$} \\  \cline{22-22}
\multicolumn{1}{c|}{$32$}  &   &   &   &  &  &   &   &   &  &  &   &   &   &  &  & \multicolumn{1}{c}{$1$}  &   & \multicolumn{1}{c}{$1$}   &   &   & \multicolumn{1}{c}{$2$} \\  \cline{22-22}
\multicolumn{1}{c|}{$31$}  &   &   &   &  &  &   &   &   &  &  &   &   &   &   & \multicolumn{1}{c}{$1$}  &   &    & \multicolumn{1}{c}{$1$}   & &   & \multicolumn{1}{c}{$2$} \\  \cline{22-22}
\multicolumn{1}{c|}{$30$}  &   &   &   &  &  &   &   &   &  &  &   &   &   &  & \multicolumn{1}{c}{$-1$}  & \multicolumn{1}{c}{$-2$}  & \multicolumn{1}{c}{$-1$}   &   &   &   & \multicolumn{1}{c}{$-3$} \\  \cline{22-22}
\multicolumn{1}{c|}{$29$}  &   &   &   &  &  &   &   &   &  &  &   &   &   & \multicolumn{1}{c}{$-2$}  & \multicolumn{1}{c}{$-2$}  & \multicolumn{1}{c}{$-2$}  & \multicolumn{1}{c}{$-2$}   &   & &   & \multicolumn{1}{c}{$-8$} \\  \cline{22-22}
\multicolumn{1}{c|}{$28$}  &   &   &   &  &  &   &   &   &  &  &   &   & \multicolumn{1}{c}{$-1$}   & \multicolumn{1}{c}{$-2$}  & \multicolumn{1}{c}{$-2$}  & \multicolumn{1}{c}{$-2$}  & \multicolumn{1}{c}{$-1$}   &   & &   & \multicolumn{1}{c}{$-8$} \\  \cline{22-22}
\multicolumn{1}{c|}{$27$}  &   &   &   &  &  &   &   &   &  &  &   & \multicolumn{1}{c}{$-1$} &    & \multicolumn{1}{c}{$-1$} & \multicolumn{1}{c}{$-1$}  & & \multicolumn{1}{c}{$-1$}   &   & &   & \multicolumn{1}{c}{$-4$} \\  \cline{22-22}
\multicolumn{1}{c|}{$26$}  &   &   &   &  &  &   &   &   &  &  &   &   & \multicolumn{1}{c}{$2$} & \multicolumn{1}{c}{$1$}   & \multicolumn{1}{c}{$2$} & &   &   &   &   & \multicolumn{1}{c}{$5$} \\  \cline{22-22}
\multicolumn{1}{c|}{$25$}  &   &   &   &  &  &   &   &   &  &  & \multicolumn{1}{c}{$1$}   & \multicolumn{1}{c}{$2$}   & \multicolumn{1}{c}{$4$}   & \multicolumn{1}{c}{$4$}  & \multicolumn{1}{c}{$2$}  & \multicolumn{1}{c}{$1$} &   &   & &   & \multicolumn{1}{c}{$14$} \\  \cline{22-22}
\multicolumn{1}{c|}{$24$}  &   &   &   &  &  &   &   &   &  &\multicolumn{1}{c}{$1$}   &\multicolumn{1}{c}{$3$}    &\multicolumn{1}{c}{$3$}    &\multicolumn{1}{c}{$6$}    &\multicolumn{1}{c}{$3$}   &\multicolumn{1}{c}{$3$}   &\multicolumn{1}{c}{$1$}  &   &   & &   & \multicolumn{1}{c}{$20$} \\  \cline{22-22}
\multicolumn{1}{c|}{$23$}  &   &   &   &  &  &   &   &   &  &\multicolumn{1}{c}{$1$}   &\multicolumn{1}{c}{$2$}    &\multicolumn{1}{c}{$2$}    &\multicolumn{1}{c}{$2$}    &\multicolumn{1}{c}{$2$}   &\multicolumn{1}{c}{$1$}   &  &   &   & &   & \multicolumn{1}{c}{$10$} \\  \cline{22-22}
\multicolumn{1}{c|}{$22$}  &   &   &   &  &  &   &   &   &\multicolumn{-1}{c}{$-1$}   &    &\multicolumn{1}{c}{$-3$}    &     &\multicolumn{1}{c}{$-1$}    &    &    &  &   &   & &   & \multicolumn{1}{c}{$-5$} \\  \cline{22-22}
\multicolumn{1}{c|}{$21$}  &   &   &   &  &  &   &   &   &\multicolumn{1}{c}{$-2$}   &\multicolumn{1}{c}{$-3$}   &\multicolumn{1}{c}{$-5$}    &\multicolumn{1}{c}{$-5$}    &\multicolumn{1}{c}{$-3$}    &\multicolumn{1}{c}{$-2$}   &  &  &   &   & &   & \multicolumn{1}{c}{$-20$} \\  \cline{22-22}
\multicolumn{1}{c|}{$20$}  &   &   &   &  &  &   &   &\multicolumn{1}{c}{$-1$}    &\multicolumn{1}{c}{$-3$}   &\multicolumn{1}{c}{$-6$}   &\multicolumn{1}{c}{$-5$}    &\multicolumn{1}{c}{$-6$}    &\multicolumn{1}{c}{$-3$}    &\multicolumn{1}{c}{$-1$}   &  &  &   &   & &   & \multicolumn{1}{c}{$-25$} \\  \cline{22-22}
\multicolumn{1}{c|}{$19$}  &   &   &   &  &  &   &\multicolumn{1}{c}{$-1$}    &\multicolumn{1}{c}{$-1$}    &\multicolumn{-1}{c}{$4$}   &\multicolumn{1}{c}{$-4$}   &\multicolumn{1}{c}{$-4$}    &\multicolumn{1}{c}{$-4$}    &\multicolumn{1}{c}{$-1$}    &\multicolumn{1}{c}{$-1$}   &  &  &   &   & &   & \multicolumn{1}{c}{$-20$} \\  \cline{22-22}
\multicolumn{1}{c|}{$18$}  &   &   &   &  &  &   &   &     &   &    &     &     &   &  &  &  &   &   & &   & \multicolumn{1}{c}{$0$} \\  \cline{22-22}
\multicolumn{1}{c|}{$17$}  &   &   &   &  &  &\multicolumn{1}{c}{$1$}    &\multicolumn{1}{c}{$1$}    &\multicolumn{1}{c}{$4$}   &\multicolumn{1}{c}{$4$}   &\multicolumn{1}{c}{$4$}    &\multicolumn{1}{c}{$4$}    &\multicolumn{1}{c}{$1$}    &\multicolumn{1}{c}{$1$}  & &  &  &   &   & &   & \multicolumn{1}{c}{$20$} \\  \cline{22-22}
\multicolumn{1}{c|}{$16$}  &   &   &   &  &  & \multicolumn{1}{c}{$1$}    &\multicolumn{1}{c}{$3$}   &\multicolumn{1}{c}{$6$}   &\multicolumn{1}{c}{$5$}    &\multicolumn{1}{c}{$6$}    &\multicolumn{1}{c}{$3$}    &\multicolumn{1}{c}{$1$}  & & &  &  &   &   & &   & \multicolumn{1}{c}{$25$} \\  \cline{22-22}
  \multicolumn{1}{c|}{$15$}  &   &   &   &  &  &\multicolumn{1}{c}{$2$}   &\multicolumn{1}{c}{$3$}   &\multicolumn{1}{c}{$5$}    &\multicolumn{1}{c}{$5$}    &\multicolumn{1}{c}{$3$}    &\multicolumn{1}{c}{$2$}   &  &   &   &   &  &  &  &  &   & \multicolumn{1}{c}{$20$} \\  \cline{22-22}
  \multicolumn{1}{c|}{$14$}  &   &   &   &  &   &\multicolumn{1}{c}{$1$}   &     &\multicolumn{1}{c}{$3$}    &    &\multicolumn{1}{c}{$1$}   &   &  &  &   &   &   &  &  &  &   & \multicolumn{1}{c}{$5$} \\  \cline{22-22}
  \multicolumn{1}{c|}{$13$}  &   &   &   &  &\multicolumn{1}{c}{$-1$}   &\multicolumn{1}{c}{$-2$}    &\multicolumn{1}{c}{$-2$}    &\multicolumn{1}{c}{$-2$}    &\multicolumn{1}{c}{$-2$}   &\multicolumn{1}{c}{$-1$}   &   &   &  &  &   &   &   &  &  &   & \multicolumn{1}{c}{$-10$} \\  \cline{22-22}
  \multicolumn{1}{c|}{$12$}  &   &   &   &\multicolumn{1}{c}{$-1$}   &\multicolumn{1}{c}{$-3$}    &\multicolumn{1}{c}{$-3$}    &\multicolumn{1}{c}{$-6$}    &\multicolumn{1}{c}{$-3$}   &\multicolumn{1}{c}{$-3$}   &\multicolumn{1}{c}{$-1$}  &   &   &  &  &   &   &   &  &  &   & \multicolumn{1}{c}{$-20$} \\  \cline{22-22}
  \multicolumn{1}{c|}{$11$}  &   &   &   & \multicolumn{1}{c}{$-1$}   & \multicolumn{1}{c}{$-2$}   & \multicolumn{1}{c}{$-4$}   & \multicolumn{1}{c}{$-4$}  & \multicolumn{1}{c}{$-2$}  & \multicolumn{1}{c}{$-1$} &   &   &   &  &  &   &   &   &  &  &   & \multicolumn{1}{c}{$-14$} \\  \cline{22-22}
  \multicolumn{1}{c|}{$10$}  &   &   &   &  & \multicolumn{1}{c}{$-2$}   & \multicolumn{1}{c}{$-1$} & \multicolumn{1}{c}{$-2$}  & &  &   &   &   &  &  &   &   &   &  &  &   & \multicolumn{1}{c}{$-5$} \\  \cline{22-22}
  \multicolumn{1}{c|}{$9$}  &   &   &\multicolumn{1}{c}{$1$} &    & \multicolumn{1}{c}{$1$} & \multicolumn{1}{c}{$1$}  &   & \multicolumn{1}{c}{$1$}   &  &   &   &   &  &  &   &   &   &  &  &   & \multicolumn{1}{c}{$4$} \\  \cline{22-22}
  \multicolumn{1}{c|}{$8$}  &   &   & \multicolumn{1}{c}{$1$}   & \multicolumn{1}{c}{$3$}  & \multicolumn{1}{c}{$2$}  & \multicolumn{1}{c}{$3$}  & \multicolumn{1}{c}{$1$}   &   &   &   &   &   &  &  &   &   &   &  &  &   & \multicolumn{1}{c}{$10$} \\  \cline{22-22}
  \multicolumn{1}{c|}{$7$}  &   &   & \multicolumn{1}{c}{$2$}  & \multicolumn{1}{c}{$3$}  & \multicolumn{1}{c}{$3$}  & \multicolumn{1}{c}{$2$}   &   &   &   &   &   &   &  &  &   &   &   &  &  &   & \multicolumn{1}{c}{$8$} \\  \cline{22-22}
  \multicolumn{1}{c|}{$6$}  &  &   & \multicolumn{1}{c}{$1$}  & \multicolumn{1}{c}{$2$}  & \multicolumn{1}{c}{$1$}  &   &   &  &   &   &   &   &  &  &   &   &   &  &  &   & \multicolumn{1}{c}{$3$} \\  \cline{22-22}
  \multicolumn{1}{c|}{$5$}  &  & \multicolumn{1}{c}{$-1$}  &   &   & \multicolumn{1}{c}{$-1$}   &   &   &  &   &   &   &   &  &  &   &   &   &  &  &   & \multicolumn{1}{c}{$-2$} \\  \cline{22-22}
  \multicolumn{1}{c|}{$4$}  &  & \multicolumn{1}{c}{$-1$}  &   & \multicolumn{1}{c}{$-1$} &  &   &   &   &   &   &   &   &  &  &   &   &   &  &  &   & \multicolumn{1}{c}{$-2$} \\  \cline{22-22}
  \multicolumn{1}{c|}{$3$}  &  & \multicolumn{1}{c}{$-1$}  & \multicolumn{1}{c}{$-1$}  &   &   &   &  &   &   &   &   &   &  &  &   &   &   &  &  &   & \multicolumn{1}{c}{$-2$} \\  \cline{22-22}
  \multicolumn{1}{c|}{$2$}  &  & \multicolumn{1}{c}{$-1$} &   &   &   &  &   &   &   &   &   &   &  &  &   &   &   &  &  &   & \multicolumn{1}{c}{$-1$} \\ \cline{22-22}
  \multicolumn{1}{c|}{$1$}  &   &   &   &  &  &  &   &   &   &   &   &   &  &  &   &   &   &  &  &   & \multicolumn{1}{c}{$0$} \\  \cline{22-22}
  \multicolumn{1}{c|}{$0$}  & \multicolumn{1}{c}{$1$}&   &   &   &   &  &   &   &   &   &   &  &  &   &   &   &  &  &   &   & \multicolumn{1}{c}{$1$}  \\ \cline{1-20} \cline{22-22}
  \multicolumn{1}{c|}{$y/x$}  & \multicolumn{1}{c}{$0$}  & \multicolumn{1}{c}{$1$}  & \multicolumn{1}{c}{$2$}  & \multicolumn{1}{c}{$3$}  & \multicolumn{1}{c}{$4$}& \multicolumn{1}{c}{$5$}  & \multicolumn{1}{c}{$6$}  & \multicolumn{1}{c}{$7$}  & \multicolumn{1}{c}{$8$} & \multicolumn{1}{c}{$9$} & \multicolumn{1}{c}{$10$} & \multicolumn{1}{c}{$11$} & \multicolumn{1}{c}{$12$} & \multicolumn{1}{c}{$13$} & \multicolumn{1}{c}{$14$} & \multicolumn{1}{c}{$15$} & \multicolumn{1}{c}{$16$} & \multicolumn{1}{c}{$17$} & \multicolumn{1}{c}{$18$} &  &   \\
                                  &    &   &   &   &   &   &  &   &   &   &   &   &  &  &   &   &   &  &  &   &  \\
  \multicolumn{1}{c|}{$\sum c$}  & \multicolumn{1}{|c|}{$1$}  & \multicolumn{1}{|c|}{$-4$}  & \multicolumn{1}{|c|}{$4$}  & \multicolumn{1}{|c|}{$2$}  & \multicolumn{1}{|c|}{$-3$} & \multicolumn{1}{|c|}{$0$}  & \multicolumn{1}{|c|}{$-7$}  & \multicolumn{1}{|c|}{$10$}   & \multicolumn{1}{|c|}{$-1$}  & \multicolumn{1}{|c|}{$0$}  & \multicolumn{1}{|c|}{$1$}  & \multicolumn{1}{|c|}{$-10$}  & \multicolumn{1}{|c|}{$7$}  & \multicolumn{1}{|c|}{$0$} & \multicolumn{1}{|c|}{$3$} & \multicolumn{1}{|c|}{$-2$} & \multicolumn{1}{|c|}{$-4$} & \multicolumn{1}{|c|}{$4$} & \multicolumn{1}{|c|}{$-1$} & & \\
\end{tabular} }
\end{equation*}

\bigskip
Note that column totals generate from
\begin{equation*}
  (1-xy^2)(1-xy^3)(1-x^2y^3)(1-xy^4)(1-x^2y^4)(1-x^3y^4)(1-xy^5)(1-x^2y^5)(1-x^3y^5)(1-x^4y^5)
\end{equation*}
where $y=1$, expanding to
\begin{equation*}
      1 - 4 x + 3 x^2 + 6 x^3 - 7 x^4 - 2 x^5 - 4 x^6 + 10 x^7 + 6 x^8 - 10 x^9 + 2 x^{10} - 10 x^{11}
\end{equation*}
\begin{equation*}
    + 6 x^{12} + 10 x^{13} - 4 x^{14} - 2 x^{15} - 7 x^{16} + 6 x^{17} + 3 x^{18} - 4 x^{19} + x^{20}.
\end{equation*}

Likewise, row totals generate from again
\begin{equation*}
  (1-xy^2)(1-xy^3)(1-x^2y^3)(1-xy^4)(1-x^2y^4)(1-x^3y^4)(1-xy^5)(1-x^2y^5)(1-x^3y^5)(1-x^4y^5)
\end{equation*}
where $x=1$, expanding to
\begin{equation*}
    1 - y^2 - 2 y^3 - 3 y^4 - 2 y^5 + 4 y^6 + 10 y^7 + 10 y^8 + 6 y^9 - 8 y^{10} - 22 y^{11} - 28 y^{12}
\end{equation*}
\begin{equation*}
 - 14 y^{13} + 10 y^{14} + 34 y^{15} + 45 y^{16} + 30 y^{17} - 5 y^{18} - 40 y^{19} - 50 y^{20} - 40 y^{21}
\end{equation*}
\begin{equation*}
 - 5 y^{22} + 30 y^{23} + 45 y^{24} + 34 y^{25} + 10 y^{26} - 14 y^{27} - 28 y^{28} - 22 y^{29} - 8 y^{30}
\end{equation*}
\begin{equation*}
  + 6 y^{31} + 10 y^{32} + 10 y^{33} + 4 y^{34} - 2 y^{35} - 3 y^{36} - 2 y^{37} - y^{38} + y^{40}.
\end{equation*}

\bigskip

\textbf{Combinatorial interpretation}: Consider the aggregate of vectors
\begin{equation}  \label{8.21}
   \begin{array}{cccc}
     \langle 1,5 \rangle & \langle 2,5 \rangle & \langle 3,5 \rangle & \langle 4,5 \rangle  \\
     \langle 1,4 \rangle & \langle 2,4 \rangle & \langle 3,4 \rangle &   \\
     \langle 1,3 \rangle & \langle 2,3 \rangle &   &     \\
     \langle 1,2 \rangle &   &   &
   \end{array}
\end{equation}
Let $p_{(2,e)}(\mathfrak{D}, \langle a,b \rangle)$, denote the number of partitions of $\langle a,b \rangle$ into an even number of distinct parts from (\ref{8.21}). Let $p_{(2,o)}(\mathfrak{D}, \langle a,b \rangle)$, denote the number of partitions of $\langle a,b \rangle$ into an odd number of distinct parts from (\ref{8.21}). Then we have the generating function
\begin{equation}  \label{8.22}
   \prod_{k=2}^{4} \prod_{j=1}^{k-1} (1-x^j y^k)
   = \sum_{b=2}^{20} \sum_{a=1}^{b-1} [p_{(2,e)}(\mathfrak{D}, \langle a,b \rangle)-p_{(2,o)}(\mathfrak{D}, \langle a,b \rangle) ] x^a y^b,
\end{equation}
encoding all  $20 \times 40$ grid numerical values of $p_{(2,e)}(\mathfrak{D}, \langle a,b \rangle)-p_{(2,o)}(\mathfrak{D}, \langle a,b \rangle)$, with respect to $\langle a,b \rangle)$.

\bigskip

\section{Introducing the VPV identities} \label{S:VPV identities}

In the 1990s and up to 2000 the author published a series of papers introducing the so-called Visible Point Vector (VPV) identities.
In these papers such as for example Campbell \cite{gC1994a}, \cite{gC1994b}, \cite{gC1998}, the identities of the present note were given.
They attracted some interest, but most people did not see them as much more than curiosities. Perhaps the proofs in the early papers were too obscure or cryptic. So, our approach here is to introduce the deeper $n$-dimensional identities by first trying to simply derive the easier 2-dimensional VPV identities.

\section{Deriving the $2D$ first quadrant VPV identity.} \label{S:2D VPV identity}

 We will derive the 2D VPV identity for the first quadrant integer lattice points. Let $(j,k)=1$ mean that $j$ and $k$ are coprime integers. We see that

\begin{equation}  \nonumber
  \left( \sum_{m=1}^{\infty}  \frac{y^m}{m^a} \right) \left( \sum_{n=1}^{\infty}  \frac{z^n}{n^b} \right)
\end{equation}

\begin{equation}  \nonumber
= \left( \frac{y^1}{1^a}+\frac{y^2}{2^a}+\frac{y^3}{3^a}+\frac{y^4}{4^a}+\frac{y^5}{5^a}+\frac{y^6}{6^a}+\frac{y^7}{7^a}+\cdots \right) \left( \frac{z^1}{1^b}+\frac{z^2}{2^b}+\frac{z^3}{3^b}+\frac{z^4}{4^b}+\frac{z^5}{5^b}+\frac{z^6}{6^b}+\frac{z^7}{7^b}+\cdots \right)
\end{equation}

\begin{equation}  \nonumber
 =\frac{y^1 z^1}{1^a 1^b}+\frac{y^2 z^1}{2^a 1^b}+\frac{y^3 z^1}{3^a 1^b}+\frac{y^4 z^1}{4^a 1^b}+\frac{y^5 z^1}{5^a 1^b}+\frac{y^6 z^1}{6^a 1^b}+\frac{y^7 z^1}{7^a 1^b}+\cdots
\end{equation}
\begin{equation}  \nonumber
 +\frac{y^1 z^2}{1^a 2^b}+\frac{y^2 z^2}{2^a 2^b}+\frac{y^3 z^2}{3^a 2^b}+\frac{y^4 z^2}{4^a 2^b}+\frac{y^5 z^2}{5^a 2^b}+\frac{y^6 z^2}{6^a 2^b}+\frac{y^7 z^2}{7^a 2^b}+\cdots
\end{equation}
\begin{equation}  \nonumber
 +\frac{y^1 z^3}{1^a 3^b}+\frac{y^2 z^3}{2^a 3^b}+\frac{y^3 z^3}{3^a 3^b}+\frac{y^4 z^3}{4^a 3^b}+\frac{y^5 z^3}{5^a 3^b}+\frac{y^6 z^3}{6^a 3^b}+\frac{y^7 z^3}{7^a 3^b}+\cdots
\end{equation}
\begin{equation}  \nonumber
 +\frac{y^1 z^4}{1^a 4^b}+\frac{y^2 z^4}{2^a 4^b}+\frac{y^3 z^4}{3^a 4^b}+\frac{y^4 z^4}{4^a 4^b}+\frac{y^5 z^4}{5^a 4^b}+\frac{y^6 z^4}{6^a 4^b}+\frac{y^7 z^4}{7^a 4^b}+\cdots
\end{equation}
\begin{equation}  \nonumber
 +\frac{y^1 z^5}{1^a 5^b}+\frac{y^2 z^5}{2^a 5^b}+\frac{y^3 z^5}{3^a 5^b}+\frac{y^4 z^5}{4^a 5^b}+\frac{y^5 z^5}{5^a 5^b}+\frac{y^6 z^5}{6^a 5^b}+\frac{y^7 z^5}{7^a 5^b}+\cdots
\end{equation}
\begin{equation}  \nonumber
 +\frac{y^1 z^6}{1^a 6^b}+\frac{y^2 z^6}{2^a 6^b}+\frac{y^3 z^6}{3^a 6^b}+\frac{y^4 z^6}{4^a 6^b}+\frac{y^5 z^6}{5^a 6^b}+\frac{y^6 z^6}{6^a 6^b}+\frac{y^7 z^6}{7^a 6^b}+\cdots
\end{equation}
\begin{equation}  \nonumber
 +\frac{y^1 z^7}{1^a 7^b}+\frac{y^2 z^7}{2^a 7^b}+\frac{y^3 z^7}{3^a 7^b}+\frac{y^4 z^7}{4^a 7^b}+\frac{y^5 z^7}{5^a 7^b}+\frac{y^6 z^7}{6^a 7^b}+\frac{y^7 z^7}{7^a 7^b}+\cdots
\end{equation}
\begin{equation}  \nonumber
 + \quad \vdots \quad + \quad \vdots \quad + \quad \vdots \quad + \quad \vdots \quad + \quad \vdots \quad + \quad \vdots \quad + \quad \vdots \quad   \ddots
\end{equation}
\begin{equation}  \nonumber
 = \sum_{m,n \geq 1}^{\infty}  \frac{y^m z^n}{m^a n^b}
\end{equation}
\begin{equation}  \nonumber
 = \sum_{\substack{ h \geq 1 \\ (j,k)=1}}^{\infty}  \frac{(y^j z^k)^h}{h^{a+b} (j^a k^b)}
\end{equation}
\begin{equation}  \nonumber
 = \sum_{(j,k)=1}^{\infty}  \frac{1}{(j^a k^b)}   \sum_{h=1}^{\infty} \frac{(y^j z^k)^h}{h^{a+b}}
\end{equation}
\begin{equation}  \nonumber
 = \sum_{(j,k)=1}^{\infty}  \frac{1}{(j^a k^b)}   \log \left( \frac{1}{1 - y^j z^k} \right) \quad if \quad a+b=1.
\end{equation}

Therefore, we have shown that
\begin{equation}  \nonumber
 \left( \sum_{m=1}^{\infty}  \frac{y^m}{m^a} \right) \left( \sum_{n=1}^{\infty}  \frac{z^n}{n^b} \right) = \sum_{(j,k)=1}^{\infty}  \frac{1}{(j^a k^b)}   \log \left( \frac{1}{1 - y^j z^k} \right) \quad if \quad a+b=1.
\end{equation}

 Exponentiating both sides gives us the 2D first quadrant VPV identity,
 \begin{equation}   \label{13.01}
    \prod_{\substack{ (j,k)=1 \\ j,k \geq 1}} \left( \frac{1}{1-y^j z^k} \right)^{\frac{1}{j^a k^b}}
    = \exp\left\{ \left( \sum_{j=1}^{\infty} \frac{y^j}{j^a}\right) \left( \sum_{k=1}^{\infty} \frac{z^k}{k^b}\right) \right\} \quad if \quad a+b=1.
  \end{equation}

 This approach will be useful to return to when considering derivation of 2D identities from first principles. Clearly, we have summed on the 2D first quadrant lattice points, that is on points with positive integer coordinates. We give immediate interesting cases of (\ref{13.01}) now.

 Taking $a = 0$, $b=1$ gives
 \begin{equation}   \label{13.02}
    \prod_{\substack{ (j,k)=1 \\ j,k \geq 1}} \left( \frac{1}{1-y^j z^k} \right)^{\frac{1}{k}}
    = \left(  \frac{1}{1-z}  \right)^{\frac{y}{1-y}}
  \end{equation}
  \begin{equation} \nonumber
    = \sum_{n=0}^{\infty} z^n (-1)^n \binom{\frac{y}{y-1}}{n}
  \end{equation}
  \begin{equation}   \nonumber
   = 1 - \frac{y z}{y - 1} + \frac{y z^2}{2! (y - 1)^2} + \frac{(y-2)y z^3}{3! (y - 1)^3}
    + \frac{(y-2)y (2 y - 3) z^4}{4! (y - 1)^4}
     \end{equation}
  \begin{equation}   \nonumber
    + \frac{(y-2)y (2 y - 3) (3 y - 4) z^5}{5! (y - 1)^5} + O(z^6) .
  \end{equation}

\bigskip

A partition grid part for $\left(  \frac{1}{1-z}  \right)^{\frac{y}{1-y}}$ for coefficients of $y^az^b$ with $0 \leq a \leq 9, \; 0 \leq b \leq 10$ is

\begin{equation} \nonumber
  \begin{array}{c|cccccccccc}
     \textbf{10} &   & \frac{362880}{10!} & \frac{1389456}{10!} & \frac{3588732}{10!} & \frac{7684388}{10!} & \frac{14669429}{10!} & \frac{25869458}{10!} & \frac{43015399}{10!} & \frac{68326540}{10!} & \frac{104604811}{10!} \\
     \textbf{9} &   & \frac{40320}{9!} & \frac{149904}{9!} & \frac{377612}{9!} & \frac{790728}{9!} & \frac{1478985}{9!} & \frac{2559101}{9!} & \frac{4179861}{9!} & \frac{6527781}{9!} & \frac{9833391}{9!} \\
     \textbf{8} &   & \frac{5040}{8!} & \frac{18108}{8!} & \frac{44308}{8!} & \frac{90409}{8!} & \frac{165140}{8!} & \frac{279512}{8!} & \frac{447168}{8!} & \frac{684762}{8!} & \frac{1012368}{8!} \\
     \textbf{7} &   & \frac{720}{7!} & \frac{2484}{7!} & \frac{5872}{7!} & \frac{11619}{7!} & \frac{20635}{7!} & \frac{34026}{7!} & \frac{53116}{7!} & \frac{79470}{7!} & \frac{114918}{7!} \\
     \textbf{6} &   & \frac{120}{6!} & \frac{394}{6!} & \frac{893}{6!} & \frac{1702}{6!} & \frac{2921}{6!} & \frac{4666}{6!} & \frac{7070}{6!} & \frac{10284}{6!} & \frac{14478}{6!} \\
     \textbf{5} &   & \frac{24}{5!} & \frac{74}{5!} & \frac{159}{5!} & \frac{289}{5!} & \frac{475}{5!} & \frac{729}{5!} & \frac{1064}{5!} & \frac{1494}{5!} & \frac{2034}{5!} \\
     \textbf{4} &   & \frac{6}{4!} & \frac{17}{4!} & \frac{34}{4!} & \frac{58}{4!} & \frac{90}{4!} & \frac{131}{4!} & \frac{182}{4!} & \frac{244}{4!} & \frac{318}{4!} \\
     \textbf{3} &   & \frac{2}{3!} & \frac{5}{3!} & \frac{9}{3!} & \frac{14}{3!} & \frac{20}{3!} & \frac{27}{3!} & \frac{35}{3!} & \frac{44}{3!} & \frac{54}{3!} \\
     \textbf{2} &   & \frac{1}{2!} & \frac{2}{2!}  & \frac{3}{2!} & \frac{4}{2!}  & \frac{5}{2!} & \frac{6}{2!}  & \frac{7}{2!} & \frac{8}{2!}  & \frac{9}{2!} \\
     \textbf{1} &   & 1 & 1 & 1 & 1 & 1 & 1 & 1 & 1 & 1 \\
     \textbf{0} & 1 &   &   &   &   &   &   &   &   &    \\ \hline
     \; b/a  & \textbf{0} & \textbf{1} & \textbf{2} & \textbf{3} & \textbf{4} & \textbf{5} & \textbf{6} & \textbf{7} & \textbf{8} & \textbf{9}
  \end{array}
  \end{equation}

\bigskip

 Or equivalent to (\ref{13.02}) is it's reciprocal of both sides,
 \begin{equation}   \label{13.03}
    \prod_{\substack{ (j,k)=1 \\ j,k \geq 1}} \left(1-y^j z^k \right)^{\frac{1}{k}}
    = \left(  1-z \right)^{\frac{y}{1-y}}
  \end{equation}
  \begin{equation} \nonumber
    = \sum_{n=0}^{\infty} z^n (-1)^n \binom{\frac{y}{1-y}}{n}
  \end{equation}
  \begin{equation}   \nonumber
   = 1 + \frac{y z}{y - 1} + \frac{y (2 y - 1) z^2}{2! (y - 1)^2} + \frac{y (2 y - 1) (3 y - 2) z^3}{3! (y - 1)^3}
    + \frac{y (2 y - 1) (3 y - 2) (4 y - 3) z^4}{4! (y - 1)^4}
     \end{equation}
  \begin{equation}   \nonumber
    + \frac{y (2 y - 1) (3 y - 2) (4 y - 3) (5 y - 4) z^5}{5! (y - 1)^5} + O(z^6) .
  \end{equation}

\bigskip

A partition grid part for $\left(  1-z \right)^{\frac{y}{1-y}}$ for coefficients of $y^az^b$ with $0 \leq a \leq 9, \; 0 \leq b \leq 10$ is

\begin{equation} \nonumber
  \begin{array}{c|cccccccccc}
     \textbf{10} &   & \frac{-362880}{10!} & \frac{663696}{10!} & \frac{517572}{10!} & \frac{-77572}{10!} & \frac{-667381}{10!} & \frac{-1003552}{10!} & \frac{-990011}{10!} & \frac{-637670}{10!} & \frac{-26939}{10!} \\
     \textbf{9} &   & \frac{-40320}{9!} & \frac{69264}{9!} & \frac{60724}{9!} & \frac{1344}{9!} & \frac{-64041}{9!} & \frac{-108509}{9!} & \frac{-119061}{9!} & \frac{-93141}{9!} & \frac{-35631}{9!} \\
     \textbf{8} &   & \frac{-5040}{8!} & \frac{8028}{8!} & \frac{7964}{8!} & \frac{1537}{8!} & \frac{-6444}{8!} & \frac{-12808}{8!} & \frac{-15728}{8!} & \frac{-14454}{8!} & \frac{-9072}{8!} \\
     \textbf{7} &   & \frac{-720}{7!} & \frac{1044}{7!} & \frac{1184}{7!} & \frac{435}{7!} & \frac{-643}{7!} & \frac{-1644}{7!} & \frac{-2296}{7!} & \frac{-2442}{7!} & \frac{-2022}{7!} \\
     \textbf{6} &   & \frac{-120}{6!} & \frac{154}{6!} & \frac{203}{6!} & \frac{112}{6!} & \frac{-49}{6!} & \frac{-224}{6!} & \frac{-370}{6!} & \frac{-456}{6!} & \frac{-462}{6!} \\
     \textbf{5} &   & \frac{-24}{5!} & \frac{26}{5!} & \frac{42}{5!} & \frac{31}{5!} & \frac{5}{5!} & \frac{-29}{5!} & \frac{-64}{5!} & \frac{-94}{5!} & \frac{-114}{5!} \\
     \textbf{4} &   & \frac{-6}{4!} & \frac{5}{4!} & \frac{10}{4!} & \frac{10}{4!} & \frac{6}{4!} & \frac{-1}{4!} & \frac{-10}{4!} & \frac{-20}{4!} & \frac{-30}{4!} \\
     \textbf{3} &   & \frac{-2}{3!} & \frac{1}{3!} & \frac{3}{3!} & \frac{4}{3!} & \frac{4}{3!} & \frac{3}{3!} & \frac{1}{3!} & \frac{-2}{3!} & \frac{-6}{3!} \\
     \textbf{2} &   & \frac{1}{2!} & \frac{0}{2!}  & \frac{-1}{2!} & \frac{-2}{2!}  & \frac{-3}{2!} & \frac{-4}{2!}  & \frac{-5}{2!} & \frac{-6}{2!}  & \frac{-7}{2!} \\
     \textbf{1} &   & 1 & 1 & 1 & 1 & 1 & 1 & 1 & 1 & 1 \\
     \textbf{0} & 1 &   &   &   &   &   &   &   &   &    \\ \hline
     \; b/a  & \textbf{0} & \textbf{1} & \textbf{2} & \textbf{3} & \textbf{4} & \textbf{5} & \textbf{6} & \textbf{7} & \textbf{8} & \textbf{9}
  \end{array}
  \end{equation}

\bigskip

From the case of (\ref{13.03}) with $y^2$ and $z^2$ replacing $y$ and $z$, and dividing both sides of (\ref{13.03}) into it we have,
 \begin{equation}   \label{13.04}
    \prod_{\substack{ (j,k)=1 \\ j,k \geq 1}} \left(1+y^j z^k \right)^{\frac{1}{k}}
    = \frac{\left(  1-z^2 \right)^{\frac{y^2}{1-y^2}}}{\left(  1-z \right)^{\frac{y}{1-y}}}  .
  \end{equation}
\begin{equation} \nonumber
  = \sum_{n=0}^{\infty} \frac{\Gamma\left(n + \frac{y}{1-y}\right)}{\Gamma\left(\frac{y}{1-y}\right)} \, _3F_2\left(\frac{1}{2} - \frac{n}{2}, -\frac{n}{2}, -\frac{y^2}{1-y^2} ;\frac{1}{2} - \frac{n}{2} - \frac{y}{2 (1-y)}, 1 - \frac{n}{2} - \frac{y}{2 (1-y)};1\right)\frac{z^n}{n!}.
\end{equation}
 So, in (\ref{13.01}) we have derived the VPV 2D first quadrant identity summed upon the visible points in that quadrant as depicted in yellow shaded area of Figure \ref{Fig18a}.

 We note that from the power series in $z$ for (\ref{13.03}), the series is finite when $y = \frac{1}{2}, \frac{2}{3}, \frac{3}{4}, \frac{4}{5}, \frac{5}{6}, \ldots$ leading us to the following identities,

 \begin{equation}   \nonumber
    \prod_{\substack{ (j,k)=1 \\ j,k \geq 1}} \left(1- \left(\frac{1}{2}\right)^j z^k \right)^{\frac{1}{k}}
    = 1-z,
  \end{equation}

 \begin{equation}   \nonumber
    \prod_{\substack{ (j,k)=1 \\ j,k \geq 1}} \left(1- \left(\frac{2}{3}\right)^j z^k \right)^{\frac{1}{k}}
    = (1-z)^2 = 1-2z+z^2,
  \end{equation}

  \begin{equation}   \nonumber
    \prod_{\substack{ (j,k)=1 \\ j,k \geq 1}} \left(1- \left(\frac{3}{4}\right)^j z^k \right)^{\frac{1}{k}}
    = (1-z)^3 = 1-3z+3z^2-z^3,
  \end{equation}

  \begin{equation}   \nonumber
    \prod_{\substack{ (j,k)=1 \\ j,k \geq 1}} \left(1- \left(\frac{4}{5}\right)^j z^k \right)^{\frac{1}{k}}
    = (1-z)^4 = 1-4z+6z^2-4z^3+z^4,
  \end{equation}

  \begin{equation}   \nonumber
    \prod_{\substack{ (j,k)=1 \\ j,k \geq 1}} \left(1- \left(\frac{5}{6}\right)^j z^k \right)^{\frac{1}{k}}
    = (1-z)^5 = 1-5z+10z^2-10z^3+5z^4-z^5,
  \end{equation}

 and so on. These infinite products tell an interesting story as they are a product of infinite series that reduce to a finite polynomial.

 Similarly, for (\ref{13.02}), the series is finite when $y = 2, \frac{3}{2}, \frac{4}{3}, \frac{5}{4}, \frac{6}{5}, \ldots$ leading us to the following identities,

 \begin{equation}   \nonumber
    \prod_{\substack{ (j,k)=1 \\ j,k \geq 1}} \left(\frac{1}{1- 2^j z^k} \right)^{\frac{1}{k}}
    = (1-z)^2 = 1-2z+z^2,
  \end{equation}

 \begin{equation}   \nonumber
    \prod_{\substack{ (j,k)=1 \\ j,k \geq 1}} \left(\frac{1}{1- \left(\frac{3}{2}\right)^j z^k} \right)^{\frac{1}{k}}
    = (1-z)^3 = 1-3z+3z^2-z^3,
  \end{equation}

  \begin{equation}   \nonumber
    \prod_{\substack{ (j,k)=1 \\ j,k \geq 1}} \left(\frac{1}{1- \left(\frac{4}{3}\right)^j z^k} \right)^{\frac{1}{k}}
    = (1-z)^4 = 1-4z+6z^2-4z^3+z^4,
  \end{equation}

  \begin{equation}   \nonumber
    \prod_{\substack{ (j,k)=1 \\ j,k \geq 1}} \left(\frac{1}{1- \left(\frac{5}{4}\right)^j z^k} \right)^{\frac{1}{k}}
    = (1-z)^5 = 1-5z+10z^2-10z^3+5z^4-z^5,
  \end{equation}

  \begin{equation}   \nonumber
    \prod_{\substack{ (j,k)=1 \\ j,k \geq 1}} \left(\frac{1}{1- \left(\frac{6}{5}\right)^j z^k} \right)^{\frac{1}{k}}
    = (1-z)^6 = 1-6z+15z^2-20z^3+15z^4-6z^5+z^6,
  \end{equation}

 and so on.

 Another obvious case of (\ref{13.01}) to consider is where $a=b=\frac{1}{2}$, so then

 \begin{equation}   \label{13.05}
    \prod_{\substack{ (j,k)=1 \\ j,k \geq 1}} \left( \frac{1}{1-y^j z^k} \right)^{\frac{1}{\sqrt{jk}}}
    = \exp\left\{ \left( \sum_{j=1}^{\infty} \frac{y^j}{\sqrt{j}}\right) \left( \sum_{k=1}^{\infty} \frac{z^k}{\sqrt{k}}\right) \right\},
  \end{equation}
 and equivalently,
 \begin{equation}   \label{13.06}
    \prod_{\substack{ (j,k)=1 \\ j,k \geq 1}} \left( 1-y^j z^k \right)^{\frac{1}{\sqrt{jk}}}
    = \exp\left\{ - \left( \sum_{j=1}^{\infty} \frac{y^j}{\sqrt{j}}\right) \left( \sum_{k=1}^{\infty} \frac{z^k}{\sqrt{k}}\right) \right\}.
  \end{equation}
 From the case of (\ref{13.05}) with $y^2$ and $z^2$ replacing $y$ and $z$, and dividing both sides of (\ref{13.05}) into it we have

 \begin{equation}   \label{13.07}
    \prod_{\substack{ (j,k)=1 \\ j,k \geq 1}} \left( 1 + y^j z^k \right)^{\frac{1}{\sqrt{jk}}}
    = \exp\left\{ \left( \sum_{j=1}^{\infty} \frac{y^j}{\sqrt{j}}\right) \left( \sum_{k=1}^{\infty} \frac{z^k}{\sqrt{k}}\right)
                - \left( \sum_{j=1}^{\infty} \frac{y^{2j}}{\sqrt{j}}\right) \left( \sum_{k=1}^{\infty} \frac{z^{2k}}{\sqrt{k}}\right)\right\}.
  \end{equation}

  \begin{figure} [ht!]
\centering
    \includegraphics[width=15.5cm,angle=0,height=12.48cm]{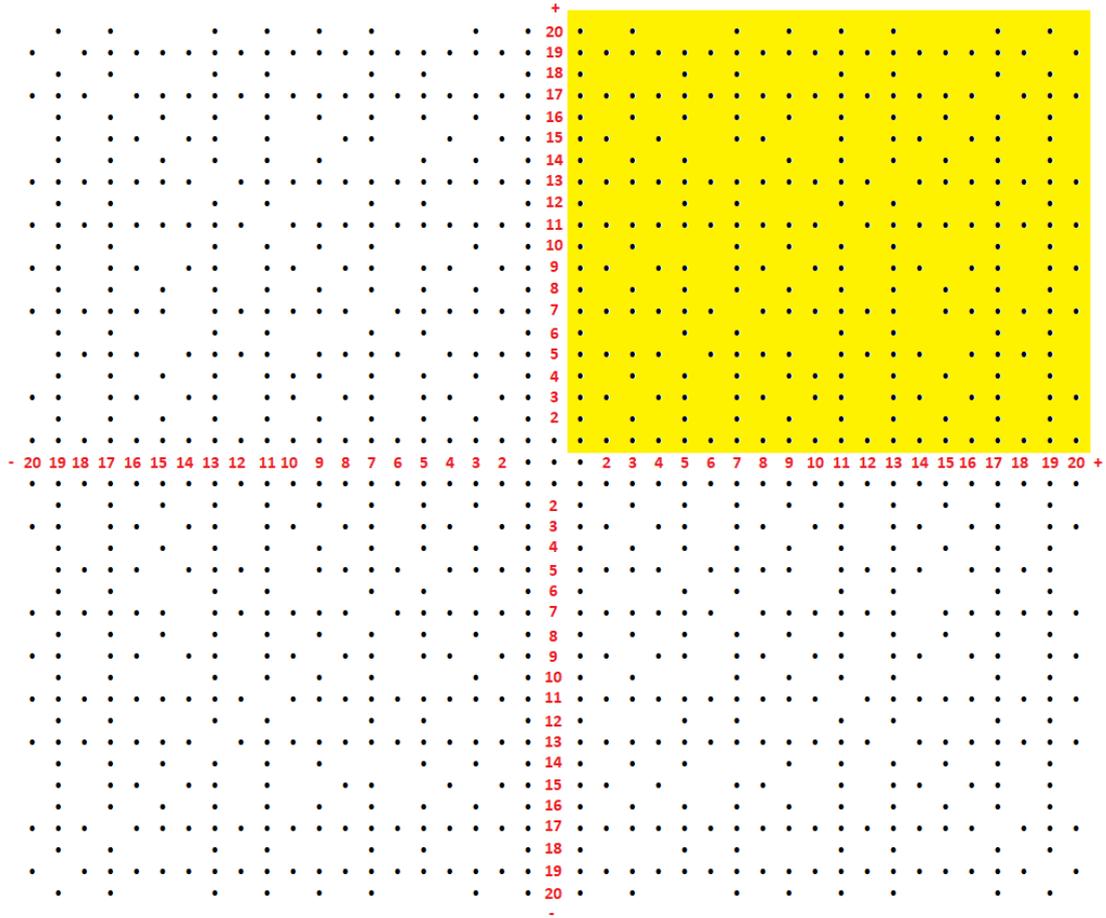}
  \caption{VPV first quadrant 2D lattice points region.}\label{Fig18a}
\end{figure}

\section{Deriving the $n$-dimensional first hyperquadrant VPV identity.} \label{S:nD VPV identity}

 They are however, generating functions for weighted vector partitions, and evidently contribute a new branch to the literature on partitions of vectors.

So, we establish the following conventions for the VPV identities.
\begin{definition}
  We use the notation $\left(x_1, x_2, x_3,...,x_n\right)$, to mean “the greatest common divisor of all of
$x_1, x_2, x_3,...,x_n$ together; the same as $\gcd\left(x_1, x_2, x_3,...,x_n\right)$”.
It is important to distinguish between this and the ordered $n$-tuple utilized for the vector
$\langle x_1, x_2, x_3,...,x_n \rangle$. In either case we will be concerned with lattice points
in the relevant Euclidean space, hence any vector or gcd will be over integer coordinates.
\end{definition}

\begin{definition}
  Any Euclidean vector $\langle x_1, x_2, x_3,...,x_n \rangle$ for which
$\left( x_1, x_2, x_3,...,x_n \right)=1$ we call a visible point vector, abbreviated VPV.
\end{definition}

\begin{theorem}   \label{8.1a}
  The first hyperquadrant VPV identity. If $i = 1, 2, 3,...,n$ then for each $x_i \in \mathbb{C}$ such that $|x_i|<1$ and $b_i \in \mathbb{C}$ such that $\sum_{i=1}^{n}b_i = 1$,
  \begin{equation}   \label{13.08}
    \prod_{\substack{ (a_1,a_2,...,a_n)=1 \\ a_1,a_2,...,a_n \geq 1}} \left( \frac{1}{1-{x_1}^{a_1}{x_2}^{a_2}{x_3}^{a_3}\cdots{x_n}^{a_n}} \right)^{\frac{1}{{a_1}^{b_1}{a_2}^{b_2}{a_3}^{b_3}\cdots{a_n}^{b_n}}}
  \end{equation}
  \begin{equation}  \nonumber
  = \exp\left\{ \prod_{i=1}^{n} \left( \sum_{j=1}^{\infty} \frac{{x_i}^j}{j^{b_i}} \right)\right\}
  = \exp\left\{ \prod_{i=1}^{n} Li_{b_i}(x_i)\right\}.
  \end{equation}
\end{theorem}

There follow numerous example corollaries of this theorem, all of them susceptible to the combinatorial analysis of the previous sections. However, firstly we give the lemma and proof underpinning theorem \ref{13.08}.

\begin{lemma}
  Consider an infinite region raying out of the origin in any Euclidean
vector space. The set of all lattice point vectors apart from the origin in that region is
precisely the set of positive integer multiples of the VPVs in that region.
\end{lemma}
\begin{proof}
  Each VPV will have integer coordinates whose greatest common divisor
is unity. Viewed from the origin, all other lattice points are obscured behind the VPV end
points. If $x$ is a VPV in the region then all vectors in that region from the origin with direction
of $x$ preserved are enumerated by a sequence $1x, 2x, 3x,...$, and the greatest
common divisor of the components of $nx$ is clearly $n$. This is because if the scalar $n$ is
non-integer at least one of the coordinates of $nx$ would be a non-integer. Therefore, if
the VPVs in the region are countably given by $x_1, x_2, x_3, ...$, then all lattice point vectors
from the origin in the region are
$1x_1,2x_1,3x_1,...; 1x_2,2x_2,3x_2,...; 1x_3,2x_3,3x_3,...$ etc.
Completion of the proof comes with recognition that the set of all VPVs in any \emph{rayed
from the origin} region in any Euclidean vector space is a countable set. Proof of this last
assertion is by induction on the dimension, knowing that the lattice points are countable
in any two dimensional region. As we count each lattice point vector in the desired region
we decide whether it is a VPV simply by observing whether its coordinates are relatively
prime as a whole.
\end{proof}

This then brings us to the proof of theorem \ref{13.08}.
\begin{proof}
  We start with the multiple sum

  \begin{equation} \nonumber
    \sum_{ a_1,a_2,...,a_n \in \mathbb{Z}^+} \frac{{x_1}^{a_1}{x_2}^{a_2}{x_3}^{a_3}\cdots{x_n}^{a_n}}{{a_1}^{b_1}{a_2}^{b_2}{a_3}^{b_3}\cdots{a_n}^{b_n}}
    = \prod_{i=1}^{n} \sum_{j=1}^{\infty} \frac{{x_i}^j}{j^{b_i}}
    \end{equation}
  which, due to Lemma \ref{13.08}, also equals, letting $b = \sum_{i=1}^{n} b_i$,
 \begin{tiny} \begin{equation} \nonumber
    \sum_{\substack{ (a_1,a_2,...,a_n)=1 \\ a_1,a_2,...,a_n \geq 1}}
           \left( \frac{{x_1}^{a_1}{x_2}^{a_2}\cdots{x_n}^{a_n}}{1^b}
            + \frac{({x_1}^{a_1}{x_2}^{a_2}\cdots{x_n}^{a_n})^2}{2^b}
            + \frac{({x_1}^{a_1}{x_2}^{a_2}\cdots{x_n}^{a_n})^3}{3^b}
             + \cdots\right)
             \frac{1}{{a_1}^{b_1}{a_2}^{b_2}{a_3}^{b_3}\cdots{a_n}^{b_n}}
  \end{equation}\end{tiny}
  \begin{equation} \nonumber
    = \sum_{\substack{ (a_1,a_2,...,a_n)=1 \\ a_1,a_2,...,a_n \geq 1}}
             \frac{- \log( 1 - {x_1}^{a_1}{x_2}^{a_2}\cdots{x_n}^{a_n})}{{a_1}^{b_1}{a_2}^{b_2}{a_3}^{b_3}\cdots{a_n}^{b_n}}
  \end{equation}
  Exponentiating both sides then yields Theorem 7.1.
\end{proof}
The cases of theorem \ref{13.08} with $n=2 , n=3, n=4, n=5$, are stated easily in the forms,

\begin{corollary} If $|y|<1, |z|<1,$ and $s+t=1$, then
  \begin{equation}   \label{13.09}
    \prod_{\substack{ (a,b)=1 \\ a,b \geq 1}} \left( \frac{1}{1-y^a z^b} \right)^{\frac{1}{a^s b^t}}
    = \exp\left\{ \left( \sum_{j=1}^{\infty} \frac{y^j}{j^s}\right) \left( \sum_{k=1}^{\infty} \frac{z^k}{k^t}\right) \right\}.
  \end{equation}
\end{corollary}

\begin{corollary} If $|x|<1, |y|<1, |z|<1 $ and $s+t+u=1$, then
  \begin{equation}   \label{13.10}
    \prod_{\substack{ (a,b,c)=1 \\ a,b,c \geq 1}} \left( \frac{1}{1-x^a y^b z^c} \right)^{\frac{1}{a^s b^t c^u}}
    = \exp\left\{ \left( \sum_{i=1}^{\infty} \frac{x^i}{i^s}\right) \left( \sum_{j=1}^{\infty} \frac{y^j}{j^t}\right)
    \left( \sum_{k=1}^{\infty} \frac{z^k}{k^u}\right)\right\}.
  \end{equation}
\end{corollary}

\begin{corollary} If $|w|<1, |x|<1, |y|<1, |z|<1 $ and $r+s+t+u=1$, then
  \begin{equation}   \label{13.11}
    \prod_{\substack{ (a,b,c,d)=1 \\ a,b,c,d \geq 1}} \left( \frac{1}{1-w^a x^b y^c z^d} \right)^{\frac{1}{a^r b^s c^t d^u}}
    = \exp\left\{ \left( \sum_{h=1}^{\infty} \frac{w^h}{h^r}\right) \left( \sum_{i=1}^{\infty} \frac{x^i}{i^s}\right)
    \left( \sum_{j=1}^{\infty} \frac{y^j}{j^t}\right) \left( \sum_{k=1}^{\infty} \frac{z^k}{k^u}\right)\right\}.
  \end{equation}
\end{corollary}

\begin{corollary} If $|v|<1, |w|<1, |x|<1, |y|<1, |z|<1 $ and $q+r+s+t+u=1$, then
  \begin{equation}   \label{13.12}
    \prod_{\substack{ (a,b,c,d,e)=1 \\ a,b,c,d,e \geq 1}} \left( \frac{1}{1-v^a w^b x^c y^d z^e} \right)^{\frac{1}{a^q b^r c^s d^t e^u}}
    \end{equation}
  \begin{equation} \nonumber
  = \exp\left\{ \left( \sum_{g=1}^{\infty} \frac{v^g}{g^q}\right) \left( \sum_{h=1}^{\infty} \frac{w^h}{h^r}\right) \left( \sum_{i=1}^{\infty} \frac{x^i}{i^s}\right)
    \left( \sum_{j=1}^{\infty} \frac{y^j}{j^t}\right) \left( \sum_{k=1}^{\infty} \frac{z^k}{k^u}\right)\right\}.
  \end{equation}
\end{corollary}

The reader will recognize the polylogarithm occurring in the right sides of (\ref{13.09}) through to (\ref{13.12}). The many known particular values of the polylogarithms as combinations of generalized Mordell–Tornheim–Witten (MTW) zeta-function values along with their derivatives, and recently found connections with multiple zeta values (MZVs) implies a lot of possible future research. For example, taking the substitution where $b_k = \frac{1}{n}$  for all  $1 \leq k \leq n$, we have interesting new identities such as,

\begin{equation}   \label{13.13}
    \prod_{\substack{ (a,b)=1 \\ a,b \geq 1}} \left( \frac{1}{1-y^a z^b} \right)^{\frac{1}{\surd{(ab)}}}
    = \exp\left\{ \left( \sum_{j=1}^{\infty} \frac{y^j}{\surd{j}}\right) \left( \sum_{k=1}^{\infty} \frac{z^k}{\surd{k}}\right) \right\},
  \end{equation}

  \begin{equation}   \label{13.14}
    \prod_{\substack{ (a,b,c)=1 \\ a,b,c \geq 1}} \left( \frac{1}{1-x^a y^b z^c} \right)^\frac{1}{\sqrt[3]{abc}}
    = \exp\left\{ \left( \sum_{i=1}^{\infty} \frac{x^i}{\sqrt[3]{i}}\right) \left( \sum_{j=1}^{\infty} \frac{y^j}{\sqrt[3]{j}}\right)
    \left( \sum_{k=1}^{\infty} \frac{z^k}{\sqrt[3]{k}}\right)\right\}.
  \end{equation}

  \begin{equation}   \label{13.15}
    \prod_{\substack{ (a,b,c,d)=1 \\ a,b,c,d \geq 1}} \left( \frac{1}{1-w^a x^b y^c z^d} \right)^{\frac{1}{\sqrt[4]{abcd}}}
    = \exp\left\{ \left( \sum_{h=1}^{\infty} \frac{w^h}{\sqrt[4]{h}}\right) \left(\sum_{i=1}^{\infty} \frac{x^i}{\sqrt[4]{i}}\right)
    \left(\sum_{j=1}^{\infty} \frac{y^j}{\sqrt[4]{j}}\right) \left(\sum_{k=1}^{\infty} \frac{z^k}{\sqrt[4]{k}}\right)\right\}.
  \end{equation}

  \begin{equation}   \label{13.16}
    \prod_{\substack{ (a,b,c,d,e)=1 \\ a,b,c,d,e \geq 1}} \left( \frac{1}{1-v^a w^b x^c y^d z^e} \right)^{\frac{1}{\sqrt[5]{abcde}}}
    \end{equation}
  \begin{equation} \nonumber
  = \exp\left\{ \left( \sum_{g=1}^{\infty} \frac{v^g}{\sqrt[5]{g}}\right) \left( \sum_{h=1}^{\infty} \frac{w^h}{\sqrt[5]{h}}\right) \left( \sum_{i=1}^{\infty} \frac{x^i}{\sqrt[5]{i}}\right) \left( \sum_{j=1}^{\infty} \frac{y^j}{\sqrt[5]{j}}\right)
   \left( \sum_{k=1}^{\infty} \frac{z^k}{\sqrt[5]{k}}\right)\right\}.
  \end{equation}

\section{Hyperquadrant lattices and their hyperdiagonal line functions}

The above equations (\ref{13.13}) to (\ref{13.16}) when represented by their 2D, 3D, 4D and 5D grids in their respective first quadrants and first hyperquadrants, obviously become unwieldy to visualize after the 2D extended rectangular lattice, or relevant 3D extended cubic lattice. What human consciousness can say they see a 4D extended tesseract (4D cube) lattice, or a 5D extended hypercube lattice? However we understand this dimensional concept, there is a logical higher space extension also of the 2D first quadrant diagonal of lattice points upon the line $y=x$. In the 3D cube first hyperquadrant in an $x$-$y$-$z$ plane of lattice points there is the diagonal line represented by the equation $x=y=z$ and the coefficients of those lattice point vectors arise from setting the generating function $f(x,y,z)$ to say $f(z,z,z)$ which then is a generating function for a particular cluster of 3D diagonal lattice point vectors in that hyperquadrant. Similarly for the 4D case generating function $f(w,x,y,z)$ calculate $f(z,z,z,z)$ to give the 4D hyperdiagonal generating function for a cluster of lattice points along the line with equation $w=x=y=z$ through the 4D extended tesseract (hypecube). And so on for the 5D function generated by $f(v,w,x,y,z)$ giving us the function $f(z,z,z,z,z)$ generating the coefficients (vector partition sums) for the lattice points along the 5D line $v=w=x=y=z$. It's tricky, especially when impossible to visualize, but this concept can be applied to any of our 2D, 3D, 4D, etc lattice point vector grids throughout our present volume. It applies to VPV identities as \textit{hyperquadrant} lattice functions, \textit{square hyperpyramid} lattice functions and \textit{skewed hyperpyramid} lattice point vector identities. This includes the many possibilities of applying \textit{polylogarithm} formulas and \textit{Parametric Euler sum} identities; examples given in our present book being just the starting place of a large number of possibilities.

We are reminded by (\ref{13.09}) of the functional equation due originally to Riemann in his famous paper \cite{gR1859} on the Riemann zeta function. Both have the $s+t=1$ caveat. Riemann's zeta function reflection
formula is equivalent to

\begin{equation}\label{13.17}
  \Gamma\left( \frac{s}{2} \right)\pi^{-s/2}\zeta(s) = \Gamma\left( \frac{t}{2} \right)\pi^{-t/2}\zeta(t)
\end{equation}
  where $s+t=1$, but equation (\ref{13.09}) is quite a different relationship in a context amenable to the critical line Riemann zeta function $\zeta\left( \frac{1}{2} + i t\right)$ for nontrivial zeroes.

There are several further corollary cases that we can state here, that may be susceptible to the analysis of the earlier sections. There are natural and simple cases of Theorem \ref{8.1a} to consider.

\section{Some hyperdiagonal line generating functions} \label{nDlattice01}

We return to the second most recent section and the (\ref{13.13}) to (\ref{13.16}) equations and find particular cases that follow easily from setting $v=w=x=y=z$ into them. So we have 2D, 3D, 4D and 5D examples of the hyperdiagonal generating functions mentioned in our previous section.

If $|z|<1$, then

\begin{equation}   \label{13.18}
    \prod_{\substack{ \gcd(a,b)=1 \\ a,b \geq 1}} \left( \frac{1}{1- z^{a+b}} \right)^{\frac{1}{\surd{(ab)}}}
    = \exp\left\{ \left( \sum_{k=1}^{\infty} \frac{z^k}{\surd{k}}\right)^2 \right\}
  \end{equation}
  \begin{equation*}
    = 1 + z^2 + \sqrt{2}  z^3 + \left(1 + \frac{2}{\sqrt{3}}\right) z^4 + \left(1 + \sqrt{\frac{2}{3}} + \sqrt{2}\right) z^5 + \left(2 + \frac{1}{\sqrt{2}} + \frac{2}{\sqrt{3}} + \frac{2}{\sqrt{5}}\right) z^6
  \end{equation*}
  \begin{equation*}
    + \left(1 + \sqrt{\frac{2}{5}} + 4 \sqrt{\frac{2}{3}} + \sqrt{2} + \frac{1}{\sqrt{3}}\right) z^7 + \left(\frac{8}{3} + \frac{3}{\sqrt{2}} + \frac{5}{\sqrt{3}} + \frac{2}{\sqrt{5}} + \frac{2}{\sqrt{7}} + \frac{2}{\sqrt{15}}\right) z^8
  \end{equation*}
    \begin{equation*}
    + \left(2 + \sqrt{\frac{2}{7}} + 3 \sqrt{\frac{2}{5}} + 4 \sqrt{\frac{2}{3}} + \frac{17}{3 \sqrt{2}} + \sqrt{3} + \frac{1}{\sqrt{5}}\right) z^9
  \end{equation*}
    \begin{equation*}
    + \left(\frac{14}{3} + 2 \sqrt{\frac{3}{5}} + \frac{3}{\sqrt{2}} + 3 \sqrt{3} + \frac{4}{\sqrt{5}} + \frac{7}{\sqrt{6}} + \frac{2}{\sqrt{7}} + \frac{2}{\sqrt{21}}\right) z^{10} + O(z^{11}).
  \end{equation*}
  The coefficients of $z^n$ for positive integer $n$ in (\ref{13.18}) are the diagonal entries on the 2D $y$-$z$ grid for identity (\ref{13.13}), namely, lattice points along the line $y=z$ in the first quadrant.

  \begin{equation}   \label{13.19}
    \prod_{\substack{ \gcd(a,b,c)=1 \\ a,b,c \geq 1}} \left( \frac{1}{1- z^{a+b+c}} \right)^\frac{1}{\sqrt[3]{abc}}
    = \exp\left\{ \left( \sum_{k=1}^{\infty} \frac{z^k}{\sqrt[3]{k}}\right)^3\right\}
  \end{equation}
      \begin{equation*}
    =1+ z^3 + \frac{3}{\sqrt[3]{2}}z^4 + \left(\frac{3}{\sqrt[3]{4}} + \sqrt[3]{9} \right) z^5 + \left(1+ \frac{3}{\sqrt[3]{4}} + \sqrt[3]{36} \right) z^6
  \end{equation*}
    \begin{equation*}
    + \left(3 + \sqrt[3]{\frac{9}{4}} + \frac{3}{\sqrt[3]{2}} + \sqrt[3]{3} + \frac{3}{\sqrt[3]{5}} \right) z^7 + + O(z^{8}).
  \end{equation*}
  The coefficients of $z^n$ for positive integer $n$ in (\ref{13.19}) are the diagonal entries on the 3D $x$-$y$-$z$ grid for identity (\ref{13.14}), namely, lattice points along the 3D line with equation $x=y=z$ in the first hyperquadrant with $x$, $y$ and $z$ all positive integers.

  \begin{equation}   \label{13.20}
    \prod_{\substack{ \gcd(a,b,c,d)=1 \\ a,b,c,d \geq 1}} \left( \frac{1}{1- z^{a+b+c+d}} \right)^{\frac{1}{\sqrt[4]{abcd}}}
    = \exp\left\{ \left(\sum_{k=1}^{\infty} \frac{z^k}{\sqrt[4]{k}}\right)^4\right\},
  \end{equation}

  \begin{equation}   \label{13.21}
    \prod_{\substack{ \gcd(a,b,c,d,e)=1 \\ a,b,c,d,e \geq 1}} \left( \frac{1}{1- z^{a+b+c+d+e}} \right)^{\frac{1}{\sqrt[5]{abcde}}}
    = \exp\left\{ \left( \sum_{k=1}^{\infty} \frac{z^k}{\sqrt[5]{k}}\right)^5\right\}.
  \end{equation}

Each of the equations (\ref{13.18}) to (\ref{13.21}) can give us a combinatorial theorem on relation between two arithmetical functions. For example, (\ref{13.18}) can be rewritten as the equation

\begin{equation}   \label{13.22}
    \prod_{k=2}^{\infty} \left( \frac{1}{1- z^k} \right)^{\sum \frac{1}{\surd{(ab)}}}
    = \exp\left\{ \left( \sum_{k=1}^{\infty} \frac{z^k}{\sqrt{k}}\right)^2 \right\},
  \end{equation}

where $\sum \frac{1}{\surd{(ab)}}$ is the sum over all possible coprime $a$ and $b$ positive integers such that $a+b=k$.

Similarly, (\ref{13.19}) can be rewritten as the equation

\begin{equation}   \label{13.23}
    \prod_{k=3}^{\infty} \left( \frac{1}{1- z^k} \right)^{\sum \frac{1}{\sqrt[3]{abc}}}
    = \exp\left\{ \left( \sum_{k=1}^{\infty} \frac{z^k}{\sqrt[3]{k}}\right)^3 \right\},
  \end{equation}

where $\sum \frac{1}{\sqrt[3]{abc}}$ is the sum over all possible triplewise coprime $a$, $b$ and $c$ positive integers such that $a+b+c=k$.

The identities (\ref{13.22}) and (\ref{13.23}) seem \textit{a priori} to have some useful aspect in relation to examining the \textit{abc Conjecture}. This is due to the feature of the identities having coprime numbers in both a sum and also in a product within each term of the infinite product left sides.

We note also that both (\ref{13.22}) and (\ref{13.23}) may fit the conditions allowing us to apply the saddle point asymptotic Theorem of Meinardus, thereby determining the asymptotic behaviour of coefficients of the infinite products that seem to relate to the \emph{abc Conjecture}.

Let us first enlarge the theorem’s positive coordinate hyperquadrant to include lattice points on each axis except for the highest or $n$th dimension. In other words, the product operator for variable $z$ on each left side of (\ref{13.24}) to (\ref{13.27}) runs over each integer 1, 2, 3,... whereas for the non-$z$ variables $v, w, x, y$, the product is over 0, 1, 2, 3,....

In our general term notation, with $b_1 + b_2 + b_3 + ... + b_n = 1$, applied to slightly enlarge the number of lattice point vectors in the $n$-space hyperquadrant \textit{radial from the origin} region, we sum on the particular lattice points with vectors defined by  $\langle b_1, b_2, b_3, ..., b_n \rangle$ such that

  \[
    b_k =
    \begin{cases}
      0,                       &\text{$1 \leq k \leq n-1$;}\\
      1,                       &\text{$k=n$.}
    \end{cases}
  \]

Using this we can easily obtain the following infinite products in respectively 2-D, 3-D, 4-D and 5-D space involving weighted VPV partitions in their combinatorial interpretations.

\begin{corollary} For each of $|v|, |w|, |x|, |y|, |z|<1,$
  \begin{equation}\label{13.24}
    \prod_{\substack{(a,b)=1 \\ a\geq0,b>0}} \left( 1-y^a z^b \right)^{\frac{1}{b}}
    = (1-z)^{\frac{1}{1-y}},
  \end{equation}
    \begin{equation}\label{13.25}
    \prod_{\substack{(a,b,c)=1 \\ a,b\geq0,c>0}} \left( 1-x^a y^b z^c \right)^{\frac{1}{c}}
    = (1-z)^{\frac{1}{(1-x)(1-y)}},
  \end{equation}
    \begin{equation}\label{13.26}
    \prod_{\substack{(a,b,c,d)=1 \\ a,b,c\geq0,d>0}} \left( 1-w^a x^b y^c z^d \right)^{\frac{1}{d}}
    = (1-z)^{\frac{1}{(1-w)(1-x)(1-y)}},
  \end{equation}
    \begin{equation}\label{13.27}
    \prod_{\substack{(a,b,c,d,e)=1 \\ a,b,c,d\geq0,e>0}} \left( 1-v^a w^b x^c y^d z^e \right)^{\frac{1}{e}}
    = (1-z)^{\frac{1}{(1-v)(1-w)(1-x)(1-y)}}.
  \end{equation}
  \end{corollary}

The above four infinite products and their reciprocals are worth deeper analysis as simple examples of weighted VPV partitions, giving us \textit{exact results} reminiscent of the integer partition theorems. (\ref{13.24}) to (\ref{13.27}) are slight variations on particular cases of (\ref{13.18}) to (\ref{13.19}) and their 4-D and 5-D forms. They may be interesting to examine in the "near bijection" context that has been applied to the classical Euler pentagonal number theorem. This is a large topic probably beyond the scope of our present note.

Let us take the example of equation (\ref{13.25}). The right side product is a case of the binomial theorem, which when applied gives us,

  \begin{equation}\label{13.28}
    \prod_{\substack{(a,b,c)=1 \\ a,b\geq0,c>0}} \left( 1-x^a y^b z^c \right)^{\frac{1}{c}}
   = (1-z)^{\frac{1}{(1-x)(1-y)}}
      \end{equation}
  \begin{equation}  \nonumber
   = 1 - \frac{z}{n 1!}  + \frac{(n-1)z^2}{n^2 2!} + \frac{(n-1)(2n-1)z^3}{n^3 3!} + \frac{(n-1)(2n-1)(3n-1)z^4}{n^4 4!}
  \end{equation}
  \begin{equation}  \nonumber
    + \frac{(n-1)(2n-1)(3n-1)(4n-1)z^5}{n^5 5!} + \cdots, \quad where \quad n=(1-x)(1-y).
  \end{equation}

  \bigskip

Looking closer at this, we see that (\ref{13.28}) encodes a theorem about weighted 3-dimensional VPV partitions in the first hyperquadrant including the $x$ and $y$ axes but not the $z$ axis.

It is also clear that the power series terminates to a polynomial in $z$ whenever $(1-x)(1-y)$ equals any of $1, \frac{1}{2}, \frac{1}{3}, \frac{1}{4}, \ldots$ .  For partitions of these VPVs in the first hyperquadrant of Euclidean 3-space, each vector $\langle a,b,c \rangle$ has integer coordinates that satisfy $a,b\geq0,c>0$. By a weighted partition, we mean a\textit{"stepping stone jump while carrying a weight determined by a coefficient"} from one integer lattice point to the next, jumping always\textit{"away from the origin by a nonincreasing length"}, that origin being the point $\langle 0,0,0 \rangle$. ie. The distance $\sqrt{a^2+b^2+c^2}$ from $\langle 0,0,0 \rangle$ to the starting point $\langle a,b,c \rangle$ of the jump is less than the distance $\sqrt{h^2+j^2+k^2}$ from $\langle 0,0,0 \rangle$ to the destination point $\langle h,j,k \rangle$ of the jump. From the perspective of the weighted partition sum vale, the order of the weighted summands is unimportant. However, switching orders of summands creates different paths to the same vector being partitioned.

The next set of natural cases we consider in our general term notation, with $b_1 + b_2 + b_3 + ... + b_n = 1$, is as follows. We adjust the number of lattice point vectors in the $n$-space hyperquadrant \textit{radial from the origin} region. We sum on the particular lattice points with vectors defined by  $\langle b_1, b_2, b_3, ..., b_n \rangle$ such that

  \[
    b_k =
    \begin{cases}
      1,                       &\text{$1 \leq k \leq n-1$;}\\
      2-n,                     &\text{$k=n$.}
    \end{cases}
  \]

This leads to cases of (\ref{13.09}) to (\ref{13.12}) given here as the following identities in 2D, 3D, 4D and 5D.

If $|y|<1, |z|<1,$ then
  \begin{equation}   \label{13.29}
    \prod_{\substack{ (a,b)=1 \\ a,b \geq 1}} \left( \frac{1}{1-y^a z^b} \right)^{\frac{1}{a^1 b^0}}
    = \exp\left\{ \left( \sum_{i=1}^{\infty} \frac{y^i}{i^1}\right) \left( \sum_{j=1}^{\infty} \frac{z^j}{j^0}\right) \right\}.
  \end{equation}

If $|x|<1, |y|<1, |z|<1 $ then
  \begin{equation}   \label{13.30}
    \prod_{\substack{ (a,b,c)=1 \\ a,b,c \geq 1}} \left( \frac{1}{1-x^a y^b z^c} \right)^{\frac{1}{a^1 b^1 c^{-1}}}
    = \exp\left\{ \left( \sum_{i=1}^{\infty} \frac{x^i}{i^1}\right) \left( \sum_{j=1}^{\infty} \frac{y^j}{j^1}\right)
    \left( \sum_{k=1}^{\infty} \frac{z^k}{k^{-1}}\right)\right\}.
  \end{equation}

If $|w|<1, |x|<1, |y|<1, |z|<1 $ then
  \begin{equation}   \label{13.31}
    \prod_{\substack{ (a,b,c,d)=1 \\ a,b,c,d \geq 1}} \left( \frac{1}{1-w^a x^b y^c z^d} \right)^{\frac{1}{a^1 b^1 c^1 d^{-2}}}
    = \exp\left\{ \left( \sum_{h=1}^{\infty} \frac{w^h}{h^1}\right) \left(\sum_{i=1}^{\infty} \frac{x^i}{i^1}\right)
    \left(\sum_{j=1}^{\infty} \frac{y^j}{j^1}\right) \left(\sum_{k=1}^{\infty} \frac{z^k}{k^{-2}}\right)\right\}.
  \end{equation}

If $|v|<1, |w|<1, |x|<1, |y|<1, |z|<1 $ then
  \begin{equation}   \label{13.32}
    \prod_{\substack{ (a,b,c,d,e)=1 \\ a,b,c,d,e \geq 1}} \left( \frac{1}{1-v^a w^b x^c y^d z^e} \right)^{\frac{1}
    {a^1 b^1 c^1 d^1 e^{-3}}}
    \end{equation}
  \begin{equation} \nonumber
  = \exp\left\{ \left( \sum_{g=1}^{\infty} \frac{v^g}{g^1}\right) \left( \sum_{h=1}^{\infty} \frac{w^h}{h^1}\right) \left( \sum_{i=1}^{\infty} \frac{x^i}{i^1}\right) \left( \sum_{j=1}^{\infty} \frac{y^j}{j^1}\right)
   \left( \sum_{k=1}^{\infty} \frac{z^k}{k^{-3}}\right)\right\}.
  \end{equation}

It is not hard to see that these identities can be easily written up to 5D, 6D, 7D space etc. So, if we do this for 2D to 5D, and sum the elementary power series in the exponential term right hand sides, we arrive at

 \begin{corollary} \label{13.32a} For each of $|v|, |w|, |x|, |y|, |z|<1,$
   \begin{equation}\label{13.33}
    \prod_{\substack{(a,b)=1 \\ a,b\geq1}} \left( 1-y^a z^b \right)^{\frac{1}{a}}
    = (1-y)^{\frac{z}{1-z}},
  \end{equation}
    \begin{equation}\label{13.34}
    \prod_{\substack{(a,b,c)=1 \\ a,b,c\geq1}} \left( 1-x^a y^b z^c \right)^{\frac{c}{ab}}
    = ((1-x)(1-y))^{\frac{z}{(1-z)^2}},
  \end{equation}
    \begin{equation}\label{13.35}
    \prod_{\substack{(a,b,c,d)=1 \\ a,b,c,d\geq1}} \left( 1-w^a x^b y^c z^d \right)^{\frac{d^2}{abc}}
    = ((1-w)(1-x)(1-y))^{\frac{z(1+z)}{(1-z)^3}},
  \end{equation}
    \begin{equation}\label{13.36}
    \prod_{\substack{(a,b,c,d,e)=1 \\ a,b,c,d,e\geq1}} \left( 1-v^a w^b x^c y^d z^e \right)^{\frac{e^3}{abcd}}
    = ((1-v)(1-w)(1-x)(1-y))^{\frac{z(1+4z+z^2)}{(1-z)^4}}.
  \end{equation}
 \end{corollary}

 The identities (\ref{13.34}) to (\ref{13.36}) are new to the literature. The index functions on the right sides will be recognized by many as the well-known values of the non-positive integer polylogarithms, $Li_0(z), Li_{-1}(z), Li_{-2}(z), Li_{-3}(z)$.

So next we briefly give definitions and versions of the \emph{abc Conjecture}, before in the section ensuing that, continuing discussion of equation (\ref{13.34}) in the context of the conjecture. This shows plausible uses of VPV identities to examine known open problems from a new perspective.

\section{Diversionary note on the \textit{abc Conjecture}}

The \textit{abc Conjecture} (also known as the Oesterl\'{e}–Masser conjecture) is a conjecture in number theory, first proposed by Joseph Oesterl\'{e} (1988) and David Masser (1985) (see \cite{jO1988} and \cite{dM1985}). It is stated in terms of three positive integers, $a$, $b$ and $c$ (hence the name) that are relatively prime and satisfy $a + b = c$. If $d$ denotes the product of the distinct prime factors of $abc$, the conjecture essentially states that $d$ is usually not much smaller than $c$. In other words: if $a$ and $b$ are composed from large powers of primes, then $c$ is usually not divisible by large powers of primes. A number of famous conjectures and theorems in number theory would follow immediately from the \textit{abc Conjecture} or its versions. Goldfeld (1996) (see \cite{dG1996}) described the \textit{abc Conjecture} as \textit{the most important unsolved problem in Diophantine analysis}.

The \textbf{abc Conjecture} originated as the outcome of attempts by Oesterl\'{e} and Masser to understand the Szpiro conjecture about elliptic curves, which involves more geometric structures in its statement than the abc Conjecture. The \textit{abc Conjecture} was shown to be equivalent to the modified Szpiro's conjecture \cite{lS1981,lS1987}.

Various attempts to prove the \textit{abc Conjecture} have been made, but none are currently accepted by the mainstream mathematical community and as of 2022, the conjecture is still largely regarded as unproven.

Before we state the conjecture we introduce the notion of the radical of an integer: for a positive integer $n$, the radical of $n$, denoted $\textmd{rad}(n)$, is the product of the distinct prime factors of $n$. For example

$\textmd{rad}(16) = \textmd{rad}(2^4) = \textmd{rad}(2) = 2$,

$\textmd{rad}(17) = 17$,

$\textmd{rad}(18) = \textmd{rad}(2 \times 3^2) = 2 \times 3 = 6$,

$\textmd{rad}(1000000) = \textmd{rad}(2^6 \times 5^6) = 2 \times 5 = 10$.

If $a$, $b$, and $c$ are coprime positive integers such that $a + b = c$, it turns out that "usually" $c < \textmd{rad}(abc)$. The \textit{abc Conjecture} deals with the exceptions. Specifically, it states that:
\begin{conjecture}

\textit{abc Conjecture (version I).} For every positive real number $\varepsilon$, there exist only finitely many triples $(a, b, c)$ of coprime positive integers, with $a + b = c$, such that $c>\textmd{rad}(abc)^{1+\varepsilon}$.

An equivalent formulation:

\textit{abc Conjecture (version II).} For every positive real number $\varepsilon$, there exists a constant $K_\varepsilon$ such that for all triples $(a, b, c)$ of coprime positive integers, with $a + b = c$: $c< K_\varepsilon \textmd{rad}(abc)^{1+\varepsilon}$.

Another equivalent formulation:

\textit{abc Conjecture (version III).} For every positive real number $\varepsilon$, there exist only finitely many triples $(a, b, c)$ of coprime positive integers with $a + b = c$ such that $q(a, b, c) > 1 + \varepsilon$.

\end{conjecture}

This third equivalent formulation of the conjecture involves the quality $q(a, b, c)$ of the triple $(a, b, c)$, defined as
\begin{equation}\nonumber
  q(a, b, c) = \frac{\log (c)}{\log\left( \textmd{rad}(abc)\right)}.
\end{equation}
For example:

$q(4, 127, 131) = \frac{\log(131)}{\log(\textmd{rad}(4\times127\times131))} = \frac{\log(131)}{\log(2\times127\times131)} = 0.46820... $

$q(3, 125, 128) = \frac{\log(128)}{\log(\textmd{rad}(3\times125\times128))}  = \frac{\log(128)}{\log(30)} = 1.426565...$

A typical triple $(a, b, c)$ of coprime positive integers with $a + b = c$ will have $c < \textmd{rad}(abc)$, i.e. $q(a, b, c) < 1$. Triples with $q > 1$ such as in the second example are rather special, they consist of numbers divisible by high powers of small prime numbers.

Whereas it is known that there are infinitely many triples $(a, b, c)$ of coprime positive integers with $a + b = c$ such that $q(a, b, c) > 1$, the conjecture predicts that only finitely many of those have $q > 1.01$ or $q > 1.001$ or even $q > 1.0001$, etc. In particular, if the conjecture is true, then there must exist a triple $(a, b, c)$ that achieves the maximal possible quality $q(a, b, c)$.

\section{Applying a 3D VPV identity to the \textit{abc Conjecture}}

Three versions of the \emph{abc Conjecture} were presented in the previous section. This enables us to relate the following identity and try to gain some insight into the conjecture. The feature in common between the \textit{abc Conjecture} and a 3D VPV identity is the aspect that $a+b=c$ may be approached with $\gcd(a,b,c)=1$. Hence, there is scope to examine this potential crossover between theories.

So, in the context of present analysis in section 3 of this note, we consider equation (\ref{13.34}), taking the case where $x=1/z, y=1/z$ so that the identity becomes

    \begin{equation} \label{13.37}
    \prod_{\substack{(a,b,c)=1 \\ a,b,c\geq1}} \left( 1- z^{c-a-b} \right)^{\frac{c}{ab}}
    = \left(\frac{1-z}{z}\right)^{\frac{2z}{(1-z)^2}},
  \end{equation}

valid evidently, for $0 < |z| < 1$. The right side of (\ref{13.37}) approaches unity as $z \rightarrow 0$. Just the fact of convergence of the infinite product to an existing limit has relevance to the \emph{abc Conjecture}. The \textit{abc Conjecture} is concerned with cases where $a+b=c$, and also asserts:

"For every positive real number $\varepsilon$, there exist only finitely many triples $(a, b, c)$ of coprime positive integers, with $a + b = c$, such that $c>\textmd{rad}(abc)^{1+\varepsilon}$."

In the infinite product we have covered off all cases where $a+b=c$, and convergence of the product to a limit implies certain things about the comparison of both $c-a-b$ and $\frac{c}{ab}$ terms as c increases.

\section{Hyperdiagonal line generating functions for different $n$D slopes}  \label{nDlattice02a}

At the start of section \ref{nDlattice01} of this note we saw how to calculate the hyperdiagonal generating function for an nD generated vector partition grid. Basically, we have the
\begin{statement}  \label{nDlattice03a}
  If the nD generating function for the entire nD grid is the $n$ variable $f_n(z_1,z_2,z_3,...,z_n)$, the equation of the nD hyperdiagonal line from the origin is $z_1=z_2=z_3=...=z_n$ and so the generating function for the nD vector partitions along that line is given by the single variable function $f_n(z,z,,...,z)$.
\end{statement}
That is, even though an nD space grid is impossible for humans to envision, we can define a straight line through the hyperdiagonal from the nD origin point, and formulate an exact value of the vector partition function at any point along that hyperdiagonal line.

Furthermore, taking an arbitrary lattice point vector $\langle a_1,a_2,a_3,...,a_n \rangle$ in the nD grid we can show that the vector partition function (or nD coefficient) for that lattice point is exactly evaluated as follows.

\begin{statement}  \label{nDlattice04a}
  If the nD generating function for the entire nD grid is the $n$ variable $f_n(z_1,z_2,z_3,...,z_n)$, the equation of the nD \textit{hyper-radial-from-origin} line from the origin to the arbitrary point $\langle a_1,a_2,a_3,...,a_n \rangle$ is based on knowing values of
  \begin{equation} \label{nDlattice05a}
    \langle a_1,a_2,a_3,...,a_n \rangle = r \langle c_1,c_2,c_3,...,c_n \rangle
  \end{equation}
  where $\langle c_1,c_2,c_3,...,c_n \rangle$ is a VPV with $\gcd( c_1,c_2,c_3,...,c_n ) =1$ and $r$ is the unique positive integer that makes this true.   So the straight line from the origin to the point $\langle a_1,a_2,a_3,...,a_n \rangle$ has the defining equation  $c_1z_1=c_2z_2=c_3z_3=...=c_nz_n$ and so the generating function for the nD vector partitions along that line is given by the single variable function $f_n(c_1z,c_2z,c_3z,...,c_nz)$, which has the coefficient of $z^r$ equal to the vector partition function (or nD coefficient) for that lattice point $\langle a_1,a_2,a_3,...,a_n \rangle$ exactly evaluated.
\end{statement}

\section{Exercises}

Derive from corollary \ref{13.32a} that:

For each of $|v|, |w|, |x|, |y|, |z|<1,$
  \begin{equation}\label{13.38}
    \prod_{\substack{\gcd(a,b)=1 \\ a\geq0,b>0}} \left( 1+y^a z^b \right)^{\frac{1}{b}}
    = \frac{(1-z^2)^{\frac{1}{1-y^2}}}{(1-z)^{\frac{1}{1-y}}}
  \end{equation}
  \begin{equation} \nonumber
  = \sum_{n=0}^{\infty} \frac{\Gamma\left(n + \frac{1}{1-y}\right)}{\Gamma\left(\frac{1}{1-y}\right)} \, _3F_2\left(\frac{1}{2} - \frac{n}{2}, -\frac{n}{2}, \frac{-1}{1-y^2} ;\frac{1}{2} - \frac{n}{2} - \frac{1}{2 (1-y)}, 1 - \frac{n}{2} - \frac{1}{2 (1-y)};1\right)\frac{z^n}{n!}.
\end{equation}
    \begin{equation}\label{13.39}
    \prod_{\substack{\gcd(a,b,c)=1 \\ a,b\geq0,c>0}} \left( 1+x^a y^b z^c \right)^{\frac{1}{c}}
    = \frac{(1-z^2)^{\frac{1}{(1-x^2)(1-y^2)}}}{(1-z)^{\frac{1}{(1-x)(1-y)}}}
    = \sum_{n=0}^{\infty}  \frac{\Gamma(n+ \frac{1}{(1-x) (1-y)}) \; \textbf{F(n)}z^n}
    {\Gamma(\frac{1}{(1-x)(1-y)})n!},
    \end{equation}
    where \textbf{F(n)} is the hypergeometric series
    \begin{equation} \nonumber
  \, _3F_2 \left(\frac{1}{2} - \frac{n}{2}, \frac{-n}{2},
    \frac{-1}{(1-x^2)(1-y^2)};\frac{1}{2}- \frac{n}{2}- \frac{1}{2(1-x)(1-y)}, 1 - \frac{n}{2} - \frac{1}{2(1-x)(1-y)};1\right).
     \end{equation}
     The reader can verify this, by entering the code

     \textbf{series} $\mathbf{(1-z^2) ^\wedge (1/((1-x^2)(1-y^2))) / (1-z) ^\wedge (1/((1-x)(1-y)))}$ \textbf{at} $\mathbf{z=0}$\textbf{.}

     \noindent at an online calculating engine.
    \begin{equation}\label{13.40}
    \prod_{\substack{\gcd(a,b,c,d)=1 \\ a,b,c\geq0,d>0}} \left( 1+w^a x^b y^c z^d \right)^{\frac{1}{d}}
    = \frac{(1-z^2)^{\frac{1}{(1-w^2)(1-x^2)(1-y^2)}}}{(1-z)^{\frac{1}{(1-w)(1-x)(1-y)}}},
  \end{equation}
    \begin{equation}\label{13.41}
    \prod_{\substack{\gcd(a,b,c,d,e)=1 \\ a,b,c,d\geq0,e>0}} \left( 1+v^a w^b x^c y^d z^e \right)^{\frac{1}{e}}
    = \frac{(1-z^2)^{\frac{1}{(1-v^2)(1-w^2)(1-x^2)(1-y^2)}}}{(1-z)^{\frac{1}{(1-v)(1-w)(1-x)(1-y)}}}.
  \end{equation}

Derive from corollary 5.2 that

 For each of $|v|, |w|, |x|, |y|, |z|<1,$
   \begin{equation}\label{13.42}
    \prod_{\substack{\gcd(a,b)=1 \\ a,b\geq1}} \left( 1+y^a z^b \right)^{\frac{1}{a}}
    = \frac{(1-y^2)^{\frac{z^2}{1-z^2}}}{(1-y)^{\frac{z}{1-z}}},
  \end{equation}
    \begin{equation}\label{13.43}
    \prod_{\substack{\gcd(a,b,c)=1 \\ a,b,c\geq1}} \left( 1+x^a y^b z^c \right)^{\frac{c}{ab}}
    = \frac{((1-x^2)(1-y^2))^{\frac{z^2}{(1-z^2)^2}}}{((1-x)(1-y))^{\frac{z}{(1-z)^2}}},
  \end{equation}
    \begin{equation}\label{13.44}
    \prod_{\substack{\gcd(a,b,c,d)=1 \\ a,b,c,d\geq1}} \left( 1+w^a x^b y^c z^d \right)^{\frac{d^2}{abc}}
    = \frac{((1-w^2)(1-x^2)(1-y^2))^{\frac{z^2(1+z^2)}{(1-z^2)^3}}}{((1-w)(1-x)(1-y))^{\frac{z(1+z)}{(1-z)^3}}},
  \end{equation}
    \begin{equation}\label{13.45}
    \prod_{\substack{\gcd(a,b,c,d,e)=1 \\ a,b,c,d,e\geq1}} \left( 1+v^a w^b x^c y^d z^e \right)^{\frac{e^3}{abcd}}
    = \frac{((1-v^2)(1-w^2)(1-x^2)(1-y^2))^{\frac{z^2(1+4z^2+z^4)}{(1-z^2)^4}}}{((1-v)(1-w)(1-x)(1-y))^{\frac{z(1+4z+z^2)}{(1-z)^4}}}.
  \end{equation}

\bigskip

\section{VPV identities in square hyperpyramid regions.} \label{S:Intro VPV hyperpyramids}

In the 1990s and up to 2000 the author published papers that culminated in the 2000 paper on hyperpyramid VPV identities. (See Campbell \cite{gC1992,gC1993,gC1994a,gC1994b,gC1997,gC1998} then \cite{gC2000}.) These were not at the time taken any further than statement of a general theorem and a few prominent examples. However, since then, these identities have not been developed further in the literature, despite there being evidently a large number of ways the Parametric Euler Sum Identities of the 21st century along with experimental computation results are definitely applicable. Add to this the possibility that light diffusion lattice models, random walk regimes, and stepping stone weighted partitions seem fundamentally applicable in contexts of VPV identities, and it becomes clear that the transition from integer partitions to vector partitions may be a path for future researches.

So, we here give the simplest $n$-space hyperpyramid VPV theorem due to the author in \cite{gC2000}. The so-called "Skewed Hyperpyramid $n$-space Identities" from \cite{gC2000} we shall cover in a later paper. The application of the determinant coefficient technique of our current earlier work is strikingly applicable and bearing some semblance to the $q$-binomial variants. Note that for each of (7.11) to (7.15) the left side products are taken over a set of integer lattice points inside an inverted hyperpyramid on the Euclidean cartesian space.

In the first 15 years of the 21st century the summations found by the Borwein brothers Peter and Jonathan, their father David with their colleagues, see \cite{dB2006} to \cite{jB2013} have renewed interest in the old Euler Sums. Their results give us particular values of polylogarithms and related functions involving the generalized Harmonic numbers. This work has been developed some way over nearly two decades so now we speak of the Mordell-Tornheim-Witten sums, which are polylogarithm generalizations all seen to be applicable to the VPV identities, but that connection is not yet fully worked through. These newer results can, many of them, be substituted into VPV identities to give us exact results for weighted vector partitions. To make sense of these new results, we need to go back to fundamental definitions and ideas for partitions of vectors as distinct from those well considered already for integer partitions.

As with our previous paper on the first hyperquadrant identities, we begin with the simple derivation of the $2D$ case, then look at the $3D$ case, before stating and proving the result in the $n$ dimensional generalization.

\section{Deriving 2D VPV identities in extended triangle regions.} \label{S:2D VPV hyperpyramids}

As we did in the hyperquadrant paper, we again start with a simple $2D$ summation. Consider

\begin{equation}  \nonumber
   \sum_{n=1}^{\infty}  \left( \sum_{m=1}^{n} \frac{y^m}{m^a} \right) \frac{z^n}{n^b}
\end{equation}

\begin{equation}  \nonumber
=\left(\frac{y^1}{1^a}\right)\frac{z^1}{1^b}
+\left(\frac{y^1}{1^a}+\frac{y^2}{2^a}\right)\frac{z^2}{2^b}
+\left(\frac{y^1}{1^a}+\frac{y^2}{2^a}+\frac{y^3}{3^a}\right)\frac{z^3}{3^b}
+\left(\frac{y^1}{1^a}+\frac{y^2}{2^a}+\frac{y^3}{3^a}+\frac{y^4}{4^a}\right)\frac{z^4}{4^b}
\end{equation}
\begin{equation}  \nonumber
+\left(\frac{y^1}{1^a}+\frac{y^2}{2^a}+\frac{y^3}{3^a}+\frac{y^4}{4^a}+\frac{y^5}{5^a}\right)\frac{z^5}{5^b}
+\left(\frac{y^1}{1^a}+\frac{y^2}{2^a}+\frac{y^3}{3^a}+\frac{y^4}{4^a}+\frac{y^5}{5^a}+\frac{y^6}{6^a}\right)\frac{z^6}{6^b}+\cdots
\end{equation}
\begin{equation}  \nonumber
 =\frac{y^1 z^1}{1^a 1^b}
\end{equation}
\begin{equation}  \nonumber
 +\frac{y^1 z^2}{1^a 2^b}+\frac{y^2 z^2}{2^a 2^b}
\end{equation}
\begin{equation}  \nonumber
 +\frac{y^1 z^3}{1^a 3^b}+\frac{y^2 z^3}{2^a 3^b}+\frac{y^3 z^3}{3^a 3^b}
\end{equation}
\begin{equation}  \nonumber
 +\frac{y^1 z^4}{1^a 4^b}+\frac{y^2 z^4}{2^a 4^b}+\frac{y^3 z^4}{3^a 4^b}+\frac{y^4 z^4}{4^a 4^b}
\end{equation}
\begin{equation}  \nonumber
 +\frac{y^1 z^5}{1^a 5^b}+\frac{y^2 z^5}{2^a 5^b}+\frac{y^3 z^5}{3^a 5^b}+\frac{y^4 z^5}{4^a 5^b}+\frac{y^5 z^5}{5^a 5^b}
\end{equation}
\begin{equation}  \nonumber
 +\frac{y^1 z^6}{1^a 6^b}+\frac{y^2 z^6}{2^a 6^b}+\frac{y^3 z^6}{3^a 6^b}+\frac{y^4 z^6}{4^a 6^b}+\frac{y^5 z^6}{5^a 6^b}+\frac{y^6 z^6}{6^a 6^b}
\end{equation}
\begin{equation}  \nonumber
 +\frac{y^1 z^7}{1^a 7^b}+\frac{y^2 z^7}{2^a 7^b}+\frac{y^3 z^7}{3^a 7^b}+\frac{y^4 z^7}{4^a 7^b}+\frac{y^5 z^7}{5^a 7^b}+\frac{y^6 z^7}{6^a 7^b}+\frac{y^7 z^7}{7^a 7^b}
\end{equation}
\begin{equation}  \nonumber
 + \quad \vdots \quad + \quad \vdots \quad + \quad \vdots \quad + \quad \vdots \quad + \quad \vdots \quad + \quad \vdots \quad + \quad \vdots \quad   \ddots
\end{equation}
\begin{equation}  \nonumber
 = \sum_{m,n \geq 1; m \leq n}^{\infty}  \frac{y^m z^n}{m^a n^b}
\end{equation}
\begin{equation}  \nonumber
 = \sum_{\substack{ h,j,k \geq 1 \\ j \leq k ; \, (j,k)=1}}  \frac{(y^j z^k)^h}{h^{a+b} (j^a k^b)}
\end{equation}
\begin{equation}  \nonumber
 = \sum_{\substack{ j,k \geq 1 \\ j \leq k ; \, (j,k)=1}}  \frac{1}{(j^a k^b)}   \sum_{h=1}^{\infty} \frac{(y^j z^k)^h}{h^{a+b}}
\end{equation}
\begin{equation}  \nonumber
 = \sum_{\substack{ j,k \geq 1 \\ j \leq k ; \, (j,k)=1}} \frac{1}{(j^a k^b)}   \log \left( \frac{1}{1 - y^j z^k} \right) \quad if \quad a+b=1.
\end{equation}

Therefore, we have shown that
\begin{equation}  \nonumber
 \sum_{n=1}^{\infty}  \left( \sum_{m=1}^{n} \frac{y^m}{m^a} \right) \frac{z^n}{n^b} = \sum_{\substack{ j,k \geq 1 \\ j \leq k ; \, (j,k)=1}} \frac{1}{(j^a k^b)}   \log \left( \frac{1}{1 - y^j z^k} \right) \quad if \quad a+b=1.
\end{equation}

 Exponentiating both sides gives us the $2D$ first extended triangle VPV identity, where in this $2D$ case the $nD$ pyramid reduces to the form of a triangle shaped array of lattice point vectors, and so we can state the

 \begin{theorem}     \label{vpv-pyramid2D-thm}
 \textbf{The $2D$ triangle VPV identity.} For $|y|<1, |z|<1,$
 \begin{equation}   \label{14.01}
    \prod_{\substack{ j,k \geq 1 \\ j \leq k ; \, (j,k)=1}} \left( \frac{1}{1-y^j z^k} \right)^{\frac{1}{j^a k^b}}
    = \exp\left\{ \sum_{n=1}^{\infty}  \left( \sum_{m=1}^{n} \frac{y^m}{m^a} \right) \frac{z^n}{n^b} \right\} \quad if \quad a+b=1.
  \end{equation}
 \end{theorem}

As with our earlier exploits into the $2D$ first quadrant case, for the present result we take some simple example cases where new and interesting results arise.

So, let us take the case where $a=0, b=1$, giving us

 \begin{equation}   \nonumber
    \prod_{\substack{ j,k \geq 1 \\ j \leq k ; \, (j,k)=1}} \left( \frac{1}{1-y^j z^k} \right)^{\frac{1}{k}}
    = \exp\left\{ \sum_{n=1}^{\infty}  \left( \sum_{m=1}^{n} y^m \right) \frac{z^n}{n} \right\}
  \end{equation}
 \begin{equation}   \nonumber
    = \exp\left\{ \sum_{n=1}^{\infty}  \left( y \frac{1-y^n}{1-y} \right) \frac{z^n}{n} \right\}
    = \exp\left\{ \frac{y}{1-y} \log \left( \frac{1-yz}{1-z} \right)   \right\}.
  \end{equation}

So, we arrive then at the following pair of equivalent results,

 \begin{equation}   \label{14.02}
    \prod_{\substack{ j,k \geq 1 \\ j \leq k ; \, (j,k)=1}} \left( \frac{1}{1-y^j z^k} \right)^{\frac{1}{k}}
        =  \left( \frac{1-yz}{1-z} \right)^{\frac{y}{1-y}} ,
  \end{equation}

 and

 \begin{equation}   \label{14.03}
    \prod_{\substack{ j,k \geq 1 \\ j \leq k ; \, (j,k)=1}} \left( 1-y^j z^k \right)^{\frac{1}{k}}
        =  \left( \frac{1-z}{1-yz} \right)^{\frac{y}{1-y}} .
  \end{equation}

From here, multiply both sides of (\ref{14.02}) and the case of (\ref{14.03}) with $y \mapsto y^2$ and $z \mapsto z^2$ to get,

 \begin{equation}   \label{14.04}
    \prod_{\substack{ j,k \geq 1; \, j \leq k \\ gcd(j,k)=1}} \left( 1+y^j z^k \right)^{\frac{1}{k}}
        =  \left( \frac{1-yz}{1-z} \right)^{\frac{y}{1-y}} \left( \frac{1-z^2}{1-y^2z^2} \right)^{\frac{y^2}{1-y^2}} .
  \end{equation}

Particular cases:

$y = \frac{1}{2}$ gives us from (\ref{14.03}) and (\ref{14.04}) the remarkable two results that

 \begin{equation}   \nonumber
    \prod_{\substack{ j,k \geq 1; \, j \leq k \\ gcd(j,k)=1}} \left( 1- \frac{z^k}{2^j} \right)^{\frac{1}{k}}
        =   \frac{2-2z}{2-z}  = 1 - \frac{z}{2} - \frac{z^2}{4} - \frac{z^3}{8} - \frac{z^4}{16} - \frac{z^5}{32} - \ldots
  \end{equation}
 \begin{equation}   \nonumber
 = \left( 1- \frac{z}{2} \right)
  \end{equation}
 \begin{equation}   \nonumber
 \sqrt{\left( 1- \frac{z^2}{2^1} \right)}
  \end{equation}
 \begin{equation}   \nonumber
 \sqrt[3]{\left( 1- \frac{z^3}{2^1} \right)\left( 1- \frac{z^3}{2^2} \right)}
  \end{equation}
 \begin{equation}   \nonumber
 \sqrt[4]{\left( 1- \frac{z^4}{2^1} \right)\left( 1- \frac{z^4}{2^3} \right)}
  \end{equation}
 \begin{equation}   \nonumber
 \sqrt[5]{\left( 1- \frac{z^5}{2^1} \right)\left( 1- \frac{z^5}{2^2} \right)\left( 1- \frac{z^5}{2^3} \right)\left( 1- \frac{z^5}{2^4} \right)}
  \end{equation}
 \begin{equation}   \nonumber
 \sqrt[6]{\left( 1- \frac{z^6}{2^1} \right)\left( 1- \frac{z^6}{2^5} \right)}
 \end{equation}
 \begin{equation}   \nonumber
 \vdots \, ,
 \end{equation}

 \begin{equation}   \nonumber
    \prod_{\substack{ j,k \geq 1; \, j \leq k \\ gcd(j,k)=1}} \left( 1+ \frac{z^k}{2^j} \right)^{\frac{1}{k}}
        =   \frac{2-z}{2-2z}      \sqrt[3]{\frac{4-z^2}{4-4z^2}}
  \end{equation}
 \begin{equation}   \nonumber
 = \left( 1+ \frac{z}{2} \right)
  \end{equation}
 \begin{equation}   \nonumber
 \sqrt{\left( 1+ \frac{z^2}{2^1} \right)}
  \end{equation}
 \begin{equation}   \nonumber
 \sqrt[3]{\left( 1+ \frac{z^3}{2^1} \right)\left( 1+ \frac{z^3}{2^2} \right)}
  \end{equation}
 \begin{equation}   \nonumber
 \sqrt[4]{\left( 1+ \frac{z^4}{2^1} \right)\left( 1+ \frac{z^4}{2^3} \right)}
  \end{equation}
 \begin{equation}   \nonumber
 \sqrt[5]{\left( 1+ \frac{z^5}{2^1} \right)\left( 1+ \frac{z^5}{2^2} \right)\left( 1+ \frac{z^5}{2^3} \right)\left( 1+ \frac{z^5}{2^4} \right)}
  \end{equation}
 \begin{equation}   \nonumber
 \sqrt[6]{\left( 1+ \frac{z^6}{2^1} \right)\left( 1+ \frac{z^6}{2^5} \right)}
 \end{equation}
 \begin{equation}   \nonumber
 \vdots .
 \end{equation}

These two equations can be easily verified on a calculating engine like Mathematica or WolframAlpha by expanding each side into it's Taylor series around $z=0$ and comparing coefficients of like powers of $z$. Next, take the cases of (\ref{14.03}) and (\ref{14.04}) with $y=2$, both of which converge if $|z|<2$, so then, after a slight adjustment to both sides by a factor of $1-2z$,

 \begin{equation}   \nonumber
    \prod_{\substack{ j,k \geq 1; \, j < k \\ gcd(j,k)=1}} \left( 1- 2^j z^k  \right)^{\frac{1}{k}}
        =   \frac{1-2z}{(1-z)^2}  = 1 - z - 2z^2 - 3z^3 - 4z^4 - 5z^5 - \ldots - n z^n - \ldots
  \end{equation}
 \begin{equation}   \nonumber
 = \sqrt{\left( 1- 2^1 z^2 \right)}
  \end{equation}
 \begin{equation}   \nonumber
 \sqrt[3]{\left( 1- 2^1 z^3 \right)\left( 1- 2^2 z^3 \right)}
  \end{equation}
 \begin{equation}   \nonumber
 \sqrt[4]{\left( 1- 2^1 z^4 \right)\left( 1- 2^3 z^4 \right)}
  \end{equation}
 \begin{equation}   \nonumber
 \sqrt[5]{\left( 1- 2^1 z^5 \right)\left( 1- 2^2 z^5 \right)\left( 1- 2^3 z^5 \right)\left( 1- 2^4 z^5 \right)}
  \end{equation}
 \begin{equation}   \nonumber
 \sqrt[6]{\left( 1- 2^1 z^6 \right)\left( 1- 2^5 z^6 \right)}
 \end{equation}
 \begin{equation}   \nonumber
 \sqrt[7]{\left( 1- 2^1 z^7 \right)\left( 1- 2^2 z^7 \right)\left( 1- 2^3 z^7 \right)\left( 1- 2^4 z^7 \right)
  \left( 1- 2^5 z^7 \right)\left( 1- 2^6 z^7 \right)}
 \end{equation}
 \begin{equation}   \nonumber
 \vdots \, ,
 \end{equation}

which is also easy to verify on a calculating engine term by term from the power series of each side. The notably simple coefficients make this result somewhat tantalizing, as there seems no obvious reason for such coefficients to come out of the products of binomial series roots.
We remark at this juncture that equations (\ref{14.03}) and it's reciprocal equation (\ref{14.04}) are amenable to applying the limit as y approaches 1. In fact we have as follows that,

\begin{equation}   \nonumber
    \lim_{y \rightarrow 1}   \left( \frac{1 - z}{1 - y z}\right)^{\frac{y}{1 - y}} = e^{\frac{z}{z-1}}
  \end{equation}
and also from considering equation (\ref{14.04}) there is the limit, easily evaluated,

\begin{equation}   \nonumber
    \lim_{y \rightarrow 1}   \left( \frac{1 - yz}{1 - z}\right)^{\frac{y}{1 - y}}   \left( \frac{1 - z^2}{1 - y^2 z^2}\right)^{\frac{y^2}{1 - y^2}}            = e^{\frac{z}{1-z^2}}.
  \end{equation}

Therefore, applying these two limits to equations (\ref{14.03}) and (\ref{14.04}) respectively we obtain the two interesting results that

\begin{equation}   \label{14.05}
    \prod_{k=1}^{\infty} \left( 1- z^k \right)^{\frac{\varphi(k)}{k}}
        =  e^{\frac{z}{z-1}},
  \end{equation}

 \begin{equation}   \label{14.06}
    \prod_{k=1}^{\infty} \left( 1+ z^k \right)^{\frac{\varphi(k)}{k}}
        =  e^{\frac{z}{1-z^2}},
  \end{equation}

where $\varphi(k)$ is the Euler totient function, the number of positive integers less than and coprime to $k$.

Next we take (\ref{14.01}) with the case that $a=1$ and $b=0$, so then

 \begin{equation}   \nonumber
    \prod_{\substack{ j,k \geq 1 \\ j \leq k ; \, (j,k)=1}} \left( \frac{1}{1-y^j z^k} \right)^{\frac{1}{j}}
    = \exp\left\{ \sum_{n=1}^{\infty}  \left( \sum_{m=1}^{n} \frac{y^m}{m} \right) z^n \right\}
  \end{equation}
 \begin{equation}   \nonumber
    = \exp\left\{ \frac{1}{1-z} \sum_{n=1}^{\infty}  \frac{y^n z^n}{n}  \right\}
    = \exp\left\{ \frac{1}{1-z} \log \left( \frac{1}{1-yz} \right)   \right\}.
  \end{equation}

This leads us to establish that

 \begin{equation}   \label{14.07}
    \prod_{\substack{ j,k \geq 1 \\ j \leq k ; \, (j,k)=1}} \left( \frac{1}{1-y^j z^k} \right)^{\frac{1}{j}}
    =     \left( \frac{1}{1-yz} \right)^{\frac{1}{1-z}} ,
  \end{equation}

which is equivalent to

 \begin{equation}   \label{14.08}
    \prod_{\substack{ j,k \geq 1 \\ j \leq k ; \, (j,k)=1}} \left( 1-y^j z^k \right)^{\frac{1}{j}}
    =     \left( 1-yz \right)^{\frac{1}{1-z}} .
  \end{equation}

From multiplying both sides of (\ref{14.07}) in which $y \mapsto y^2$ and $z \mapsto z^2$ with both sides of (\ref{14.08}) we obtain

 \begin{equation}   \label{14.09}
    \prod_{\substack{ j,k \geq 1 \\ j \leq k ; \, (j,k)=1}} \left( 1+y^j z^k \right)^{\frac{1}{j}}
    =     \frac{\left( 1-y^2z^2 \right)^{\frac{1}{1-z^2}}}{\left( 1-yz \right)^{\frac{1}{1-z}}} .
  \end{equation}

 Particular cases:

$z = \frac{1}{2}$ gives us from (\ref{14.08}) and (\ref{14.09}) the remarkable result that

 \begin{equation}   \nonumber
    \prod_{\substack{ j,k \geq 1; \, j \leq k \\ gcd(j,k)=1}} \left( 1- \frac{y^j}{2^k} \right)^{\frac{1}{j}}
        =  \left( 1- \frac{y}{2} \right)^2 = 1 - \frac{y}{4} + \frac{y^2}{4}.
  \end{equation}
   \begin{equation}   \nonumber
 = \left( 1- \frac{y^1}{2^1} \right)
  \end{equation}
 \begin{equation}   \nonumber
 \left( 1- \frac{y^1}{2^2} \right)
  \end{equation}
 \begin{equation}   \nonumber
 \left( 1- \frac{y^1}{2^3} \right)\sqrt{\left( 1- \frac{y^2}{2^3} \right)}
  \end{equation}
 \begin{equation}   \nonumber
 \left( 1- \frac{y^1}{2^4} \right)\sqrt[3]{\left( 1- \frac{y^3}{2^4} \right)}
  \end{equation}
 \begin{equation}   \nonumber
 \left( 1- \frac{y^1}{2^5} \right)\sqrt{\left( 1- \frac{y^2}{2^5} \right)}\sqrt[3]{\left( 1- \frac{y^3}{2^5} \right)}\sqrt[4]{\left( 1- \frac{y^4}{2^5} \right)}
  \end{equation}
 \begin{equation}   \nonumber
 \left( 1- \frac{y^1}{2^6} \right)\sqrt[5]{\left( 1- \frac{y^5}{2^6} \right)}
 \end{equation}
 \begin{equation}   \nonumber
 \vdots \, ,
 \end{equation}
and the curious result,

 \begin{equation}   \nonumber
    \prod_{\substack{ j,k \geq 1; \, j \leq k \\ gcd(j,k)=1}} \left( 1+ \frac{y^j}{2^k} \right)^{\frac{1}{j}}
        =             \sqrt[3]{\left( \frac{2+y}{2-y} \right)^2}  = 1 +  \frac{2 y}{3} + \frac{2 y^2}{9} + \frac{17 y^3}{162} + \frac{11 y^4}{243} + \ldots
  \end{equation}
    \begin{equation}   \nonumber
 = \left( 1+ \frac{y^1}{2^1} \right)
  \end{equation}
 \begin{equation}   \nonumber
 \left( 1+ \frac{y^1}{2^2} \right)
  \end{equation}
 \begin{equation}   \nonumber
 \left( 1+ \frac{y^1}{2^3} \right)\sqrt{\left( 1+ \frac{y^2}{2^3} \right)}
  \end{equation}
 \begin{equation}   \nonumber
 \left( 1+ \frac{y^1}{2^4} \right)\sqrt[3]{\left( 1+ \frac{y^3}{2^4} \right)}
  \end{equation}
 \begin{equation}   \nonumber
 \left( 1+ \frac{y^1}{2^5} \right)\sqrt{\left( 1+ \frac{y^2}{2^5} \right)}\sqrt[3]{\left( 1+ \frac{y^3}{2^5} \right)}\sqrt[4]{\left( 1+ \frac{y^4}{2^5} \right)}
  \end{equation}
 \begin{equation}   \nonumber
 \left( 1+ \frac{y^1}{2^6} \right)\sqrt[5]{\left( 1+ \frac{y^5}{2^6} \right)}
 \end{equation}
 \begin{equation}   \nonumber
 \vdots \, .
 \end{equation}

These two equations can be verified on a calculating engine like Mathematica or WolframAlpha by expanding each side into it's Taylor series around $z=0$ and comparing coefficients of like powers of $y$. However, the calculation is an infinite series for each coefficient, unlike in the previous examples, where it is a finite sum.

\section{Deriving 3D VPV identities in square pyramid regions.} \label{S:3D VPV hyperpyramids}

As we did in the hyperquadrant section, we start with a simple $3D$ summation. Consider the sum, whose shape resembles a $3D$ pyramid here, given by

\begin{equation}  \nonumber
   \sum_{n=1}^{\infty}  \left( \sum_{l=1}^{n} \frac{x^l}{l^a} \right) \left( \sum_{m=1}^{n} \frac{y^m}{m^b} \right) \frac{z^n}{n^c}
\end{equation}

\begin{equation}  \nonumber
=\left(\frac{x^1}{1^a}\right)\left(\frac{y^1}{1^b}\right)\frac{z^1}{1^c}
\end{equation}
\begin{equation}  \nonumber
+\left(\frac{x^1}{1^a}+\frac{x^2}{2^a}\right)\left(\frac{y^1}{1^b}+\frac{y^2}{2^b}\right)\frac{z^2}{2^c}
\end{equation}
\begin{equation}  \nonumber
+\left(\frac{x^1}{1^a}+\frac{x^2}{2^a}+\frac{x^3}{3^a}\right)\left(\frac{y^1}{1^b}+\frac{y^2}{2^b}+\frac{y^3}{3^b}\right)\frac{z^3}{3^c}
\end{equation}
\begin{equation}  \nonumber
+\left(\frac{x^1}{1^a}+\frac{x^2}{2^a}+\frac{x^3}{3^a}+\frac{x^4}{4^a}\right)
 \left(\frac{y^1}{1^b}+\frac{y^2}{2^b}+\frac{y^3}{3^b}+\frac{y^4}{4^b}\right)\frac{z^4}{4^c}
\end{equation}
\begin{equation}  \nonumber
+\left(\frac{x^1}{1^a}+\frac{x^2}{2^a}+\frac{x^3}{3^a}+\frac{x^4}{4^a}+\frac{x^5}{5^a}\right)
 \left(\frac{y^1}{1^b}+\frac{y^2}{2^b}+\frac{y^3}{3^b}+\frac{y^4}{4^b}+\frac{y^5}{5^b}\right)\frac{z^5}{5^c}
\end{equation}
\begin{equation}  \nonumber
+\left(\frac{x^1}{1^a}+\frac{x^2}{2^a}+\frac{x^3}{3^a}+\frac{x^4}{4^a}+\frac{x^5}{5^a}+\frac{x^6}{6^a}\right)
 \left(\frac{y^1}{1^b}+\frac{y^2}{2^b}+\frac{y^3}{3^b}+\frac{y^4}{4^b}+\frac{y^5}{5^b}+\frac{y^6}{6^b}\right)\frac{z^6}{6^c}+\cdots
\end{equation}
\begin{equation}  \nonumber
 =\frac{x^1 y^1 z^1}{1^a 1^b 1^c}
\end{equation}

\begin{equation}  \nonumber
 +\frac{x^1 y^1 z^2}{1^a 1^b 2^c}+\frac{x^1 y^2 z^2}{1^a 2^b 2^c}
 \end{equation}
\begin{equation}  \nonumber
 +\frac{x^2 y^1 z^2}{2^a 1^b 2^c}+\frac{x^2 y^2 z^2}{2^a 2^b 2^c}
\end{equation}

\begin{equation}  \nonumber
 +\frac{x^1y^1 z^3}{1^a 1^b 3^c}+\frac{x^1 y^2 z^3}{1^a 2^b 3^c}+\frac{x^1 y^3 z^3}{1^a 3^b 3^c}
\end{equation}
\begin{equation}  \nonumber
 +\frac{x^2y^1 z^3}{2^a 1^b 3^c}+\frac{x^2 y^2 z^3}{2^a 2^b 3^c}+\frac{x^2 y^3 z^3}{2^a 3^b 3^c}
\end{equation}
\begin{equation}  \nonumber
 +\frac{x^3y^1 z^3}{3^a 1^b 3^c}+\frac{x^3 y^2 z^3}{3^a 2^b 3^c}+\frac{x^3 y^3 z^3}{3^a 3^b 3^c}
\end{equation}

\begin{equation}  \nonumber
 +\frac{x^1 y^1 z^4}{1^a 1^b 4^c}+\frac{x^1 y^2 z^4}{1^a 2^b 4^c}+\frac{x^1 y^3 z^4}{1^a 3^b 4^c}+\frac{x^1 y^4 z^4}{1^a 4^b 4^c}
\end{equation}
\begin{equation}  \nonumber
 +\frac{x^2 y^1 z^4}{2^a 1^b 4^c}+\frac{x^2 y^2 z^4}{2^a 2^b 4^c}+\frac{x^2 y^3 z^4}{2^a 3^b 4^c}+\frac{x^2 y^4 z^4}{2^a 4^b 4^c}
\end{equation}
\begin{equation}  \nonumber
 +\frac{x^3 y^1 z^4}{3^a 1^b 4^c}+\frac{x^3 y^2 z^4}{3^a 2^b 4^c}+\frac{x^3 y^3 z^4}{3^a 3^b 4^c}+\frac{x^3 y^4 z^4}{3^a 4^b 4^c}
\end{equation}
\begin{equation}  \nonumber
 +\frac{x^4 y^1 z^4}{4^a 1^b 4^c}+\frac{x^4 y^2 z^4}{4^a 2^b 4^c}+\frac{x^4 y^3 z^4}{4^a 3^b 4^c}+\frac{x^4 y^4 z^4}{4^a 4^b 4^c}
\end{equation}

\begin{equation}  \nonumber
 +\frac{x^1y^1z^5}{1^a1^b5^c}+\frac{x^1y^2z^5}{1^a2^b5^c}+\frac{x^1y^3z^5}{1^a3^b5^c}+\frac{x^1y^4z^5}{1^a4^b5^c}+\frac{x^1y^5z^5}{1^a5^b5^c}
\end{equation}
\begin{equation}  \nonumber
 +\frac{x^2y^1z^5}{2^a1^b5^c}+\frac{x^2y^2z^5}{2^a2^b5^c}+\frac{x^2y^3z^5}{2^a3^b5^c}+\frac{x^2y^4z^5}{2^a4^b5^c}+\frac{x^2y^5z^5}{2^a5^b5^c}
\end{equation}
\begin{equation}  \nonumber
 +\frac{x^3y^1z^5}{3^a1^b5^c}+\frac{x^3y^2z^5}{3^a2^b5^c}+\frac{x^3y^3z^5}{3^a3^b5^c}+\frac{x^3y^4z^5}{3^a4^b5^c}+\frac{x^3y^5z^5}{3^a5^b5^c}
\end{equation}
\begin{equation}  \nonumber
 +\frac{x^4y^1z^5}{4^a1^b5^c}+\frac{x^4y^2z^5}{4^a2^b5^c}+\frac{x^4y^3z^5}{4^a3^b5^c}+\frac{x^4y^4z^5}{4^a4^b5^c}+\frac{x^4y^5z^5}{4^a5^b5^c}
\end{equation}
\begin{equation}  \nonumber
 +\frac{x^5y^1z^5}{5^a1^b5^c}+\frac{x^5y^2z^5}{5^a2^b5^c}+\frac{x^5y^3z^5}{5^a3^b5^c}+\frac{x^5y^4z^5}{5^a4^b5^c}+\frac{x^5y^5z^5}{5^a5^b5^c}
\end{equation}

\begin{equation}  \nonumber
 + \quad \vdots \; \quad \; + \; \quad \vdots \; \; \quad + \; \quad \vdots \; \; \quad + \; \quad \vdots \; \; \quad + \; \quad \vdots \;   \ddots
\end{equation}

\begin{equation}  \nonumber
 = \sum_{l,m,n \geq 1; \;  l,m \leq n}^{\infty}  \frac{x^l y^m z^n}{l^a m^b n^c}
\end{equation}
\begin{equation}  \nonumber
 = \sum_{\substack{ h,l,m,n \geq 1 \\ l,m \leq n ; \, \gcd(l,m,n)=1}}  \frac{(x^l y^m z^n)^h}{h^{a+b+c} (l^a m^b n^c)}
\end{equation}
\begin{equation}  \nonumber
 = \sum_{\substack{ l,m,n \geq 1 \\ l,m \leq n ; \, \gcd(l,m,n)=1}}  \frac{1}{(l^a m^b n^c)}   \sum_{h=1}^{\infty} \frac{(x^l y^n z^n)^h}{h^{a+b+c}}
\end{equation}
\begin{equation}  \nonumber
 = \sum_{\substack{ l,m,n \geq 1 \\ l,m \leq n ; \, \gcd(l,m,n)=1}}  \frac{1}{(l^a m^b n^c)}   \log \left( \frac{1}{1 - x^l y^b z^c} \right) \quad if \quad a+b+c=1.
\end{equation}

Therefore, we have shown that if $a+b+c=1$ then
\begin{equation}  \nonumber
 \sum_{n=1}^{\infty}  \left( \sum_{l=1}^{n} \frac{x^l}{l^a} \right) \left( \sum_{m=1}^{n} \frac{y^m}{m^b} \right) \frac{z^n}{n^c}
 = \sum_{\substack{ l,m,n \geq 1 \\ l,m \leq n ; \, \gcd(l,m,n)=1}} \frac{1}{(l^a m^b n^c)}   \log \left( \frac{1}{1 - x^l y^b z^c} \right).
\end{equation}

 Exponentiating both sides gives us the $3D$ "pyramid VPV identity", where in this $3D$ case the pyramid takes the form of layered square shaped arrays of lattice point vectors as shown in the above workings.

 \begin{figure} [ht!]
\centering
    \includegraphics[width=10cm,angle=0,height=12cm]{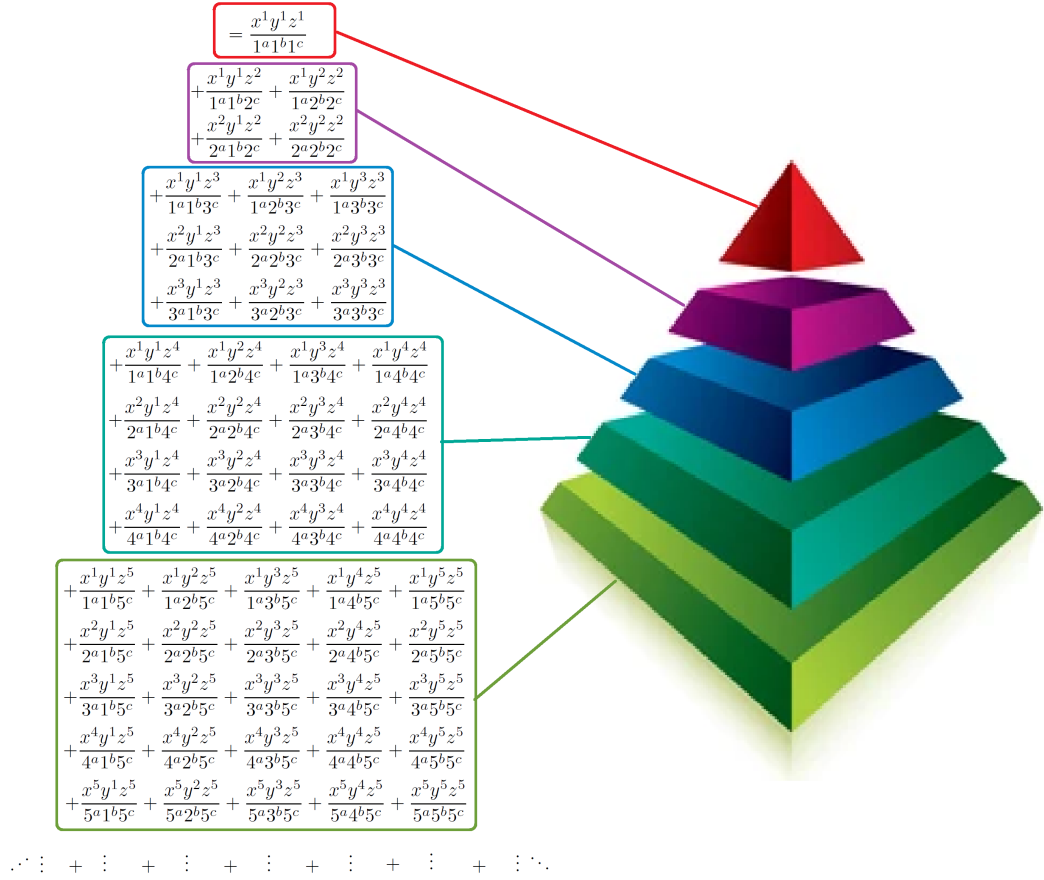}
  \caption{3D hyperpyramid sum form (pyramid version).}\label{Fig19}
\end{figure}

 The identity is summarized in the

\begin{theorem}     \label{vpv-pyramid3D-thm}
\textbf{The $3D$ square pyramid VPV identity.} If $|x|, |y|, |z| < 1$, with $a+b+c=1$,
 \begin{equation}   \label{14.10}
    \prod_{\substack{l,m,n \geq 1 \\ l,m \leq n ; \, \gcd(l,m,n)=1}} \left( \frac{1}{1-x^l y^m z^n} \right)^{\frac{1}{l^a m^b n^c}}
    = \exp\left\{ \sum_{n=1}^{\infty}  \left( \sum_{l=1}^{n} \frac{x^l}{l^a} \right) \left( \sum_{m=1}^{n} \frac{y^m}{m^b} \right) \frac{z^n}{n^c} \right\}.
  \end{equation}
\end{theorem}

As we did for the $2D$ particular cases, we can examine some obvious example corollaries arising from this theorem. Firstly, take the case where $a=b=0, c=1$, so then,

\begin{equation}   \nonumber
    \prod_{\substack{l,m,n \geq 1 \\ l,m \leq n ; \, \gcd(l,m,n)=1}} \left( \frac{1}{1-x^l y^m z^n} \right)^{\frac{1}{n}}
    = \exp\left\{ \sum_{n=1}^{\infty}  \left( \sum_{l=1}^{n} x^l \right) \left( \sum_{m=1}^{n} y^m \right) \frac{z^n}{n} \right\}
  \end{equation}
\begin{equation}   \nonumber
    = \exp\left\{ \frac{xy}{(1-x)(1-y)} \log \left( \frac{1-xyz}{1-z} \right)    \right\},
  \end{equation}

which brings us after exponentiating both sides to a set of $3D$ infinite products. So, we have

\begin{equation}   \label{14.11}
    \prod_{\substack{l,m,n \geq 1 \\ l,m \leq n ; \, \gcd(l,m,n)=1}} \left( \frac{1}{1-x^l y^m z^n} \right)^{\frac{1}{n}}
    = \left( \frac{1-xyz}{1-z} \right)^{\frac{xy}{(1-x)(1-y)}},
  \end{equation}

and the equivalent identity,

\begin{equation}   \label{14.12}
    \prod_{\substack{l,m,n \geq 1 \\ l,m \leq n ; \, \gcd(l,m,n)=1}} \left( 1-x^l y^m z^n \right)^{\frac{1}{n}}
    = \left( \frac{1-z}{1-xyz} \right)^{\frac{xy}{(1-x)(1-y)}}.
  \end{equation}

We see that (\ref{14.11}) and (\ref{14.12}) are generalizations of the 2D identities (\ref{14.02}) and (\ref{14.03}) from the previous section. Writing (\ref{14.12}) in longhand gives us,

\begin{equation}   \nonumber
 \left( \frac{1-z}{1-xyz} \right)^{\frac{xy}{(1-x)(1-y)}}
  \end{equation}
\begin{equation}   \nonumber
 =\sqrt{1-xyz^2}
  \end{equation}
\begin{equation}   \nonumber
  \sqrt[3]{(1-xyz^3)(1-x^2yz^3)}
  \end{equation}
\begin{equation}   \nonumber
  \sqrt[3]{(1-xy^2z^3)(1-x^2y^2z^3)}
  \end{equation}
\begin{equation}   \nonumber
  \sqrt[4]{(1-xyz^4)(1-x^2yz^4)(1-x^3yz^4)}
  \end{equation}
\begin{equation}   \nonumber
  \sqrt[4]{(1-xy^2z^4)(1-x^3y^2z^4)}
  \end{equation}
\begin{equation}   \nonumber
  \sqrt[4]{(1-xy^3z^4)(1-x^2y^3z^4)(1-x^3y^3z^4)}
  \end{equation}
\begin{equation}   \nonumber
  \sqrt[5]{(1-xyz^5)(1-x^2yz^5)(1-x^3yz^5)(1-x^4yz^5)}
  \end{equation}
\begin{equation}   \nonumber
  \sqrt[5]{(1-xy^2z^5)(1-x^2y^2z^5)(1-x^3y^2z^5)(1-x^3y^2z^5)}
  \end{equation}
\begin{equation}   \nonumber
  \sqrt[5]{(1-xy^3z^5)(1-x^2y^3z^5)(1-x^3y^3z^5)(1-x^3y^3z^5)}
  \end{equation}
\begin{equation}   \nonumber
  \sqrt[5]{(1-xy^4z^5)(1-x^2y^4z^5)(1-x^3y^4z^5)(1-x^3y^4z^5)}
  \end{equation}
\begin{equation}   \nonumber
  \sqrt[6]{(1-xyz^6)(1-x^2yz^6)(1-x^3yz^6)(1-x^4yz^6)(1-x^5yz^6)}
  \end{equation}
\begin{equation}   \nonumber
  \sqrt[6]{(1-xy^2z^6)(1-x^3y^2z^6)(1-x^5y^2z^6)}
  \end{equation}
\begin{equation}   \nonumber
  \sqrt[6]{(1-xy^3z^6)(1-x^2y^3z^6)(1-x^4y^3z^6)(1-x^5y^3z^6)}
  \end{equation}
\begin{equation}   \nonumber
  \sqrt[6]{(1-xy^4z^6)(1-x^3y^4z^6)(1-x^5y^4z^6)}
  \end{equation}
\begin{equation}   \nonumber
  \sqrt[6]{(1-xy^5z^6)(1-x^2y^5z^6)(1-x^3y^5z^6)(1-x^4y^5z^6)(1-x^5y^5z^6)} \quad etc.
  \end{equation}

\section{VPV identities in $nD$ square hyperpyramid regions.} \label{S:VPV hyperpyramids}

The $n$ dimensional square hyperpyramid VPV Identity is encoded in the following

\begin{theorem}   \label{9.1a}
  \textbf{The $nD$ square hyperpyramid VPV identity.} If $i = 1, 2, 3,...,n$ then for each $x_i \in \mathbb{C}$ such that $|x_i|<1$ and $b_i \in \mathbb{C}$ such that $\sum_{i=1}^{n}b_i = 1$,
  \begin{equation}   \label{14.13}
        \prod_{\substack{ \gcd(a_1,a_2,...,a_n)=1 \\ a_1,a_2,...,a_{n-1} < a_n \\ a_1,a_2,...,a_n \geq 1}} \left( \frac{1}{1-{x_1}^{a_1}{x_2}^{a_2}{x_3}^{a_3}\cdots{x_n}^{a_n}} \right)^{\frac{1}{{a_1}^{b_1}{a_2}^{b_2}{a_3}^{b_3}\cdots{a_n}^{b_n}}}
  \end{equation}
  \begin{equation}  \nonumber
  = \exp\left\{ \sum_{k=1}^{\infty} \prod_{i=1}^{n-1} \left( \sum_{j=1}^{k} \frac{{x_i}^j}{j^{b_i}} \right)\frac{{x_n}^k}{k^{b_n}} \right\}
  \end{equation}
  \begin{equation}  \nonumber
  = \exp\left\{ \sum_{k=1}^{\infty} \left( \sum_{j=1}^{k} \frac{{x_1}^j}{j^{b_1}} \right) \left( \sum_{j=1}^{k} \frac{{x_2}^j}{j^{b_2}} \right)
  \left( \sum_{j=1}^{k} \frac{{x_3}^j}{j^{b_3}} \right)  \cdots   \left( \sum_{j=1}^{k} \frac{{x_{n-1}}^j}{j^{b_{n-1}}} \right)
  \frac{{x_n}^k}{k^{b_n}} \right\}.
  \end{equation}
  \end{theorem}

This result is quite straight-forward to prove using the technique of our two previous sections. It was also given in Campbell \cite{gC2000} by
summing on the VPV’s in the $n$-space hyperpyramid, defined by the inequalities
  \begin{equation}\label{14.14}
   x_1<x_n, x_2<x_n, x_3<x_n, ... , x_{n-1}<x_n
  \end{equation}

in the first $n$-space hyperquadrant, and applying the following
\begin{lemma} \label{lemma4.1}
  Consider an infinite region raying out of the origin in any Euclidean
vector space. The set of all lattice point vectors apart from the origin in that region is
precisely the set of positive integer multiples of the VPVs in that region.
\end{lemma}

The corresponding theorem from Campbell \cite{gC1994} was summed very simply over all
lattice point vectors in the first hyperquadrant.

Further consequences of the above theorem are given as follows.

The 2D case of theorem \ref{9.1a} is

\begin{corollary}   \label{9.3a}
  If $|yz|$ and $|z|<1$ and $s+t=1$ then,
  \begin{equation}   \label{14.15}
    \exp\left\{ \left( \sum_{k=1}^{\infty} \frac{{z}^k}{k^{t}} \right)\right\}
    \prod_{\substack{ (a,b)=1 \\ a < b \\ a > 0, b > 1}} \left( \frac{1}{1-{y}^{a}{z}^{b}} \right)^{\frac{1}{{a}^{s}{b}^{t}}}
  \end{equation}
    \begin{equation}   \nonumber
    = \exp\left\{ \frac{{z}^1}{1^{t}} + \left(1+ \frac{{y}^1}{1^{s}}\right) \frac{{z}^2}{2^{t}} + \left(1+ \frac{{y}^1}{1^{s}}+\frac{{y}^2}{2^{s}}\right)\frac{{z}^3}{3^{t}}+\cdots \right\}
  \end{equation}
\end{corollary}

The 3D case of theorem \ref{9.1a} is

\begin{corollary}   \label{9.4a}
  If $|xyz|$, $|yz|$ and $|z|<1$ and $r+s+t=1$ then,
  \begin{equation}   \label{14.16}
    \exp\left\{ \sum_{k=1}^{\infty} \frac{{z}^k}{k^{t}} \right\}
    \prod_{\substack{ (a,b,c)=1 \\ a,b < c \\ a,b > 0, c > 1}} \left( \frac{1}{1-{x}^{a}{y}^{b}{z}^{c}} \right)^{\frac{1}{{a}^{r}{b}^{s}{c}^{t}}}
  \end{equation}
    \begin{equation}   \nonumber
    = \exp\left\{ \frac{{z}^1}{1^{t}} + \left(1+ \frac{{x}^1}{1^{r}}\right)\left(1+ \frac{{y}^1}{1^{s}}\right) \frac{{z}^2}{2^{t}}
    + \left(1+ \frac{{x}^1}{1^{r}}+\frac{{x}^2}{2^{r}}\right)\left(1+ \frac{{y}^1}{1^{s}}+\frac{{y}^2}{2^{s}}\right)\frac{{z}^3}{3^{t}}+\cdots \right\}
  \end{equation}
\end{corollary}

The 4D case of theorem \ref{9.1a} is

\begin{corollary}   \label{9.5a}
  If $|wxyz|$, $|xyz|$, $|yz|$ and $|z|<1$ and $r+s+t+u=1$ then,
  \begin{equation}   \label{14.16a}
    \exp\left\{\sum_{k=1}^{\infty} \frac{{z}^k}{k^{u}}\right\}
    \prod_{\substack{ (a,b,c,d)=1 \\ a,b,c < d \\ a,b,c > 0, d > 1}} \left( \frac{1}{1-{w}^{a}{x}^{b}{y}^{c}{z}^{d}} \right)^{\frac{1}{{a}^{r}{b}^{s}{c}^{t}{d}^{u}}} = \exp \left\{ \mathrm{P}_3(r,w;s,x;t,y;u,z)\right\}
  \end{equation}
  where $\mathrm{P}_3$, is a 4D hyperpyramid function,
    \begin{multline*}   \nonumber
   \mathrm{P}_3(r,w;s,x;t,y;u,z) = \frac{{z}^1}{1^{u}} + \left(1+ \frac{{w}^1}{1^{r}}\right)\left(1+ \frac{{x}^1}{1^{s}}\right)\left(1+ \frac{{y}^1}{1^{t}}\right) \frac{{z}^2}{2^{u}}  \\
     + \left(1+ \frac{{w}^1}{1^{r}}+\frac{{w}^2}{2^{r}}\right)\left(1+ \frac{{x}^1}{1^{s}}+\frac{{x}^2}{2^{s}}\right)
      \left(1+ \frac{{y}^1}{1^{t}}+\frac{{y}^2}{2^{t}}\right)\frac{{z}^3}{3^{u}}+\cdots
  \end{multline*}
\end{corollary}

The approach we adopt to give the reader an intuitive sense for these identities is to state corollaries and then examples from them.
The 2D case through to the 5D case of (\ref{14.13}) are given in the following examples of the \textit{square hyperpyramid identity}.

\begin{corollary} For $|y|, |z|<1,$
  \begin{equation}\label{14.17}
    \prod_{\substack{(a,b)=1 \\ a<b \\ a\geq0,b>0}} \left( \frac{1}{1-y^a z^b} \right)^{\frac{1}{b}}
    = \left(\frac{1-yz}{1-z}\right)^{\frac{1}{1-y}}
  \end{equation}
        \begin{equation}  \nonumber
= 1 + \frac{z}{1!} + \begin{vmatrix}
    1 & -1 \\
    \frac{1-y^2}{1-y} & 1 \\
  \end{vmatrix} \frac{z^2}{2!}
  + \begin{vmatrix}
    1 & -1 & 0 \\
    \frac{1-y^2}{1-y} & 1 & -2 \\
    \frac{1-y^3}{1-y} & \frac{1-y^2}{1-y} & 1 \\
  \end{vmatrix} \frac{z^3}{3!}
+ \begin{vmatrix}
    1 & -1 & 0 & 0 \\
    \frac{1-y^2}{1-y} & 1 & -2 & 0 \\
    \frac{1-y^3}{1-y} & \frac{1-y^2}{1-y} & 1 & -3 \\
    \frac{1-y^4}{1-y} & \frac{1-y^3}{1-y} & \frac{1-y^2}{1-y} & 1 \\
  \end{vmatrix} \frac{z^4}{4!}
+ etc.
\end{equation}
   \end{corollary}
  In this case it is fairly easy to find the Taylor coefficients for the (\ref{14.17}) right side function. Hence we get a closed form evaluation of the determinant coefficients. In Mathematica, and WolframAlpha one easily sees that the Taylor series is

  \begin{equation}  \nonumber
    \left(\frac{1-yz}{1-z}\right)^{\frac{1}{1-y}} = 1 + z + (y + 2) \frac{z^2}{2!} + (2 y^2 + 5 y + 6) \frac{z^3}{3!} + (6 y^3 + 17 y^2 + 26 y + 24) \frac{z^4}{4!}
  \end{equation}
  \begin{equation} \nonumber
   + (24 y^4 + 74 y^3 + 129 y^2 + 154 y + 120) \frac{z^5}{5!} + O(z^6)
  \end{equation}
  and that the expansion is encapsulated by
   $\sum_{n=0}^{\infty} c_n z^n$ where $c_0 = 1$, $c_1 = 1$ with the recurrence
   \begin{equation} \nonumber
   ny c_n + (n+2) c_{n+2} = (2 + n + y + ny) c_{n+1}.
   \end{equation}

   Incidentally, also in Mathematica, and WolframAlpha one easily sees, for example, that the code
   \begin{equation} \nonumber
   Det[\{1,-1,0,0\},\{(1-y^2)/(1-y),1,-2,0\},\{(1-y^3)/(1-y),(1-y^2)/(1-y),1,-3\},
  \end{equation}
  \begin{equation} \nonumber
  \{(1-y^4)/(1-y),(1-y^3)/(1-y),(1-y^2)/(1-y),1\}]
   \end{equation}
   nicely verifies the coefficient given by

  \begin{equation}  \nonumber
\begin{vmatrix}
    1 & -1 & 0 & 0 \\
    \frac{1-y^2}{1-y} & 1 & -2 & 0 \\
    \frac{1-y^3}{1-y} & \frac{1-y^2}{1-y} & 1 & -3 \\
    \frac{1-y^4}{1-y} & \frac{1-y^3}{1-y} & \frac{1-y^2}{1-y} & 1 \\
  \end{vmatrix}
  = 6 y^3 + 17 y^2 + 26 y + 24. \\
\end{equation}

\begin{corollary} For each of $|x|, |y|, |z|<1,$
    \begin{equation}\label{14.18}
    \prod_{\substack{(a,b,c)=1 \\ a,b<c \\ a,b\geq0,c>0}} \left( \frac{1}{1-x^a y^b z^c} \right)^{\frac{1}{c}}
    = \left(\frac{(1-xz)(1-yz)}{(1-z)(1-xyz)}\right)^{\frac{1}{(1-x)(1-y)}}
  \end{equation}
      \begin{equation}  \nonumber
= 1 + \frac{z}{1!} + \begin{vmatrix}
    1 & -1 \\
    \frac{(1-x^2)(1-y^2)}{(1-x)(1-y)} & 1 \\
  \end{vmatrix} \frac{z^2}{2!}
  + \begin{vmatrix}
    1 & -1 & 0 \\
    \frac{(1-x^2)(1-y^2)}{(1-x)(1-y)} & 1 & -2 \\
    \frac{(1-x^3)(1-y^3)}{(1-x)(1-y)} & \frac{(1-x^2)(1-y^2)}{(1-x)(1-y)} & 1 \\
  \end{vmatrix} \frac{z^3}{3!}
      \end{equation}
  \begin{equation}  \nonumber
+ \begin{vmatrix}
    1 & -1 & 0 & 0 \\
    \frac{(1-x^2)(1-y^2)}{(1-x)(1-y)} & 1 & -2 & 0 \\
    \frac{(1-x^3)(1-y^3)}{(1-x)(1-y)} & \frac{(1-x^2)(1-y^2)}{(1-x)(1-y)} & 1 & -3 \\
    \frac{(1-x^4)(1-y^4)}{(1-x)(1-y)} & \frac{(1-x^3)(1-y^3)}{(1-x)(1-y)} & \frac{(1-x^2)(1-y^2)}{(1-x)(1-y)} & 1 \\
  \end{vmatrix} \frac{z^4}{4!}
+ etc.
\end{equation}
   \end{corollary}

\begin{corollary} For each of $|w|, |x|, |y|, |z|<1,$
    \begin{equation}\label{14.19}
    \prod_{\substack{(a,b,c,d)=1 \\ a,b,c<d \\ a,b,c\geq0,d>0}} \left( \frac{1}{1-w^a x^b y^c z^d} \right)^{\frac{1}{d}}
    = \left(\frac{(1-wz)(1-xz)(1-yz)(1-wxyz)}{(1-z)(1-wxz)(1-wyz)(1-xyz)}\right)^{\frac{1}{(1-w)(1-x)(1-y)}},
  \end{equation}
    \begin{equation}  \nonumber
= 1 + \frac{z}{1!} + \begin{vmatrix}
    1 & -1 \\
    \frac{(1-w^2)(1-x^2)(1-y^2)}{(1-w)(1-x)(1-y)} & 1 \\
  \end{vmatrix} \frac{z^2}{2!}
\end{equation}
  \begin{equation}  \nonumber
  + \begin{vmatrix}
    1 & -1 & 0 \\
    \frac{(1-w^2)(1-x^2)(1-y^2)}{(1-w)(1-x)(1-y)} & 1 & -2 \\
    \frac{(1-w^3)(1-x^3)(1-y^3)}{(1-w)(1-x)(1-y)} & \frac{(1-w^2)(1-x^2)(1-y^2)}{(1-w)(1-x)(1-y)} & 1 \\
  \end{vmatrix} \frac{z^3}{3!}
      \end{equation}
  \begin{equation}  \nonumber
+ \begin{vmatrix}
    1 & -1 & 0 & 0 \\
    \frac{(1-w^2)(1-x^2)(1-y^2)}{(1-w)(1-x)(1-y)} & 1 & -2 & 0 \\
    \frac{(1-w^3)(1-x^3)(1-y^3)}{(1-w)(1-x)(1-y)} & \frac{(1-w^2)(1-x^2)(1-y^2)}{(1-w)(1-x)(1-y)} & 1 & -3 \\
    \frac{(1-w^4)(1-x^4)(1-y^4)}{(1-w)(1-x)(1-y)} & \frac{(1-w^3)(1-x^3)(1-y^3)}{(1-w)(1-x)(1-y)} & \frac{(1-w^2)(1-x^2)(1-y^2)}{(1-w)(1-x)(1-y)} & 1 \\
  \end{vmatrix} \frac{z^4}{4!}
+ etc.
\end{equation}
   \end{corollary}

\begin{corollary} For each of $|v|, |w|, |x|, |y|, |z|<1,$
    \begin{equation}\label{14.20}
    \prod_{\substack{(a,b,c,d,e)=1 \\ a,b,c,d<e \\ a,b,c,d\geq0,e>0}} \left( \frac{1}{1-v^a w^b x^c y^d z^e} \right)^{\frac{1}{e}}
   \end{equation}
   \begin{gather}\nonumber
     =       \left(\frac{(1-vz)(1-wz)(1-xz)(1-yz)}{(1-z)(1-vwz)(1-vxz)(1-vyz)}\right)^{\frac{1}{(1-v)(1-w)(1-x)(1-y)}} \\  \nonumber
      \times \left(\frac{(1-vwxz)(1-vwyz)(1-vxyz)(1-wxyz)}{(1-wxz)(1-wyz)(1-xyz)(1-vwxyz)}\right)^{\frac{1}{(1-v)(1-w)(1-x)(1-y)}}.
   \end{gather}
    \begin{equation}  \nonumber
= 1 + \frac{z}{1!} + \begin{vmatrix}
    1 & -1 \\
    \frac{(1-v^2)(1-w^2)(1-x^2)(1-y^2)}{(1-v)(1-w)(1-x)(1-y)} & 1 \\
  \end{vmatrix} \frac{z^2}{2!}
\end{equation}
  \begin{equation}  \nonumber
  + \begin{vmatrix}
    1 & -1 & 0 \\
    \frac{(1-v^2)(1-w^2)(1-x^2)(1-y^2)}{(1-v)(1-w)(1-x)(1-y)} & 1 & -2 \\
    \frac{(1-v^3)(1-w^3)(1-x^3)(1-y^3)}{(1-v)(1-w)(1-x)(1-y)} & \frac{(1-v^2)(1-w^2)(1-x^2)(1-y^2)}{(1-v)(1-w)(1-x)(1-y)} & 1 \\
  \end{vmatrix} \frac{z^3}{3!}
      \end{equation}
  \begin{equation}  \nonumber
+ \begin{vmatrix}
    1 & -1 & 0 & 0 \\
    \frac{(1-v^2)(1-w^2)(1-x^2)(1-y^2)}{(1-v)(1-w)(1-x)(1-y)} & 1 & -2 & 0 \\
    \frac{(1-v^3)(1-w^3)(1-x^3)(1-y^3)}{(1-v)(1-w)(1-x)(1-y)} & \frac{(1-v^2)(1-w^2)(1-x^2)(1-y^2)}{(1-v)(1-w)(1-x)(1-y)} & 1 & -3 \\
    \frac{)1-v^4)(1-w^4)(1-x^4)(1-y^4)}{(1-v)(1-w)(1-x)(1-y)} & \frac{(1-v^3)(1-w^3)(1-x^3)(1-y^3)}{(1-v)(1-w)(1-x)(1-y)} & \frac{(1-v^2)(1-w^2)(1-x^2)(1-y^2)}{(1-v)(1-w)(1-x)(1-y)} & 1 \\
  \end{vmatrix} \frac{z^4}{4!}
       \end{equation}
  \begin{equation}  \nonumber
+ etc.
 \end{equation}
    \end{corollary}

\bigskip
Next we drill down further by taking corollaries of corollaries of equations (\ref{14.18}) through to (\ref{14.20}). Hence, we observe the emergence of binomial coefficients in the right side indices of the following (\ref{14.21}) through to (\ref{14.23}) results.

\bigskip

\begin{corollary} For each of $|y|, |z|<1,$
    \begin{equation}\label{14.21}
    \prod_{\substack{(a,b,c)=1 \\ a,b<c \\ a,b\geq0,c>0}} \left( \frac{1}{1-y^{a+b} z^c} \right)^{\frac{1}{c}}
    = \left(\frac{(1-yz)^2}{(1-z)(1-y^2z)}\right)^{\frac{1}{(1-y)^2}}
  \end{equation}
      \begin{equation}  \nonumber
= 1 + \frac{z}{1!} + \begin{vmatrix}
    1 & -1 \\
    \left(\frac{1-y^2}{1-y}\right)^2 & 1 \\
  \end{vmatrix} \frac{z^2}{2!}
  + \begin{vmatrix}
    1 & -1 & 0 \\
    \left(\frac{1-y^2}{1-y}\right)^2 & 1 & -2 \\
    \left(\frac{1-y^3}{1-y}\right)^2 & \left(\frac{1-y^2}{1-y}\right)^2 & 1 \\
  \end{vmatrix} \frac{z^3}{3!}
      \end{equation}
  \begin{equation}  \nonumber
+ \begin{vmatrix}
    1 & -1 & 0 & 0 \\
    \left(\frac{1-y^2}{1-y}\right)^2 & 1 & -2 & 0 \\
    \left(\frac{1-y^3}{1-y}\right)^2 & \left(\frac{1-y^2}{1-y}\right)^2 & 1 & -3 \\
    \left(\frac{1-y^4}{1-y}\right)^2 & \left(\frac{1-y^3}{1-y}\right)^2 & \left(\frac{1-y^2}{1-y}\right)^2 & 1 \\
  \end{vmatrix} \frac{z^4}{4!}
+ etc.
\end{equation}
   \end{corollary}

\begin{corollary} For each of $|y|, |z|<1,$
    \begin{equation}\label{14.22}
    \prod_{\substack{(a,b,c,d)=1 \\ a,b,c<d \\ a,b,c\geq0,d>0}} \left( \frac{1}{1-y^{a+b+c} z^d} \right)^{\frac{1}{d}}
    = \left(\frac{(1-yz)^3(1-y^3z)}{(1-z)(1-y^2z)^3}\right)^{\frac{1}{(1-y)^3}}
  \end{equation}
    \begin{equation}  \nonumber
= 1 + \frac{z}{1!} + \begin{vmatrix}
    1 & -1 \\
    \left(\frac{1-y^2}{1-y}\right)^3 & 1 \\
  \end{vmatrix} \frac{z^2}{2!}
  + \begin{vmatrix}
    1 & -1 & 0 \\
    \left(\frac{1-y^2}{1-y}\right)^3 & 1 & -2 \\
    \left(\frac{1-y^3}{1-y}\right)^3 & \left(\frac{1-y^2}{1-y}\right)^3 & 1 \\
  \end{vmatrix} \frac{z^3}{3!}
      \end{equation}
  \begin{equation}  \nonumber
+ \begin{vmatrix}
    1 & -1 & 0 & 0 \\
    \left(\frac{1-y^2}{1-y}\right)^3 & 1 & -2 & 0 \\
    \left(\frac{1-y^3}{1-y}\right)^3 & \left(\frac{1-y^2}{1-y}\right)^3 & 1 & -3 \\
    \left(\frac{1-y^4}{1-y}\right)^3 & \left(\frac{1-y^3}{1-y}\right)^3 & \left(\frac{1-y^2}{1-y}\right)^3 & 1 \\
  \end{vmatrix} \frac{z^4}{4!}
+ etc.
\end{equation}
   \end{corollary}

\begin{corollary} For each of $|y|, |z|<1,$
    \begin{equation}\label{14.23}
    \prod_{\substack{(a,b,c,d,e)=1 \\ a,b,c,d<e \\ a,b,c,d\geq0,e>0}} \left( \frac{1}{1-y^{a+b+c+d} z^e} \right)^{\frac{1}{e}}
       =   \left(\frac{(1-yz)^4 (1-y^3z)^4}{(1-z)(1-y^2z)^6 (1-y^4z)}\right)^{\frac{1}{(1-y)^4}}
   \end{equation}
    \begin{equation}  \nonumber
= 1 + \frac{z}{1!} + \begin{vmatrix}
    1 & -1 \\
    \left(\frac{1-y^2}{1-y}\right)^4 & 1 \\
  \end{vmatrix} \frac{z^2}{2!}
  + \begin{vmatrix}
    1 & -1 & 0 \\
    \left(\frac{1-y^2}{1-y}\right)^4 & 1 & -2 \\
    \left(\frac{1-y^3}{1-y}\right)^4 & \left(\frac{1-y^2}{1-y}\right)^4 & 1 \\
  \end{vmatrix} \frac{z^3}{3!}
      \end{equation}
  \begin{equation}  \nonumber
+ \begin{vmatrix}
    1 & -1 & 0 & 0 \\
    \left(\frac{1-y^2}{1-y}\right)^4 & 1 & -2 & 0 \\
    \left(\frac{1-y^3}{1-y}\right)^4 & \left(\frac{1-y^2}{1-y}\right)^4 & 1 & -3 \\
    \left(\frac{1-y^4}{1-y}\right)^4 & \left(\frac{1-y^3}{1-y}\right)^4 & \left(\frac{1-y^2}{1-y}\right)^4 & 1 \\
  \end{vmatrix} \frac{z^4}{4!}
+ etc.
 \end{equation}
    \end{corollary}

\bigskip
Beyond 3D space, it is humanly impossible to visualize the 4D or 5D grids that would depict the coefficients in these identities. However, the above equation (\ref{14.23}) represents combinatorially a statement about radial-from-origin 5D lattice points pertaining to a select cluster of vectors or lattice points (depending on preference) on a 2D manifold through the 5D lattice space. The radial-from-origin aspect is always evident from the product restriction on left side where $\gcd(a,b,c,d,e)=1$.

\section{Finite Euler Sums}

Consider the following well known summations:

\begin{equation} \label{16.50}
\sum_{k=1}^{n} k^1 = \frac{n}{2} + \frac{n^2}{2},
\end{equation}
\begin{equation} \label{16.51}
\sum_{k=1}^{n} k^2 = \frac{n}{6} + \frac{n^2}{2} + \frac{n^3}{3},
\end{equation}
\begin{equation} \label{16.52}
\sum_{k=1}^{n} k^3 = \frac{n^2}{4} + \frac{n^3}{3} + \frac{n^4}{4},
\end{equation}
\begin{equation} \label{16.53}
\sum_{k=1}^{n} k^4 = \frac{-n}{30} + \frac{n^3}{3} + \frac{n^4}{2} + \frac{n^5}{5},
\end{equation}

and their generalizations:

\begin{equation} \label{16.54}
\sum_{k=1}^{n} k^1 z^k = \frac{z - (1 + n) z^{n+1} + n z^{n+2}}{(1-z)^2},
\end{equation}
\begin{equation} \label{16.55}
\sum_{k=1}^{n} k^2 z^k
= \frac{ - n^2 z^{n+3} + (2n^2+2n-1)z^{n+2} -(n^2+2n+1) z^{n+1} + z^2 + z}{(1-z)^3},
\end{equation}
\begin{equation} \label{16.56}
\sum_{k=1}^{n} k^3 z^k = \frac{A}{(1-z)^4}, \quad where
\end{equation}
\begin{equation} \nonumber
A=n^3 z^{n+4} +(3n^3+6n^2-4) z^{n+2} -(3n^3+3n^2-3n+1) z^{n+3} -(n+1)^3 z^{n+1} +z^3 + 4z^2 + z.
\end{equation}

\begin{equation} \label{16.57}
\sum_{k=1}^{n} k^4 z^k = - \frac{B}{(1-z)^5}, \quad where
\end{equation}
\begin{equation} \nonumber
B = n^4z^{n+5}+(-4n^4-4n^3+6n^2-4n+1)z^{n+4} +(6n^4+12n^3-6n^2-12n+11)z^{n+3}
\end{equation}
\begin{equation} \nonumber
+ (4n^4+12n^3+6n^2-12n-11)z^{n+2} + (n+1)^4 z^{n+1} - z^4 - 11z^3 - 11z^2 - z.
\end{equation}

\bigskip
These kind of finite sums lend themselves to substitution in the hyperpyramid identities; in fact for both the square and skewed hyperpyramid versions. It is illustrative here to give a few examples in 2D and 3D of the uses for the Euler Sums in the context of VPV identities here.

\subsection{The $2D$ square hyperpyramid VPV identity.}

Recall from an earlier paper that:

If $|y|, |z| < 1$, with $a+b=1$,
 \begin{equation}   \label{16.57a}
    \prod_{\substack{m,n \geq 1 \\ m \leq n ; \, \gcd(m,n)=1}} \left( \frac{1}{1- y^m z^n} \right)^{\frac{1}{m^a n^b}}
    = \exp\left\{\sum_{n=1}^{\infty} \left( \sum_{m=1}^{n} \frac{y^m}{m^a} \right) \frac{z^n}{n^b} \right\}.
  \end{equation}

Let us set by definition $\gcd(m,n):=(m,n)$ and apply equations (\ref{16.50}) through to (\ref{16.53}) substituted respectively into (\ref{16.57a}). This corresponds to the cases of (\ref{16.57a}) with:

 $y=1$, $a=-1$, implying from condition $a+b=1$ that $b=2$;

 $y=1$, $a=-2$, implying from condition $a+b=1$ that $b=3$;

 $y=1$, $a=-3$, implying from condition $a+b=1$ that $b=4$;

 $y=1$, $a=-4$, implying from condition $a+b=1$ that $b=5$;

 \begin{flushleft}
 so the resulting substitutions give us respectively the four equations
 \end{flushleft}

 \begin{equation}   \nonumber
    \prod_{\substack{m,n \geq 1 \\ m \leq n ; \, (m,n)=1}} \left( \frac{1}{1-z^n} \right)^{\frac{m^1 }{n^2}}
    = \exp\left\{ \sum_{n=1}^{\infty}  \left( \frac{n}{2} + \frac{n^2}{2} \right) \frac{z^n}{n^2} \right\}
\end{equation}
 \begin{equation}   \nonumber
    = \exp\left\{ \frac{1}{2} \log\left(\frac{1}{1-z}\right) + \frac{1}{2}\frac{z}{1-z} \right\},
\end{equation}

 \begin{equation}   \nonumber
    \prod_{\substack{m,n \geq 1 \\ m \leq n ; \, (m,n)=1}} \left( \frac{1}{1-z^n} \right)^{\frac{m^2 }{n^3}}
    = \exp\left\{ \sum_{n=1}^{\infty}  \left( \frac{n}{6} + \frac{n^2}{2} + \frac{n^3}{3} \right) \frac{z^n}{n^3} \right\}
\end{equation}
  \begin{equation}   \nonumber
    = \exp\left\{\frac{1}{6} Li_2(z) + \frac{1}{2} \log\left(\frac{1}{1-z}\right) + \frac{1}{3}\frac{z}{1-z} \right\},
\end{equation}

 \begin{equation}   \nonumber
    \prod_{\substack{m,n \geq 1 \\ m \leq n ; \, (m,n)=1}} \left( \frac{1}{1-z^n} \right)^{\frac{m^3 }{n^4}}
    = \exp\left\{ \sum_{n=1}^{\infty}  \left( \frac{n^2}{4} + \frac{n^3}{3} + \frac{n^4}{4} \right) \frac{z^n}{n^4} \right\}
\end{equation}
 \begin{equation}   \nonumber
    = \exp\left\{\frac{1}{4} Li_2(z) + \frac{1}{3} \log\left(\frac{1}{1-z}\right) + \frac{1}{4}\frac{z}{1-z} \right\},
\end{equation}

 \begin{equation}   \nonumber
    \prod_{\substack{m,n \geq 1 \\ m \leq n ; \, (m,n)=1}} \left( \frac{1}{1-z^n} \right)^{\frac{m^4 }{n^5}}
    = \exp\left\{ \sum_{n=1}^{\infty}  \left( \frac{-n}{30} + \frac{n^3}{3} + \frac{n^4}{2} + \frac{n^5}{5} \right) \frac{z^n}{n^5} \right\}
\end{equation}
 \begin{equation}   \nonumber
    = \exp\left\{\frac{-1}{30} Li_3(z) + \frac{1}{3} \log\left(\frac{1}{1-z}\right) + \frac{1}{2}\frac{z}{(1-z)^2}
    + \frac{1}{5}\frac{z}{1-z} \right\}.
\end{equation}

Based on the above equations and their reciprocal equations, we may assert the following four theorems.

\begin{theorem}    \label{16.57b}
For $|z|<1$,
 \begin{equation}   \nonumber
    \prod_{\substack{m,n \geq 1 \\ m \leq n ; \, (m,n)=1}} \left( \frac{1}{1-z^n} \right)^{\frac{m^1 }{n^2}}
    = \sqrt{\frac{1}{1-z}} \; \exp\left\{ \frac{1}{2}\frac{z}{1-z} \right\},
\end{equation}
 \begin{equation}   \nonumber
    \prod_{\substack{m,n \geq 1 \\ m \leq n ; \, (m,n)=1}} \left(1-z^n\right)^{\frac{m^1 }{n^2}}
    = \sqrt{1-z} \; \exp\left\{ \frac{-1}{2}\frac{z}{1-z} \right\}.
\end{equation}
\end{theorem}

\begin{theorem}   \label{16.57c}
For $|z|<1$,
 \begin{equation}   \nonumber
    \prod_{\substack{m,n \geq 1 \\ m \leq n ; \, (m,n)=1}} \left( \frac{1}{1-z^n} \right)^{\frac{m^2 }{n^3}}
    = \sqrt{\frac{1}{1-z}} \; \exp\left\{\frac{1}{6} Li_2(z) + \frac{1}{3}\frac{z}{1-z} \right\},
\end{equation}
 \begin{equation}   \nonumber
    \prod_{\substack{m,n \geq 1 \\ m \leq n ; \, (m,n)=1}} \left( 1-z^n\right)^{\frac{m^2 }{n^3}}
    = \sqrt{1-z} \; \exp\left\{\frac{-1}{6} Li_2(z) - \frac{1}{3}\frac{z}{1-z} \right\}.
\end{equation}
\end{theorem}

\begin{theorem}   \label{16.57d}
For $|z|<1$,
 \begin{equation}   \nonumber
    \prod_{\substack{m,n \geq 1 \\ m \leq n ; \, (m,n)=1}} \left( \frac{1}{1-z^n} \right)^{\frac{m^3 }{n^4}}
    = \sqrt[3]{\frac{1}{1-z}} \; \exp\left\{\frac{1}{4} Li_2(z) + \frac{1}{4}\frac{z}{1-z} \right\},
\end{equation}
 \begin{equation}   \nonumber
    \prod_{\substack{m,n \geq 1 \\ m \leq n ; \, (m,n)=1}} \left( 1-z^n \right)^{\frac{m^3 }{n^4}}
    = \sqrt[3]{1-z} \; \exp\left\{\frac{-1}{4} Li_2(z) - \frac{1}{4}\frac{z}{1-z} \right\}.
\end{equation}
\end{theorem}

\begin{theorem}   \label{16.57e}
For $|z|<1$,
 \begin{equation}   \nonumber
    \prod_{\substack{m,n \geq 1 \\ m \leq n ; \, (m,n)=1}} \left( \frac{1}{1-z^n} \right)^{\frac{m^4 }{n^5}}
    = \sqrt[3]{\frac{1}{1-z}} \; \exp\left\{\frac{z(7-2z)}{10(1-z)^2} - \frac{1}{30} Li_3(z) \right\},
\end{equation}
 \begin{equation}   \nonumber
    \prod_{\substack{m,n \geq 1 \\ m \leq n ; \, (m,n)=1}} \left( 1-z^n \right)^{\frac{m^4 }{n^5}}
    = \sqrt[3]{1-z} \; \exp\left\{\frac{z(2z-7)}{10(1-z)^2} + \frac{1}{30} Li_3(z) \right\}.
\end{equation}
\end{theorem}

The above four theorems give us a single variable equation in each case that encodes a statement about "weighted integer partitions" of a kind not normally discussed in the partition literature. However equations (\ref{16.54}) through to (\ref{16.57}) substituted respectively into (\ref{16.57a}) can supply us with two-variable 2D generalizations of theorem \ref{16.57b} to theorem \ref{16.57e}. This corresponds to the cases of (\ref{16.57a}) with:

 $a=-1$, implying from condition $a+b=1$ that $b=2$;

 $a=-2$, implying from condition $a+b=1$ that $b=3$;

 $a=-3$, implying from condition $a+b=1$ that $b=4$;

 $a=-4$, implying from condition $a+b=1$ that $b=5$;

 \begin{flushleft}
 so the resulting substitutions (putting $y$ for $z$ in each case) give us respectively the four 2D equation workings,
 \end{flushleft}

 \begin{equation}   \nonumber
    \prod_{\substack{m,n \geq 1 \\ m \leq n ; \, (m,n)=1}} \left( \frac{1}{1- y^m z^n} \right)^{\frac{m^1 }{n^2}}
    = \exp\left\{ \sum_{n=1}^{\infty}  \left( \frac{y - (1 + n) y^{n+1} + n y^{n+2}}{(1-y)^2} \right) \frac{z^n}{n^2} \right\}
\end{equation}
 \begin{equation}   \nonumber
    = \exp\left\{ \frac{y Li_2(z)}{(1-y)^2} - \frac{y Li_2(yz)}{(1-y)^2} - \frac{y}{(1-y)^2} \log\left(\frac{1}{1-yz}\right)
     + \frac{y^2}{(1-y)^2} \log\left(\frac{1}{1-yz}\right) \right\}
\end{equation}
 \begin{equation}   \nonumber
    = \left(1-yz\right)^{\frac{y}{1-y}} \exp\left\{ \frac{y}{(1-y)^2} \left(Li_2(z)- Li_2(yz)\right) \right\};
\end{equation}

 \begin{equation}   \nonumber
    \prod_{\substack{m,n \geq 1 \\ m \leq n ; \, (m,n)=1}} \left( \frac{1}{1- y^m z^n} \right)^{\frac{m^2 }{n^3}}
    \end{equation}
  \begin{equation}   \nonumber
    = \exp\left\{ \sum_{n=1}^{\infty}  \left(\frac{ - n^2 y^{n+3} + (2n^2+2n-1)y^{n+2} -(n^2+2n+1) y^{n+1} + y^2 + y}{(1-y)^3}\right) \frac{z^n}{n^3} \right\}
\end{equation}
\begin{equation}   \nonumber
    = \exp\left\{ \sum_{n=1}^{\infty}  \left(\frac{ -y(y+1) (y^n -1) + 2 n (y-1) y^{n+1} - n^2 (y-1)^2 y^{n+1} }{(1-y)^3}\right) \frac{z^n}{n^3} \right\}
\end{equation}
  \begin{equation}   \nonumber
    = \left(\frac{1}{1-yz}\right)^{\frac{y}{1-y}} \exp\left\{ \frac{y(y+1)}{(1-y)^3}(Li_3(z) -Li_3(yz)) - \frac{2 y}{(1-y)^2} Li_2(yz) \right\};
\end{equation}

 \begin{equation}   \nonumber
    \prod_{\substack{m,n \geq 1 \\ m \leq n ; \, (m,n)=1}} \left( \frac{1}{1- y^m z^n} \right)^{\frac{m^3 }{n^4}}
    = \exp\left\{ \sum_{n=1}^{\infty}  \left( \frac{n^3 y^{n+4} +(3n^3+6n^2-4) y^{n+2}}{(1-y)^4} \right) \frac{z^n}{n^4} \right\}
\end{equation}
 \begin{equation}   \nonumber
    \times \exp\left\{ \sum_{n=1}^{\infty}  \left( \frac{-(n+1)^3 y^{n+1} +y^3 + 4y^2 + y}{(1-y)^4} \right) \frac{z^n}{n^4} \right\}
\end{equation}
 \begin{equation}   \nonumber
    = \exp\left\{ \sum_{n=1}^{\infty}  \left( \frac{-y (y^2 + 4 y + 1) (y^n - 1) + 3 n (y^2 -1) y^{n+1} - 3 n^2 ((y-1)^2 y^{n+1}) + n^3 (y-1)^3 y^{n+1}}{(1-y)^4} \right) \frac{z^n}{n^4} \right\}
\end{equation}
 \begin{equation}   \nonumber
    = \left( 1-yz \right)^{\frac{y}{1-y}} \exp\left\{  \frac{y (y^2 +4y +1)}{(1-y)^4}(Li_4(z)-Li_4(yz)) + \frac{3y(1+y)}{(1-y)^3} Li_3(yz) + \frac{3y}{(1-y)^2}  Li_2(yz) \right\};
\end{equation}

 \begin{equation}   \nonumber
    \prod_{\substack{m,n \geq 1 \\ m \leq n ; \, (m,n)=1}} \left( \frac{1}{1- y^m z^n} \right)^{\frac{m^4 }{n^5}}
\end{equation}
 \begin{equation}   \nonumber
    = \exp\left\{ \sum_{n=1}^{\infty}  \left( - \frac{n^4y^{n+5}+(-4n^4-4n^3+6n^2-4n+1)y^{n+4} }{(1-y)^5} \right) \frac{z^n}{n^5} \right\}
\end{equation}
 \begin{equation}   \nonumber
    \times \exp\left\{ \sum_{n=1}^{\infty}  \left( - \frac{(6n^4+12n^3-6n^2-12n+11)y^{n+3} }{(1-y)^5} \right) \frac{z^n}{n^5} \right\}
\end{equation}
 \begin{equation}   \nonumber
    \times \exp\left\{ \sum_{n=1}^{\infty}  \left( - \frac{(-4n^4-12n^3-6n^2+12n+11)y^{n+2} + (n+1)^4 y^{n+1} }{(1-y)^5} \right) \frac{z^n}{n^5} \right\}
\end{equation}
 \begin{equation}   \nonumber
    \times \exp\left\{ \sum_{n=1}^{\infty}  \left( + \frac{y^4 + 11y^3 + 11y^2 + z}{(1-y)^5} \right) \frac{z^n}{n^5} \right\}
\end{equation}

\begin{equation}   \nonumber
    = \exp\left\{ \frac{y (y + 1) (y^2 + 10 y + 1)}{(1-y)^5}(Li_5(yz) - Li_5(z)) \right\}
\end{equation}
\begin{equation}   \nonumber
    \times \exp\left\{ \frac{4}{(1-y)^5} ((-y^4 +3y^3 +y)Li_4(yz) + 3Li_4(z)) \right\}
\end{equation}
 \begin{equation}   \nonumber
    \times \exp\left\{ \frac{6}{(1-y)^5} ((y^4-y^3+y) Li_3(yz) - Li_3(z))\right\}
\end{equation}
\begin{equation}   \nonumber
    \times \exp\left\{\frac{4}{(1-y)^5}((y-3y^3-y^4) Li_2(yz) - 3Li_2(z)) \right\}
\end{equation}
\begin{equation}   \nonumber
    \times \exp\left\{ \frac{1}{(1-y)^5}((y^5 - 4 y^4 + 6y^3 +y) Li_1(yz) - 4Li_1(z))\right\}
\end{equation}

 \begin{equation}   \nonumber
    = \frac{\left(1-z\right)^{\frac{4}{(1-y)^5}}}{\left(1-yz\right)^{\frac{y^5 - 4 y^4 + 6y^3 +y}{(1-y)^5}}}
    \end{equation}
\begin{equation}   \nonumber
    \times \exp\left\{ \frac{y (y + 1) (y^2 + 10 y + 1)}{(1-y)^5}(Li_5(yz) - Li_5(z)) \right\}
\end{equation}
\begin{equation}   \nonumber
    \times \exp\left\{ \frac{4}{(1-y)^5} ((-y^4 +3y^3 +y)Li_4(yz) + 3Li_4(z)) \right\}
\end{equation}
 \begin{equation}   \nonumber
    \times \exp\left\{ \frac{6}{(1-y)^5} ((y^4-y^3+y) Li_3(yz) - Li_3(z))\right\}
\end{equation}
\begin{equation}   \nonumber
    \times \exp\left\{\frac{4}{(1-y)^5}((y-3y^3-y^4) Li_2(yz) - 3Li_2(z)) \right\}.
\end{equation}

\bigskip

Based on the above equations and their reciprocal equations, we may assert the following four 2D theorems.

\begin{theorem}    \label{16.57f}
For $|z|<1$,
 \begin{equation}   \nonumber
    \prod_{\substack{m,n \geq 1 \\ m \leq n ; \, (m,n)=1}} \left( \frac{1}{1- y^m z^n} \right)^{\frac{m^1 }{n^2}}
    = \left(1-yz\right)^{\frac{y}{1-y}} \exp\left\{ \frac{y}{(1-y)^2} \left(Li_2(z)- Li_2(yz)\right) \right\},
\end{equation}
 \begin{equation}   \nonumber
    \prod_{\substack{m,n \geq 1 \\ m \leq n ; \, (m,n)=1}} \left(1- y^m z^n\right)^{\frac{m^1 }{n^2}}
    = \left(\frac{1}{1-yz}\right)^{\frac{y}{1-y}} \exp\left\{- \frac{y}{(1-y)^2} \left(Li_2(z)- Li_2(yz)\right) \right\}.
\end{equation}
\end{theorem}

\begin{theorem}   \label{16.57g}
For $|z|<1$,
 \begin{equation}   \nonumber
    \prod_{\substack{m,n \geq 1 \\ m \leq n ; \, (m,n)=1}} \left( \frac{1}{1- y^m z^n} \right)^{\frac{m^2 }{n^3}}
    = \left(\frac{1}{1-yz}\right)^{\frac{y}{1-y}}
    \end{equation}
 \begin{equation}   \nonumber
    \times \exp\left\{ \frac{y(y+1)}{(1-y)^3}(Li_3(z) -Li_3(yz)) - \frac{2 y}{(1-y)^2} Li_2(yz) \right\},
\end{equation}
 \begin{equation}   \nonumber
    \prod_{\substack{m,n \geq 1 \\ m \leq n ; \, (m,n)=1}} \left( 1- y^m z^n\right)^{\frac{m^2 }{n^3}}
    = \left(1-yz\right)^{\frac{y}{1-y}}
    \end{equation}
 \begin{equation}   \nonumber
    \times \exp\left\{- \frac{y(y+1)}{(1-y)^3}(Li_3(z) -Li_3(yz)) + \frac{2 y}{(1-y)^2} Li_2(yz) \right\},
\end{equation}
\end{theorem}

\begin{theorem}   \label{16.57h}
For $|z|<1$,
 \begin{equation}   \nonumber
    \prod_{\substack{m,n \geq 1 \\ m \leq n ; \, (m,n)=1}} \left( \frac{1}{1- y^m z^n} \right)^{\frac{m^3 }{n^4}}
    = \left( 1-yz \right)^{\frac{y}{1-y}}
 \end{equation}
 \begin{equation}   \nonumber
     \times \exp\left\{  \frac{y (y^2 +4y +1)}{(1-y)^4}(Li_4(z)-Li_4(yz)) + \frac{3y(1+y)}{(1-y)^3} Li_3(yz) + \frac{3y}{(1-y)^2}  Li_2(yz) \right\},
\end{equation}
 \begin{equation}   \nonumber
    \prod_{\substack{m,n \geq 1 \\ m \leq n ; \, (m,n)=1}} \left( 1- y^m z^n \right)^{\frac{m^3 }{n^4}}
    = \left(\frac{1}{1-yz} \right)^{\frac{y}{1-y}}
    \end{equation}
 \begin{equation}   \nonumber
 \times \exp\left\{ -\frac{y (y^2 +4y +1)}{(1-y)^4}(Li_4(z)-Li_4(yz)) - \frac{3y(1+y)}{(1-y)^3} Li_3(yz) - \frac{3y}{(1-y)^2}  Li_2(yz) \right\}.
\end{equation}
\end{theorem}

\begin{theorem}   \label{16.57i}
For $|z|<1$,
 \begin{equation}   \nonumber
    \prod_{\substack{m,n \geq 1 \\ m \leq n ; \, (m,n)=1}} \left( \frac{1}{1- y^m z^n} \right)^{\frac{m^4 }{n^5}}
    = \left\{\frac{\left(1-z\right)^{4}}{\left(1-yz\right)^{y^5 - 4 y^4 + 6y^3 +y}}\right\}^{\frac{1}{(1-y)^5}}
    \end{equation}
\begin{equation}   \nonumber
    \times \exp\left\{ \frac{y (y + 1) (y^2 + 10 y + 1)}{(1-y)^5}(Li_5(yz) - Li_5(z)) \right\}
\end{equation}
\begin{equation}   \nonumber
    \times \exp\left\{ \frac{4}{(1-y)^5} ((-y^4 +3y^3 +y)Li_4(yz) + 3Li_4(z)) \right\}
\end{equation}
 \begin{equation}   \nonumber
    \times \exp\left\{ \frac{6}{(1-y)^5} ((y^4-y^3+y) Li_3(yz) - Li_3(z))\right\}
\end{equation}
\begin{equation}   \nonumber
    \times \exp\left\{\frac{4}{(1-y)^5}((y-3y^3-y^4) Li_2(yz) - 3Li_2(z)) \right\},
\end{equation}

\begin{equation}   \nonumber
    \prod_{\substack{m,n \geq 1 \\ m \leq n ; \, (m,n)=1}} \left( 1- y^m z^n \right)^{\frac{m^4 }{n^5}}
    = \left\{\frac{\left(1-yz\right)^{y^5 - 4 y^4 + 6y^3 +y}}{\left(1-z\right)^{4}}\right\}^{\frac{1}{(1-y)^5}}
    \end{equation}
\begin{equation}   \nonumber
    \times \exp\left\{ \frac{-y (y + 1) (y^2 + 10 y + 1)}{(1-y)^5}(Li_5(yz) - Li_5(z)) \right\}
\end{equation}
\begin{equation}   \nonumber
    \times \exp\left\{ \frac{-4}{(1-y)^5} ((-y^4 +3y^3 +y)Li_4(yz) + 3Li_4(z)) \right\}
\end{equation}
 \begin{equation}   \nonumber
    \times \exp\left\{ \frac{-6}{(1-y)^5} ((y^4-y^3+y) Li_3(yz) - Li_3(z))\right\}
\end{equation}
\begin{equation}   \nonumber
    \times \exp\left\{\frac{-4}{(1-y)^5}((y-3y^3-y^4) Li_2(yz) - 3Li_2(z)) \right\}.
\end{equation}
\end{theorem}

\bigskip

\subsection{The $3D$ square hyperpyramid VPV identity.}

Recall from an earlier paper that:

If $|x|, |y|, |z| < 1$, with $a+b+c=1$,
 \begin{equation}   \label{16.58}
    \prod_{\substack{l,m,n \geq 1 \\ l,m \leq n ; \, \gcd(l,m,n)=1}} \left( \frac{1}{1-x^l y^m z^n} \right)^{\frac{1}{l^a m^b n^c}}
    = \exp\left\{ \sum_{n=1}^{\infty}  \left( \sum_{l=1}^{n} \frac{x^l}{l^a} \right) \left( \sum_{m=1}^{n} \frac{y^m}{m^b} \right) \frac{z^n}{n^c} \right\}.
  \end{equation}

In one of the simpler examples we can give, let us apply equations (\ref{16.50}) and (\ref{16.51}) to (\ref{16.58}). This corresponds to the case of (\ref{16.58}) with $x=y=1$, $a=-1$, $b=-2$, implying from condition $a+b+c=1$ that $c=4$. The resulting substitution is

 \begin{equation}   \nonumber
    \prod_{\substack{l,m,n \geq 1 \\ l,m \leq n ; \, \gcd(l,m,n)=1}} \left( \frac{1}{1-z^n} \right)^{\frac{ l m^2 }{n^4}}
    = \exp\left\{ \sum_{n=1}^{\infty}  \left( \frac{n}{2} + \frac{n^2}{2} \right) \left( \frac{n}{6} + \frac{n^2}{2} + \frac{n^3}{3} \right) \frac{z^n}{n^4} \right\}
\end{equation}
\begin{equation} \nonumber
= \exp\left\{ \sum_{n=1}^{\infty} \left( \frac{n^2}{12} + \frac{n^3}{3} + \frac{5 n^4}{12} + \frac{n^5}{6} \right) \frac{z^n}{n^4} \right\}
\end{equation}
\begin{equation} \nonumber
= \exp\left\{ \frac{1}{12}Li_2(z) + \frac{1}{3}Li_1(z) + \frac{5}{12}Li_0(z) + \frac{1}{6}Li_{-1}(z)  \right\}
\end{equation}
\begin{equation} \nonumber
= \exp\left\{ \frac{1}{12}Li_2(z) + \frac{1}{3} \log\left( \frac{1}{1-z} \right) + \frac{5}{12} \frac{z}{1-z} + \frac{1}{6} \frac{z}{(1-z)^2}  \right\}
\end{equation}
\begin{equation} \nonumber
= \exp\left\{ \frac{1}{12}Li_2(z) + \frac{z(7-5z))}{12(1-z)^2} + \frac{1}{3} \log\left( \frac{1}{1-z} \right) \right\}
\end{equation}
\begin{equation} \nonumber
=    \sqrt[3]{\left( \frac{1}{1-z} \right)} \;  \exp\left\{ \frac{1}{12}Li_2(z) + \frac{z(7-5z))}{12(1-z)^2} \right\}.
\end{equation}

Therefore we have proven that if $|z|<1$,

 \begin{equation}   \label{16.59}
    \prod_{\substack{l,m,n \geq 1 \\ l,m \leq n ; \, \gcd(l,m,n)=1}} \left( \frac{1}{1-z^n} \right)^{\frac{ l m^2 }{n^4}}
    =    \sqrt[3]{\left( \frac{1}{1-z} \right)} \;  \exp\left\{ \frac{1}{12} \left( Li_2(z) + \frac{z(7-5z)}{(1-z)^2}\right) \right\},
\end{equation}

and the equivalent result,

 \begin{equation}   \label{16.60}
    \prod_{\substack{l,m,n \geq 1 \\ l,m \leq n ; \, \gcd(l,m,n)=1}} \left( 1-z^n \right)^{\frac{ l m^2 }{n^4}}
    =    \sqrt[3]{\left( 1-z \right)} \;  \exp\left\{ - \frac{1}{12} \left( Li_2(z) + \frac{z(7-5z)}{(1-z)^2}\right) \right\}.
\end{equation}

There is a corresponding example we can give, if we apply equations (\ref{16.54}) and (\ref{16.55}) to (\ref{16.58}). This corresponds to the case of (\ref{16.58}) with $a=-1$, $b=-2$, implying from condition $a+b+c=1$ that $c=4$. The resulting substitution gives us

 \begin{equation}     \label{16.61}
    \prod_{\substack{l,m,n \geq 1 \\ l,m \leq n ; \, \gcd(l,m,n)=1}} \left( \frac{1}{1-x^l y^m z^n} \right)^{\frac{ l m^2 }{n^4}}
\end{equation}
\begin{equation} \nonumber
    = \exp\left\{ \sum_{n=1}^{\infty} A(n,x) B(n,y) \frac{z^n}{n^4} \right\}
\end{equation}
where
\begin{equation} \nonumber
 A(n,x)   = \frac{x - (1 + n) x^{n+1} + n x^{n+2}}{(1-x)^2}
\end{equation}
and
\begin{equation} \nonumber
 B(n,y)   =  \frac{ - n^2 y^{n+3} + (2n^2+2n-1)y^{n+2} -(n^2+2n+1) y^{n+1} + y^2 + y}{(1-y)^3},
\end{equation}

which, after expanding and simplifying is equivalent to

 \begin{equation}     \label{16.62}
    \prod_{\substack{l,m,n \geq 1 \\ l,m \leq n ; \, \gcd(l,m,n)=1}} \left( \frac{1}{1-x^l y^m z^n} \right)^{\frac{ l m^2 }{n^4}}
\end{equation}

\begin{equation} \nonumber
    =  \left(\frac{1}{1-xyz} \right)^{\frac{xy}{(1-x)^2(1-y)^3}}
\end{equation}
\begin{equation} \nonumber
    \times \exp\left\{ \frac{-xy(2x+y-3)}{(1-x)^3(1-y)^4} Li_2(xyz) + \frac{xy}{(1-x)^3(1-y)^3} Li_2(yz) + \frac{xy(y+1)}{(1-x)^3(1-y)^5} Li_4(xyz) \right\}
\end{equation}
\begin{equation} \nonumber
    \times \exp\left\{ \frac{-xy(xy+x+y-3)}{(1-x)^3(1-y)^5} Li_3(xyz) - \frac{xy(y+1)}{(1-x)^2(1-y)^5}Li_3(xz) - \frac{xy(y+1)}{(1-x)^3(1-y)^5} Li_4(x z) \right\}
\end{equation}
\begin{equation} \nonumber
    \times \exp\left\{ \frac{-2xy}{(1-x)^3(1-y)^4}Li_3(yz) - \frac{xy(1+y)}{(1-x)^3(1-y)^5}Li_4(yz) + \frac{xy(1+y) }{(1-x)^3(1-y)^5}Li_4(z) \right\}.
\end{equation}

The necessary workings and simplification in going from (\ref{16.61}) to arrive at (\ref{16.62}) are given in the following page and a half of analysis.

 \begin{equation}   \nonumber
    \prod_{\substack{l,m,n \geq 1 \\ l,m \leq n ; \, \gcd(l,m,n)=1}} \left( \frac{1}{1-x^l y^m z^n} \right)^{\frac{ l m^2 }{n^4}}
\end{equation}

\begin{equation} \nonumber
    = \exp\left\{ \frac{1}{(1-x)^3(1-y)^5} \sum_{n=1}^{\infty} (n^3x^{n+1}y^{n+1}-2n^3x^{n+1}y^{n+2}+n^3x^{n+1}y^{n+3}) \frac{z^n}{n^4} \right\}
\end{equation}
\begin{equation} \nonumber
    \times \exp\left\{ \frac{1}{(1-x)^3(1-y)^5} \sum_{n=1}^{\infty} (-n^3x^{n+2}y^{n+1}+2n^3x^{n+2}y^{n+2}-n^3x^{n+2}y^{n+3}) \frac{z^n}{n^4} \right\}
\end{equation}
\begin{equation} \nonumber
    \times \exp\left\{ \frac{1}{(1-x)^3(1-y)^5} \sum_{n=1}^{\infty} (+3n^2x^{n+1}y^{n+1}-4n^2x^{n+1}y^{n+2}+n^2x^{n+1}y^{n+3}) \frac{z^n}{n^4} \right\}
\end{equation}
\begin{equation} \nonumber
    \times \exp\left\{ \frac{1}{(1-x)^3(1-y)^5} \sum_{n=1}^{\infty} (-2n^2x^{n+2}y^{n+1}+2n^2x^{n+2}y^{n+2}-n^2xy^{n+1}) \frac{z^n}{n^4} \right\}
\end{equation}
\begin{equation} \nonumber
    \times \exp\left\{ \frac{1}{(1-x)^3(1-y)^5} \sum_{n=1}^{\infty} (+2n^2xy^{n+2}-n^2xy^{n+3}+3nx^{n+1}y^{n+1}) \frac{z^n}{n^4} \right\}
\end{equation}
\begin{equation} \nonumber
    \times \exp\left\{ \frac{1}{(1-x)^3(1-y)^5} \sum_{n=1}^{\infty} (+x^{n+1}y^{n+1}-nx^{n+1}y^{n+2}+x^{n+1}y^{n+2}) \frac{z^n}{n^4} \right\}
\end{equation}
\begin{equation} \nonumber
    \times \exp\left\{ \frac{1}{(1-x)^3(1-y)^5} \sum_{n=1}^{\infty} (-nx^{n+2}y^{n+1}-nx^{n+2}y^{n+2}-ny^2x^{n+1}-y^2x^{n+1}) \frac{z^n}{n^4} \right\}
\end{equation}
\begin{equation} \nonumber
    \times \exp\left\{ \frac{1}{(1-x)^3(1-y)^5} \sum_{n=1}^{\infty} (+ny^2x^{n+2}-nyx^{n+1}-yx^{n+1}+nyx^{n+2}-2nxy^{n+1}) \frac{z^n}{n^4} \right\}
\end{equation}
\begin{equation} \nonumber
    \times \exp\left\{ \frac{1}{(1-x)^3(1-y)^5} \sum_{n=1}^{\infty} (-xy^{n+1}+2nxy^{n+2}-xy^{n+2}+xy^2+xy) \frac{z^n}{n^4} \right\}
\end{equation}

which works out through the following analysis to

 \begin{equation}   \nonumber
    \prod_{\substack{l,m,n \geq 1 \\ l,m \leq n ; \, \gcd(l,m,n)=1}} \left( \frac{1}{1-x^l y^m z^n} \right)^{\frac{ l m^2 }{n^4}}
\end{equation}

\begin{equation} \nonumber
    = \exp\left\{ \frac{1}{(1-x)^3(1-y)^5}  (xyLi_1(xyz)-2xy^2Li_1(xyz)+xy^3Li_1(xyz))  \right\}
\end{equation}
\begin{equation} \nonumber
    \times \exp\left\{ \frac{1}{(1-x)^3(1-y)^5} (-x^2yLi_1(xyz)+2x^2y^2Li_1(xyz)-x^2y^3Li_1(xyz)) \right\}
\end{equation}
\begin{equation} \nonumber
    \times \exp\left\{ \frac{1}{(1-x)^3(1-y)^5} (+3xyLi_2(xyz)-4xy^2Li_2(xyz)+xy^3Li_2(xyz)) \right\}
\end{equation}
\begin{equation} \nonumber
    \times \exp\left\{ \frac{1}{(1-x)^3(1-y)^5} (-2x^2yLi_2(xyz)+2x^2y^2Li_2(xyz)-xyLi_2(yz)) \right\}
\end{equation}
\begin{equation} \nonumber
    \times \exp\left\{ \frac{1}{(1-x)^3(1-y)^5} (+2xy^2Li_2(yz)-xy^3Li_2(yz)+3xyLi_3(xyz)) \right\}
\end{equation}
\begin{equation} \nonumber
    \times \exp\left\{ \frac{1}{(1-x)^3(1-y)^5} (+xyLi_4(xyz)-xy^2Li_3(xyz)+xy^2Li_4(xyz)) \right\}
\end{equation}
\begin{equation} \nonumber
    \times \exp\left\{ \frac{1}{(1-x)^3(1-y)^5} (-x^2yLi_3(xyz)-x^2y^2Li_3(xyz)-xy^2Li_3(xz)-xy^2Li_4(xz)) \right\}
\end{equation}
\begin{equation} \nonumber
    \times \exp\left\{ \frac{1}{(1-x)^3(1-y)^5} (+x^2y^2Li_3(xz)-xyLi_3(xz)-xyLi_4(xz)+x^2yLi_3(xz)-2xyLi_3(yz)) \right\}
\end{equation}
\begin{equation} \nonumber
    \times \exp\left\{ \frac{1}{(1-x)^3(1-y)^5} (-xyLi_4(yz)+2xy^2Li_3(yz)-xy^2Li_4(yz)+xy^2Li_4(z)+xyLi_4(z)) \right\}
\end{equation}

\begin{equation} \nonumber
    = \exp\left\{ \frac{1}{(1-x)^3(1-y)^5}  ((x-1)x(y - 1)^2y \log(1-xyz))  \right\}
\end{equation}
\begin{equation} \nonumber
    \times \exp\left\{ \frac{1}{(1-x)^3(1-y)^5} (x(y-1)y(2x+y-3) Li_2(xyz)) \right\}
\end{equation}
\begin{equation} \nonumber
    \times \exp\left\{ \frac{1}{(1-x)^3(1-y)^5} (-x(y-1)^2 y Li_2(yz)) \right\}
\end{equation}
\begin{equation} \nonumber
    \times \exp\left\{ \frac{1}{(1-x)^3(1-y)^5} (xy(y + 1) Li_4(xyz)) \right\}
\end{equation}
\begin{equation} \nonumber
    \times \exp\left\{ \frac{1}{(1-x)^3(1-y)^5} (-xy(xy+x+y-3) Li_3(xyz)) \right\}
\end{equation}
\begin{equation} \nonumber
    \times \exp\left\{ \frac{1}{(1-x)^3(1-y)^5} ((x-1)xy(y+1)Li_3(xz)-xy(y+1)Li_4(x z)) \right\}
\end{equation}
\begin{equation} \nonumber
    \times \exp\left\{ \frac{1}{(1-x)^3(1-y)^5} (+2x(y-1)yLi_3(y z)-xy(1+y)Li_4(yz)+xy(1+y)Li_4(z)) \right\}
\end{equation}

\begin{equation} \nonumber
    =  \left(\frac{1}{1-xyz} \right)^{\frac{xy}{(1-x)^2(1-y)^3}}
\end{equation}
\begin{equation} \nonumber
    \times \exp\left\{ \frac{-xy(2x+y-3)}{(1-x)^3(1-y)^4} Li_2(xyz) + \frac{xy}{(1-x)^3(1-y)^3} Li_2(yz) + \frac{xy(y+1)}{(1-x)^3(1-y)^5} Li_4(xyz) \right\}
\end{equation}
\begin{equation} \nonumber
    \times \exp\left\{ \frac{-xy(xy+x+y-3)}{(1-x)^3(1-y)^5} Li_3(xyz) - \frac{xy(y+1)}{(1-x)^2(1-y)^5}Li_3(xz) - \frac{xy(y+1)}{(1-x)^3(1-y)^5} Li_4(x z) \right\}
\end{equation}
\begin{equation} \nonumber
    \times \exp\left\{ \frac{-2xy}{(1-x)^3(1-y)^4}Li_3(yz) - \frac{xy(1+y)}{(1-x)^3(1-y)^5}Li_4(yz) + \frac{xy(1+y) }{(1-x)^3(1-y)^5}Li_4(z) \right\}.
\end{equation}

The casual reader will appreciate that the simple summations arising from substituted equations (\ref{16.50}) through to (\ref{16.53}) yield much more immediate and simple results than the more elaborate and impressive substituted equations (\ref{16.54}) through to (\ref{16.57}).

\bigskip

\section{Elementary ideas for binary Partitions} \label{S:binary identities}

A binary partition is an integer partition where each number (part) in the sum is of the form $2^n$ for non-negative integers $n$. The unrestricted binary integer partitions function $b(n)$ was first studied by Euler in the 18th century. We extend the unrestricted and the distinct integer binary partition function to partitions of 2D and 3D vectors into parts that are of forms $\langle 2^a,2^b \rangle$ and $\langle 2^c,2^d,2^e \rangle$ respectively for non-negative integers $a,b,c,d$. Unrestricted partitions of this type are $b(i,j)$ is the number of solutions of $\langle i,j \rangle = \sum \langle 2^a,2^b \rangle$ for non-increasing, non-negative integers $a,b$ in 2D; and $b(i,j,k)$ is the number of solutions of $\langle i,j,k \rangle = \sum \langle 2^c,2^d,2^e \rangle$ for non-increasing, non-negative integers $c,d,e$ in 3D.

The generating function for unrestricted binary partitions is, for $|q|<1$,
\begin{eqnarray}
\label{7.23.1}  \sum_{n \geq 0} b(n) q^n &=& \sum_{n \geq 0} p(n | \textmd{ integer parts }2^j) q^n = \prod_{n=0}^{\infty} \frac{1}{1-q^{2^n}} \\
\nonumber    &=& 1 + q + 2 q^2 + 2 q^3 + 4 q^4 + 4 q^5 + 6 q^6 + 6 q^7 + 10 q^8 \\
\nonumber    & & + 10 q^9 + 14 q^{10} + 14 q^{11} + 20 q^{12} + 20 q^{13} + 26 q^{14} + 26 q^{15}  \\
\nonumber    & & + 36 q^{16} + 36 q^{17} + 46 q^{18} + 46 q^{19} + 60 q^{20} + O(q^{21}).
\end{eqnarray}
Then, for example the four binary partitions of $5$ are
\begin{center}
 $4+1, \quad 2+2+1, \quad 2+1+1+1, \quad 1+1+1+1+1,$
\end{center}
so $p(5)= 4$.

\section{A few finite and infinite products for binary partitions} \label{S:binary finit}

Consider the following finite products:-
\begin{eqnarray*}
\nonumber (1+x) &=& \frac{1 -  x^{2}}{1 - x}, \\
\nonumber (1 + x) (1 + x^2) &=& \frac{1 -  x^{4}}{1 - x}, \\
\nonumber (1 + x) (1 + x^2) (1 + x^4) &=& \frac{1 -  x^{8}}{1 - x}, \\
\nonumber (1 + x) (1 + x^2) (1 + x^4) (1 + x^8) &=& \frac{1 -  x^{16}}{1 - x},
\end{eqnarray*}

and generally

\begin{equation}  \nonumber
(1 + x) (1 + x^2) (1 + x^4) \cdots  (1 + x^{2^n}) = \frac{1 -  x^{2^{(n+1)}}}{1 - x}.
\end{equation}

It is easy to see how this leads to the well-known results:-

\begin{equation}  \label{11.06a}
(1 + x) (1 + x^2) (1 + x^4) \cdots  (1 + x^{2^n})\cdots = \frac{1}{1 - x} = \sum_{n \geq 0} x^n,
\end{equation}

for $|x|<1$; and

\begin{equation}  \nonumber
\left(\frac{1 + x^{1/2}}{2}\right) \left(\frac{1 + x^{1/4}}{2}\right) \left(\frac{1 + x^{1/8}}{2}\right) \cdots  \left(\frac{1 + x^{2^{-n}}}{2}\right) \cdots = \frac{1 - x}{\log( \frac{1}{x})},
\end{equation}

for $|x| \neq 1$ or $0$.

We note that (\ref{11.06a}) encodes that every positive integer is a unique partition into distinct non-negative powers of 2.

\bigskip

\section{2D version of every integer is a unique sum of distinct binary powers}

We are fortunate to have the identities from the previous section. They permit us easy access to some very simple 2D binary partition results. Let us start by considering the following identities, each easy to prove from the finite products above.

\begin{equation} \label{11b01}
   \begin{array}{cc}
     (1+xy) & \times(1+x^2y) \\
     \times(1+xy^2) & \times(1+x^2y^2)
   \end{array}
   =
   \begin{array}{cc}
     \frac{1-x^3y^3}{(1-xy)} & \times\frac{1-x^4y^2}{(1-x^2y)} \\
     \times\frac{(1-xy^2)}{1-x^2y^4} &
   \end{array}
\end{equation}
\begin{equation} \nonumber
=  x^6 y^6 + x^5 y^5 + x^5 y^4 + x^4 y^5 + x^4 y^4 + x^4 y^3 + x^3 y^4 + 2 x^3 y^3 + x^3 y^2 + x^2 y^3 + x^2 y^2 + x^2 y + x y^2 + x y + 1
\end{equation}
\begin{equation} \nonumber
=   1 +
\begin{array}{cccccc}
   &  &  &  &  & +x^6 y^6 \\
   &  &  & + x^4 y^5 & + x^5 y^5 &  \\
   &  & + x^3 y^4 & + x^4 y^4 & + x^5 y^4 &  \\
   & + x^2 y^3 & + 2 x^3 y^3 & + x^4 y^3 &  &  \\
  + x y^2 & + x^2 y^2 & + x^3 y^2 &  &  &  \\
  + x y & + x^2 y &  &  &  &
\end{array}
,
\end{equation}
where in this last equation we've taken poetic licence to mimick the partition grid, which would look like this
\begin{equation} \nonumber
\begin{array}{c|ccccccc}
 \textbf{6}  &    &    &   &   &   &   & 1 \\
 \textbf{5}  &    &    &   &   & 1 & 1 &   \\
 \textbf{4}  &    &    &   & 1 & 1 & 1 &   \\
 \textbf{3}  &    &    & 1 & 2 & 1 &   &   \\
 \textbf{2}  &    &  1 & 1 & 1 &   &   &   \\
 \textbf{1}  &    &  1 & 1 &   &   &   &   \\
 \textbf{0}  &  1 &    &   &   &   &   &   \\ \hline
   & \textbf{0}  &  \textbf{1} & \textbf{2} & \textbf{3} & \textbf{4} & \textbf{5} & \textbf{6}
\end{array}
\end{equation}

In a similar fashion we see easily that

\begin{equation} \label{11b02}
  \begin{array}{ccc}
     (1+xy) & (1+x^2y) & (1+x^4y) \\
     (1+xy^2) & (1+x^2y^2) & (1+x^4y^2) \\
     (1+xy^4) & (1+x^2y^4) & (1+x^4y^4)
   \end{array}
   =
   \begin{array}{ccc}
     \frac{(1-x^8y^8)}{(1-xy)} & \frac{(1-x^8y^4)}{(1-x^2y)} & \frac{(1-x^8y^2)}{(1-x^4y)} \\
     \frac{(1-x^4y^8)}{(1-xy^2)} &   &   \\
     \frac{(1-x^2y^8)}{(1-xy^4)} &   &
   \end{array}
\end{equation}
\begin{equation} \nonumber
 = x^{21} y^{21} + x^{20} y^{20} + x^{19} y^{20} + x^{17} y^{20} + x^{20} y^{19} + x^{19} y^{19}+ x^{18} y^{19} + x^{17} y^{19}+ x^{16} y^{19}
  \end{equation}
\begin{equation} \nonumber
+ x^{15} y^{19} + x^{19} y^{18} + 2 x^{18} y^{18}+ x^{17} y^{18} + 2 x^{16} y^{18} + 2 x^{15} y^{18} + x^{14} y^{18}+ x^{13} y^{18} + x^{20} y^{17}
\end{equation}
\begin{equation} \nonumber
+ x^{19} y^{17} + x^{18} y^{17} + 2 x^{17} y^{17} + 2 x^{16} y^{17} + 2 x^{15} y^{17} + 3 x^{14} y^{17}+ x^{13} y^{17} + x^{12} y^{17} + x^{11} y^{17}
\end{equation}
\begin{equation} \nonumber
+ x^{19} y^{16} + 2 x^{18} y^{16} + 2 x^{17} y^{16}+ 3 x^{16} y^{16} + 3 x^{15} y^{16}+ 3 x^{14} y^{16} + 3 x^{13} y^{16} + 2 x^{12} y^{16} + x^{11} y^{16}
\end{equation}
\begin{equation} \nonumber
+ x^{10} y^{16} + x^{19} y^{15} + 2 x^{18} y^{15} + 2 x^{17} y^{15}+ 3 x^{16} y^{15} + 4 x^{15} y^{15}+ 4 x^{14} y^{15} + 4 x^{13} y^{15} + 3 x^{12} y^{15}
\end{equation}
\begin{equation} \nonumber
+ 2 x^{11} y^{15} + 2 x^{10} y^{15}+ x^9 y^{15}+ x^{18} y^{14} + 3 x^{17} y^{14} + 3 x^{16} y^{14} + 4 x^{15} y^{14} y^{14}+ 6 x^{14} y^{14}+ 5 x^{13}
\end{equation}
\begin{equation} \nonumber
+ 5 x^{12} y^{14} + 4 x^{11} y^{14}+ 2 x^{10} y^{14} + 2 x^9 y^{14} + x^8 y^{14}+ x^{18} y^{13} + x^{17} y^{13} + 3 x^{16} y^{13} + 4 x^{15} y^{13}
\end{equation}
\begin{equation} \nonumber
+ 5 x^{14} y^{13}+ 6 x^{13} y^{13}+ 6 x^{12} y^{13} + 5 x^{11} y^{13} + 5 x^{10} y^{13} + 3 x^9 y^{13} + 2 x^8 y^{13} + x^7 y^{13}+ x^17 y^{12}
\end{equation}
\begin{equation} \nonumber
+ 2 x^{16} y^{12} + 3 x^{15} y^{12} + 5 x^{14} y^{12} + 6 x^{13} y^{12}+ 7 x^{12} y^{12} + 6 x^{11} y^{12} + 6 x^{10} y^{12} + 4 x^9 y^{12}
\end{equation}
\begin{equation} \nonumber
+ 3 x^8 y^{12}+ 2 x^7 y^{12} + x^6 y^{12} + x^{17} y^{11} + x^{16} y^{11} + 2 x^{15} y^{11} + 4 x^{14} y^{11}+ 5 x^{13} y^{11} + 6 x^{12} y^{11}
\end{equation}
\begin{equation} \nonumber
+ 7 x^{11} y^{11} + 6 x^{10} y^{11} + 6 x^9 y^{11}+ 5 x^8 y^{11} + 2 x^7 y^{11} + 2 x^6 y^{11} + x^5 y^{11} + x^{16} y^{10}+ 2 x^{15} y^{10}
\end{equation}
\begin{equation} \nonumber
+ 2 x^{14} y^{10} + 5 x^{13} y^{10} + 6 x^{12} y^{10} + 6 x^{11} y^{10} + 7 x^{10} y^{10}+ 6 x^9 y^{10}+ 5 x^8 y^{10} + 4 x^7 y^{10} + 2 x^6 y^{10}
\end{equation}
\begin{equation} \nonumber
+ x^5 y^{10} + x^4 y^{10}+ x^{15} y^9 + 2 x^{14} y^9 + 3 x^{13} y^9+ 4 x^{12} y^9 + 6 x^{11} y^9 + 6 x^{10} y^9+ 7 x^9 y^9
\end{equation}
\begin{equation} \nonumber
+ 6 x^8 y^9 + 5 x^7 y^9 + 3 x^6 y^9 + 2 x^5 y^9 + x^4 y^9+ x^{14} y^8 + 2 x^{13} y^8 + 3 x^{12} y^8 + 5 x^{11} y^8
\end{equation}
\begin{equation} \nonumber
+ 5 x^{10} y^8 + 6 x^9 y^8+ 6 x^8 y^8 + 5 x^7 y^8 + 4 x^6 y^8 + 3 x^5 y^8 + x^4 y^8 + x^3 y^8+ x^13 y^7 + 2 x^{12} y^7
\end{equation}
\begin{equation} \nonumber
+ 2 x^{11} y^7 + 4 x^{10} y^7 + 5 x^9 y^7 + 5 x^8 y^7+ 6 x^7 y^7 + 4 x^6 y^7 + 3 x^5 y^7 + 3 x^4 y^7 + x^3 y^7
\end{equation}
\begin{equation} \nonumber
+ x^{12} y^6+ 2 x^{11} y^6 + 2 x^{10} y^6 + 3 x^9 y^6 + 4 x^8 y^6 + 4 x^7 y^6 + 4 x^6 y^6+ 3 x^5 y^6 + 2 x^4 y^6
\end{equation}
\begin{equation} \nonumber
+ 2 x^3 y^6 + x^2 y^6 + x^{11} y^5 + x^{10} y^5+ 2 x^9 y^5 + 3 x^8 y^5 + 3 x^7 y^5 + 3 x^6 y^5 + 3 x^5 y^5
\end{equation}
\begin{equation} \nonumber
+ 2 x^4 y^5+ 2 x^3 y^5 + x^2 y^5 + x^{10} y^4 + x^9 y^4 + x^8 y^4 + 3 x^7 y^4 + 2 x^6 y^4+ 2 x^5 y^4+ 2 x^4 y^4
\end{equation}
\begin{equation} \nonumber
+ x^3 y^4 + x^2 y^4 + x y^4 + x^8 y^3 + x^7 y^3+ 2 x^6 y^3 + 2 x^5 y^3 + x^4 y^3+ 2 x^3 y^3 + x^2 y^3
\end{equation}
\begin{equation} \nonumber
+ x^6 y^2+ x^5 y^2 + x^4 y^2 + x^3 y^2 + x^2 y^2 + x y^2 + x^4 y+ x^2 y + x y + 1
\end{equation}

It is clear that this identity can be extended to any finite case which can then be proved by induction. The infinite case is evident also. We see the coefficients in the above expansion are represented by the grid below.

\begin{equation} \nonumber
\tiny{
  \begin{array}{c|cccccccccccccccccccccc}
    \textbf{21} &   &   &   &   &   &   &   &   &   &   &   &   &   &   &   &   &   &   &   &   &   & \mathbf{1} \\
    \textbf{20} &   &   &   &   &   &   &   &   &   &   &   &   &   &   &   &   &   & 1 &   & 1 & \mathbf{1} &   \\
    \textbf{19} &   &   &   &   &   &   &   &   &   &   &   &   &   &   &   & 1 & 1 & 1 & 1 & \mathbf{1} & 1 &   \\
    \textbf{18} &   &   &   &   &   &   &   &   &   &   &   &   &   & 1 & 1 & 2 & 2 & 1 & \mathbf{2} & 1 &   &   \\
    \textbf{17} &   &   &   &   &   &   &   &   &   &   &   & 1 & 1 & 1 & 3 & 2 & 2 & \mathbf{2} & 1 & 1 & 1 &   \\
    \textbf{16} &   &   &   &   &   &   &   &   &   &   & 1 & 1 & 2 & 3 & 3 & 3 & \mathbf{3} & 2 & 2 & 1 &   &   \\
    \textbf{15} &   &   &   &   &   &   &   &   &   & 1 & 2 & 3 & 3 & 4 & 4 & \mathbf{4} & 3 & 2 & 2 & 1 &   &   \\
    \textbf{14} &   &   &   &   &   &   &   &   & 1 & 2 & 2 & 4 & 5 & 5 & \mathbf{6} & 4 & 3 & 3 & 1 &   &   &   \\
    \textbf{13} &   &   &   &   &   &   & 1 & 1 & 2 & 3 & 5 & 5 & 6 & \mathbf{6} & 5 & 4 & 3 & 1 & 1 &   &   &   \\
    \textbf{12} &   &   &   &   &   &   & 1 & 2 & 3 & 4 & 6 & 6 & \mathbf{7} & 6 & 5 & 3 & 2 & 1 &   &   &   &   \\
    \textbf{11} &   &   &   &   &   & 1 & 2 & 2 & 5 & 6 & 6 & \mathbf{7} & 6 & 5 & 4 & 2 & 1 & 1 &   &   &   &   \\
    \textbf{10} &   &   &   &   & 1 & 1 & 2 & 4 & 5 & 6 & \mathbf{7} & 6 & 6 & 5 & 2 & 2 & 1 &   &   &   &   &   \\
     \textbf{9} &   &   &   &   & 1 & 2 & 3 & 5 & 6 & \mathbf{7} & 6 & 6 & 4 & 3 & 2 & 1 &   &   &   &   &   &   \\
     \textbf{8} &   &   &   & 1 & 1 & 3 & 4 & 5 & \mathbf{6} & 6 & 5 & 5 & 3 & 2 & 1 &   &   &   &   &   &   &   \\
     \textbf{7} &   &   &   & 1 & 3 & 3 & 4 & \mathbf{6} & 5 & 5 & 4 & 2 & 2 & 1 &   &   &   &   &   &   &   &   \\
     \textbf{6} &   &   & 1 & 2 & 2 & 3 & \mathbf{4} & 4 & 4 & 3 & 2 & 2 & 1 &   &   &   &   &   &   &   &   &   \\
     \textbf{5} &   &   & 1 & 2 & 2 & \mathbf{3} & 3 & 3 & 3 & 2 & 1 & 1 &   &   &   &   &   &   &   &   &   &   \\
     \textbf{4} &   & 1 & 1 & 1 & \mathbf{2} & 2 & 2 & 3 & 1 & 1 & 1 &   &   &   &   &   &   &   &   &   &   &   \\
     \textbf{3} &   &   & 1 & \mathbf{2} & 1 & 2 & 2 & 1 & 1 &   &   &   &   &   &   &   &   &   &   &   &   &   \\
     \textbf{2} &   & 1 & \mathbf{1} & 1 & 1 & 1 & 1 &   &   &   &   &   &   &   &   &   &   &   &   &   &   &   \\
     \textbf{1} &   & \mathbf{1} & 1 &   & 1 &   &   &   &   &   &   &   &   &   &   &   &   &   &   &   &   &   \\
     \textbf{0} & \mathbf{1} &   &   &   &   &   &   &   &   &   &   &   &   &   &   &   &   &   &   &   &   &   \\  \hline
      \mathbf{x/y} & \textbf{0} & \textbf{1} & \textbf{2} & \textbf{3} & \textbf{4} & \textbf{5} & \textbf{6} & \textbf{7} & \textbf{8} & \textbf{9} & \textbf{10} & \textbf{11} & \textbf{12} & \textbf{13} & \textbf{14} & \textbf{15} & \textbf{16} & \textbf{17} & \textbf{18} & \textbf{19} & \textbf{20} & \textbf{21}
  \end{array} }
  \end{equation}

Note the generating function for the diagonal where $y=x$ in \textbf{bold} is (if $xy=z$)
\begin{equation*}
  z^{21} + z^{20} + z^{19} + 2 z^{18} + 2 z^{17}+ 3 z^{16} + 4 z^{15} + 6 z^{14}+ 6 z^{13}+ 7 z^{12}+ 7 z^{11} + 7 z^{10}
 \end{equation*}
\begin{equation*}
+ 7 z^9 + 6 z^8 + 6 z^7+ 4 z^6 + 3 z^5 + 2 z^4 + 2 z^3 + z^2 + z + 1
 \end{equation*}
\begin{equation*}
= (z + 1) (z^2 + 1) (z^4 + 1) (z^4 + z^3 + z^2 + z + 1) (z^{10} - z^9 + z^7 - z^6 + 2 z^5 - z^4 + z^3 - z + 1).
 \end{equation*}

\bigskip

\section{The 2D binary, n-ary and 10-ary formulas}

A 2D generalization of our earlier formula for distinct binary partitions has the cases:

\begin{equation} \label{11b03}
  (p+q)= \frac{p^{2}-q^{2}}{p-q},
\end{equation}
\begin{equation} \label{11b04}
  (p+q)(p^2+q^2) = \frac{p^4-q^4}{p-q},
\end{equation}
\begin{equation} \label{11b05}
  (p+q)(p^2+q^2)(p^4+q^4) = \frac{p^8-q^8}{p-q},
\end{equation}
\begin{equation} \nonumber
  \vdots
\end{equation}
\begin{equation} \label{11b06}
  (p+q)(p^2+q^2)(p^4+q^4)\cdots(p^{2^{n-1}}+q^{2^{n-1}}) = \frac{p^{{2^{n}}}-q^{{2^{n}}}}{p-q}.
\end{equation}

Hence, adding together both sides of (\ref{11b03}) to (\ref{11b06}) gives the identity,

\begin{equation} \label{11b07}
1 + (p+q) + (p+q)(p^2+q^2) + (p+q)(p^2+q^2)(p^4+q^4) + \ldots + (p+q)(p^2+q^2)(p^4+q^4)\cdots(p^{2^{n-1}}+q^{2^{n-1}})
\end{equation}
\begin{equation} \nonumber
  = \frac{(p+p^2+p^4+ \ldots +p^{{2^{n}}}) -(q+q^2+q^4 + \ldots + q^{{2^{n}}})}{p-q}.
\end{equation}

This is an interesting fundamental result in distinct 2D binary partitions. Before going further on this we repeat these equations, this time including the expansions.

\begin{equation} \label{11b08}
  (p+q)= \frac{p^{2}-q^{2}}{p-q}
\end{equation}
\begin{equation} \label{11b09}
  (p+q)(p^2+q^2) = \frac{p^4-q^4}{p-q},
\end{equation}
\begin{equation} \nonumber
  =p^3 + p^2 q + p q^2 + q^3,
\end{equation}
\begin{equation} \label{11b10}
  (p+q)(p^2+q^2)(p^4+q^4) = \frac{p^8-q^8}{p-q},
\end{equation}
\begin{equation} \nonumber
  =p^7 + p^6 q + p^5 q^2 + p^4 q^3 + p^3 q^4 + p^2 q^5 + p q^6 + q^7,
\end{equation}
\begin{equation} \label{11b11}
  (p+q)(p^2+q^2)(p^4+q^4)(p^8+q^8) = \frac{p^{16}-q^{16}}{p-q}
\end{equation}
\begin{equation} \nonumber
  =p^{15} + p^{14} q + p^{13} q^2 + p^{12} q^3 + p^{11} q^4 + p^{10} q^5 + p^9 q^6 + p^8 q^7
\end{equation}
\begin{equation} \nonumber
  + p^7 q^8 + p^6 q^9 + p^5 q^{10} + p^4 q^{11} + p^3 q^{12} + p^2 q^{13} + p q^{14} + q^{15},
\end{equation}
\begin{equation} \nonumber
  \vdots
\end{equation}
\begin{equation} \label{11b12}
  (p+q)(p^2+q^2)(p^4+q^4)\cdots(p^{2^{n-1}}+q^{2^{n-1}}) = \frac{p^{{2^{n}}}-q^{{2^{n}}}}{p-q}
\end{equation}
\begin{equation} \nonumber
  =p^{2^n-1} + p^{2^n-2} q + p^{2^n-3} q^2 + p^{2^n-4} q^3 + \ldots + p^3 q^{2^n-4} + p^2 q^{2^n-3} + p q^{2^n-2} + q^{2^n-1},
\end{equation}
where each term $p^a q^b$ has the restriction that $a$ and $b$ are positive integers and $a+b=2^{m}-1$, for every $m \in N^+$.

\bigskip

The above identities combine to give us a theorem for 2D binary component partitions into distinct parts. Therefore we have the

\begin{theorem}
  The number $B_2(a,b)$ for $a$ and $b$ both non negative integers, of vector partitions with parts of the form $\langle 2^\alpha, 2^\beta \rangle$ with $\alpha \geq 0$ and $\beta \geq 0$, into distinct parts is
  \begin{equation} \nonumber
    B_2(a,b)= \left\{
       \begin{array}{ll}
         1, & \hbox{if condition $B_{2}$ holds,} \\
         0, & \hbox{otherwise.}
       \end{array}
     \right.
  \end{equation}
where condition $B_{2}$ is : \textbf{$a$ and $b$ are positive integers and $a+b=2^{m}-1$, for every $m \in N^+$.} The generating function identity encoding this is

\begin{equation} \label{11b14}
1 + \sum_{k=0}^{\infty} \prod_{j=0}^{k}(p^{2^{k}}+q^{2^{k}}) = 1 + \sum_{a,b \geq0} B_2(a,b) p^a q^b
\end{equation}
 \begin{equation} \nonumber
= \frac{(p+p^2+p^4+ \ldots +p^{{2^{m}}}+ \ldots) -(q+q^2+q^4 + \ldots + q^{{2^{m}}}+ \ldots)}{p-q}.
\end{equation}
 \begin{equation} \nonumber
= 1 + (p+q) + (p+q)(p^2+q^2) + \ldots + (p+q)(p^2+q^2)(p^4+q^4)\cdots(p^{2^{m}}+q^{2^{m}})+ \ldots
\end{equation}
\end{theorem}

\bigskip

Part of the partition grid for the generating function form
\begin{equation} \nonumber
 \frac{(p+p^2+p^4+ \ldots +p^{{2^{m}}}+ \ldots) -(q+q^2+q^4 + \ldots + q^{{2^{m}}}+ \ldots)}{p-q},
\end{equation}
is as shown below.

\begin{equation} \nonumber
\tiny{
  \begin{array}{c|cccccccccccccccccccccc}
    \textbf{21} &   &   &   &   &   &   &   &   &   &   & 1 &   &   &   &   &   &   &   &   &   &   &   \\
    \textbf{20} &   &   &   &   &   &   &   &   &   &   &   & 1 &   &   &   &   &   &   &   &   &   &   \\
    \textbf{19} &   &   &   &   &   &   &   &   &   &   &   &   & 1 &   &   &   &   &   &   &   &   &   \\
    \textbf{18} &   &   &   &   &   &   &   &   &   &   &   &   &   & 1 &   &   &   &   &   &   &   &   \\
    \textbf{17} &   &   &   &   &   &   &   &   &   &   &   &   &   &   & 1 &   &   &   &   &   &   &   \\
    \textbf{16} &   &   &   &   &   &   &   &   &   &   &   &   &   &   &   & 1 &   &   &   &   &   &   \\
    \textbf{15} & 1 &   &   &   &   &   &   &   &   &   &   &   &   &   &   &   & 1 &   &   &   &   &   \\
    \textbf{14} &   & 1 &   &   &   &   &   &   &   &   &   &   &   &   &   &   &   & 1 &   &   &   &   \\
    \textbf{13} &   &   & 1 &   &   &   &   &   &   &   &   &   &   &   &   &   &   &   & 1 &   &   &   \\
    \textbf{12} &   &   &   & 1 &   &   &   &   &   &   &   &   &   &   &   &   &   &   &   & 1 &   &   \\
    \textbf{11} &   &   &   &   & 1 &   &   &   &   &   &   &   &   &   &   &   &   &   &   &   & 1 &   \\
    \textbf{10} &   &   &   &   &   & 1 &   &   &   &   &   &   &   &   &   &   &   &   &   &   &   & 1 \\
     \textbf{9} &   &   &   &   &   &   & 1 &   &   &   &   &   &   &   &   &   &   &   &   &   &   &   \\
     \textbf{8} &   &   &   &   &   &   &   & 1 &   &   &   &   &   &   &   &   &   &   &   &   &   &   \\
     \textbf{7} & 1 &   &   &   &   &   &   &   & 1 &   &   &   &   &   &   &   &   &   &   &   &   &   \\
     \textbf{6} &   & 1 &   &   &   &   &   &   &   & 1 &   &   &   &   &   &   &   &   &   &   &   &   \\
     \textbf{5} &   &   & 1 &   &   &   &   &   &   &   & 1 &   &   &   &   &   &   &   &   &   &   &   \\
     \textbf{4} &   &   &   & 1 &   &   &   &   &   &   &   & 1 &   &   &   &   &   &   &   &   &   &   \\
     \textbf{3} & 1 &   &   &   & 1 &   &   &   &   &   &   &   & 1 &   &   &   &   &   &   &   &   &   \\
     \textbf{2} &   & 1 &   &   &   & 1 &   &   &   &   &   &   &   & 1 &   &   &   &   &   &   &   &   \\
     \textbf{1} &   &   & 1 &   &   &   & 1 &   &   &   &   &   &   &   & 1 &   &   &   &   &   &   &   \\
     \textbf{0} & 1 &   &   & 1 &   &   &   & 1 &   &   &   &   &   &   &   & 1 &   &   &   &   &   &   \\  \hline
       & \textbf{0} & \textbf{1} & \textbf{2} & \textbf{3} & \textbf{4} & \textbf{5} & \textbf{6} & \textbf{7} & \textbf{8} & \textbf{9} & \textbf{10} & \textbf{11} & \textbf{12} & \textbf{13} & \textbf{14} & \textbf{15} & \textbf{16} & \textbf{17} & \textbf{18} & \textbf{19} & \textbf{20} & \textbf{21}
  \end{array} }
  \end{equation}

\bigskip
We note that in the above grid, each 1 occurs at coordinates that add to $2^m-1$ for some positive integer $m$. Note also, that the above grid depicts the binary case of the $n$-ary generalization for $n=2,3,4,5,6,\ldots$. Considering examples using the $n=10$, 10-ary, or decimal, case allows us to bring in familiar base ten numbering, so providing a familiar presentation instead of asking the reader to think in number systems base $n$ generally.

So next, our objective is to write down the base $n$, $n$-ary version of the above binary 2D generating function.

To start the rationale, let us rewrite equations (\ref{11b08}) to (\ref{11b12}), using binary numbers as the indices in the expanded polynomials. This will enable us to write the base $n$ version of the binary cases given above.

\begin{equation} \label{11b15}
  (p+q)= \frac{p^{2}-q^{2}}{p-q} = p^{1_2}+q^{1_2},
\end{equation}
\begin{equation} \label{11b16}
  (p+q)(p^2+q^2) = \frac{p^4-q^4}{p-q} =p^{11_2} + p^{10_2} q + p q^{10_2} + q^{11_2},
\end{equation}
\begin{equation} \label{11b17}
  (p+q)(p^2+q^2)(p^4+q^4) = \frac{p^8-q^8}{p-q}
\end{equation}
\begin{equation} \nonumber
  =p^{111_2} + p^{110_2} q + p^{101_2} q^{10_2} + p^{100_2} q^{11_2} + p^{11_2} q^{100_2} + p^{10_2} q^{101_2} + p q^{110_2} + q^{111_2},
\end{equation}
\begin{equation} \label{11b18}
  (p+q)(p^2+q^2)(p^4+q^4)(p^8+q^8) = \frac{p^{16}-q^{16}}{p-q}
\end{equation}
\begin{equation} \nonumber
  =p^{1111_2} + p^{1110_2} q + p^{1101_2} q^{10_2} + p^{1100_2} q^{11_2} + p^{1011_2} q^{100_2} + p^{1010_2} q^{101_2} + p^{1001_2} q^{110_2} + p^{1000_2} q^{111_2}
\end{equation}
\begin{equation} \nonumber
  + p^{111_2} q^{1000_2} + p^{110_2} q^{1001_2} + p^{101_2} q^{1010_2} + p^{100_2} q^{1011_2} + p^{11_2} q^{1100_2} + p^{10_2} q^{1101_2} + p q^{1110_2} + q^{1111_2},
\end{equation}
\begin{equation} \nonumber
  \vdots
\end{equation}
\begin{equation} \label{11b19}
  (p+q)(p^2+q^2)(p^4+q^4)\cdots(p^{2^{n-1}}+q^{2^{n-1}}) = \frac{p^{{2^{m}}}-q^{{2^{m}}}}{p-q} =1+\sum_{\substack{a,b \geq 0; \\  a+b \neq 0}} B_2(a,b)p^aq^b
\end{equation}
\begin{equation} \nonumber
  =p^{\sum_{k=0}^{m-1}2^k} + p^{\sum_{k=0}^{m-2}2^k} q + p^{\sum_{k=0}^{m-3}2^k} q^2 + \ldots + p^2 q^{\sum_{k=0}^{m-3}2^k} + p q^{\sum_{k=0}^{m-2}2^k} + q^{\sum_{k=0}^{m-1}2^k},
\end{equation}
where each term $p^a q^b$ has the restriction that $a$ and $b$ are positive integers and $a+b=2^{m}-1$, for every $m \in N^+$.

The above identities combine to give us an approach for 2D $n$-ary component partitions into distinct parts. All we need do is replace $a_2$ in binary, with $a_n$ in $n$-ary.

Therefore we see the \textbf{base n} or \textbf{n-ary} partition analogy case is:

\begin{equation} \label{11b20}
  (p+q)= p^{1_n}+q^{1_n},
\end{equation}
\begin{equation} \label{11b21}
  (p+q)(p^{n}+q^{n}) =p^{11_n} + p^{10_n} q + p q^{10_n} + q^{11_n},
\end{equation}
\begin{equation} \label{11b22}
  (p+q)(p^{n}+q^{n})(p^{n^2}+q^{n^2})
\end{equation}
\begin{equation} \nonumber
 =p^{111_n} + p^{110_n} q + p^{101_n} q^{10_n} + p^{100_n} q^{11_n} + p^{11_n} q^{100_n} + p^{10_n} q^{101_n} + p q^{110_n} + q^{111_n},
\end{equation}
\begin{equation} \label{11b23}
  (p+q)(p^{n}+q^{n})(p^{n^2}+q^{n^2})(p^{n^3}+q^{n^3})
\end{equation}
\begin{equation} \nonumber
  =p^{1111_n} + p^{1110_n} q + p^{1101_n} q^{10_n} + p^{1100_n} q^{11_n} + p^{1011_n} q^{100_n} + p^{1010_n} q^{101_n} + p^{1001_n} q^{110_n} + p^{1000_n} q^{111_n}
\end{equation}
\begin{equation} \nonumber
  + p^{111_n} q^{1000_n} + p^{110_n} q^{1001_n} + p^{101_n} q^{1010_n} + p^{100_n} q^{1011_n} + p^{11_n} q^{1100_n} + p^{10_n} q^{1101_n} + p q^{1110_n} + q^{1111_n},
\end{equation}
\begin{equation} \nonumber
  \vdots
\end{equation}
\begin{equation} \label{11b24}
  (p+q)(p^{10}+q^{10})(p^{10^2}+q^{10^2})\cdots(p^{10^{n-1}}+q^{10^{n-1}})
\end{equation}
\begin{equation} \nonumber
  =p^{\sum_{k=0}^{m-1}n^k} + p^{\sum_{k=0}^{m-2}n^k} q^1 + p^{\sum_{k=0}^{m-3}n^k} q^{10_n} + \ldots + p^{10_n} q^{\sum_{k=0}^{m-3}n^k} + p^1 q^{\sum_{k=0}^{m-2}n^k} + q^{\sum_{k=0}^{m-1}n^k},
\end{equation}
where each term $p^a q^b$ has the restriction $a$ and $b$ are positive integers base $n$ with digits comprised of only 1s and 0s, and $a+b=\sum_{k=0}^{m-1}n^k$, for every $m \in N^+$.

The above identities combine to give us a theorem for 2D \textbf{base n} or \textbf{n-ary} component partitions into distinct parts. Therefore we have the

\begin{theorem}
  The number $B_n(a,b)$ for $a$ and $b$ both non negative integers, of vector partitions with parts of the form $\langle 2^\alpha, 2^\beta \rangle$ with $\alpha \geq 0$ and $\beta \geq 0$, into distinct parts is
  \begin{equation} \nonumber
    B_{n}(a,b)= \left\{
       \begin{array}{ll}
         1, & \hbox{if condition $B_{n}$ holds,} \\
         0, & \hbox{otherwise,}
       \end{array}
     \right.
  \end{equation}
where condition $B_{n}$ is : \textbf{$a$ and $b$ are positive integers base n with digits comprised of only 1s and 0s, and $a+b=\sum_{k=0}^{m-1}n^k$, for every $m \in N^+$.} The generating function identity encoding this is

\begin{equation} \label{11b25}
1 + \sum_{k=0}^{\infty} \prod_{j=0}^{k}(p^{n^{k}}+q^{n^{k}}) = 1+ \sum_{\substack{a,b \geq 0; \\  a+b \neq 0}} B_{n}(a,b) p^a q^b
\end{equation}
 \begin{equation} \nonumber
= 1 + (p+q) + (p+q)(p^{n}+q^{n}) + \ldots + (p+q)(p^{n}+q^{n})(p^{n^2}+q^{n^2})\cdots(p^{n^{m}}+q^{n^{m}})+ \ldots
\end{equation}
\end{theorem}

The base 10 or decimal case in which $n=10$ is derived by the set of successive equations:

\begin{equation} \label{11b26}
  (p+q)= p^1+q^1
\end{equation}
\begin{equation} \label{11b27}
  (p+q)(p^{10}+q^{10}) =p^{11} + p^{10} q^1 + p^1 q^{10} + q^{11},
\end{equation}
\begin{equation} \label{11b28}
  (p+q)(p^{10}+q^{10})(p^{10^2}+q^{10^2})
\end{equation}
\begin{equation} \nonumber
 =p^{111} + p^{110} q^1 + p^{101} q^{10} + p^{100} q^{11} + p^{11} q^{100} + p^{10} q^{101} + p^1 q^{110} + q^{111},
\end{equation}
\begin{equation} \label{11b29}
  (p+q)(p^{10}+q^{10})(p^{10^2}+q^{10^2})(p^{10^3}+q^{10^3})
\end{equation}
\begin{equation} \nonumber
  =p^{1111} + p^{1110} q^1 + p^{1101} q^{10} + p^{1100} q^{11} + p^{1011} q^{100} + p^{1010} q^{101} + p^{1001} q^{110} + p^{1000} q^{111}
\end{equation}
\begin{equation} \nonumber
  + p^{111} q^{1000} + p^{110} q^{1001} + p^{101} q^{1010} + p^{100} q^{1011} + p^{11} q^{1100} + p^{10} q^{1101} + p^1 q^{1110} + q^{1111},
\end{equation}
\begin{equation} \nonumber
  \vdots
\end{equation}
\begin{equation} \label{11b30}
  (p+q)(p^{10}+q^{10})(p^{10^2}+q^{10^2})\cdots(p^{10^{m-1}}+q^{10^{m-1}})
\end{equation}
\begin{equation} \nonumber
  =p^{\sum_{k=0}^{m-1}10^k} + p^{\sum_{k=0}^{m-2}10^k} q^1 + p^{\sum_{k=0}^{m-3}10^k} q^{10} + \ldots + p^{10} q^{\sum_{k=0}^{m-3}10^k} + p^1 q^{\sum_{k=0}^{m-2}10^k} + q^{\sum_{k=0}^{m-1}10^k},
\end{equation}
where each term $p^a q^b$ has the restriction that $a$ and $b$ are positive integers in base 10 with digits comprised of only 1s and 0s, and $a+b=\sum_{k=0}^{m-1}10^k$, for every $m \in N^+$.

The above identities combine to give us a theorem for 2D \textbf{base 10} component partitions into distinct parts. Therefore we have the

\begin{theorem}
  The number $B_{10}(a,b)$ for $a$ and $b$ both non negative integers, of vector partitions with parts of the form $\langle 2^\alpha, 2^\beta \rangle$ with $\alpha \geq 0$ and $\beta \geq 0$, with $\alpha + \beta \neq 0$, into distinct parts is
  \begin{equation} \nonumber
    B_{10}(a,b)= \left\{
       \begin{array}{ll}
         1, & \hbox{if condition $B_{10}$ holds,} \\
         0, & \hbox{otherwise,}
       \end{array}
     \right.
  \end{equation}
where condition $B_{10}$ is : \textbf{$a$ and $b$ are positive integers base 10 with digits comprised of only 1s and 0s, and $a+b=\sum_{k=0}^{m-1}10^k$, for every $m \in N^+$.} The generating function identity encoding this is

\begin{equation} \label{11b31}
1 + \sum_{k=0}^{\infty} \prod_{j=0}^{k}(p^{10^{k}}+q^{10^{k}}) =1 + \sum_{\substack{a,b \geq 0; \\  a+b \neq 0}} B_{10}(a,b) p^a q^b
\end{equation}
 \begin{equation} \nonumber
= 1 + (p+q) + (p+q)(p^{10}+q^{10}) + \ldots + (p+q)(p^{10}+q^{10})(p^{10^2}+q^{10^2})\cdots(p^{10^{m}}+q^{10^{m}})+ \ldots
\end{equation}
\end{theorem}


\section{Some binary integer partition preliminary results}
We start by revisiting the elementary identity,

\begin{equation}  \label{11.07z}
(1 + x) (1 + x^2) (1 + x^4) \cdots  (1 + x^{2^n})\cdots = \frac{1}{1 - x}, \; \; for \, |x|<1.
\end{equation}

By considering both sides of (\ref{11.07z}) as power series it is clear that every positive integer is a sum of distinct binary powers in exactly one way. Hence, each positive integer has a unique binary representation.  That is $1_{10} = 1_2$,
$2_{10} = 10_2$, $3_{10} = 11_2$, $4_{10} = 100_2$, $5_{10} = 101_2$, $6_{10} = 110_2$, etc.

Next, we give an infinite product relationship that is easily derived from the following cases of (\ref{11.07z}).

\begin{eqnarray*}
    (1 + x) (1 + x^2) (1 + x^4) (1 + x^8) (1 + x^{16}) (1 + x^{32})\cdots &=& \frac{1}{1 - x},  \\
    (1 + x^2) (1 + x^4) (1 + x^8) (1 + x^{16})(1 + x^{32})\cdots &=& \frac{1}{1 - x^2}, \\
    (1 + x^4) (1 + x^8) (1 + x^{16})(1 + x^{32})\cdots &=& \frac{1}{1 - x^4}, \\
    (1 + x^8) (1 + x^{16})(1 + x^{32})\cdots &=& \frac{1}{1 - x^8}, \\
    (1 + x^{16})(1 + x^{32})\cdots &=& \frac{1}{1 - x^{16}}, \\
    etc.
  \end{eqnarray*}

Now, the product of all left sides here is equal to the product of all right sides here so we arrive at for $|x|<1,$

\begin{equation}  \label{11.08}
(1 + x) (1 + x^2)^2 (1 + x^4)^3 (1 + x^8)^4 \cdots  (1 + x^{2^n})^{n+1} \ldots
\end{equation}
\begin{equation}  \nonumber
= \frac{1}{(1-x) (1-x^2) (1-x^4) (1-x^8) \cdots} = B(x).
\end{equation}

This gives us an alternative generating function for the integer binary partition function. It is also suggestive of the theorem,
\begin{theorem} \label{thm 7.22}
  The number of unrestricted binary partitions $b(n)$ of a positive integer $n$ is equal to the number of partitions of $n$ such that
\begin{equation}  \nonumber
n = a_1 2^0 + a_2 2^1 + a_3 2^2 + a_4 2^3 + \ldots + a_m 2^{m-1},
\end{equation}
where $0 \leq a_1 \leq 1, 0 \leq a_2 \leq 2, 0 \leq a_3 \leq 3, \ldots, 0 \leq a_m \leq m$. Another way of putting this is partitions of $n$ into powers of $2$ where $2^0$ is used at most once, $2^1$ is used at most twice, $2^2$ is used at most 3 times, etc and generally, $2^{m-1}$ is used at most $m$ times.
\end{theorem}

Although this theorem and (\ref{11.08}) is shallow in the sense of it being fundamental for integer binary partitions, it doesn't seem to be in the literature. Whether this result is known is immaterial, as we will apply it to vector partitions with binary components for $2D$ and $3D$ cases, and assert the $nD$ generalization.

Next, let us consider the 2D infinite binary product given as $\mathbf{B}_2(y,z)$ below. This product enumerates the number of partitions of vector $\langle a, b \rangle$ into distinct parts of the form $\langle 2^c, 2^d \rangle$ for $c$ and $d$ non-negative integers under vector addition.

We see that each rightward, downward sloping diagonal of $\mathbf{B}_2(y,z)$ is a case of (\ref{11.07z}). Therefore, the following 2D equivalence is seen, leading to a new theorem on 2D binary partitions of vectors $\langle a,b \rangle$ for $a$ and $b$ positive integers.


\section{Some easy 2D binary partition transform generating functions}

\begin{theorem} \label{thm 7.23}
\textbf{Distinct 2D binary vector partition transform.} The following infinite 2D binary product relation holds for $|y|<1$, $|z|<1$. The first product here is the generating function for the number of partitions of 2D vector $\langle m,n \rangle$ into distinct vectors whose components are non-negative powers of 2. The relation is
\begin{equation} \label{7.23a}
  \prod_{m, n \geq 0} (1+y^{2^m} z^{2^n}) = \frac{1}{1-xy}  \prod_{k \geq 1} \frac{1}{(1-y^{2^k} z)(1-y z^{2^k})} = \sum_{m, n \geq 0}  \textbf{b}_2(m,n)y^m z^n.
\end{equation}
\end{theorem}

In longhand expansion this equation is as per below. Each diagonal of the first tableaux here is a case of (\ref{11.07z}). The second tableaux here relates to the right side of the equation in the theorem, after successive application of (\ref{11.07z}). We arrive then at

\begin{equation} \label{7.23b}
\mathbf{B}_2(y,z) =
  \end{equation}
\begin{equation}\nonumber
(1+y^{2^0} z^{2^0})(1+y^{2^0} z^{2^1})(1+y^{2^0} z^{2^2})(1+y^{2^0} z^{2^3})(1+y^{2^0} z^{2^4})(1+y^{2^0} z^{2^5}) \cdots
  \end{equation}
\begin{equation}\nonumber
(1+y^{2^1} z^{2^0})(1+y^{2^1} z^{2^1})(1+y^{2^1} z^{2^2})(1+y^{2^1} z^{2^3})(1+y^{2^1} z^{2^4})(1+y^{2^1} z^{2^5}) \cdots
  \end{equation}
\begin{equation}\nonumber
(1+y^{2^2} z^{2^0})(1+y^{2^2} z^{2^1})(1+y^{2^2} z^{2^2})(1+y^{2^2} z^{2^3})(1+y^{2^2} z^{2^4})(1+y^{2^2} z^{2^5}) \cdots
  \end{equation}
\begin{equation}\nonumber
(1+y^{2^3} z^{2^0})(1+y^{2^3} z^{2^1})(1+y^{2^3} z^{2^2})(1+y^{2^3} z^{2^3})(1+y^{2^3} z^{2^4})(1+y^{2^3} z^{2^5}) \cdots
  \end{equation}
\begin{equation}\nonumber
(1+y^{2^4} z^{2^0})(1+y^{2^4} z^{2^1})(1+y^{2^4} z^{2^2})(1+y^{2^4} z^{2^3})(1+y^{2^4} z^{2^4})(1+y^{2^4} z^{2^5}) \cdots
  \end{equation}
\begin{equation}\nonumber
(1+y^{2^5} z^{2^0})(1+y^{2^5} z^{2^1})(1+y^{2^5} z^{2^2})(1+y^{2^5} z^{2^3})(1+y^{2^5} z^{2^4})(1+y^{2^5} z^{2^5}) \cdots
  \end{equation}
\begin{equation}\nonumber
 \;\;\;\;\;\;\;\;\;\; \vdots \;\;\;\;\;\;\;\;\;\;\;\;\;\;\;\;\; \vdots \;\;\;\;\;\;\;\;\;\;\;\;\;\;\;\;\; \vdots \;\;\;\;\;\;\;\;\;\;\;\;\;\;\;\;\; \vdots \;\;\;\;\;\;\;\;\;\;\;\;\;\;\;\;\;\; \vdots \;\;\;\;\;\;\;\;\;\;\;\;\;\;\;\;\; \vdots \;\;\;\;\;\;\;\;\;  \ddots
  \end{equation}

\begin{equation}\nonumber
= \frac{1}{(1-yz)}
  \end{equation}
\begin{equation}\nonumber
\times \frac{1}{(1-y z^{2^1})(1-y z^{2^2})(1-y z^{2^3})(1-y z^{2^4})(1- y z^{2^5}) \cdots}
  \end{equation}
\begin{equation}\nonumber
\times \frac{1}{(1-z y^{2^1})(1-z y^{2^2})(1-z y^{2^3})(1-z y^{2^4})(1- z y^{2^5}) \cdots} .
  \end{equation}

  \bigskip

This identity is equivalent to stating that
\begin{theorem}
  The number of distinct vector 2D binary partitions of a 2D vector $\langle m,n \rangle$ in the first quadrant is equal to the number of unrestricted 2D binary partitions of 2D vector $\langle m,n \rangle$ into binary component vectors with parts having a component as 1.
\end{theorem}

  \bigskip

It is almost intuitively obvious from the above tableau that

\begin{equation}\nonumber
\mathbf{B}_2(y,z) = (1+yz) \frac{\mathbf{B}_2(y^2,z)\mathbf{B}_2(y,z^2)}{\mathbf{B}_2(y^2,z^2)}.
  \end{equation}

Clearly also it is seen that our earlier discussed 2D binary vector unrestricted partition generating function $B_2(y,z)$, satisfies

\begin{equation}\nonumber
\mathbf{B}_2(y,z) = \frac{B_2(y,z)}{B_2(y^2,z^2)}.
  \end{equation}

Obviously then this has implications from the set of functional equations already given. Again consider the equation
\begin{equation}   \label{7.23c}
(1+y^{2^0} z^{2^0})(1+y^{2^0} z^{2^1})(1+y^{2^0} z^{2^2})(1+y^{2^0} z^{2^3})(1+y^{2^0} z^{2^4})(1+y^{2^0} z^{2^5}) \cdots
  \end{equation}
\begin{equation} \nonumber
= 1+y^{\textbf{b}(1)}z+y^{\textbf{b}(2)}z^2+y^{\textbf{b}(3)}z^3+ \cdots +y^{\textbf{b}(n)}z^n + \cdots
:= \sum_{n \geq 0} y^{\textbf{b}(n)}z^n,
\end{equation}
where $\textbf{b}(n)$ is the number of $1$s in the binary number representing $n$. Hence for example, $\textbf{b}(2^5)=1$, $\textbf{b}(2^5+2^3+2^1)=3$, $\textbf{b}(2^{11}+2^5+2^3+2^1)=4$, $\textbf{b}(7)=3$, $\textbf{b}(2^{100}+2^6+2^5+2^3+1)=5$, and so on, with the convenient definition $\textbf{b}(0)=0$.

Hence, substituting cases of  (\ref{7.23c}) into  (\ref{7.23b}) and comparing them to (\ref{7.23a}) we see that
\begin{equation}   \label{7.23d}
\mathbf{B}_2(y,z) = (\sum_{n \geq 0} y^{\textbf{b}(n)}z^n) (\sum_{n \geq 0} y^{2\textbf{b}(n)}z^n)
(\sum_{n \geq 0} y^{4\textbf{b}(n)}z^n)(\sum_{n \geq 0} y^{8\textbf{b}(n)}z^n) \cdots
  \end{equation}
\begin{equation} \nonumber
= \sum_{m, n \geq 0}  \textbf{b}_2(m,n)y^m z^n.
\end{equation}

\bigskip
\bigskip
\bigskip

In the future notes we can explore such things a bit further. We next give a different example, albeit a little contrived.

Next, let us consider the 2D infinite binary product tableau given below. This product enumerates the number of partitions of vector $\langle a, b \rangle$ into at most $d+1$ parts of the form $\langle 2^c, 2^d \rangle$ for $c$ and $d$ non-negative integers under vector addition. The tableau is easily transformed by applying (\ref{11.08}) repeatedly to each diagonal from left down, and replacing each successive term of LHS of the case of (\ref{11.08}) with the corresponding successive term of the RHS of (\ref{11.08}).

Hence the following 2D binary vector partition identity result ensues.

\newpage

\begin{theorem} \label{thm 7.24}
\textbf{Unrestricted 2D binary vector partition transform.} The following infinite 2D binary product relation holds for $|y|<1$, $|z|<1$. The second product here is the generating function for the number of partitions of 2D vector $\langle m,n \rangle$ into unrestricted vectors whose components are non-negative powers of 2. The relation is
\begin{equation} \nonumber
  \prod_{m, n \geq 0} (1+y^{2^m} z^{2^n})^{n+1} = \prod_{m, n \geq 0} \frac{1}{(1-y^{2^m} z^{2^n})} = \sum_{m, n \geq 0}  b_2(m,n)y^m z^n.
\end{equation}
\end{theorem}

\textbf{Proof}: The equation from the theorem in longhand is as follows. Each diagonal of the first tableaux here is a case of (\ref{11.07z}). The second tableaux here relates to the right side of the equation in the theorem, after successive application of (\ref{11.07z}). We arrive then at

\begin{equation}\nonumber
(1+y^{2^0} z^{2^0})^1(1+y^{2^0}z^{2^1})^1(1+y^{2^0}z^{2^2})^1(1+y^{2^0}z^{2^3})^1(1+y^{2^0}z^{2^4})^1(1+y^{2^0}z^{2^5})^1 \cdots
  \end{equation}
\begin{equation}\nonumber
(1+y^{2^1} z^{2^0})^1(1+y^{2^1}z^{2^1})^2(1+y^{2^1}z^{2^2})^2(1+y^{2^1}z^{2^3})^2(1+y^{2^1}z^{2^4})^2(1+y^{2^1}z^{2^5})^2 \cdots
  \end{equation}
\begin{equation}\nonumber
(1+y^{2^2} z^{2^0})^1(1+y^{2^2}z^{2^1})^2(1+y^{2^2}z^{2^2})^3(1+y^{2^2}z^{2^3})^3(1+y^{2^2}z^{2^4})^3(1+y^{2^2}z^{2^5})^3 \cdots
  \end{equation}
\begin{equation}\nonumber
(1+y^{2^3} z^{2^0})^1(1+y^{2^3}z^{2^1})^2(1+y^{2^3}z^{2^2})^3(1+y^{2^3}z^{2^3})^4(1+y^{2^3}z^{2^4})^4(1+y^{2^3}z^{2^5})^4 \cdots
  \end{equation}
\begin{equation}\nonumber
(1+y^{2^4} z^{2^0})^1(1+y^{2^4}z^{2^1})^2(1+y^{2^4} z^{2^2})^3(1+y^{2^4}z^{2^3})^4(1+y^{2^4}z^{2^4})^5(1+y^{2^4}z^{2^5})^5 \cdots
  \end{equation}
\begin{equation}\nonumber
(1+y^{2^5} z^{2^0})^1(1+y^{2^5}z^{2^1})^2(1+y^{2^5} z^{2^2})^3(1+y^{2^5}z^{2^3})^4(1+y^{2^5}z^{2^4})^5(1+y^{2^5}z^{2^5})^6 \cdots
  \end{equation}
\begin{equation}\nonumber
 \;\;\;\;\;\;\;\;\;\; \vdots \;\;\;\;\;\;\;\;\;\;\;\;\;\;\;\;\; \vdots \;\;\;\;\;\;\;\;\;\;\;\;\;\;\;\;\; \vdots \;\;\;\;\;\;\;\;\;\;\;\;\;\;\;\;\;\;\;\; \vdots \;\;\;\;\;\;\;\;\;\;\;\;\;\;\;\;\;\;\;\;\; \vdots \;\;\;\;\;\;\;\;\;\;\;\;\;\;\;\;\;\; \vdots \;\;\;\;\;\;\;\;\;\;\;\;\;\;  \ddots
  \end{equation}

\begin{equation}\nonumber
= \frac{1}{(1-y^{2^0} z^{2^0}) (1-y^{2^0}z^{2^1}) (1-y^{2^0}z^{2^2}) (1-y^{2^0}z^{2^3}) (1-y^{2^0}z^{2^4}) (1-y^{2^0}z^{2^5})  \cdots}
  \end{equation}
\begin{equation}\nonumber
\times \frac{1}{(1-y^{2^1} z^{2^0}) (1-y^{2^1}z^{2^1}) (1-y^{2^1}z^{2^2}) (1-y^{2^1}z^{2^3}) (1-y^{2^1}z^{2^4}) (1-y^{2^1}z^{2^5})  \cdots}
  \end{equation}
\begin{equation}\nonumber
\times \frac{1}{(1-y^{2^2} z^{2^0}) (1-y^{2^2}z^{2^1}) (1-y^{2^2}z^{2^2}) (1-y^{2^2}z^{2^3}) (1-y^{2^2}z^{2^4}) (1-y^{2^2}z^{2^5})  \cdots}
  \end{equation}
\begin{equation}\nonumber
\times \frac{1}{(1-y^{2^3} z^{2^0}) (1-y^{2^3}z^{2^1}) (1-y^{2^3}z^{2^2}) (1-y^{2^3}z^{2^3}) (1-y^{2^3}z^{2^4}) (1-y^{2^3}z^{2^5})  \cdots}
  \end{equation}
\begin{equation}\nonumber
\times \frac{1}{(1-y^{2^4} z^{2^0}) (1-y^{2^4}z^{2^1}) (1-y^{2^4} z^{2^2}) (1-y^{2^4}z^{2^3}) (1-y^{2^4}z^{2^4}) (1-y^{2^4}z^{2^5})  \cdots}
  \end{equation}
\begin{equation}\nonumber
\times \frac{1}{(1-y^{2^5} z^{2^0}) (1-y^{2^5}z^{2^1}) (1-y^{2^5} z^{2^2}) (1-y^{2^5}z^{2^3}) (1-y^{2^5}z^{2^4}) (1-y^{2^5}z^{2^5})  \cdots}
  \end{equation}
\begin{equation}\nonumber
  \vdots
  \end{equation}
\begin{equation}\nonumber
  = B_2(y,z). \; \; \; \blacksquare
  \end{equation}

So, the generating function $B_2(y,z)$ for 2D unrestricted binary vector partitions has an alternate generating function, as did $B(z)$ for binary integer unrestricted partitions. Hence we can state the following

\begin{theorem} \label{thm 7.25}
  The number of unrestricted binary vector partitions, $\, b_2(m,n)$ of a vector $\langle m,n \rangle$ is equal to the number of partitions of vector $\langle m,n \rangle$ into 2D vectors with each $\lambda_i = 0,1,2,3,\ldots$
\begin{equation} \nonumber
  \langle 2^{\lambda_1}, 2^0 \rangle, \,  \langle 2^{\lambda_2}, 2^1\rangle, \,
  \langle 2^{\lambda_3}, 2^2\rangle, \,  \langle 2^{\lambda_4}, 2^3\rangle, \,  \ldots , \,  \langle 2^{\lambda_m}, 2^{m-1}\rangle, \,  \ldots ,
\end{equation}
 where $2^{\lambda_1}=2^0$ is used at most once, $2^{\lambda_2}=2^1$ is used at most twice, $2^{\lambda_3}=2^2$ is used at most 3 times, etc and generally, $2^{\lambda_m}=2^{m-1}$ is used at most $m$ times.
\end{theorem}


\begin{theorem} \label{thm 7.25a}
\textbf{2D binary vector partition transform with triangular numbers.} The following infinite 2D binary product relation holds for $|y|<1$, $|z|<1$. The second product here is the generating function for the number of partitions of 2D vector $\langle m,n \rangle$ into unrestricted vectors whose components are non-negative powers of 2.
\begin{equation} \nonumber
  \prod_{m, n \geq 0} (1+y^{2^m} z^{2^n})^{\frac{(m+1)(m+2)}{2}} = \prod_{m, n \geq 0} \frac{1}{(1-y^{2^m} z^{2^n})^{m+1}} .
\end{equation}
\end{theorem}

\textbf{Proof}: The equation from the theorem in longhand is as follows. Each diagonal of the first tableaux here is a case of (\ref{11.08}). The second tableaux here relates to the right side of the equation in the theorem, after successive application of (\ref{11.08}). We arrive then at

\begin{equation}\nonumber
(1+y^{2^0} z^{2^0})^1(1+y^{2^0}z^{2^1})^1(1+y^{2^0}z^{2^2})^1(1+y^{2^0}z^{2^3})^1(1+y^{2^0}z^{2^4})^1(1+y^{2^0}z^{2^5})^1 \cdots
  \end{equation}
\begin{equation}\nonumber
(1+y^{2^1} z^{2^0})^1(1+y^{2^1}z^{2^1})^{3}(1+y^{2^1}z^{2^2})^{3}(1+y^{2^1}z^{2^3})^{3}(1+y^{2^1}z^{2^4})^{3}(1+y^{2^1}z^{2^5})^3 \cdots
  \end{equation}
\begin{equation}\nonumber
(1+y^{2^2} z^{2^0})^1(1+y^{2^2}z^{2^1})^{3}(1+y^{2^2}z^{2^2})^6(1+y^{2^2}z^{2^3})^6(1+y^{2^2}z^{2^4})^6(1+y^{2^2}z^{2^5})^{6} \cdots
  \end{equation}
\begin{equation}\nonumber
(1+y^{2^3} z^{2^0})^1(1+y^{2^3}z^{2^1})^{3}(1+y^{2^3}z^{2^2})^6(1+y^{2^3}z^{2^3})^{10}(1+y^{2^3}z^{2^4})^{10}(1+y^{2^3}z^{2^5})^{10} \cdots
  \end{equation}
\begin{equation}\nonumber
(1+y^{2^4} z^{2^0})^1(1+y^{2^4}z^{2^1})^{3}(1+y^{2^4} z^{2^2})^6(1+y^{2^4}z^{2^3})^{10}(1+y^{2^4}z^{2^4})^{15}(1+y^{2^4}z^{2^5})^{15} \cdots
  \end{equation}
\begin{equation}\nonumber
(1+y^{2^5} z^{2^0})^1(1+y^{2^5}z^{2^1})^{3}(1+y^{2^5} z^{2^2})^6(1+y^{2^5}z^{2^3})^{10}(1+y^{2^5}z^{2^4})^{15}(1+y^{2^5}z^{2^5})^{21} \cdots
  \end{equation}
\begin{equation}\nonumber
 \;\;\;\;\;\;\;\;\;\; \vdots \;\;\;\;\;\;\;\;\;\;\;\;\;\;\;\;\; \vdots \;\;\;\;\;\;\;\;\;\;\;\;\;\;\;\;\; \vdots \;\;\;\;\;\;\;\;\;\;\;\;\;\;\;\;\;\;\;\; \vdots \;\;\;\;\;\;\;\;\;\;\;\;\;\;\;\;\;\;\;\;\; \vdots \;\;\;\;\;\;\;\;\;\;\;\;\;\;\;\;\;\; \vdots \;\;\;\;\;\;\;\;\;\;\;\;\;\;  \ddots
  \end{equation}
\begin{equation}\nonumber
= \frac{1}{(1-y^{2^0} z^{2^0}) (1-y^{2^0}z^{2^1}) (1-y^{2^0}z^{2^2}) (1-y^{2^0}z^{2^3}) (1-y^{2^0}z^{2^4}) (1-y^{2^0}z^{2^5})  \cdots}
  \end{equation}
\begin{equation}\nonumber
\times \frac{1}{(1-y^{2^1} z^{2^0}) (1-y^{2^1}z^{2^1})^2 (1-y^{2^1}z^{2^2})^2 (1-y^{2^1}z^{2^3})^2 (1-y^{2^1}z^{2^4})^2 (1-y^{2^1}z^{2^5})^2  \cdots}
  \end{equation}
\begin{equation}\nonumber
\times \frac{1}{(1-y^{2^2} z^{2^0}) (1-y^{2^2}z^{2^1})^2 (1-y^{2^2}z^{2^2})^3 (1-y^{2^2}z^{2^3})^3 (1-y^{2^2}z^{2^4})^3 (1-y^{2^2}z^{2^5})^3  \cdots}
  \end{equation}
\begin{equation}\nonumber
\times \frac{1}{(1-y^{2^3} z^{2^0}) (1-y^{2^3}z^{2^1})^2 (1-y^{2^3}z^{2^2})^3 (1-y^{2^3}z^{2^3})^4 (1-y^{2^3}z^{2^4})^4 (1-y^{2^3}z^{2^5})^4  \cdots}
  \end{equation}
\begin{equation}\nonumber
\times \frac{1}{(1-y^{2^4} z^{2^0}) (1-y^{2^4}z^{2^1})^2 (1-y^{2^4} z^{2^2})^3 (1-y^{2^4}z^{2^3})^4 (1-y^{2^4}z^{2^4})^5 (1-y^{2^4}z^{2^5})^5  \cdots}
  \end{equation}
\begin{equation}\nonumber
\times \frac{1}{(1-y^{2^5} z^{2^0}) (1-y^{2^5}z^{2^1})^2 (1-y^{2^5} z^{2^2})^3 (1-y^{2^5}z^{2^3})^4 (1-y^{2^5}z^{2^4})^5 (1-y^{2^5}z^{2^5})^6  \cdots}
  \end{equation}
\begin{equation}\nonumber
  \vdots
  \end{equation}
\begin{equation}\nonumber
 \blacksquare
  \end{equation}

So, we have here stated that the generating functions for two differently derived classes of 2D partitions, say $\mathbf{C1}_2(y,z)$ and $\mathbf{C2}_2(y,z)$ have equivalence. Hence we can state the following
\begin{theorem} \label{thm 7.26}
  Define $\mathbf{C1}_2(y,z)$ as the generating function for partitions of vector $\langle m,n \rangle$ for $m$ and $n$ positive integers such that each partition involves precisely for $\lambda = 0, 1, 2, 3,\ldots$, the parts,
\begin{equation} \nonumber
  \langle 2^\lambda, 2^0 \rangle, \langle 2^\lambda, 2^1\rangle,
  \langle 2^\lambda, 2^2\rangle, \langle 2^\lambda, 2^3\rangle, \ldots , \langle 2^\lambda, 2^{m-1}\rangle, \ldots ,
\end{equation}
 where $2^\lambda=2^0$ is used at most once, $2^\lambda=2^1$ is used at most three times, $2^\lambda=2^2$ is used at most six times, etc and generally, $2^\lambda=2^{m-1}$ is used at most $\frac{m(m+1)}{2}$ times.

 Define $\mathbf{C2}_2(y,z)$ as the generating function for partitions of vector $\langle m,n \rangle$ for $m$ and $n$ positive integers such that each partition involves precisely for $\lambda = 0, 1, 2, 3,\ldots$, the parts,
\begin{equation} \nonumber
  \langle 2^\lambda, 2^0 \rangle, \langle 2^\lambda, 2^1\rangle,
  \langle 2^\lambda, 2^2\rangle, \langle 2^\lambda, 2^3\rangle, \ldots , \langle 2^\lambda, 2^{m-1}\rangle, \ldots ,
\end{equation}
 where $2^\lambda=2^0$ is used at most once, $2^\lambda=2^1$ is used at most three times, $2^\lambda=2^2$ is used at most six times, etc and generally, $2^\lambda=2^{m-1}$ is used at most $\frac{m(m+1)}{2}$ times.
\end{theorem}

\section{A binary partition \textit{2}-space variation of extended \textit{q}-binomial theorem.} \label{S:section 9a}

We can apply the method to other vector partition generating functions.
An example is now given. The following theorem is based around the ideas associated with the elementary identity

\begin{equation}   \nonumber
  \left( 1 + x \right)\left( 1 + x^2 \right)\left( 1 + x^4 \right)\left( 1 + x^8 \right)\cdots = \frac{1}{1-x}.
\end{equation}

In fact the combinatorial interpretation of this is "\textit{each positive integer is uniquely represented by a sum of distinct powers of 2}". So, we are here looking at an extension of this result in the

\begin{theorem}
 \begin{equation}\label{12.01}
    \prod_{k \geq 0} \left( \frac{1}{1- q t^{2^k}} \right)
    = \prod_{\substack{j,k \geq 0 \\ j \leq k}} \left( 1+ q^{2^j} t^{2^k} \right)
    = 1 + \sum_{k=1}^{\infty} A_k t^k
   \end{equation}
   where
   \begin{equation}  \nonumber
A_k = \frac{1}{k!}
   \begin{tiny}
   \begin{vmatrix}
    q & -1 & 0 & 0 & \cdots & 0 \\
    q^2+2q & q & -2 & 0 & \cdots & 0 \\
    q^3 & q^2+q & q & -3 & \cdots & 0 \\
    q^4+2q^2+4q & q^3 & q^2+q & q & \ddots & \vdots \\
    \vdots & \vdots & \vdots & \vdots & \ddots & -(k-1) \\
    \sum_{\substack{2^j|k \\ j \geq 0}} 2^j q^{k/2^j} & \sum_{\substack{2^j|(k-1) \\ j \geq 0}} 2^j q^{{(k-1)}/2^j} & \sum_{\substack{2^j|(k-2) \\ j \geq 0}} 2^j q^{{(k-2)}/2^j} & \sum_{\substack{2^j|(k-3) \\ j \geq 0}} 2^j q^{{(k-3)}/2^j} & \cdots & q \\
  \end{vmatrix} \end{tiny}
 \end{equation}
\end{theorem}

The combinatorial interpretation of (\ref{12.01}) is

\begin{theorem}  \label{th12.01}
  If $B(j,k)$ is the number of vector partitions of $\langle j, k \rangle$ into distinct parts of kind
$\langle 2^a , 2^b \rangle$ in which $a \leq b$ with non-negative integers $a$ and $b$, then $B(j,k)$ equals also the number of partitions into “unrestricted” parts of kind $\langle 1, 2^b \rangle$ in which $b$ is a non negative integer, and $B(j,k)$ is
the coefficient of $q^j t^k$ in (\ref{12.01}).
\end{theorem}

Each side of (\ref{12.01}) satisfies the equation $f(t)(1 - qt) = f(t^2)$ and this equation also
leads to a set of recurrences solvable using Cramer’s rule.

In Mathematica, Maple or Wolframalpha online we can easily check that

\begin{equation} \label{12.02}
  \prod_{k \geq 0} \left( \frac{1}{1- q t^{2^k}} \right)
  = 1 + q t + q (q + 1) t^2 + q^2 (q + 1) t^3 + q (q^3 + q^2 + q + 1) t^4
  \end{equation}
  \begin{equation}  \nonumber
  + q^2 (q^3 + q^2 + q + 1) t^5 + q^2 (q^4 + q^3 + q^2 + 2 q + 1) t^6
  \end{equation}
  \begin{equation}  \nonumber
  + q^3 (q^4 + q^3 + q^2 + 2 q + 1) t^7
  \end{equation}
  \begin{equation}  \nonumber
  + q (q^7 + q^6 + q^5 + 2 q^4 + 2 q^3 + q^2 + q + 1) t^8
\end{equation}
  \begin{equation}  \nonumber
  + q^2 (q^7 + q^6 + q^5 + 2 q^4 + 2 q^3 + q^2 + q + 1) t^9
\end{equation}
  \begin{equation}  \nonumber
  + q^2 (q^8 + q^7 + q^6 + 2 q^5 + 2 q^4 + 2 q^3 + 2 q^2 + 2 q + 1) t^{10} + \ldots
\end{equation}

\bigskip

Also, as a matter of interest, utilizing a form $\prod_{k=0}^\infty \prod_{j=0}^k (1+ q^{2^j} t^{2^k})$,
on calculating engines, the two product expansions in (\ref{12.01}) can be easily verified; both of them
yielding the series given in (\ref{12.02}).


So in longhand we have that

\begin{multline*}
\frac{1}{(1-qt)(1-qt^2)(1-qt^4)(1-qt^8)(1-qt^{16})...} \\
  =(1+qt) \\
  (1+qt^2)(1+q^2t^2) \\
  (1+qt^4)(1+q^2t^4)(1+q^4t^4) \\
  (1+qt^8)(1+q^2t^8)(1+q^4t^8)(1+q^4t^8) \\
  (1+qt^{16})(1+q^2t^{16})(1+q^4t^{16})(1+q^4t^{16})(1+q^8t^{16}) \\
  \vdots
\end{multline*}

So let us define $\alpha_k(q)$ and $\beta_2(j,k)$ from

\begin{equation}\label{12.03}
    \prod_{k \geq 0} \left( \frac{1}{1- q t^{2^k}} \right)
    = \prod_{\substack{j,k \geq 0 \\ j \leq k}} \left( 1+ q^{2^j} t^{2^k} \right)
    = 1 + \sum_{k=1}^{\infty} \alpha_k(q) t^k = \sum_{j=0,k=0}^{\infty} \beta_2(j,k) q^j t^k.
   \end{equation}

So we see from (\ref{12.02}) that

\begin{eqnarray}
\nonumber   \alpha_1(q) &=& q, \\
\nonumber   \alpha_2(q) &=& q (q + 1), \\
\nonumber   \alpha_3(q) &=& q^2 (q + 1), \\
\nonumber   \alpha_4(q) &=& q (q^3 + q^2 + q + 1), \\
\nonumber   \alpha_5(q) &=& q^2 (q^3 + q^2 + q + 1), \\
\nonumber   \alpha_6(q) &=& q^2 (q^4 + q^3 + q^2 + 2 q + 1), \\
\nonumber   \alpha_7(q) &=& q^3 (q^4 + q^3 + q^2 + 2 q + 1), \\
\nonumber   \alpha_8(q) &=& q (q^7 + q^6 + q^5 + 2 q^4 + 2 q^3 + q^2 + q + 1), \\
\nonumber   \alpha_9(q) &=& q^2 (q^7 + q^6 + q^5 + 2 q^4 + 2 q^3 + q^2 + q + 1), \\
\nonumber   \alpha_{10}(q) &=& q^2 (q^8 + q^7 + q^6 + 2 q^5 + 2 q^4 + 2 q^3 + 2 q^2 + 2 q + 1),
\end{eqnarray}
and so on. Hence, from this we see that

\begin{equation}\nonumber
  \beta_2(1,1)=1, \beta_2(1,2)=1, \beta_2(2,2)=1, \beta_2(1,3)=0, \beta_2(2,3)=1,  \beta_2(1,4)=1,  \beta_2(2,4)=1, etc.
\end{equation}
All of the $ \beta_2(j,k)$ values are encoded in the following grid table.

\textbf{Partition grid for $\mathbf{1/((1-qt)(1-qt^2)(1-qt^4)(1-qt^8)...)= \sum_{j,k\geq 0}^{\infty} \beta_2(j,k) q^j t^k}$}
\begin{tabular}{c|ccccccccccccc}
\vdots & \vdots & \vdots & \vdots & \vdots & \vdots & \vdots & \vdots & \vdots & \vdots & \vdots & \vdots & \vdots & \vdots \\
13 &  &   & 1 & 2 & 2 & 3 & 3 & 2 & 2 & 2 & 1 & 1 & 1  \\
12 &  & 1 & 2 & 2 & 3 & 3 & 2 & 2 & 2 & 1 & 1 & 1 &    \\
11 &  &   & 1 & 2 & 2 & 2 & 2 & 2 & 1 & 1 & 1 &   &    \\
10 &  & 1 & 2 & 2 & 2 & 2 & 2 & 1 & 1 & 1 &   &   &    \\
 9 &  & 1 & 1 & 1 & 2 & 2 & 1 & 1 & 1 &   &   &   &    \\
 8 &1 & 1 & 1 & 2 & 2 & 1 & 1 & 1 &   &   &   &   &    \\
 7 &  &   & 1 & 2 & 1 & 1 & 1 &   &   &   &   &   &    \\
 6 &  & 1 & 2 & 1 & 1 & 1 &   &   &   &   &   &   &    \\
 5 &  & 1 & 1 & 1 & 1 &   &   &   &   &   &   &   &    \\
 4 & 1& 1 & 1 & 1 &   &   &   &   &   &   &   &   &    \\
 3 &  & 1 & 1 &   &   &   &   &   &   &   &   &   &    \\
 2 & 1& 1 &   &   &   &   &   &   &   &   &   &   &    \\
 1 & 1&   &   &   &   &   &   &   &   &   &   &   &    \\ \hline
   & 1 & 2 & 3 & 4 & 5 & 6 & 7 & 8 & 9 & 10 & 11 & 12 & 13
\end{tabular}

To illustrate theorem \ref{th12.01}, we give an arbitrary case for
the 2D vector $\langle 3, 6 \rangle$:

\begin{corollary}
 $\beta_2(3,6)=2$ is the number of vector partitions of $\langle 3, 6 \rangle$ into distinct parts of kind
$\langle 2^a , 2^b \rangle$ in which $a \leq b$ with non-negative integers $a$ and $b$. The two partitions are $\langle 2, 4 \rangle + \langle 1, 2 \rangle$ and  $\langle 1, 4 \rangle + \langle 2, 2 \rangle$. Also $\beta_2(3,6)=2$ equals the number of partitions into
“unrestricted” parts of kind $\langle 1, 2^b \rangle$ in which $b$ is a non negative integer. The two partitions are $\langle 1, 2 \rangle + \langle 1, 2 \rangle + \langle 1, 2 \rangle$ and  $\langle 1, 4 \rangle + \langle 1, 1 \rangle + \langle 1, 1 \rangle$. Then also $\beta_2(3,6)=2$ is
the coefficient of $q^3 t^6$ in (\ref{12.01}).
\end{corollary}
The vector partitions $\beta_2(j,k)$ defined for (\ref{12.01}) are easily visualized by number entries in the 2D grid above.

We see that our methods set up the platform possible research for partitions into binary component vectors, and for other $n$-space infinite product generating functions.


\section{First quadrant 2D binary partitions}

We employ the first quadrant expansion of binary powers in $x$ and $y$.

  \begin{equation}  \nonumber
(1+xy)(1+x^2y)(1+x^4y)(1+x^8y)(1+x^{16}y)(1+x^{32}y) (1+x^{64}y)(1+x^{128}y) (1+x^{256}y)\cdots
  \end{equation}
  \begin{equation}  \nonumber
(1+xy^2)(1+x^2y^2)(1+x^4y^2)(1+x^8y^2)(1+x^{16}y^2)(1+x^{32}y^2) (1+x^{64}y^2)(1+x^{128}y^2)\cdots
  \end{equation}
  \begin{equation}  \nonumber
 (1+xy^4)(1+x^2y^4)(1+x^4y^4)(1+x^8y^4)(1+x^{16}y^4)(1+x^{32}y^4) (1+x^{64}y^4) (1+x^{128}y^4)\cdots
  \end{equation}
  \begin{equation}  \nonumber
 (1+xy^8)(1+x^2y^8)(1+x^4y^8)(1+x^8y^8)(1+x^{16}y^8)(1+x^{32}y^8) (1+x^{64}y^8) (1+x^{128}y^8)\cdots
  \end{equation}
  \begin{equation}  \nonumber
 (1+xy^{16})(1+x^2y^{16})(1+x^4y^{16})(1+x^8y^{16})(1+x^{16}y^{16})(1+x^{32}y^{16}) (1+x^{64}y^{16})(1+x^{128}y^{16})\cdots
  \end{equation}
  \begin{equation}  \nonumber
 (1+xy^{32})(1+x^2y^{32})(1+x^4y^{32})(1+x^8y^{32})(1+x^{16}y^{32})(1+x^{32}y^{32}) (1+x^{64}y^{32})(1+x^{128}y^{32})\cdots
  \end{equation}
  \begin{equation}  \nonumber
etc.
  \end{equation}

We note that each diagonal infinite product with terms going one row down and one column right is a case of the product
  \begin{equation}  \nonumber
 (1+z)(1+z^2)(1+z^4)(1+z^8)(1+z^{16})\cdots = \frac{1}{1-z},
  \end{equation}

so the above tableau product is therefore equal to

   \begin{equation}  \nonumber
\frac{1}{(1-xy) }
  \end{equation}
  \begin{equation}  \nonumber
\times \frac{1}{(1-x^2y)(1-x^4y)(1-x^8y)(1-x^{16}y)(1-x^{32}y) (1-x^{64}y)(1-x^{128}y) (1-x^{256}y)\cdots}
  \end{equation}
  \begin{equation}  \nonumber
\times \frac{1}{(1-xy^2)(1-xy^4)(1-xy^8)(1-xy^{16})(1-xy^{32})(1-xy^{64})(1-xy^{128})(1-xy^{256})\ldots}
  \end{equation}

  \bigskip

Hence, we have described two equivalent generating functions for 2D first quadrant binary component partitions. We see that this approach works to give us a theorem on 2D vector partitions where the two vector components $a$ and $b$ of vector $\langle a,b \rangle$ are integers of the form $a=2^j$, $b=2^k$ for all non-negative integers $j$ and $k$. In the usual notation the above identity is the generating function version of the theorem just below it

\begin{equation}\label{12.04}
    \frac{1}{1-xy} \prod_{j \geq 1} \left( \frac{1}{(1- x^{2^j} y)(1- x y^{2^j})} \right)
    = \prod_{\substack{j,k \geq 0 \\ j \leq k}} \left( 1+ x^{2^j} y^{2^k} \right)
   \end{equation}

So then we can formulate a

\begin{theorem}
  Consider the set $V_1$ of all first quadrant 2D vectors $\langle a,b \rangle$, where the components $a=2^j$ and $b=2^k$ are binary powers with $j,k \geq 0$. Next consider the set $V_2$ of all first quadrant 2D vectors $\langle c,d \rangle$, where the components are either:
\begin{enumerate}
                                     \item $c=1$ and $d=1$; or
                                     \item $c=1$ and $d=2^j$ with $j \geq 1$; or
                                     \item $c=2^k$ and $d=1$ with $k \geq 1$.
                                   \end{enumerate}
  Then the number of partitions of vector $\langle m,n \rangle$ with $m \geq 1$, $n \geq 1$, into distinct partitions from the set $V_1$ is equal to the number of unrestricted partitions from the set $V_2$.
\end{theorem}

\section{First quadrant lower diagonal 2D binary partitions}
Consider the first quadrant lower diagonal expansion of binary powers in $x$ and $y$. Then for this we state the

\begin{theorem} \label{thm 12.1}
  \begin{equation}  \nonumber
(1+xy)
  \end{equation}
  \begin{equation}  \nonumber
(1+xy^2)(1+x^2y^2)
  \end{equation}
  \begin{equation}  \nonumber
 (1+xy^4)(1+x^2y^4)(1+x^4y^4)
  \end{equation}
  \begin{equation}  \nonumber
 (1+xy^8)(1+x^2y^8)(1+x^4y^8)(1+x^8y^8)
  \end{equation}
  \begin{equation}  \nonumber
 (1+xy^{16})(1+x^2y^{16})(1+x^4y^{16})(1+x^8y^{16})(1+x^{16}y^{16})
  \end{equation}
  \begin{equation}  \nonumber
 (1+xy^{32})(1+x^2y^{32})(1+x^4y^{32})(1+x^8y^{32})(1+x^{16}y^{32})(1+x^{32}y^{32})
  \end{equation}
  \begin{equation}  \nonumber
etc.
  \end{equation}
  \begin{equation}  \nonumber
 = \frac{1}{(1-xy) (1-xy^2)(1-xy^4)(1-xy^8)(1-xy^{16})(1-xy^{32})(1-xy^{64})(1-xy^{128})\cdots}
  \end{equation}
 \end{theorem}

Hence, we have again described two equivalent generating functions for 2D first quadrant binary component partitions. We then have another theorem on 2D vector partitions.

\begin{theorem} \label{thm 12.2}
  Consider the set $U_1$ of all first quadrant 2D vectors $\langle a,b \rangle$, where the components $a=2^j$ and $b=2^k$ are binary powers with $0 \leq j \leq k$ or equivalently $a \leq b$. Next consider the set $U_2$ of all first quadrant 2D vectors $\langle c,d \rangle$, where the components are either:
\begin{enumerate}
                                     \item $c=1$ and $d=1$; or
                                     \item $c=1$ and $d=2^j$ with $j \geq 0$.
                                   \end{enumerate}
  Then the number of partitions of vector $\langle m,n \rangle$ with $m \geq 1$, $n \geq 1$, into distinct partitions from the set $U_1$ is equal to the number of unrestricted partitions from the set $U_2$.
\end{theorem}

\section{First hyperquadrant 3D binary partitions}

Following similar logic for the 3D first hyperquadrant points for binary numbers $2^a$ with positive integers $a$ components of vector $\langle m,n,p \rangle$ it is easy to guess (and we shall prove this) that

\begin{equation}\label{12.05}
    \frac{1}{1-xyz} \prod_{a,b \geq 1} \left( \frac{1}{(1- x y^{2^a}z^{2^b})(1- x^{2^a}y z^{2^b})(1- x^{2^a}y^{2^b}z)} \right)
    = \prod_{a,b,c \geq 0 } \left( 1+ x^{2^a} y^{2^b} z^{2^c} \right)
   \end{equation}

So let us next look at the infinite product of binary powers in the 3D pyramid.

\begin{equation}\nonumber
\mathbf{(1 + x y z)}
\end{equation}
\begin{equation}\nonumber
(1 + x y^2 z)(1 + x^2 y^2 z)
\end{equation}
\begin{equation}\nonumber
(1 + x y^2 z^2)\mathbf{(1 + x^2 y^2 z^2)}
\end{equation}
\begin{equation}\nonumber
(1 + x y^4 z)(1 + x^2 y^4 z)(1 + x^4 y^4 z)
\end{equation}
\begin{equation}\nonumber
(1 + x y^4 z^2)(1 + x^2 y^4 z^2)(1 + x^4 y^4 z^2)
\end{equation}
\begin{equation}\nonumber
(1 + x y^4 z^4)(1 + x^2 y^4 z^4)\mathbf{(1 + x^4 y^4 z^4)}
\end{equation}
\begin{equation}\nonumber
(1 + x y^8 z)(1 + x^2 y^8 z)(1 + x^4 y^8 z)(1 + x^8 y^8 z)
\end{equation}
\begin{equation}\nonumber
(1 + x y^8 z^2)(1 + x^2 y^8 z^2)(1 + x^4 y^8 z^2)(1 + x^8 y^8 z^2)
\end{equation}
\begin{equation}\nonumber
(1 + x y^8 z^4)(1 + x^2 y^8 z^4)(1 + x^4 y^8 z^4)(1 + x^8 y^8 z^4)
\end{equation}
\begin{equation}\nonumber
(1 + x y^8 z^8)(1 + x^2 y^8 z^8)(1 + x^4 y^8 z^8)\mathbf{(1 + x^8 y^8 z^8)}
\end{equation}
\begin{equation}\nonumber
etc
\end{equation}

where we see for example, that the \textbf{bold} terms form the infinite product for $\mathbf{\frac{1}{1-xyz}}$. Hence, picking out every instance of the basic binary infinite product as we have done for the \textbf{bold} product, we see quite easily that

\begin{equation}\nonumber
(1 + x y z)
\end{equation}
\begin{equation}\nonumber
(1 + x^1 y^2 z^1)(1 + x^2 y^2 z^1)
\end{equation}
\begin{equation}\nonumber
(1 + x^1 y^2 z^2)(1 + x^2 y^2 z^2)
\end{equation}
\begin{equation}\nonumber
(1 + x^1 y^4 z)(1 + x^2 y^4 z^1)(1 + x^4 y^4 z^1)
\end{equation}
\begin{equation}\nonumber
(1 + x^1 y^4 z^2)(1 + x^2 y^4 z^2)(1 + x^4 y^4 z^2)
\end{equation}
\begin{equation}\nonumber
(1 + x^1 y^4 z^4)(1 + x^2 y^4 z^4)(1 + x^4 y^4 z^4)
\end{equation}
\begin{equation}\nonumber
(1 + x^1 y^8 z^1)(1 + x^2 y^8 z^1)(1 + x^4 y^8 z^1)(1 + x^8 y^8 z^1)
\end{equation}
\begin{equation}\nonumber
(1 + x^1 y^8 z^2)(1 + x^2 y^8 z^2)(1 + x^4 y^8 z^2)(1 + x^8 y^8 z^2)
\end{equation}
\begin{equation}\nonumber
(1 + x^1 y^8 z^4)(1 + x^2 y^8 z^4)(1 + x^4 y^8 z^4)(1 + x^8 y^8 z^4)
\end{equation}
\begin{equation}\nonumber
(1 + x^1 y^8 z^8)(1 + x^2 y^8 z^8)(1 + x^4 y^8 z^8)(1 + x^8 y^8 z^8)
\end{equation}
\begin{equation}\nonumber
etc
\end{equation}

\begin{equation}\nonumber
= \frac{1}{(1 - x^1 y^1 z^1)}
\end{equation}
\begin{equation}\nonumber
\frac{1}{(1 - x^1 y^2 z^1)(1 - x^2 y^2 z^1)}
\end{equation}
\begin{equation}\nonumber
\frac{1}{(1 - x^1 y^2 z^2)}
\end{equation}
\begin{equation}\nonumber
\frac{1}{(1 - x^1 y^4 z^1)(1 - x^2 y^4 z^1)(1 - x^4 y^4 z^1)}
\end{equation}
\begin{equation}\nonumber
\frac{1}{(1 - x^1 y^4 z^2)}
\end{equation}
\begin{equation}\nonumber
\frac{1}{(1 - x^1 y^4 z^4)}
\end{equation}
\begin{equation}\nonumber
\frac{1}{(1 - x^1 y^8 z^1)(1 - x^2 y^8 z^1)(1 - x^4 y^8 z^1)(1 - x^8 y^8 z^1)}
\end{equation}
\begin{equation}\nonumber
\frac{1}{(1 - x^1 y^8 z^2)}
\end{equation}
\begin{equation}\nonumber
\frac{1}{(1 - x^1 y^8 z^4)}
\end{equation}
\begin{equation}\nonumber
\frac{1}{(1 - x^1 y^8 z^8)}
\end{equation}
\begin{equation}\nonumber
etc
\end{equation}

Both sides of this 3D equation enumerate vector partitions in 3-space.  So the transformation identity above has an interpretation in vector partitions implying the
\begin{theorem}
  Consider the infinitely extended 3D pyramid defined in the x-y-z Euclidean space by the inequalities
  \begin{equation} \nonumber
   0 \leq x \leq y, \; 1 \leq z \leq y, \; y \geq 1.
    \end{equation}
   Define all integer component lattice point vectors inside that pyramid to be of form $\langle a,b,c \rangle$ where:
\begin{equation} \nonumber
   0 \leq a \leq b, \; 1 \leq c \leq b, \; b \geq 1.
    \end{equation}
  Then $b_3(a,b,c)$, the number of 3D binary partitions of $\langle a,b,c \rangle$ have the two equivalent forms:

    Form A) All distinct 3D binary vector partitions of kind
\begin{equation}  \label{12.06}
\langle a,b,c \rangle = \sum_{a_i,a_j,a_k \geq 0} \langle 2^{a_i},2^{a_j},2^{a_k} \rangle,
\end{equation}
    \; \; Form B) All unrestricted 3D binary vector partitions of kind
\begin{equation}  \label{12.07}
\langle a,b,c \rangle = \sum_{a_i,a_j,a_k \geq 0} \left(\langle 1,2^{a_j},2^{a_k} \rangle + \langle 2^{a_i},1,2^{a_k} \rangle + \langle 2^{a_i},2^{a_j},1 \rangle \right).
\end{equation}
 Another way of putting this is

 "The number of distinct 3D binary vector partitions of $\langle a,b,c \rangle$ inside the infinite extended pyramid $0 \leq x \leq y, \; 1 \leq z \leq y, \; y \geq 1$ equals the number of unrestricted 3D binary vector partitions of $\langle a,b,c \rangle$ with at least one component $a$, $b$, or $c$ equal to 1, in the same pyramid."
\end{theorem}

Because this idea of vector partitions and their generating functions is greatly assisted by visualizations, the above theorem is depicted next yet again, but with the added help of a pyramidal structure beside the equations.


  \begin{figure} [ht!]
\centering
    \includegraphics[width=14.50cm,angle=0,height=12.50cm]{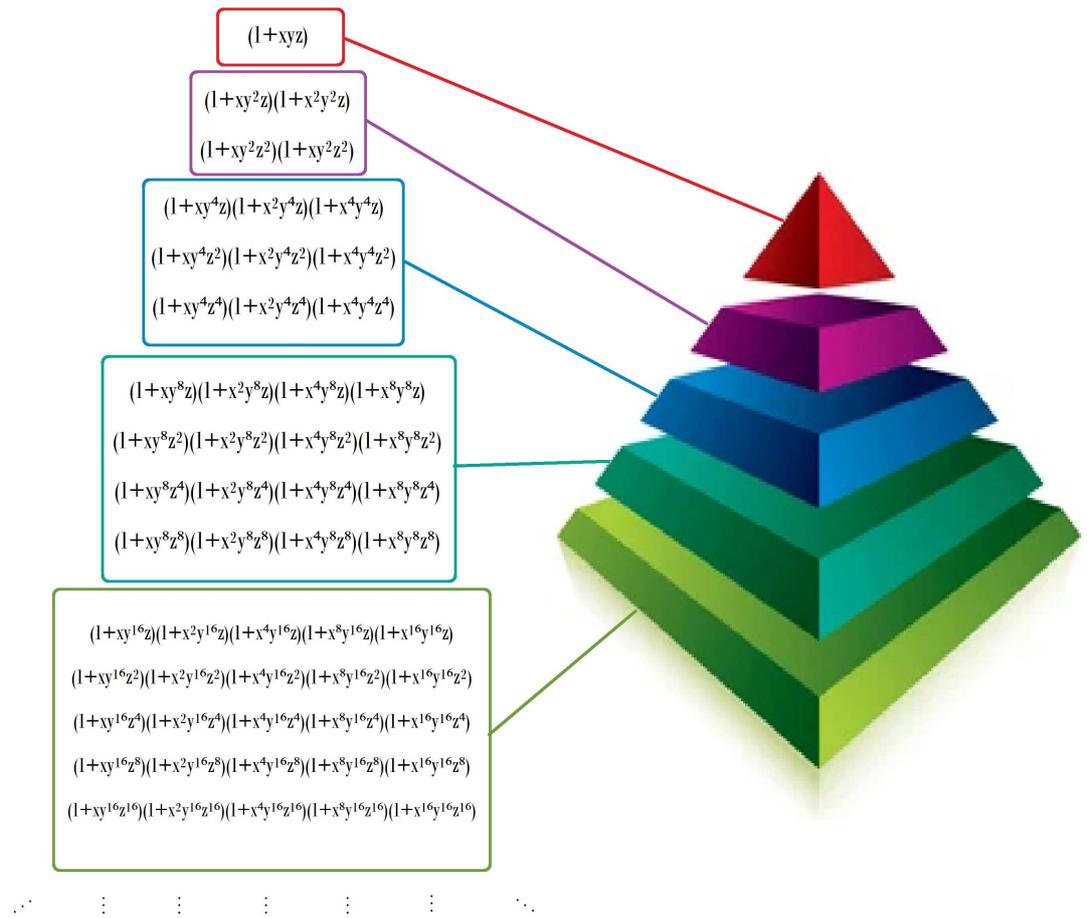}
  \caption{3D Pyramid Binary Product Form.}\label{Fig18}
\end{figure}

\begin{equation}\label{12.08}
    = \frac{1}{(1-xyz)}
 \times \frac{1}{\begin{Bmatrix}
    (1-xy^2z) & (1-x^2y^2z) \\
    (1-xy^2z^2) &   \\
  \end{Bmatrix}}
\times \frac{1}{\begin{Bmatrix}
    (1-xy^4z) & (1-x^2y^4z) & (1-x^4y^4z) \\
    (1-xy^4z^2) &   &   \\
    (1-xy^4z^4) &   &   \\
  \end{Bmatrix}}
  \end{equation}
  \begin{equation}  \nonumber
 \times \frac{1}{\begin{Bmatrix}
   (1-xy^8z) & (1-x^2y^8z) & (1-x^4y^8z) & (1-x^4y^8z) \\
   (1-xy^8z^2) &   &   &   \\
   (1-xy^8z^4) &   &   &   \\
   (1-xy^8z^8) &   &   &   \\
  \end{Bmatrix}}
\times etc.
\end{equation}

\bigskip

\section{Concluding remarks}

This paper summarizes several chapters of a book on vector partitions currently being written. The present paper excludes many areas related to the text herein. For example, there is further work on the Rogers-Ramanujan identities links to physics in statistical mechanics; to asymptotics of partitions, plane partitions and functional equations as well as new work associated with congruences and continued fractions. It is hoped the present paper will give the reader an appetite to follow through developments of the topics given already, as it is clear that the research into higher space lattice systems is only starting out.


\bigskip


\begin{thebibliography}{99}
\bibitem{mA1972}
ABRAMOWITZ, M., and STEGUN, I. Handbook of Mathematical Functions,
Dover Publications Inc., New York, 1972.
\bibitem{gA1976}
ANDREWS, G.E.  The Theory of Partitions, Addison-Wesley Publishing Company,
Advanced Book Program, Reading, Massachusetts, 1976.
\bibitem{gA2004}
ANDREWS, G.E. and ERIKSSON, K. Integer Partitions, Cambridge University Press,
Cambridge, UK, New York, USA, Port Melbourne, Australia, Madrid, Spain, Cape Town, South Africa, 2004.
\bibitem{tA1976}
APOSTOL, T.  Introduction to Analytic Number Theory,
Springer-Verlag, New York, 1976.
\bibitem{rB1982}
BAXTER, R. J. Exactly Solved Models in Statistical Mechanics, Academic Press,
New York, 1982.
\bibitem{gB1977}
BIRKHOFF, G. and MACLAINE, S. A survey of modern algebra, fourth ed., N.Y.,
Macmillan, 1977.
\bibitem{dB2006}
BORWEIN, D.; BORWEIN, J.M.; BRADLEY, D.M. Parametric Euler sum identities. J. Math. Anal. Appl. 2006, 316, 328–338.
\bibitem{jB2013}
BORWEIN, J. M.; GLASSER, M. L.; MCPHEDRAN, R. C.; WAN, J. G.; and ZUCKER, I. J. (2013) "Lattice Sums then and now". Encyclopedia of Mathematics and Its Applications 150: Cambridge University Press.
\bibitem{gC1979}
CAMPBELL, G. B. \textit{Generalization of a Formula of Hardy}, Pure Math. Research Paper 79-5, La Trobe University, Melbourne, Australia, 1979.
\bibitem{gC1992}
CAMPBELL, G. B. \textit{Multiplicative functions over Riemann zeta function products}, J. Ramanujan Soc. 7 No. 1, 1992, 52-63.
\bibitem{gC1993}
CAMPBELL, G. B. \textit{Dirichlet summations and products over primes}, Int. J. Math. Math. Sci., Vol 16, No 2, (1993) 359-372.
\bibitem{gC1994}
CAMPBELL, G. B. \textit{A generalized formula of Hardy}, Int. J. Math. Math. Sci., Vol 17, No 2, (1994) 369-378.
\bibitem{gC1994a}
CAMPBELL, G. B. \textit{A new class of infinite products, and Euler's totient}, International Journal of Mathematics and Mathematical Sciences, vol. 17, no. 3, pp. 417-422, 1994. https://doi.org/10.1155/S0161171294000591.
\bibitem{gC1994b}
CAMPBELL, G. B. \textit{Infinite products over visible lattice points}, International Journal of Mathematics and Mathematical Sciences, vol. 17, no. 4, pp. 637-654, 1994. https://doi.org/10.1155/S0161171294000918.
\bibitem{gC1997}
CAMPBELL, G. B. Combinatorial identities in number theory related to \textit{q}-series and arithmetical functions, Doctor of Philosophy Thesis, School of Mathematical Sciences, The Australian National University, October 1997.
\bibitem{gC1998}
CAMPBELL, G. B. \textit{A closer look at some new identities}, International Journal of Mathematics and Mathematical Sciences, vol. 21, no. 3, pp. 581-586, 1998. https://doi.org/10.1155/S0161171298000805.
\bibitem{gC2000}
CAMPBELL, G. B. \textit{Infinite products over hyperpyramid lattices}, International Journal of Mathematics and Mathematical Sciences, vol. 23, no. 4, pp. 271-277, 2000. https://doi.org/10.1155/S0161171200000764.
\bibitem{gC2019}
CAMPBELL, G. B. \textit{Some n-space q-binomial theorem extensions and similar identities}, arXiv:1906.07526v1 [math.NT], Jun 2019. (https://arxiv.org/abs/1906.07526)
\bibitem{gC2020}
CAMPBELL, G. B. \textit{An interview with Rodney James Baxter}, Aust. Math. Soc. Gazette, Volume 47, No1, pp24-32, March 2020. (https://austms.org.au/wp-content/uploads/2020/07/471Web.pdf)
\bibitem{gC2022}
CAMPBELL, G. B. \textit{Fun with numbers: Rational solutions to $x^y y^x = v^w w^v$},  Aust. Math. Soc. Gazette, Volume 49, No5, pp210-211, November 2022. (https://austms.org.au/publications/gazette/gazette495/)
\bibitem{aC1893}
CAUCHY, A. M\'{e}moire sur les fonctions dont plusieurs \ldots , C. R. Acad. Sci. Paris,
T. XVII, p. 523, Oeuvres de Cauchy, 1re s\'{e}rie, T. VIII, Gauthier-Villars, Paris, 1893, 42-
50.
\bibitem{mC1966}
CHEEMA, M. S., \textit{Vector partitions and combinatorial identities}, Math. Comp. 18,
1966 414-420.
\bibitem{mC1971}
CHEEMA, M. S. and MOTZKIN, T. S., \textit{Multipartitions and multipermutations}, Proc.
Symp. Pure Math. 19, 1971, 37-39.
\bibitem{gG1990}
GASPER, G.  and  RAHMAN, M.  Basic Hypergeometric Series,
Encyclopedia of Mathematics and its Applications, Vol 35,
Cambridge University Press, (Cambridge - New York - Port Chester -
Melbourne - Sydney), 1990.
\bibitem{cG1813}
GAUSS, C.F.  Disquisitiones generales circa seriem infinitam
\ldots, Comm. soc. reg. sci. G\"{o}tt. rec., Vol II; reprinted in
Werke 3 (1876), pp. 123--162.
\bibitem{dG1996}
GOLDFELD, D. \textit{Beyond the last theorem.} Math Horizons. 4 (September): 26–34. (1996). doi:10.1080/10724117.1996.11974985. JSTOR 25678079.
\bibitem{bG1963}
GORDON, B. \textit{Two theorems on multipartite partitions}, J. London Math. Soc. 38,
1963, 459-464.
\bibitem{gH1974a}
HARDY, G. H. \textit{An extension of a theorem on oscillating series}, Collected Papers, Vol
VI, Clarendon Press, Oxford, 1974, 500-506.
\bibitem{gH1974b}
HARDY, G. H. \textit{On certain oscillating series}, Collected Papers, Vol VI, Clarendon
Press, Oxford, 1974, 146-167.
\bibitem{gH1974c}
HARDY, G. H., and LITTLEWOOD, J. E. \textit{A further note on the converse of Abel's
theorem.} Collected Papers of Hardy, Vol VI, Clarendon Press, Oxford, 1974, 699-716.
\bibitem{eH1847}
HEINE, E.  \textit{Untersuchungen uber die Reihe ... }, J. Reine angew.
Math. 34, 1847, 285-328.
\bibitem{eH1878}
HEINE, E.  Handbuch der Kugelfunctionen, Theorie und Andwendungen,
Vol. 1, Reimer, Berlin, 1878.
\bibitem{iM1995}
MACDONALD, I. G. Symmetric Functions And Hall Polynomials, 2nd ed., Oxford :
Clarendon Press ; New York : Oxford University Press, 1995.
\bibitem{dM1985}
MASSER, D. W. (1985). "Open problems". In Chen, W. W. L. (ed.). Proceedings of the Symposium on Analytic Number Theory. London: Imperial College.
\bibitem{jO1988}
OESTERL\'{E}, J. \textit{Nouvelles approches du "théorème" de Fermat}, Astérisque, Séminaire Bourbaki exp 694 (161): 165–186, (1988), ISSN 0303-1179, MR 0992208.
\bibitem{gR1859}
RIEMANN, G. F. B. "\"{U}ber die Anzahl der Primzahlen unter einer gegebenen
Gr\"{o}sse." Monatsber. Königl. Preuss. Akad. Wiss. Berlin, 671-680, Nov. 1859.
\bibitem{oeiseuler}
SLOANE, N. J. A., The On-Line Encyclopedia of Integer Sequences (OEIS) Euler transform. ${\rm https://oeis.org/wiki/Euler\_transform}$.
\bibitem{nS2015}
SLOANE, N. J. A., The On-Line Encyclopedia of Integer Sequences (OEIS) sequence A061159 Numerators in expansion of Euler transform of b(n)=1/2 https://oeis.org/A061159.
\bibitem{nS2015a}
SLOANE, N. J. A., The On-Line Encyclopedia of Integer Sequences (OEIS) sequence A061160 Numerators in expansion of Euler transform of b(n)=1/3 https://oeis.org/A061160.
\bibitem{lS1981}
SZPIRO, L. (1981). "Propriétés numériques du faisceau dualisant rélatif". Seminaire sur les pinceaux des courbes de genre au moins deux (PDF). Astérisque. Vol. 86. pp. 44–78. Zbl 0517.14006.
\bibitem{lS1987}
SZPIRO, L. (1987), "Présentation de la théorie d'Arakelov", Contemp. Math., Contemporary Mathematics, 67: 279–293, doi:10.1090/conm/067/902599, ISBN 9780821850749, Zbl 0634.14012
\bibitem{eW1956}
WRIGHT, E. M. \textit{Partitions of multipartite numbers}, Proc. Amer. Math. Soc. 28, 1956,
880-890.
\end{thebibliography}
\end{document}